\documentclass[a4paper,12pt]{article}
\usepackage{amsmath}
\usepackage{amssymb}
\usepackage{amsthm}
\usepackage{latexsym}
\usepackage{amsmath}
\usepackage{amsthm}
\usepackage{graphicx}
\usepackage{txfonts} 
\usepackage{bm}
\usepackage{epic,eepic}
\usepackage{braket}
\usepackage{color}
\usepackage[author-year]{amsrefs} 

\usepackage{geometry}
\newtheorem{theorem}{Theorem}[section]
\newtheorem{proposition}[theorem]{Proposition}

\newtheorem{remark}{Remark}[section]
\newtheorem{example}{Example}[section]

\numberwithin{equation}{section}

\newcommand{\todaye}{\the\year/\the\month/\the\day}

\newcommand{\finbox}{\hspace*{\fill}$\rule{0.15cm}{0.2cm}$}

\newcommand{\RED}[1]{{\color{red}#1}} 
 \renewcommand{\RED}[1]{{#1}}

\newcommand{\RR}{\mathbb{R}}
\newcommand{\ZZ}{\mathbb{Z}}

\newcommand{\calI}{\mathcal{I}}

\newcommand{\calT}{\mathcal{T}}
\newcommand{\vecone}{{\bf 1}}
\newcommand{\veczero}{{\bf 0}}
\newcommand{\sign}{{\rm sign\,}}

\newcommand{\dom}{{\rm dom\,}}
\newcommand{\domR}{{\rm dom_{\RR}}}
\newcommand{\domZ}{{\rm dom_{\ZZ}}}

\newcommand{\suppp}{{\rm supp}\sp{+}}
\newcommand{\suppm}{{\rm supp}\sp{-}}
\newcommand{\unitvec}[1]{\chi_{#1}}
\newcommand{\argmax}{\arg \max}
\newcommand{\argmaxZ}{\arg \max_{\ZZ}}

\newcommand{\argmin}{\arg \min}

\newcommand{\conv}{\Box\,}

\newcommand{\OMIT}[1]{{\bf [OMIT:} #1 \ {\bf --- end OMIT] }}  
   \renewcommand{\OMIT}[1]{}            
\newcommand{\citeH}[2][]{\cite[#1]{#2}} 
  \renewcommand{\citeH}[2][]{}   

\title{Discrete Convex Analysis:
\\
A Tool for Economics and Game Theory%
\footnote{
\RED{
This is a revised version of the paper with the same title
published in
\textit{Journal of Mechanism and Institution Design} 
{\bf 1} (2016), 151--273.
The revision consists of corrections in 
definitions \eqref{viapkarcdef} and \eqref{gamdef2k}
as well as updates of bibliographic information.
}
}
}

\author{Kazuo Murota\thanks{
The Institute of Statistical Mathematics, and 
Tokyo Metropolitan University,
e-mail: murota@tmu.ac.jp.}}

\date{October 2016 / December 2022}
\begin{document}

\maketitle

\begin{abstract}
This paper presents discrete convex analysis
as a tool for economics and game theory.
Discrete convex analysis is a new framework of discrete mathematics and optimization,
developed during the last two decades.
Recently, it is being recognized 
as a powerful tool for analyzing economic or game models with indivisibilities.
The main feature of discrete convex analysis
is the distinction of two convexity concepts,
M-convexity and L-convexity, for functions in integer or binary variables,
together with their conjugacy relationship.
The crucial fact is that M-concavity, or its variant called M$\sp{\natural}$-concavity, 
is equivalent to the (gross) substitutes property in economics.
Fundamental theorems in discrete convex analysis such as the M-L conjugacy theorems,
discrete separation theorems and discrete fixed point theorems
yield structural results in economics such as the existence of equilibria
and the lattice structure of equilibrium price vectors.
Algorithms in discrete convex analysis give iterative auction algorithms
as well as computational methods for equilibria.
\end{abstract}



\section{Introduction}
\label{SCintro}

Convex analysis and fixed point theorems have played a crucial role 
in economic and game-theoretic analysis, for instance, 
to prove the existence of competitive equilibrium and Nash equilibrium; 
see Debreu (1959)\citeH{Deb59}, 
Arrow and Hahn (1971)\citeH{AH71}, 
and Fudenberg and Tirole (1991)\citeH{FT91}.
Traditionally, in such studies,
it is assumed that commodities are perfectly divisible,
or mixed strategies can be used,
or the space of strategies is continuous. 
However, this traditional approach cannot be equally applied to economic models
which involve significant indivisibilities or to game-theoretic models
 where the space of strategies is discrete and mixed strategies do not make much sense.
In this paper we will present a new approach
based on discrete convex analysis and discrete fixed point theorems, 
which have been recently developed in the field of discrete mathematics
and optimization and
become a powerful tool for analyzing economic or game models with indivisibilities.

Discrete convex analysis
(Murota 1998, 2003)\citeH{Mdca,Mdcasiam} 
is a general theoretical framework 
constructed through a combination of convex analysis
and combinatorial mathematics.
The framework of convex analysis 
is adapted to discrete settings and 
the mathematical results in matroid/submodular function theory are generalized%
\footnote{
The readers who are interested in general backgrounds are referred to
Rockafellar (1970)\citeH{Roc70}
for convex analysis,
Schrijver (1986)\citeH{Sch86}
for linear and integer programming,
Korte and Vygen (2012)\citeH{KV12}
and Schrijver (2003)\citeH{Sch03} 
for combinatorial optimization,
Oxley (2011)\citeH{Oxl11}
for matroid theory, and 
Fujishige (2005)\citeH{Fuj05}
and Topkis (1998)\citeH{Top98}
for submodular function theory.
}. 
The theory extends the direction set forth in discrete optimization around 1980
by Edmonds (1970)\citeH{Edm70}, 
Frank (1982)\citeH{Fra82}, 
Fujishige (1984)\citeH{Fuj84}, 
and Lov{\'a}sz (1983)\citeH{Lov83};
see also Fujishige (2005)\citeH[Chapter VII]{Fuj05}.
The main feature of discrete convex analysis
is the distinction of two convexity concepts
for functions in integer or binary variables,
M-convexity and L-convexity%
\footnote{
``M'' stands for ``Matroid'' and ``L''  for ``Lattice.'' 
}, 
together with their conjugacy relationship 
with respect to the (continuous or discrete) Legendre--Fenchel transformation.
Roughly speaking, M-convexity is defined in terms of an exchange property
and L-convexity by submodularity.

The interaction between discrete convex analysis and mathematical economics
was initiated by 
Danilov, Koshevoy, and Murota (1998, 2001)\citeH{DKM98,DKM01} 
for the  Walrasian equilibrium of indivisible markets
(see also Chapter 11 of Murota 2003)\citeH[Chapter 11]{Mdcasiam}.
The next stage of the interaction was brought about 
by the crucial observation of 
Fujishige and Yang (2003)\citeH{FY03gr}
that M-concavity, or its variant called M$\sp{\natural}$-concavity%
\footnote{
``M$\sp{\natural}$'' and ``L$\sp{\natural}$'' are read 
``em natural'' and ``ell natural,'' respectively.
}, 
is equivalent to the gross substitutability (GS) of 
Kelso and Crawford (1982)\citeH{KC82}.
The survey papers by
Murota and Tamura (2003b)\citeH{MTcompeq03}
and Tamura (2004)\citeH{Tam04}
describe the interaction at the earlier stages.

Concepts, theorems, and algorithms in discrete convex analysis
have turned out to be useful in modeling and analysis of economic problems.
The M-L conjugacy corresponds to the conjugacy
between commodity bundles and price vectors in economics.
The conjugacy theorem in discrete convex analysis implies,
for example, that a valuation (utility) function has substitutes property 
(M$\sp{\natural}$-concavity)
if and only if
the indirect utility function is an L$\sp{\natural}$-convex function,
where L$\sp{\natural}$-convexity is a variant of L-convexity.

One of the most successful examples of the discrete convex analysis approach
is Fujishige and Tamura's model 
(Fujishige and Tamura 2006, 2007)\citeH{FT06market,FT07market}
of two-sided matching, which unifies the stable matching of 
Gale and Shapley (1962)\citeH{GS62}
and the assignment model of 
Shapley and Shubik (1972)\citeH{SS72}.
The existence of a market equilibrium is established 
by revealing a novel duality-related property of M$\sp{\natural}$-concave functions.
Tamura's monograph (Tamura 2009)\citeH{Tam09book},
though in Japanese, gives a comprehensive account of this model.

Another significant instance of the discrete convex analysis approach 
is the design and analysis of auction algorithms.
Based on the Lyapunov function approach of 
Ausubel (2006)\citeH{Aus06},
Murota, Shioura, and Yang (2013a, 2016)\citeH{MSY13,MSY16auction}
shed a new light on a variety of 
iterative auctions by making full use of
the M-L conjugacy theorem
and L$\sp{\natural}$-convex function minimization algorithms.
The lattice structure of equilibrium price vectors
is obtained as an immediate consequence of the L$\sp{\natural}$-convexity
of the Lyapunov function.

The contents of this paper are as follows:

\begin{quote}
\noindent
Section \ref{SCintro}: {Introduction}
\\
Section \ref{SCnotation}: {Notation}
\\
Section \ref{SCmnatconcav01}: 
{M$\sp{\natural}$-concave set function}
\\
Section \ref{SCmnatconcavZ}: 
{M$\sp{\natural}$-concave function on $\ZZ\sp{n}$}
\\
Section \ref{SCmnatconcavR}: 
{M$\sp{\natural}$-concave function on $\RR\sp{n}$}
\\
Section \ref{SCoperM}: 
{Operations for M$\sp{\natural}$-concave functions}
\\
Section \ref{SCconjuLconv}: 
{Conjugacy and L$\sp{\natural}$-convexity}
\\
Section \ref{SCauction}: 
{Iterative auctions}
\\
Section \ref{SCintersepar}: 
{Intersection and separation theorems}
\\
Section \ref{SCmarrigeassigngame}: 
{Stable marriage and assignment game}
\\
Section \ref{SCviap}: 
{Valuated assignment problem}
\\
Section \ref{SCsubmflow}:
{Submodular flow problem}
\\
Section \ref{SCfixedpoint}:
{Discrete fixed point theorem}
\\
Section \ref{SCothertopic}:
{Other topics}
\end{quote}

Following the introduction of notations in Section \ref{SCnotation},
Sections \ref{SCmnatconcav01} to \ref{SCmnatconcavR}
present the definition of M$\sp{\natural}$-concave functions
and the characterizations of (or equivalent conditions for)
M$\sp{\natural}$-concavity in terms of demand functions and choice functions.
Section \ref{SCoperM} shows the operations
valid for M$\sp{\natural}$-concave functions,
including convolution operation used for the aggregation of utility functions.
Section \ref{SCconjuLconv} introduces L$\sp{\natural}$-convexity
as the conjugate concept of M$\sp{\natural}$-concavity,
and Section \ref{SCauction} presents the application to
iterative auctions. 
Section \ref{SCintersepar} deals with duality theorems
of fundamental importance,
including the discrete separation theorems and 
the Fenchel-type minimax relations.
Section \ref{SCmarrigeassigngame} is a succinct description of
Fujishige and Tamura's model.
Combinations of M$\sp{\natural}$-concave functions
with graph/network structures are considered in 
Sections \ref{SCviap} and \ref{SCsubmflow}. 
Section \ref{SCfixedpoint} explains the basic idea 
underlying the discrete fixed point theorems.
Finally in Section \ref{SCothertopic},
some topics not covered in the main body of the paper
are touched upon briefly.

\bigskip

Beside economics and game theory, discrete convex analysis 
has found applications in many different areas,
including systems analysis 
(Murota 2000)\citeH{Mspr2000} in engineering,
and 
resource allocation
(Katoh et al.~2013)\citeH{KSI13}
and inventory theory 
 (Simchi-Levi et al.~2014)\citeH{SCB14}
in operations research.
The survey paper (Murota 2009)\citeH{Mbonn09}
describes other  applications including
those to finite metric spaces and eigenvalues of Hermitian matrices.



\section{Notation}
\label{SCnotation}

Basic notations are listed here.

\begin{itemize}
\item
The set of all real numbers is denoted by $\RR$, and
the sets of nonnegative reals and positive reals
are denoted, respectively, by $\RR_{+}$ and $\RR_{++}$.
The set of all integers is denoted by $\ZZ$, and
the sets of nonnegative integers and positive integers
are denoted, respectively, by $\ZZ_{+}$ and $\ZZ_{++}$.

\item 
We consistently assume  $N = \{ 1,2, \ldots, n \}$
for a positive integer $n$.
Then  
$2\sp{N}$ denotes the set of all subsets of $N$,
i.e., the power set of $N$.

\item
The characteristic vector of a subset
$A \subseteq N= \{ 1,2, \ldots, n \}$
is denoted by
$\chi_{A} \in \{ 0, 1 \}\sp{n}$.  
That is,
\begin{equation} \label{charvecdefnotat}
 (\chi_{A})_{i} =
   \left\{  \begin{array}{ll}
     1      & (i \in A) , \\
     0      & (i \in N \setminus A). \\
                      \end{array}  \right.
\end{equation}
For $i \in \{ 1,2, \ldots, n \}$,
we write $\chi_{i}$ for $\chi_{ \{ i \} }$, 
which is the $i$th unit vector.
We define $\chi_{0}=\veczero$
where
$\veczero =(0,0,\ldots,0)$. 
We also define $\vecone=(1,1,\ldots,1)$.

\item
For a vector $x=(x_{1},x_{2}, \ldots, x_{n})$
 and a subset $A \subseteq \{ 1,2, \ldots, n \}$,
$x(A)$ denotes the component sum within $A$,
i.e.,
$x(A) = \sum_{i \in A} x_{i}$.

\item
For two vectors 
$x=(x_{1},x_{2}, \ldots, x_{n})$
and 
$y=(y_{1},y_{2}, \ldots, y_{n})$,
$x \leq y$ means the componentwise inequality.
That is, $x \leq y$ is true if and only if
$x_{i} \leq y_{i}$ is true for all $i=1,2,\ldots,n$.

\item
For two integer vectors $a$ and $b$ in $\ZZ\sp{n}$
with $a \leq b$,
$[a,b]_{\ZZ}$ denotes the integer interval between $a$ and $b$ (inclusive),
i.e.,
$[a,b]_{\ZZ} = \{ x \in \ZZ\sp{n} \mid a \leq x \leq b \}$.

\item 
For two vectors $x$ and $y$,
$x \vee y$ and $x \wedge y$ denote
the vectors of componentwise maximum and minimum.
That is,
$(x \vee y)_{i}   = \max(x_{i}, y_{i})$ and $(x \wedge y)_{i} = \min(x_{i}, y_{i})$
for $i=1,\ldots,n$.

\item
For a real number $z \in \RR$, 
$\left\lceil  z   \right\rceil$ 
denotes the smallest integer not smaller than $z$
(rounding-up to the nearest integer)
and $\left\lfloor  z  \right\rfloor$
the largest integer not larger than $z$
(rounding-down to the nearest integer).
This operation is extended to a vector
by componentwise application.

\item
For a vector $x$, 
$\suppp(x) = \{ i \mid x_{i} > 0 \}$ and 
$\suppm(x) = \{ i \mid x_{i} < 0 \}$
denote the positive and negative supports of $x$,
respectively.

\item 
The $\ell_{\infty}$-norm of a vector $x$
is denoted as $\| x \|_{\infty}$, i.e.,
$\| x \|_{\infty} = \max( |x_{1}|, |x_{2}|, \ldots, |x_{n}| )$.
Variants are:
$\| x \|_\infty\sp{+} = \max(0, x_{1}, x_{2}, \ldots, x_{n} )$
and 
$\| x \|_\infty\sp{-} = \max(0, -x_{1}, -x_{2}, \ldots, -x_{n} )$.

\item
For two vectors 
$p=(p_{1},p_{2}, \ldots, p_{n})$
and 
$x=(x_{1},x_{2}, \ldots, x_{n})$,
their inner product is denoted by 
$ \langle p, x \rangle$, 
i.e.,
$ \langle p, x \rangle = p\sp{\top} x = \sum_{i=1}\sp{n} p_{i} x_{i}$,
where $p\sp{\top}$ is the transpose of $p$ viewed as a column vector.

\item
For a function $f: \mathbb{R}\sp{n} \to \RR \cup \{ +\infty \}$
or $f: \mathbb{R}\sp{n} \to \RR \cup \{  -\infty \}$,
\begin{eqnarray*} 
 \dom f &=& \{ x  \mid -\infty < f(x) < +\infty  \} ,
\\
 \argmin f &=& \{ x  \mid f(x) \leq f(y) \  \ \mbox{for all $y$} \, \},
\\
 \argmax f &=& \{ x  \mid f(x) \geq f(y) \  \ \mbox{for all $y$} \, \}. 
\end{eqnarray*}
These notations are used also for
$f: \mathbb{Z}\sp{n} \to \RR \cup \{ +\infty \}$
or $f: \mathbb{Z}\sp{n} \to \RR \cup \{  -\infty \}$.
We sometimes use $\domR f $ and $\domZ f$ to emphasize that 
$\dom f \subseteq \mathbb{R}\sp{n}$ and $\dom f \subseteq \mathbb{Z}\sp{n}$. 

\item
For a set function $f: 2\sp{N} \to \RR \cup \{ +\infty \}$ or
$f: 2\sp{N} \to \RR \cup \{ -\infty \}$,
\begin{eqnarray*} 
 \dom f &=& \{ X \subseteq N  \mid -\infty < f(X) < +\infty  \} ,
\\
 \argmin f &=& \{ X \subseteq N  \mid f(X) \leq f(Y) \  \ \mbox{for all $Y \subseteq N$} \, \},
\\
 \argmax f &=& \{ X \subseteq N  \mid f(X) \geq f(Y) \  \ \mbox{for all $Y \subseteq N$} \, \}. 
\end{eqnarray*}

\item
For a function $f$ and a vector $p$,
$f[-p]$  means the function defined by 
\begin{align*} \label{f-pdefG}
f[-p](x) &= f(x) - p\sp{\top} x = f(x) - \langle p, x \rangle .
\end{align*}
If $f$ is a set function,
$f[-p]$ is the set function defined by 
$f[-p](X) = f(X) - p(X)$.

\item  
For a function $f$, four variants of the conjugate function of $f$ are denoted as
\begin{align*} 
 f\sp{\bullet}(p) 
 &= \sup\{  \langle p, x \rangle - f(x)  \},
\qquad
 f\sp{\circ}(p) 
 = \inf\{  \langle p, x \rangle - f(x)  \},
\\
 f\sp{\triangledown}(p) 
 &= \sup\{  f(x) - \langle p, x \rangle  \},
\qquad
 f\sp{\triangle}(p) 
 = \inf\{  f(x) + \langle p, x \rangle  \}.
\end{align*}

\item 
The convex closure of a function $f$ is denoted by $\overline{f}$.
The convex hull of a set $S$ is denoted by $\overline{S}$.

\item
$D(p ; f)$ denotes the demand correspondence
for a price vector $p$ and a valuation function $f$,
defined in (\ref{Dpdef}) and (\ref{DpdefZ}).

\item
$C( \cdot )$ denotes a choice function.
$C( \cdot \, ; f)$ denotes the choice function
determined by a valuation function $f$,
defined in (\ref{choiceByVal01}) and (\ref{choiceByValZ}).

\item
$\mathrm{tw}( \cdot )$ denotes the twisting of a set 
or a vector, defined in (\ref{twistsetdef01}) and 
(\ref{twistsetdefZ}), respectively.

\item
For an arc $a$ in a directed graph,
$\partial\sp{+} a$ denotes the initial (tail) vertex of $a$, and
$\partial\sp{-}a$ the terminal (head) vertex of $a$.
That is, $\partial\sp{+} a = u $ and 
$\partial\sp{-} a = v$ if $a =(u,v)$.

\item
For a flow $\xi$ in a network,
$\partial \xi$ denotes the boundary vector on the vertex set,
defined  in (\ref{flowbounddefZ}).
For a matching $M$,
$\partial M$ denotes the set of the vertices incident to some edge in $M$.

\item
For a potential $p$ defined on the vertex set of a network,
$\delta p$ denotes the coboundary of $p$, 
the vector on the arc set defined in (\ref{cobounddef5}).
\end{itemize}




\section{M$\sp{\natural}$-concave Set Function}
\label{SCmnatconcav01}

First we introduce M$\sp{\natural}$-concavity
for set functions.
Let $N$ be a finite set, say,
$N = \{ 1,2,\ldots, n \}$,
$\mathcal{F}$ be a nonempty family of subsets of $N$,
and 
$f: \mathcal{F} \to \RR$
be a real-valued function on $\mathcal{F}$.
In economic applications, we may think of $f$ as
a single-unit valuation (binary valuation)
over combinations of indivisible commodities $N$,
where $\mathcal{F}$ represents the set of feasible combinations.

\subsection{Exchange property}
\label{SCexchange01}

Let $\mathcal{F}$ be a nonempty family of subsets of 
a finite set $N = \{ 1,2,\ldots, n \}$. 
We say that a function 
$f: \mathcal{F} \to \RR$
is {\em M$\sp{\natural}$-concave}, if,
for any $X, Y \in \mathcal{F}$ and $i \in X \setminus Y$,
we have (i)
$X - i \in \mathcal{F}$, $ Y + i \in \mathcal{F}$ and
\begin{equation}  \label{mconcav1}
f( X) + f( Y ) \leq f( X - i ) + f( Y + i ),
\end{equation}
or (ii) there exists some $j \in Y \setminus X$ such that
$X - i +j \in \mathcal{F}$, $ Y + i -j  \in \mathcal{F}$ and
\begin{equation}  \label{mconcav2}
f( X) + f( Y ) \leq 
 f( X - i + j) + f( Y + i -j).
\end{equation}
Here we use short-hand notations
$X - i = X \setminus  \{ i \}$,
$Y + i = Y \cup \{ i \}$,
$X - i + j =(X \setminus  \{ i \}) \cup \{ j \}$,
and
$Y + i - j =(Y \cup \{ i \}) \setminus \{ j \}$.
This property
is referred to as the {\em exchange property}.

A more compact way of defining M$\sp{\natural}$-concavity, 
free from explicit reference to the domain $\mathcal{F}$,
is to define 
a function
$f: 2\sp{N} \to \RR \cup \{ -\infty \}$ 
to be M$\sp{\natural}$-concave
 if it has the following property: 
\begin{description}
\item[(M$\sp{\natural}$-EXC)] 
For any $X, Y \subseteq N$ and $i \in X \setminus Y$, we have
\begin{align}
f( X) + f( Y )   &\leq 
   \max\left( f( X - i ) + f( Y + i ), \ 
 \max_{j \in Y \setminus X}  \{ f( X - i + j) + f( Y + i -j) \}
       \right) ,
\label{mnatconcavexc2}
\end{align}
\end{description}
where
$(-\infty) + a = a + (-\infty) = (-\infty) + (-\infty)  = -\infty$ for $a \in \RR$,
$-\infty \leq -\infty$, and
a maximum taken over an empty set
is defined to be $-\infty$.
The family of subsets $X$ for which $f(X)$ is finite
is called the {\em effective domain} of $f$, and denoted as 
$\dom f$, i.e.,
$\dom f = \{ X \mid f(X) > -\infty \}$.
When $f$
is regarded as a function on  $\mathcal{F} = \dom f$,
it is an M$\sp{\natural}$-concave function in the original sense.

As a (seemingly) stronger condition than (M$\sp{\natural}$-EXC)
we may also conceive the
{\em multiple exchange property}:
\begin{description}
\item[(M$\sp{\natural}$-EXC$_{\rm\bf m}$)]
For any $X, Y \subseteq N$ and $I \subseteq X \setminus Y$,
there exists $J \subseteq Y \setminus X$ such that
$ f( X) + f( Y )   \leq  f((X \setminus I) \cup J) +f((Y \setminus J) \cup I) $, 
i.e., 
\begin{align}
f( X) + f( Y )   \leq 
 \max_{J \subseteq Y \setminus X} 
 \{  f((X \setminus I) \cup J) +f((Y \setminus J) \cup I)  \}.
\label{mnatconcavexcmult}
\end{align}
\end{description}

Recently it has been shown 
\RED{
(Murota 2018)
}
 \citeH{Mmultexc16}%
that {\rm (M$\sp{\natural}$-EXC$_{\rm m}$)} is equivalent to
{\rm (M$\sp{\natural}$-EXC)}. 

\begin{theorem} \label{THmultexchmnat}
A function
$f: 2\sp{N} \to \RR \cup \{ -\infty \}$ 
satisfies {\rm (M$\sp{\natural}$-EXC)}
if and only if it satisfies
{\rm (M$\sp{\natural}$-EXC$_{\rm m}$)}.
Hence, every M$\sp{\natural}$-concave function
has the multiple exchange property 
{\rm (M$\sp{\natural}$-EXC$_{\rm m}$)}.
\end{theorem}

\begin{remark} \rm  \label{RMmnat=SNC=GS}
The multiple exchange property (M$\sp{\natural}$-EXC$_{\rm m}$) here
is the same as the ``strong no complementarities property (SNC)''
introduced by
Gul and Stacchetti (1999),
\citeH{GS99}%
where it is shown that (SNC)
implies the gross substitutes property (GS).
On the other hand, (GS) is known 
(Fujishige and Yang 2003)
\citeH{FY03gr}%
to be equivalent to 
(M$\sp{\natural}$-EXC) 
(see Theorem~\ref{THmconcavgross}).
Therefore,  
Theorem~\ref{THmultexchmnat} above 
reveals that (SNC) is equivalent to (GS).
This settles the question since 1999:
Is (SNC)  strictly stronger than (GS) or not?
We now know that (SNC) is equivalent to (GS).
See 
\RED{
Murota (2018) 
\citeH{Mmultexc16}%
}%
for details.
\finbox
\end{remark}

It follows from the definition of an M$\sp{\natural}$-concave function
that the (effective) domain $\mathcal{F}$
of an M$\sp{\natural}$-concave function
has the following exchange property:
\begin{description}
\item[(B$\sp{\natural}$-EXC)] 
For any $X, Y \in \mathcal{F}$ and $i \in X \setminus Y$, \ 
we have
(i) $X - i \in \mathcal{F}$, $ Y + i \in \mathcal{F}$ 
\ or \  
\\
(ii) there exists some $j \in Y \setminus X$ such that
$X - i +j \in \mathcal{F}$, $ Y + i -j  \in \mathcal{F}$.
\end{description}
This means that $\mathcal{F}$ forms a matroid-like structure%
\footnote{
See, e.g.,
Murota (2000a)\citeH{Mspr2000},
Oxley (2011)\citeH{Oxl11},
and Schrijver (2003)\citeH{Sch03} 
for matroids.
}, 
called a {\em generalized matroid} ({\em g-matroid}), or 
an {\em M$\sp{\natural}$-convex family}%
\footnote{
A subset of $N$ can be identified with a 0-1 vector
 (characteristic vector in (\ref{charvecdefnotat})),
and accordingly, a family of subsets can be identified with a set of 0-1 vectors.
We call a family of subsets an {\em M$\sp{\natural}$-convex family}
if the corresponding set of 0-1 vectors is an M$\sp{\natural}$-convex set
as a subset of $\ZZ\sp{N}$.
}. 
An M$\sp{\natural}$-convex family $\mathcal{F}$ containing the empty set
forms the family of independent sets of a matroid.
For example, for integers $a,b$ with $0 \leq a \leq b \leq n$,
$\mathcal{F}_{ab} = \{ X \mid a \leq |X| \leq b \}$
is an M$\sp{\natural}$-convex family,
and $\mathcal{F}_{0b}$ (with  $a=0$)
 forms the family of independent sets of a matroid.

\begin{remark} \rm  \label{RMmultBexc}
It follows from Theorem~\ref{THmultexchmnat} that
a nonempty family $\mathcal{F} \subseteq 2\sp{N}$
satisfies (B$\sp{\natural}$-EXC) if and only if it satisfies
the multiple exchange axiom:
\begin{description}
\item[(B$\sp{\natural}$-EXC$_{\rm\bf m}$)]
For any $X, Y \in \mathcal{F}$ and $I \subseteq X \setminus Y$,
there exists $J \subseteq Y \setminus X$ such that
$(X \setminus I) \cup J \in \mathcal{F}$
and $(Y \setminus J) \cup I \in \mathcal{F}$.
\finbox
\end{description}
\end{remark}

M$\sp{\natural}$-concavity can be characterized by a local exchange property
under the assumption 
that function $f$ is (effectively) defined on an M$\sp{\natural}$-convex family of sets
(Murota 1996c, 2003; Murota and Shioura 1999).
\citeH{Mstein, Mdcasiam, MS99gp}%
The conditions (\ref{mnatconcavexc1loc})--(\ref{mnatconcavexc3loc}) below
are ``local'' in the sense that they require 
the exchangeability of the form of (\ref{mnatconcavexc2}) 
only for 
$(X,Y)$ with $\max(| X \setminus Y | , | Y \setminus X |) \leq 2$.

\begin{theorem}\label{THmconcavlocexc01}
A set function  $f: 2\sp{N} \to \RR \cup \{ -\infty \}$ 
is M$\sp{\natural}$-concave
if and only if 
$\dom f$ is an M$\sp{\natural}$-convex family and 
the following three conditions hold:
\begin{align} 
&
f( X + i + j ) + f( X ) \leq f(X + i) + f(X + j) 
\qquad 
(\forall X \subseteq N, \ \forall i,j \in N \setminus X, \  i \not= j),
\label{mnatconcavexc1loc}
\\
&
 f( X + i + j ) + f( X + k)  \leq 
\max\left[ f(X + i + k) + f(X + j),  f(X + j + k) + f(X + i) \right]  
\notag \\
& \hspace{0.5\textwidth}
 (\forall X \subseteq N, \ \forall i,j,k \, \mbox{\rm (distinct}) \in N \setminus X) ,
\label{mnatconcavexc2loc}
\\
&
 f( X + i + j ) + f( X + k + l)  
\leq 
\max\left[  f(X + i + k) + f(X + j +l ),  f(X + j + k) + f(X + i + l) \right]  
\notag \\
& \hspace{0.5\textwidth}
 (\forall X \subseteq N, \ \forall i,j,k,l \, \mbox{\rm (distinct}) \in N \setminus X).
\label{mnatconcavexc3loc}
\end{align}
\end{theorem}

When the effective domain $\dom f$ contains the emptyset,
the local exchange condition for M$\sp{\natural}$-concavity
takes a simpler form without involving (\ref{mnatconcavexc3loc})
(Reijnierse et al.~2002,
\citeH[Theorem~10]{RGP02}%
M{\"u}ller 2006,
\citeH[Theorem~13.5]{Mul06}%
Shioura and Tamura 2015). 
\citeH[Theorem~6.5]{ST15jorsj}%

\begin{theorem}\label{THmconcavlocexc01B}
Let $f: 2\sp{N} \to \RR \cup \{ -\infty \}$
be a set function such that 
$\dom f$ is an M$\sp{\natural}$-convex family
containing $\emptyset$ (the empty set). 
Then $f$  
is M$\sp{\natural}$-concave
if and only if 
{\rm (\ref{mnatconcavexc1loc})} 
and {\rm (\ref{mnatconcavexc2loc})}
hold.
\end{theorem}

It is known 
(Theorem 6.19 of Murota 2003) 
\citeH[Theorem 6.19]{Mdcasiam}%
that an M$\sp{\natural}$-concave function is {\em submodular}, i.e.,
\begin{equation} \label{setfnsubm}
 f(X) + f(Y) \geq f(X \cup Y) + f(X \cap Y)
 \qquad (X, Y \subseteq N). 
\end{equation}
More precisely, the condition (\ref{mnatconcavexc1loc}) above
is equivalent to the submodularity (\ref{setfnsubm})
as long as $\dom f$ is M$\sp{\natural}$-convex
(Proposition 6.1 of Shioura and Tamura 2015).
\citeH[Proposition 6.1]{ST15jorsj}%

Because of the additional condition (\ref{mnatconcavexc2loc})
for M$\sp{\natural}$-concavity,
not every submodular set function is
M$\sp{\natural}$-concave.  Thus, M$\sp{\natural}$-concave set functions
form a proper subclass of submodular set functions.

\begin{remark} \rm  \label{RMmnatexcsize01}
It follows from (M$\sp{\natural}$-EXC) that
M$\sp{\natural}$-concave set functions enjoy the following 
exchange properties under cardinality constraints
(Lemmas 4.3 and 4.6 of Murota and Shioura 1999):
\citeH[Lemmas 4.3 and 4.6]{MS99gp}%
\\ $\bullet$ 
For any $X, Y \subseteq N$ with $|X| < |Y|$, 
\begin{align}
f( X) + f( Y )   &\leq 
 \max_{j \in Y \setminus X}  \{ f( X  + j) + f( Y  -j) \}.
\label{mnatexccard1}
\end{align}
\\ $\bullet$ 
For any $X, Y \subseteq N$ with $|X| = |Y|$ 
and $i \in X \setminus Y$, 
\begin{align}
f( X) + f( Y )   &\leq 
 \max_{j \in Y \setminus X}  \{ f( X - i + j) + f( Y + i -j) \}.
\label{mnatexccard2}
\end{align}
The former property, in particular, implies 
the cardinal-monotonicity of the induced choice function;
see Theorem~\ref{THchomonoMnat01uniq} and its proof.
\finbox
\end{remark}

\begin{remark} \rm  \label{RMmconcave}
For a set family $\mathcal{F}$ consisting of equi-cardinal sets
(i.e., $|X| =|Y|$ for all  $X, Y \in \mathcal{F}$)
the exchange property (B$\sp{\natural}$-EXC) takes a simpler form:
For any $X, Y \in \mathcal{F}$ and $i \in X \setminus Y$,
there exists some $j \in Y \setminus X$ such that
$X - i +j \in \mathcal{F}$, $ Y + i -j  \in \mathcal{F}$.
This means that $\mathcal{F}$ forms the family of bases of a matroid.
An M$\sp{\natural}$-concave function defined on matroid bases
is called 
a {\em valuated matroid}  
(Dress and Wenzel 1990,1992;
\citeH{DW90}\citeH{DW92}%
Chapter 5 of Murota 2000a),
\citeH[Chapter 5]{Mspr2000}%
 or an {\em M-concave set function}
(Murota 1996c, 2003).
\citeH{Mstein}\citeH{Mdcasiam}%
The exchange property for M-concavity reads:
A set function $f$ is M-concave
if and only if
(\ref{mnatexccard2}) holds for any $X, Y \subseteq N$ and $i \in X \setminus Y$.
A corollary of Theorem~\ref{THmultexchmnat}:
Every M-concave function (valuated matroid) $f$ has the multiple exchange property 
(M$\sp{\natural}$-EXC$_{\rm m}$) with $|J|=|I|$.
A further corollary of this fact is a classical result in matroid theory:
The base family of a matroid has the multiple exchange property 
(B$\sp{\natural}$-EXC$_{\rm m}$) with $|J|=|I|$;
see, e.g., 
Section 39.9a of Schrijver (2003)\citeH[Section 39.9a]{Sch03}.
\finbox
\end{remark}

\subsection{Maximization and single improvement property}
\label{SCmaximization01}

For an M$\sp{\natural}$-concave function,
the maximality of a function value is characterized by a local condition
(Theorem 6.26 of Murota 2003).
\citeH[Theorem 6.26]{Mdcasiam}%

\begin{theorem}
\label{THmnatsetfnlocmax}
Let $f: 2\sp{N} \to \RR \cup \{ -\infty \}$ 
be an M$\sp{\natural}$-concave function and  $X \in \dom f$. 
Then $X$ is a maximizer of $f$
if and only if
\begin{align} \label{mnatsetfnlocmax}
 f(X)  &\geq  f(X - i + j) 
  \quad (\forall \,  i \in X, \ \forall \, j \in N \setminus X) ,
\\ 
 f(X)  &\geq  f(X - i) 
  \qquad \ \ (\forall \,  i \in X) ,
\\ 
 f(X)  &\geq  f(X + j) 
  \qquad \ \ (\forall \,  j \in N \setminus X) .
\end{align}
\end{theorem}

As a discrete analogue of the subgradient inequality for convex functions,
we have the inequality (\ref{matchbndsbgr}) in the following theorem%
\footnote{
This is a reformulation of the ``upper-bound lemma'' 
(Lemma 5.2.29 of Murota 2000a)\citeH[Lemma 5.2.29]{Mspr2000}
for valuated matroids
to M$\sp{\natural}$-concave functions.
See also 
Proposition 6.25 of Murota (2003).  
\citeH[Proposition 6.25]{Mdcasiam}%
}. 

\begin{theorem}  \label{THmnatfnupperbnd}
Let $f: 2\sp{N} \to \RR \cup \{ -\infty \}$ 
be an M$\sp{\natural}$-concave function and  $X, Y \in \dom f$. 
Then 
\begin{equation} \label{matchbndsbgr}
 f(Y) - f(X) \leq \hat f(X,Y) ,
\end{equation}
where
$\hat f(X,Y)$ is defined as follows:
\begin{itemize}
\item
When $|X|=|Y|$, 
\[  
\hat f(X,Y) = \max_{\sigma} \bigg(
 \sum_{i \in X \setminus Y} [f(X - i + \sigma(i)) - f(X)]
\bigg),
\]
where the maximum is
taken over all one-to-one correspondences $\sigma: X \setminus Y \to Y \setminus X$.

\item
When $|X| < |Y|$, 
\[  
\hat f(X,Y) = \max_{\sigma} \bigg(
\sum_{i \in X \setminus Y} [f(X - i + \sigma(i)) - f(X)] + 
\sum_{j \in Y \setminus (X\cup \sigma(X))} [f(X + j) - f(X)] 
\bigg),
\]
where the maximum is
taken over all injections $\sigma: X \setminus Y \to Y \setminus X$.

\item
When $|X| > |Y|$, 
\[ 
\hat f(X,Y) = \max_{\tau} \bigg(
\sum_{j \in Y \setminus X} [f(X - \tau(j) + j) - f(X)] + 
\sum_{i \in X \setminus (Y\cup \tau(Y))} [f(X - i ) - f(X)]
\bigg),
\]
where the maximum is
taken over all injections $\tau: Y \setminus X \to X \setminus Y$.
\end{itemize}
\end{theorem}

For a vector
$p = ( p_{i} \mid i \in N) \in \RR\sp{N}$
we use the notation $f[-p]$ to mean the function $f(X) - p(X)$, where
$X \subseteq N$ and $p(X) = \sum_{i \in X} p_{i}$.  That is, 
\begin{equation} \label{f-pdef01}
 f[-p](X) = f(X) - p(X)
\qquad (X \subseteq N).
\end{equation}
Note that $f[-p]$ is M$\sp{\natural}$-concave if and only if 
$f$ is  M$\sp{\natural}$-concave.

The ``if'' part of Theorem~\ref{THmnatsetfnlocmax}, 
which is the content of the theorem, can be restated as follows:
If $X$ is not a maximizer of $f$,   
there exists $Y \subseteq N$ such that
$|X \setminus Y| \leq 1$, $|Y \setminus X| \leq 1$, 
and $f(X) < f(Y)$.
By considering this property for $f[-p]$ with varying $p$,
we are naturally led to the {\em single improvement property} of 
Gul and Stacchetti (1999):
\citeH{GS99}%
\begin{description}
\item[(SI)]
For any $p \in \RR\sp{N}$, if
$X$ is not a maximizer of $f[-p]$,   
 there exists $Y \subseteq N$ such that
$|X \setminus Y| \leq 1$, $|Y \setminus X| \leq 1$, 
and $f[-p](X) < f[-p](Y)$.
\end{description}
The above argument shows that (SI) is true for M$\sp{\natural}$-concave functions. 
In fact, (SI) is equivalent to M$\sp{\natural}$-concavity
(Fujishige and Yang 2003).
\citeH{FY03gr}%

\subsection{Maximizers and gross substitutability}
\label{SCmaximizers01}

For a vector
$p = ( p_{i} \mid i \in N) \in \RR\sp{N}$
we consider the  maximizers of the function $f[-p](X) = f(X) - p(X)$,
where $p(X) = \sum_{i \in X} p_{i}$ for $X \subseteq N$.
We denote the set of these  maximizers by
\begin{equation} \label{Dpdef}
 D(p ; f) = \argmax_{X} \{ f(X) - p(X) \mid X \subseteq N \}.
\end{equation}
In economic applications, 
$p$ is a price vector and
$D(p) = D(p ; f)$ represents the demand correspondence.

It is one of the most fundamental facts in 
discrete convex analysis that the
M$\sp{\natural}$-concavity of a function 
is characterized in terms of the M$\sp{\natural}$-convexity of its maximizers 
(Murota 1996c; Theorem 6.30 of Murota 2003; Murota and Shioura 1999).
\citeH{Mstein}\citeH[Theorem 6.30]{Mdcasiam}\citeH{MS99gp}%

\begin{theorem}
\label{THmconcavargmax}
A set function $f: 2\sp{N} \to \RR \cup \{ -\infty \}$ 
is M$\sp{\natural}$-concave
if and only if, for every vector $p \in \RR\sp{N}$, 
$D(p ; f)$ is an M$\sp{\natural}$-convex family.
That is, $f$ satisfies {\rm (M$\sp{\natural}$-EXC)} 
if and only if, for every $p \in \RR\sp{N}$, 
$D(p ; f)$ satisfies {\rm (B$\sp{\natural}$-EXC)}.
\end{theorem}

The following are two versions of the multiple exchange property of $D(p ; f)$:
\begin{description}
\item[(NC)]
For any $p \in \RR\sp{N}$, if
$X, Y \in D(p ; f)$ and $I \subseteq X \setminus Y$,
there exists $J \subseteq Y \setminus X$ such that
$(X \setminus I) \cup J \in D(p ; f)$,
\end{description}
\begin{description}
\item[(NCsim)] 
For any $p \in \RR\sp{N}$, if
$X, Y \in D(p ; f)$ and $I \subseteq X \setminus Y$,
there exists $J \subseteq Y \setminus X$ such that
$(X \setminus I) \cup J \in D(p ; f)$ and $(Y \setminus J) \cup I \in D(p ; f)$.
\end{description}
The condition (NC), introduced by 
Gul and Stacchetti (1999),
\citeH{GS99}%
is called ``no complementarities property''
and (NCsim) is a simultaneous (or symmetric) version of (NC)
introduced by
\RED{
Murota (2018).
}
\citeH{Mmultexc16}%
These conditions, (NC) and (NCsim), are equivalent to each other, 
and are equivalent to the M$\sp{\natural}$-concavity of $f$;
see Remark \ref{RMmnat=SNC=GS} as well as 
\RED{
(Murota 2018)
}
\citeH{Mmultexc16}%
for details.

In the above we have looked at
the family $D(p ; f)$ of the maximizers 
for each $p \in \RR\sp{N}$.
We now investigate how $D(p ; f)$
changes with the variation of $p$.

A set function (single-unit valuation function) 
$f: 2\sp{N} \to \RR \cup \{ -\infty \}$ 
is said to have the
{\em gross substitutes property} if
\footnote{
To be precise, 
Kelso and Crawford (1982)\citeH{KC82}
and also Gul and Stacchetti (1999)\citeH{GS99}
treat the case of $f: 2\sp{N} \to \RR$.
}  
\begin{description}
\item[(GS)]
For any $p,q  \in \RR\sp{N}$ 
with $p \leq q$ and $X \in D(p ; f)$,
 there exists $Y \in D(q ; f)$ such that
$ \{ i \in X \mid p_{i} = q_{i} \} \subseteq Y$.
\end{description}
The concept of gross substitutes property,  introduced by 
Kelso and Crawford (1982)\citeH{KC82},
has turned out to be crucial in economics; see, e.g., 
Roth and Sotomayor (1990)\citeH{RS90},
Bikhchandani and Mamer (1997)\citeH{BM97eco},
Gul and Stacchetti (1999)\citeH{GS99},
Ausubel and Milgrom (2002)\citeH{AM02},
Milgrom (2004)\citeH{Mil04book},
Hatfield and Milgrom (2005)\citeH{HM05},
Ausubel (2006)\citeH{Aus06},
Sun and Yang (2006)\citeH{SY06},
Milgrom and Strulovici (2009)\citeH{MS09sbst},
and 
\RED{
Hatfield et al. (2019)\citeH{HKNOW16fullsubst}.
}%

The following theorem, due to 
Fujishige and Yang (2003)\citeH{FY03gr},
plays the key role to connect discrete convex analysis and economics.

\begin{theorem} \label{THmconcavgross}
A set function 
$f: 2\sp{N} \to \RR \cup \{ -\infty \}$ 
has the gross substitutes property {\rm (GS)}
if and only if it is M$\sp{\natural}$-concave.
\end{theorem}

It is known 
(Hatfield and Milgrom 2005, 
\citeH{HM05}%
Milgrom and Strulovici 2009)
\citeH{MS09sbst}%
that the gross substitutes property, and hence M$\sp{\natural}$-concavity,
implies the law of aggregate demand
in the following form:

\begin{description}
\item[(LAD)] 
For any $p,q  \in \RR\sp{N}$ 
with $p \leq q$ and $X \in D(p ; f)$,
there exists $Y \in D(q ; f)$ such that
$|X| \geq |Y|$.
\end{description}

Gross substitutes properties for multi-unit valuations are treated in 
Section \ref{SCmaximizersZ}.

\subsection{Choice function}
\label{SCchoiceMnat01}

A function 
$C: 2\sp{N} \to 2\sp{N}$ is called a {\em choice function} if
$C(Z) \subseteq Z$ for all $Z \subseteq N$.
We have 
$C(\emptyset)=\emptyset$ and, 
possibly, $C(Z)=\emptyset$ for some nonempty subsets $Z$.
A choice function $C$ is said to be {\em consistent} if
$C(X) \subseteq Y \subseteq X$ implies $C(Y) = C(X)$.
Here we discuss two other properties of choice functions,  
substitutability and cardinal monotonicity,
which are closely related to M$\sp{\natural}$-concavity.

The {\em substitutability} of a choice function $C$ 
means the following property
(Roth 1984,
\citeH{Rot84}%
Roth and Sotomayor 1990):
\citeH{RS90}%
\begin{description}
\item[(SC$_{\rm ch}$)]
 For any $Z_{1}, Z_{2} \subseteq N$ 
with $Z_{1} \supseteq Z_{2}$ it holds that
$  Z_{2} \cap C(Z_{1}) \subseteq  C(Z_{2})$.
\end{description}
Several apparently different formulations of substitutability,
each equivalent to (SC$_{\rm ch}$),
are found in the literature:
\begin{itemize}
\item
 For any $Z_{1}, Z_{2} \subseteq N$ 
with $Z_{1} \supseteq Z_{2}$ it holds that
$Z_{1} \setminus C(Z_{1}) \supseteq  Z_{2} \setminus C(Z_{2})$.

\item
$i \in C(X)$ implies $i \in C(Y \cup \{ i \})$ for $Y \subseteq X$.

\item
For any $X \subseteq N$ and any distinct $i,j \in X$, \ 
$i \in C(X)$ implies $i \in C(X \setminus \{ j \})$.
\end{itemize}

A choice function $C$ is said to be {\em cardinal-monotone} if
$|C(Y)| \leq |C(X)|$ for all $Y \subseteq X$
 \citeH{Alk02}.
This property is called
{\em increasing property} by  
Fleiner (2003)\citeH{Flei03}
and {\em law of aggregate demand} by
Hatfield and Milgrom (2005)\citeH{HM05}.

\begin{remark} \rm  
As is well known, consistency and substitutability together
are equivalent to {\em path independence} of 
Plott (1973),
\citeH{Plo73}%
 which is characterized by the condition:
$C(C(X) \cup Y) =C(X \cup Y)$ for all $X, Y \subseteq N$.
This condition is equivalent to:
$C(C(X) \cup C(Y)) =C(X \cup Y)$ for all $X, Y \subseteq N$.
\finbox
\end{remark}

\begin{remark} \rm  \label{RMsignif3cond}
The above-mentioned properties of choice functions
are well-known key properties in economics and game theory.
In the stable matching problem, for example,
consistency and substitutability (i.e., path independence)
guarantee, roughly, the existence of a stable matching.
If, in addition, the choice functions are cardinal-monotone,
then the stable matchings form a nice lattice
(with simple lattice operations, being distributive, etc.).
To quote 
Theorem 10 of Alkan (2002):
\citeH[Theorem 10]{Alk02}%
``The set of stable matchings in any two-sided market with path-independent
cardinal-monotone choice functions is a distributive lattice under
the common preferences of all agents on one side of the market. The supremum
(infimum) operation of the lattice for each side consists componentwise of the
join (meet) operation in the revealed preference ordering of associated agents.
The lattice has the polarity, unicardinality and complementarity properties.''
\finbox
\end{remark}

\begin{remark} \rm  
A function $C: 2\sp{N} \to 2\sp{N}$ is called {\em comonotone}  
if there exists a monotone function  
$g: 2\sp{N} \to 2\sp{N}$ such that 
$C(X) = X \setminus g(X)$ for all $X \subseteq N$
(Fleiner 2003).
\citeH{Flei03}%
A function $C: 2\sp{N} \to 2\sp{N}$ is  comonotone   
if and only if $C$ is a choice function with substitutability.
The fixed point approach to stable matchings of 
Fleiner (2003)
\citeH{Flei03}%
is based on the observation
that stable matchings correspond to
fixed points of a certain monotone function 
associated with the choice functions
and
the deferred acceptance algorithm of 
Gale and Shapley (1962) 
\citeH{GS62}%
can be regarded as an iteration of this function. 
See also 
Farooq et al.~(2012).
\citeH{FFT12}%
\finbox
\end{remark}

A {\em choice correspondence} means a function
$C: 2\sp{N} \to 2\sp{2\sp{N}}$ such that
$\emptyset \not= C(Z) \subseteq 2\sp{Z}$ for all $Z \subseteq N$.
It should be clear that the value $C(Z)$ is not a subset of $N$ but
a family of subsets of $N$.
If $C(Z)$ consists of a single subset for each $Z \subseteq N$,
then $C$ can be identified with a choice function
$C: 2\sp{N} \to 2\sp{N}$.

The {\em substitutability} of a choice correspondence $C$ 
is formulated as follows 
(Definition~4 of Sotomayor 1999):
\citeH[Definition~4]{Sot99three}%
\begin{description}
\item[(SC$_{\rm ch}\sp{1}$)]
For any $Z_{1}, Z_{2} \subseteq N$ 
with $Z_{1} \supseteq Z_{2}$ 
and any  $X_{1} \in C(Z_{1})$, 
there exists 
$X_{2} \in C(Z_{2})$ such that
$Z_{2} \cap X_{1}\subseteq  X_{2}$.

\item[(SC$_{\rm ch}\sp{2}$)]
For any $Z_{1}, Z_{2} \subseteq N$ 
with $Z_{1} \supseteq Z_{2}$ 
and any 
$X_{2} \in C(Z_{2})$,
there exists $X_{1} \in C(Z_{1})$
such that
$Z_{2} \cap X_{1}\subseteq  X_{2}$.
\end{description}
For a choice function 
$C: 2\sp{N} \to 2\sp{N}$,
(SC$_{\rm ch}\sp{1}$) and (SC$_{\rm ch}\sp{2}$)
are each equivalent to (SC$_{\rm ch}$).

\paragraph{Choice function induced from a valuation function:}

A valuation function $f: 2\sp{N} \to \RR \cup \{ -\infty \}$
with $\emptyset \in \dom f$
induces a choice correspondence
$C: 2\sp{N} \to 2\sp{2\sp{N}}$
 by 
\begin{equation} \label{choiceByVal01}
C(Z) = C(Z ; f) = \argmax \{ f(Y) \mid Y \subseteq Z \} .
\end{equation}
The assumption 
``$\emptyset \in \dom f$'' ensures that  $C(Z;f) \not= \emptyset$
for every $Z \subseteq N$.
In general, the maximizer is not unique, and accordingly,
$C$ is a choice correspondence
(i.e., $C(Z;f)$ is a family of subsets of $N$).

While (SC$_{\rm ch}\sp{1}$) and (SC$_{\rm ch}\sp{2}$) above 
formulate the substitutability for a choice correspondence,
(SC$\sp{1}$) and (SC$\sp{2}$) below
are the corresponding conditions for a valuation function $f$.
That is,  a valuation function $f$ satisfies 
(SC$\sp{1}$) if and only if the induced choice correspondence
$C(\, \cdot \, ;f)$ satisfies (SC$_{\rm ch}\sp{1}$),
and similarly for (SC$\sp{2}$) and (SC$_{\rm ch}\sp{2}$).

\begin{description}
\item[(SC$\sp{1}$)]
For any $Z_{1}, Z_{2} \subseteq N$ 
with $Z_{1} \supseteq Z_{2}$ 
and any 
$X_{1} \in C(Z_{1};f)$, 
there exists $X_{2} \in C(Z_{2};f)$ such that $Z_{2} \cap X_{1}\subseteq  X_{2}$.

\item[(SC$\sp{2}$)]
 For any $Z_{1}, Z_{2} \subseteq N$ 
with $Z_{1} \supseteq Z_{2}$ 
and any 
$X_{2} \in C(Z_{2};f)$,
there exists $X_{1} \in C(Z_{1};f)$ such that $Z_{2} \cap X_{1}\subseteq  X_{2}$.
\end{description}
These two conditions are independent of each other; see Examples 3.1 and 3.2 in 
Farooq and Tamura (2004).
\citeH{FT04sbst}%

A connection to M$\sp{\natural}$-concavity
is pointed out by 
Eguchi et al.~(2003);
\citeH{EFT03}%
see also 
Fujishige and Tamura (2006).
\citeH{FT06market}%
This is another important finding, on top of Theorem~\ref{THmconcavgross}
(equivalence of M$\sp{\natural}$-concavity to (GS)),
which has reinforced the connection 
between discrete convex analysis and economics.

\begin{theorem} \label{THchosubstMnat01}
Every M$\sp{\natural}$-concave function
$f: 2\sp{N} \to \RR \cup \{ -\infty \}$
with $\emptyset \in \dom f$
satisfies {\rm (SC$\sp{1}$)} and {\rm (SC$\sp{2}$)}.
That is, the choice correspondence induced from 
an M$\sp{\natural}$-concave set function
has the substitutability properties
{\rm (SC$_{\rm ch}\sp{1}$)} and {\rm (SC$_{\rm ch}\sp{2}$)}.
\end{theorem}
\begin{proof}
Assume $Z_{1} \supseteq Z_{2}$.

Proof of (SC$\sp{1}$):
Let $X_{1} \in C(Z_{1};f)$
and take
$X_{2} \in C(Z_{2};f)$ with minimum
$|(Z_{2} \cap X_{1} ) \setminus X_{2}|$.
To prove by contradiction, suppose that there exists 
$i \in (Z_{2} \cap X_{1} ) \setminus X_{2}$.
Since
$i \in X_{1} \setminus X_{2}$, 
(M$\sp{\natural}$-EXC) implies
(i) $f( X_{1} ) + f( X_{2} ) \leq f( X_{1} - i ) + f( X_{2} + i )$
or
(ii) there exists $j \in X_{2} \setminus X_{1}$ such that
$f( X_{1} ) + f( X_{2} ) \leq  f( X_{1} - i + j) + f( X_{2} + i -j)$.
In case (i) we note $X_{1} - i \subseteq Z_{1}$ 
and $X_{2} + i \subseteq Z_{2}$, 
from which follow
$f( X_{1} -i  ) \leq f( X_{1})$
and
$f( X_{2} + i ) \leq f(X_{2})$.
Therefore,
the inequalities are in fact equalities, and
$X_{1} - i \in C(Z_{1};f)$ and $X_{2} + i \in C(Z_{2};f)$.
But we have
$|(Z_{2} \cap X_{1} ) \setminus (X_{2} + i)| =
 |(Z_{2} \cap X_{1} ) \setminus X_{2}| -  1$,
which contradicts the choice of $X_{2}$.
In case (ii) we note $X_{1} - i + j \subseteq Z_{1}$ 
and $X_{2} + i -j  \subseteq Z_{2}$, 
from which follow
$f( X_{1} -i +j ) \leq f( X_{1})$
and
$f( X_{2} + i -j ) \leq f(X_{2})$.
Therefore,
the inequalities are in fact equalities, and
$X_{1} - i +j \in C(Z_{1};f)$ and $X_{2} + i -j  \in C(Z_{2};f)$.
But we have
$|(Z_{2} \cap X_{1} ) \setminus (X_{2} + i -j)| =
 |(Z_{2} \cap X_{1} ) \setminus X_{2}| -  1$,
which contradicts the choice of $X_{2}$.

Proof of (SC$\sp{2}$):
Let $X_{2} \in C(Z_{2};f)$
and take
$X_{1} \in C(Z_{1};f)$ with minimum
$|(Z_{2} \cap X_{1} ) \setminus X_{2}|$.
By the same argument as above we obtain
(i) $X_{1} - i \in C(Z_{1};f)$ with
$|(Z_{2} \cap (X_{1} - i) ) \setminus X_{2}| =
 |(Z_{2} \cap X_{1} ) \setminus X_{2}| -  1$,
or (ii) 
$X_{1} - i + j \in C(Z_{1};f)$ with
$|(Z_{2} \cap (X_{1}-i+j) ) \setminus X_{2}| =
 |(Z_{2} \cap X_{1} ) \setminus X_{2}| -  1$.
This is a contradiction to the choice of $X_{1}$.
\end{proof}

When the maximizer is unique in (\ref{choiceByVal01}) for every $Z$,  we say
that $f$ is {\em unique-selecting}.
In this case, $C$ in (\ref{choiceByVal01}) is a choice function
(i.e., $C(Z;f)$ is a subset of $N$ for every $Z$),
and {\rm (SC$\sp{1}$)} and {\rm (SC$\sp{2}$)} both reduce to
the following condition: 
\begin{description}
\item[(SC)]
 For any $Z_{1}, Z_{2} \subseteq N$ 
with $Z_{1} \supseteq Z_{2}$ it holds that
$Z_{2} \cap C(Z_{1};f) \subseteq  C(Z_{2};f)$.
\end{description}

Theorem~\ref{THchosubstMnat01} 
yields, as a corollary, the following result of 
Eguchi and Fujishige (2002)\citeH{EF02}.

\begin{theorem} \label{THchosubstMnat01uniq}
Every unique-selecting M$\sp{\natural}$-concave function
$f: 2\sp{N} \to \RR \cup \{ -\infty \}$
with $\emptyset \in \dom f$
satisfies {\rm (SC)}.
That is, the choice function induced from 
a unique-selecting M$\sp{\natural}$-concave set function
has the substitutability property 
{\rm (SC$_{\rm ch}$)}.
\end{theorem}

Unique-selecting 
M$\sp{\natural}$-concave functions are well-behaved also
with respect to cardinal monotonicity.
The following is a special case of
Lemma 4.5 of Murota and Yokoi (2015)\citeH[Lemma 4.5]{MY15mor}.

\begin{theorem} \label{THchomonoMnat01uniq}
Every unique-selecting M$\sp{\natural}$-concave function
$f: 2\sp{N} \to \RR \cup \{ -\infty \}$
with $\emptyset \in \dom f$
induces a choice function with cardinal monotonicity.
\end{theorem}
\begin{proof}
The proof is based on the exchange property 
(\ref{mnatexccard1}) in Remark~\ref{RMmnatexcsize01}.
To prove by contradiction, suppose that there exist 
$X$ and $Y$ such that $X \supseteq Y$ and
$|C(X)|<|C(Y)|$.
Set $X\sp{*}=C(X)$ and $Y\sp{*}=C(Y)$. Then $|X\sp{*}|<|Y\sp{*}|$.
By the exchange property (\ref{mnatexccard1}) 
there exists  $j \in Y\sp{*} \setminus X\sp{*}$
such that
$ f(X\sp{*}) + f( Y\sp{*} )   \leq f( X\sp{*}  + j) + f( Y\sp{*}  -j) $.
Here we have 
$f( X\sp{*}  + j) < f(X\sp{*})$
since $X\sp{*}+ j \subseteq X$ and $X\sp{*}$ is the unique maximizer,
and also $f( Y\sp{*}  - j) < f(Y\sp{*})$
since $Y\sp{*}- j \subseteq Y$ and $Y\sp{*}$ is the unique maximizer.
This is a contradiction.
\end{proof}

Thus, M$\sp{\natural}$-concave valuation functions 
entail the three desirable properties.
Recall Remark~\ref{RMsignif3cond} for the implications of this fact.

\begin{theorem} \label{THchoiceMnat01}
The choice function induced from 
a unique-selecting M$\sp{\natural}$-concave set function
$f$ with $\emptyset \in \dom f$
has consistency, substitutability, and cardinal monotonicity.
\end{theorem}

 Finally, we mention a theorem that 
characterizes  M$\sp{\natural}$-concavity 
in terms of a parametrized version of 
(SC$\sp{1}$) and (SC$\sp{2}$).
Recall from (\ref{f-pdef01}) the 
notation $f[-p](X) = f(X) - p(X)$ for 
$p \in \RR\sp{N}$ and $X \subseteq N$.
If $f$ is an M$\sp{\natural}$-concave function
 (not assumed to be unique-selecting),
$f[-p]$ is also M$\sp{\natural}$-concave, and hence
is equipped with the properties (SC$\sp{1}$) and (SC$\sp{2}$)
by Theorem~\ref{THchosubstMnat01}.
 In other words, an M$\sp{\natural}$-concave function $f$
has the following properties.

\begin{description}
\item[(SC$\sp{1}_{\rm\bf {G}}$)]  
For any $p\in \RR^{N}$,  $f[-p]$ satisfies {\rm (SC$\sp{1}$)}.
\item[(SC$\sp{2}_{\rm\bf {G}}$)]  
For any $p\in \RR^{N}$,  $f[-p]$ satisfies {\rm (SC$\sp{2}$)}.
\end{description}
The following theorem, due to 
Farooq and Tamura (2004),
\citeH{FT04sbst}%
states that these two conditions are equivalent, and each of them 
characterizes M$\sp{\natural}$-concavity.

\begin{theorem} \label{THchoMnat01}
For a set function $f: 2\sp{N} \to \RR \cup \{ -\infty \}$
with $\dom f \not= \emptyset$,
we have the equivalence:
$f$ is M$\sp{\natural}$-concave 
$\iff$ 
{\rm (SC$\sp{1}_{\rm {G}}$)}
$\iff$ 
{\rm (SC$\sp{2}_{\rm {G}}$)}.
\end{theorem}

\subsection{Twisted M$\sp{\natural}$-concavity}
\label{SCtwistMnat01}

Let $W$ be a subset of $N$.
For any subset $X$ of $N$
we define
\begin{equation}  \label{twistsetdef01}
\mathrm{tw}(X) = (X \setminus W) \cup (W \setminus X).
\end{equation}
A set function $f: 2\sp{N} \to \RR \cup \{ -\infty \}$ 
is said to be a {\em twisted M$\sp{\natural}$-concave function}
with respect to $W$,
if the function $\tilde{f}: 2\sp{N} \to \RR \cup \{ -\infty \}$ 
defined by
\begin{equation}\label{ftwistMnatdef01}
\tilde{f}(X) = f(\mathrm{tw}(X))
\qquad (X \subseteq N) 
\end{equation}
is an M$\sp{\natural}$-concave function 
(Ikebe and Tamura 2015)\citeH{IT15scnet}.
The same concept was introduced earlier by 
Sun and Yang (2006, 2009)\citeH{SY06}\citeH{SY09}
under the name of {\em GM-concave functions}. 
Note that $f$ is twisted M$\sp{\natural}$-concave with respect to $W$
if and only if 
it is twisted M$\sp{\natural}$-concave with respect to $U = N \setminus W$.

Mathematically, 
twisted M$\sp{\natural}$-concavity is equivalent to the original 
M$\sp{\natural}$-concavity  through twisting, 
and all the properties and theorems
about M$\sp{\natural}$-concave functions 
can be translated into those about
twisted M$\sp{\natural}$-concave functions.
However, 
twisted M$\sp{\natural}$-concave functions
are convenient sometimes in the modeling in economics.

For example, as pointed out by 
Ikebe and Tamura (2015),
\citeH{IT15scnet}%
twisted M$\sp{\natural}$-concavity implies
the same-side substitutability (SSS) and the cross-side complementarity (CSC) 
proposed by
Ostrovsky (2008)
\citeH{Ost08}%
in discussing supply chain networks.
For a choice function $C: 2\sp{N} \to 2\sp{N}$
the {\em same-side substitutability} (SSS)
with respect to the bipartition $(U,W)$ of $N$
means the following property:
\begin{description}
\item[(SSS)]
(i)
 For any $Z_{1}, Z_{2} \subseteq N$ 
with 
$Z_{1} \cap U \supseteq Z_{2} \cap U$ and $Z_{1} \cap W =Z_{2} \cap W$,
we have
$Z_{2} \cap C(Z_{1}) \cap U \subseteq  C(Z_{2}) \cap U$,
and (ii) the same statement with $U$ and $W$ interchanged,
\end{description}
and the {\em cross-side complementarity} (CSC)
means 
\begin{description}
\item[(CSC)]
(i)
 For any $Z_{1}, Z_{2} \subseteq N$ 
with $Z_{1} \cap U \supseteq Z_{2} \cap U$ and $Z_{1} \cap W =Z_{2} \cap W$,
we have $C(Z_{1}) \cap W \supseteq  C(Z_{2}) \cap W$,
and (ii) the same statement with $U$ and $W$ interchanged.
\end{description}
For our exposition it is convenient to combine these two into
a single property:
\begin{description}
\item[(SSS-CSC)]
(i)
 For any $Z_{1}, Z_{2} \subseteq N$ 
with $Z_{1} \cap U \supseteq Z_{2} \cap U$ and $Z_{1} \cap W =Z_{2} \cap W$,
we have
$Z_{2} \cap C(Z_{1}) \cap U \subseteq  C(Z_{2}) \cap U$
and $C(Z_{1}) \cap W \supseteq  C(Z_{2}) \cap W$,
and (ii) the same statement with $U$ and $W$ interchanged.
\end{description}

The connection to twisted M$\sp{\natural}$-concavity is given 
in the following theorem%
\footnote{
Theorem~\ref{THchotwistMnatuniq01} can be understood as
 a twisted version of 
Theorem~\ref{THchosubstMnat01uniq},
though a straightforward translation of Theorem~\ref{THchosubstMnat01uniq}
via twisting does not seem to yield Theorem~\ref{THchotwistMnatuniq01}.
Theorem~\ref{THchotwistMnatuniq01} can be proved 
as a special case of 
Theorem~\ref{THchotwistMnat01} below, for which a direct proof is given.
},  
to be ascribed to 
Ikebe and Tamura (2015).
\citeH{IT15scnet}%
Recall from (\ref{choiceByVal01}) the definition of 
the choice function induced from a valuation function:
$C(Z) = C(Z ; f) = \argmax \{ f(Y) \mid Y \subseteq Z \}$.

\begin{theorem} \label{THchotwistMnatuniq01}
The choice function induced from 
a unique-selecting twisted M$\sp{\natural}$-concave set function
$f: 2\sp{N} \to \RR \cup \{ -\infty \}$
with $\emptyset \in \dom f$
has the property {\rm (SSS-CSC)}.
\end{theorem}

For choice correspondences we need to consider
the following  pair of conditions.

\begin{description}
\item[(SSS-CSC$\sp{1}$)]
(i)
 For any $Z_{1}, Z_{2} \subseteq N$ 
with $Z_{1} \cap U \supseteq Z_{2} \cap U$ and $Z_{1} \cap W =Z_{2} \cap W$
and any  $X_{1} \in C(Z_{1})$, 
there exists $X_{2} \in C(Z_{2})$
 such that
$Z_{2} \cap X_{1} \cap U \subseteq  X_{2} \cap U$
and
$X_{1} \cap W \supseteq X_{2} \cap W$,
and (ii) the same statement with $U$ and $W$ interchanged.

\item[(SSS-CSC$\sp{2}$)]
(i)
 For any $Z_{1}, Z_{2} \subseteq N$ 
with $Z_{1} \cap U \supseteq Z_{2} \cap U$ and $Z_{1} \cap W =Z_{2} \cap W$
and any  $X_{2} \in C(Z_{2})$, 
there exists $X_{1} \in C(Z_{1})$
 such that
$Z_{2} \cap X_{1} \cap U \subseteq  X_{2} \cap U$
and	
$X_{1} \cap W \supseteq X_{2} \cap W$,
and (ii) the same statement with $U$ and $W$ interchanged.
\end{description}

The following theorem  
(Ikebe and Tamura 2015)
\citeH{IT15scnet}%
states that these two properties are
implied by twisted M$\sp{\natural}$-concavity.

\begin{theorem} \label{THchotwistMnat01}
The choice correspondence induced from 
a twisted M$\sp{\natural}$-concave set function
$f: 2\sp{N} \to \RR \cup \{ -\infty \}$
with $\emptyset \in \dom f$
has the properties 
{\rm (SSS-CSC$\sp{1}$)} and {\rm (SSS-CSC$\sp{2}$)}.
\end{theorem}
\begin{proof}
We prove (SSS-CSC$\sp{1}$)-(i) and (SSS-CSC$\sp{2}$)-(i);
the proofs of (SSS-CSC$\sp{1}$)-(ii) and (SSS-CSC$\sp{2}$)-(ii)
are obtained by interchanging $U$ and $W$.
Assume 
$Z_{1} \cap U \supseteq Z_{2} \cap U$ and $Z_{1} \cap W =Z_{2} \cap W$,
and let
 $\tilde{f}$ be the M$\sp{\natural}$-concave function
in (\ref{ftwistMnatdef01}) associated with $f$.
For 
$X_{1} \subseteq Z_{1}$ and $X_{2} \subseteq Z_{2}$
define
\[
\varPhi(X_{1}, X_{2}) = 
    |(Z_{2} \cap X_{1} \cap U ) \setminus (X_{2} \cap U ) |
  + |(X_{2} \cap W ) \setminus (X_{1} \cap W ) | .
\]

Proof of (SSS-CSC$\sp{1}$)-(i): \ 
Let $X_{1} \in C(Z_{1};f)$
and take
$X_{2} \in C(Z_{2};f)$ with
$\varPhi(X_{1}, X_{2})$ minimum.
To prove by contradiction, suppose that there exists 
$i \in \big( ( Z_{2} \cap X_{1} \cap U ) \setminus (X_{2} \cap U ) \big)
\cup \big( (X_{2} \cap W ) \setminus (X_{1} \cap W ) \big)$.
Since
$i \in \mathrm{tw}(X_{1})  \setminus \mathrm{tw}(X_{2})$,
(M$\sp{\natural}$-EXC) for $\tilde{f}$ implies
\begin{enumerate}
\item[(i)]
$\tilde{f}( \mathrm{tw}(X_{1}) ) + \tilde{f}( \mathrm{tw}(X_{2}) ) 
\leq \tilde{f}( \mathrm{tw}(X_{1}) - i ) + \tilde{f}( \mathrm{tw}(X_{2}) + i )$
\ \ or

\item[(ii)]
 there exists 
$j \in \mathrm{tw}(X_{2})  \setminus \mathrm{tw}(X_{1})$  
such that
$\tilde{f}( \mathrm{tw}(X_{1}) ) + \tilde{f}(  \mathrm{tw}(X_{2})) 
\leq  \tilde{f}( \mathrm{tw}(X_{1}) - i + j) 
      + \tilde{f}( \mathrm{tw}(X_{2}) + i -j)$.
\end{enumerate}
Letting 
\begin{align*} 
\hat{X}_{1} =  
   \left\{  \begin{array}{ll}
   \mathrm{tw}(\mathrm{tw}(X_{1}) - i))   & (\mbox{in (i)}),  \\
  \mathrm{tw}(\mathrm{tw}(X_{1}) - i + j))&  (\mbox{in (ii)}),  \\
             \end{array}  \right.
\quad
\hat{X}_{2} =  
   \left\{  \begin{array}{ll}
   \mathrm{tw}(\mathrm{tw}(X_{2}) + i))   & (\mbox{in (i)}),  \\
  \mathrm{tw}(\mathrm{tw}(X_{2}) + i - j))&  (\mbox{in (ii)}),  \\
             \end{array}  \right.
\end{align*}
we can express the above inequalities in (i) and (ii) as
\[
f( X_{1} ) + f( X_{2} ) \leq f( \hat{X}_{1} ) + f( \hat{X}_{2} ) .
\]
As can be verified easily,  we have $\hat{X}_{1} \subseteq Z_{1}$ 
and $\hat{X}_{2} \subseteq Z_{2}$, 
from which follow
$f( \hat{X}_{1} ) \leq f( X_{1})$
and
$f( \hat{X}_{2} ) \leq f(X_{2})$
since  $X_{1} \in C(Z_{1};f)$
and $X_{2} \in C(Z_{2};f)$.
Therefore,
the inequalities are in fact equalities, and
$\hat{X}_{1} \in C(Z_{1};f)$ and $\hat{X}_{2} \in C(Z_{2};f)$.
But we have
$\varPhi(X_{1}, \hat{X}_{2}) = \varPhi(X_{1}, X_{2}) - 1$,
which contradicts the choice of $X_{2}$.

Proof of (SSS-CSC$\sp{2}$)-(i): \ 
Let $X_{2} \in C(Z_{2};f)$
and take
$X_{1} \in C(Z_{1};f)$ with
$\varPhi(X_{1}, X_{2})$
minimum.
By the same argument as above we obtain
$\hat{X}_{1} \in C(Z_{1};f)$ 
with 
$\varPhi(\hat{X}_{1}, X_{2}) = \varPhi(X_{1}, X_{2}) - 1$.
This is a contradiction to the choice of $X_{1}$.
\end{proof}

The concept of twisted M$\sp{\natural}$-concavity
can also be defined for functions on integer vectors $\ZZ\sp{N}$
to be used for multi-unit models. See Section \ref{SCtwistMnatZ}.

\subsection{Examples}
\label{SCmnatexample01}

Here are some examples of M$\sp{\natural}$-concave set functions.

\begin{enumerate}
\item

For real numbers $a_{i}$ indexed by $i \in N$,
the {\em additive valuation}
\begin{equation} \label{mlinfn01}
 f(X)=   \sum_{i \in X} a_{i}  
\qquad (X \subseteq N)
\end{equation}
is an M$\sp{\natural}$-concave function.

\item
For a set of nonnegative numbers $a_{i}$ indexed by $i \in N$,
the {\em maximum-value function}
({\em unit-demand utility})
\begin{equation} \label{unitdemutil01}
 f(X) = \max_{i \in X} a_{i}
\qquad (X \subseteq N)
\end{equation}
with $f(\emptyset) = 0$
is an M$\sp{\natural}$-concave function.

\item
For a univariate concave function
$\varphi: \ZZ \to \RR \cup \{ -\infty \}$
(i.e., if $\varphi(t-1) + \varphi(t+1) \leq 2\varphi(t)$ for all integers $t$),
the function $f$ defined by
\begin{equation}  \label{cardconv01}
f(X)=\varphi(|X|) 
\qquad (X \subseteq N)
\end{equation}
is M$\sp{\natural}$-concave.
Such $f$ is called a {\em symmetric concave valuation}.

\item
For a family of univariate concave functions 
$\{ \varphi_{A} \mid A \in \calT\}$ indexed by 
a family $\calT$ of subsets of $N$,
the function 
\begin{equation}  \label{laminarconv01}
 f(X)  =  \sum_{A \in \calT} \varphi_{A}(|A \cap X|)
\qquad (X \subseteq N)
\end{equation}
is submodular. 
A function $f$ of the form (\ref{laminarconv01})
is called {\em laminar concave},
if $\calT$ is  a laminar family, i.e., if
[$A, B \in \calT   \Rightarrow A \cap B = \emptyset$
or $A \subseteq B$ or $A \supseteq B$].
A laminar concave function is M$\sp{\natural}$-concave.  See 
Note 6.11 of Murota (2003)
\citeH[Note 6.11]{Mdcasiam}%
for a proof. A special case of
(\ref{laminarconv01}) with $\calT = \{ N \}$ reduces to (\ref{cardconv01}).

\item
Given a matroid%
\footnote{
For matroids, see, e.g., 
Murota (2000a),
\citeH{Mspr2000}%
Oxley (2011),
\citeH{Oxl11}%
and Schrijver (2003).
\citeH{Sch03}%
} 
 on $N$ in terms of the family $\calI$ of 
independent sets, the {\em rank function} $f$ is defined by
\begin{equation}  \label{matroidrank}
 f(X) = \max\{|I| \mid I \in {\calI},\ I \subseteq X\}
\qquad (X \subseteq N),
\end{equation}
which denotes the maximum size of an independent set contained in $X$.
A matroid rank function {\rm (\ref{matroidrank})}
is M$\sp{\natural}$-concave.
A {\em weighted matroid rank function} 
(or {\em weighted matroid valuation}) is a function represented as
\begin{equation}  \label{weightedrank}
 f(X) = \max\{ w(I) \mid I \in {\calI},\ I \subseteq X\}
\qquad (X \subseteq N)
\end{equation}
with some weight $w \in \RR\sp{N}$,
where $w(I) = \sum_{i \in I} w_{i}$. 
A weighted matroid rank function (\ref{weightedrank}) is M$\sp{\natural}$-concave
(Shioura 2012).
\citeH{Shi12rank}%
See Murota (2010)
\citeH{Mrims10}%
for an elementary proof
for the M$\sp{\natural}$-concavity of 
(\ref{weightedrank}) as well as (\ref{matroidrank}).

\item
Let $G=(S,T; E)$ be a bipartite graph with vertex bipartition
$(S, T)$ and edge set $E$,
and suppose that each edge $e \in E$ 
is associated with weight $w_{e} \in \RR$.
For $M \subseteq E$, 
we denote by $\partial M$ the set of the vertices incident to some edge in $M$,
and call $M$ a {\em matching} if
$|S \cap  \partial M \, | = |M \, | = |T \cap  \partial M \, |$.
For $X \subseteq T$ denote by 
$f(X)$ the maximum weight
of a matching that precisely matches $X$ in $T$, i.e.,
\begin{equation}  \label{weightbipartmatch}
 f(X) = \max \{ w(M) 
         \mid \mbox{$M$ is a matching, $T \cap \partial M=X$}  \}
\end{equation}
with $w(M) = \sum_{e \in M} w_{e}$, where $f(X)= -\infty$ 
if no such $M$ exists for $X$.
Then 
$f: 2\sp{T} \to \RR \cup \{ -\infty \}$
is an M$\sp{\natural}$-concave function.
See 
Example 3.3 of Murota (1996a) or
\citeH[Example 3.3]{MvmiI96}%
Example 5.2.4 of Murota (2000a)
\citeH[Example 5.2.4]{Mspr2000}%
for proofs.
Such function is called an {\em assignment valuation} by
Hatfield and Milgrom (2005).
\citeH{HM05}%
Assignment valuations cover a fairly large class of 
M$\sp{\natural}$-concave functions, but not every
M$\sp{\natural}$-concave function can be represented in the form of
(\ref{weightbipartmatch}), as shown by
Ostrovsky and Paes Leme (2015).
\citeH{OP15gs}%

\item

Let $G=(S,T; E)$ be a bipartite graph with vertex bipartition
$(S, T)$ and edge set $E$,
with weight $w_{e} \in \RR$ associated with each edge $e \in E$.
Furthermore, suppose that a matroid 
on $S$ is given in terms of the family $\calI$ of independent sets
(see Fig.~\ref{FGindassgraph}).
For $X \subseteq T$ denote by 
$f(X)$ the maximum weight
of a matching such that the end-vertices in $S$ form an independent set
and the end-vertices in $T$ are equal to $X$, i.e.,
\begin{equation}  \label{weightindepmatch}
 f(X) = \max \{ w(M) 
         \mid \mbox{$M$ is a matching, $S \cap \partial M \in \calI$, 
                $T \cap \partial M=X$}  \},
\end{equation}
where $f(X)= -\infty$ if no such $M$ exists for $X$.
We call such $f$ an {\em independent assignment valuation}.
It is known that an independent assignment valuation
is  M$\sp{\natural}$-concave.
For proofs, see 
Example 5.2.18 of Murota (2000a),
\citeH[Example 5.2.18]{Mspr2000}%
Section 9.6.2 of Murota (2003),
\citeH[Section 9.6.2]{Mdcasiam}%
and Kobayashi et al.~(2007).
\citeH{KMT07jump}%
If the given matroid is a free matroid with $\calI = 2\sp{S}$,
(\ref{weightindepmatch}) reduces to (\ref{weightbipartmatch}).
\end{enumerate}

\begin{figure}\begin{center}
 \includegraphics[width=0.25\textwidth,clip]{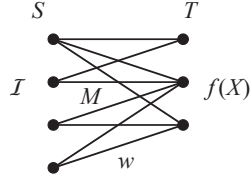}
 \caption{Independent assignment valuation}
 \label{FGindassgraph}
\end{center}\end{figure}

\subsection{Concluding remarks of section \ref{SCmnatconcav01}}

We collect here the conditions that characterize
M$\sp{\natural}$-concave set functions:

-- Exchange property
(M$\sp{\natural}$-EXC)  
\hfill (Section~\ref{SCexchange01})

-- Multiple exchange property
(M$\sp{\natural}$-EXC$_{\rm m}$)
\\ \qquad \ \ 
= Strong no complementarities property (SNC)
\hfill (Section~\ref{SCexchange01})

-- Local exchange property 
(Theorems \ref{THmconcavlocexc01} and \ref{THmconcavlocexc01B})
\hfill (Section~\ref{SCexchange01})

-- Single improvement property (SI)
\hfill (Section~\ref{SCmaximization01})

-- Exchange property (B$\sp{\natural}$-EXC) for the maximizers $D(p;f)$
\hfill (Section~\ref{SCmaximizers01})

-- Multiple (one-sided) exchange property for the maximizers $D(p;f)$
\\ \qquad \ \ 
= No complementarities property (NC)
\hfill (Section~\ref{SCmaximizers01})

-- Multiple exchange property (NCsim) for the maximizers $D(p;f)$
\hfill (Section~\ref{SCmaximizers01})

-- Gross substitutability (GS)
\hfill (Section~\ref{SCmaximizers01})

-- Parametrized substitutability
(SC$\sp{1}_{\rm {G}}$)
\hfill (Section~\ref{SCchoiceMnat01})

-- Parametrized substitutability
(SC$\sp{2}_{\rm {G}}$)
\hfill (Section~\ref{SCchoiceMnat01})



\section{M$\sp{\natural}$-concave Function on $\ZZ\sp{n}$}
\label{SCmnatconcavZ}

In Section \ref{SCmnatconcav01}
we have considered M$\sp{\natural}$-concave set functions,
which correspond to single-unit valuations with substitutability.
In this section we deal with M$\sp{\natural}$-concave functions 
defined on integer vectors,
$f: \mathbb{Z}\sp{n} \to \RR \cup \{ -\infty \}$,
which correspond to multi-unit valuations with substitutability.

\subsection{Exchange property}
\label{SCexchangeZ}

Let $N$ be a finite set, say, $N = \{ 1,2,\ldots, n \}$ for $n \geq 1$. 
For a vector $z \in \RR\sp{N}$ in general, define 
the {\em positive}\index{positive support}
 and {\em negative supports}\index{negative support}
 of $z$ as
\begin{equation} \label{vecsupportdef}
 \suppp(z) = \{ i \mid z_{i} > 0 \},
\qquad 
 \suppm(z) = \{ j \mid z_{j} < 0 \}.
\end{equation}
Recall that, for $i \in N$, the $i$th unit vector is denoted by $\chi_{i}$.

We say that a function
$f: \mathbb{Z}\sp{N} \to \RR \cup \{ -\infty \}$
with $\dom f \not= \emptyset$
is {\em M$\sp{\natural}$-concave}, if,
for any $x, y \in \ZZ\sp{N}$ and $i \in \suppp(x-y)$, 
we have (i)
\begin{equation}  \label{mconcav1Z}
f(x) + f(y)  \leq  f(x -\unitvec{i}) + f(y+\unitvec{i})
\end{equation}
or (ii) there exists some $j \in \suppm(x-y)$ such that
\begin{equation}  \label{mconcav2Z}
f(x) + f(y)   \leq 
 f(x-\unitvec{i}+\unitvec{j}) + f(y+\unitvec{i}-\unitvec{j}) .
\end{equation}
This property is referred to as the {\em exchange property}.
See Fig.~\ref{FGmnatfndef}, in which 
$(x',y')= (x -\unitvec{i}, y+\unitvec{i})$ and
$(x'',y'')= ( x-\unitvec{i}+\unitvec{j}, y+\unitvec{i}-\unitvec{j})$.

\begin{figure}\begin{center}
 \includegraphics[width=0.35\textwidth,clip]{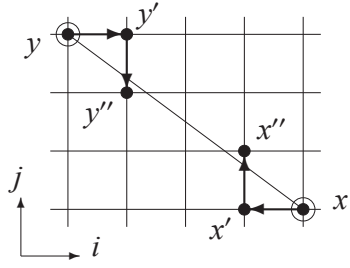}
\caption{Nearer pair in the definition of M$\sp{\natural}$-concave functions}
 \label{FGmnatfndef}
\end{center}\end{figure}

A more compact expression of the exchange property is as follows:
\begin{description}
\item[(M$\sp{\natural}$-EXC\hbox{[$\ZZ$]})] 
 For any $x, y \in \ZZ\sp{N}$ and $i \in \suppp(x-y)$, we have
\begin{equation} \label{mnatconcavexc2Z}
f(x) + f(y)   \leq 
 \max_{j \in \suppm(x - y) \cup \{ 0 \}} 
 \{ f(x - \unitvec{i} + \unitvec{j}) + f(y + \unitvec{i} - \unitvec{j}) \},
\end{equation}
\end{description}
where $\chi_{0}=\veczero$ (zero vector).
In the above statement we may change
``For any $x, y \in \ZZ\sp{N}$\,'' to ``For any $x, y \in \dom f$\,''
 since if $x \not\in \dom f$ or $y \not\in \dom f$,
(\ref{mnatconcavexc2Z}) trivially holds with $f(x) + f(y)  = -\infty$. 
An M$\sp{\natural}$-concave function $f$ with $\dom f \subseteq \{ 0,1 \}\sp{N}$
can be identified with an M$\sp{\natural}$-concave set function
introduced in Section \ref{SCexchange01}.
A function $f$ is called {\em M$\sp{\natural}$-convex} 
if $-f$ is M$\sp{\natural}$-concave.

It follows from 
(M$\sp{\natural}$-EXC[$\ZZ$]) 
that the effective domain $B = \dom f$
of an M$\sp{\natural}$-concave function $f$
has the following exchange property:
\begin{description}
\item[(B$\sp{\natural}$-EXC\hbox{[$\ZZ$]})] 
For any $x, y \in B$ and $i \in \suppp(x-y)$, we have
(i)
$x -\unitvec{i} \in B$, $y+\unitvec{i} \in B$
\ or \  
\\
(ii) there exists some $j \in \suppm(x-y)$ such that
$x-\unitvec{i}+\unitvec{j}  \in B$, $y+\unitvec{i}-\unitvec{j} \in B$.
\end{description}
A set $B \subseteq \ZZ\sp{N}$ having this property 
is called an {\em M$\sp{\natural}$-convex set}
(or {\em integral generalized polymatroid}, {\em integral g-polymatroid}).  
An M$\sp{\natural}$-convex set contained 
in the unit cube $\{ 0,1 \}\sp{N}$
can be identified with an M$\sp{\natural}$-convex family of subsets
(Section \ref{SCexchange01}).

M$\sp{\natural}$-concavity can be characterized by a local exchange property
under the assumption 
that function $f$ is (effectively) defined on an M$\sp{\natural}$-convex set
(Murota 1996c, 2003; Murota and Shioura 1999).
\citeH{Mstein, Mdcasiam, MS99gp}%
The conditions (\ref{mnatconcavexc0loc2Z})--(\ref{mnatconcavexc4loc2Z}) below
are ``local'' in the sense that they require 
the exchangeability of the form of (\ref{mnatconcavexc2Z}) 
only for some
$(x,y)$ with $\| x - y \|_{1} \leq 4$.

\begin{theorem}\label{THmconcavlocexcZ}
A function 
$f: \mathbb{Z}\sp{N} \to \RR \cup \{ -\infty \}$
is M$\sp{\natural}$-concave
if and only if 
$\dom f$ is an M$\sp{\natural}$-convex set and 
the following conditions hold:
\begin{align} 
& f( x + 2 \unitvec{i} ) + f( x ) 
\leq 2 f(x + \unitvec{i})  
\qquad 
(\forall x \in \ZZ\sp{N}, \ \forall i \in N),
\label{mnatconcavexc0loc2Z}
\\
&
f( x + \unitvec{i} + \unitvec{j} ) + f( x ) 
\leq f(x + \unitvec{i}) + f(x + \unitvec{j}) 
\qquad 
(\forall x \in \ZZ\sp{N}, \ \forall i,j \in N, \  i \not= j),
\label{mnatconcavexc1loc2Z}
\\
&
 f( x + 2 \unitvec{i} ) + f( x + \unitvec{k})  \leq 
 f(x + \unitvec{i} + \unitvec{k}) + f(x + \unitvec{i})
\qquad 
 (\forall x \in \ZZ\sp{N}, \ \forall i,k \in N,  \  i \not= k),
\label{mnatconcavexc2loc2Z}
\\
&
 f( x + \unitvec{i} + \unitvec{j} ) + f( x + \unitvec{k})  \leq 
\max\left[ f(x + \unitvec{i} + \unitvec{k}) + f(x + \unitvec{j}), 
           f(x + \unitvec{j} + \unitvec{k}) + f(x + \unitvec{i}) \right]  
\notag \\
& \hspace{0.5\textwidth}
 (\forall x \in \ZZ\sp{N}, \ \forall i,j,k \, \mbox{\rm (distinct}) \in N) ,
\label{mnatconcavexc3loc2Z}
\\
&
 f( x + \unitvec{i} + \unitvec{j} ) + f( x + \unitvec{k} + \unitvec{l})  
\notag  \\ &
\leq 
\max\left[  f(x + \unitvec{i} + \unitvec{k}) + f(x + \unitvec{j} + \unitvec{l} ),  
       f(x + \unitvec{j} + \unitvec{k}) + f(x + \unitvec{i} + \unitvec{l}) \right]  
\notag \\
& \hspace{0.4\textwidth}
 (\forall x \in \ZZ\sp{N},
 \ \forall i,j,k, l \in N \ \mbox{\rm with} \ \{ i,j \} \cap \{ k,l \} = \emptyset),
\label{mnatconcavexc4loc2Z}
\end{align}
where in {\rm (\ref{mnatconcavexc4loc2Z})} we allow the possibility of
$i= j$ or $k=l$. 
\end{theorem}

When the effective domain $\dom f$ is
an M$\sp{\natural}$-convex set such that
$\bm{0} \in \dom f \subseteq \mathbb{Z}_{+}\sp{N}$,
the local exchange condition above
takes a simpler form that does not involve (\ref{mnatconcavexc4loc2Z})
(Theorem~6.8 of Shioura and Tamura 2015).
\citeH[Theorem~6.8]{ST15jorsj}%
To cover the case of $\dom f = \mathbb{Z}\sp{N}$
we weaken the assumption on $\dom f$ to:
\begin{equation} \label{domcondlocZ}
 x, y \in \dom f \Longrightarrow x \wedge y \in \dom f.
\end{equation}

\begin{theorem}\label{THmconcavlocexcZB}
Let 
$f: \mathbb{Z}\sp{N} \to \RR \cup \{ -\infty \}$
be a function such that 
$\dom f$ is an M$\sp{\natural}$-convex set 
satisfying {\rm (\ref{domcondlocZ})}.
Then 
$f$ is M$\sp{\natural}$-concave
if and only if 
{\rm (\ref{mnatconcavexc0loc2Z})},
{\rm (\ref{mnatconcavexc1loc2Z})},
{\rm (\ref{mnatconcavexc2loc2Z})} and
{\rm (\ref{mnatconcavexc3loc2Z})} hold.
\end{theorem}
\begin{proof}
The proof of 
Theorem~6.8 of Shioura and Tamura (2015)
\citeH[Theorem~6.8]{ST15jorsj}%
works under the weaker condition (\ref{domcondlocZ}).
\end{proof}

The local exchange property above
admits a natural reformulation in terms of the discrete Hessian matrix
when $\dom f = \ZZ\sp{N}$.
For $x\in\ZZ\sp{N}$ and $i, j \in N$ define
\begin{equation}\label{mhessianDefZ}
H_{ij}(x) = f(x+\unitvec{i}+\unitvec{j})-f(x+\unitvec{i})-f(x+\unitvec{j})+f(x),
\end{equation}
and let
$H_{f}(x) =(H_{ij}(x) \mid i,j \in N)$ 
be the matrix consisting of those components.
This matrix $H_{f}(x)$ is called the {\em discrete Hessian matrix} of $f$ at $x$. 
The following theorem, due to 
Hirai and Murota (2004) 
\citeH{HM03tree}%
and Murota (2007),
\citeH{Mdcaprimer07}%
can be derived from Theorem \ref{THmconcavlocexcZB}.

\begin{theorem} \label{Mmnathessian}
A function $f: \ZZ\sp{N} \to \RR$ is M$\sp{\natural}$-concave
if and only if the discrete Hessian matrix $H_f(x)=(H_{ij}(x))$ 
satisfies the following conditions
for each $x\in\ZZ\sp{N}$:
\begin{alignat}{2}
&  H_{ij}(x)  \leq 0 
&\quad & \mbox{for any $(i,j)$} ,
\label{mnatfnhesseZ1} 
\\
&  H_{ij}(x)  \leq \max ( H_{ik}(x) ,  H_{jk}(x) ) 
&\quad & \mbox{if } \   \{ i,j \} \cap \{ k \} = \emptyset .
\label{mnatfnhesseZ2} 
\end{alignat}
\end{theorem}

\begin{proof}
The correspondence between  the conditions in 
Theorems \ref{THmconcavlocexcZB} and \ref{Mmnathessian} 
is quite straightforward.
With the use of (\ref{mhessianDefZ})
we can easily verify:
{\rm (\ref{mnatconcavexc0loc2Z})}
$\Leftrightarrow$ 
$H_{ii}(x)  \leq 0$, \ 
{\rm (\ref{mnatconcavexc1loc2Z})}
$\Leftrightarrow$ 
$H_{ij}(x)  \leq 0$ $(i \not= j)$, \ 
{\rm (\ref{mnatconcavexc2loc2Z})}
$\Leftrightarrow$ 
$H_{ii}(x)  \leq H_{ik}(x) $ $(i \not= k)$, \ 
and
{\rm (\ref{mnatconcavexc3loc2Z})} 
$\Leftrightarrow$ 
$ H_{ij}(x)  \leq \max ( H_{ik}(x) ,  H_{jk}(x) )$
($i,j,k$: distinct).
\end{proof}

It is known 
(Theorem 6.19 of Murota 2003) 
\citeH[Theorem 6.19]{Mdcasiam}%
that an M$\sp{\natural}$-concave function
$f: \mathbb{Z}\sp{N} \to \RR \cup \{ -\infty \}$
 is {\em submodular} on the integer lattice, i.e.,
\begin{equation} \label{setfnsubmZ}
 f(x) + f(y) \geq f(x \vee y) + f(x \wedge y)
 \qquad (x, y \in \ZZ\sp{N}) .
\end{equation}
More precisely, the condition (\ref{mnatconcavexc1loc2Z}) above
is equivalent to the submodularity (\ref{setfnsubmZ})
as long as $\dom f$ is M$\sp{\natural}$-convex
(Proposition 6.1 of Shioura and Tamura 2015). 
\citeH[Proposition 6.1]{ST15jorsj}%
Because of the additional conditions
for M$\sp{\natural}$-concavity,
not every submodular function is
M$\sp{\natural}$-concave.  Thus, M$\sp{\natural}$-concave functions
form a proper subclass of submodular functions on $\ZZ\sp{N}$.

It is also known 
(Theorem 4.6 of Murota 1996c,
\citeH[Theorem 4.6]{Mstein}%
 Theorem 6.42 of Murota 2003)
\citeH[Theorem 6.42]{Mdcasiam}%
that an M$\sp{\natural}$-concave function 
$f: \mathbb{Z}\sp{N} \to \RR \cup \{ -\infty \}$
is {\em concave-extensible}, 
i.e., there exists a concave function 
$\overline{f}: \RR\sp{N} \to \RR \cup \{ -\infty \}$
such that
$\overline{f}(x) = f(x)$ for all $x \in \ZZ\sp{N}$.

\begin{remark} \rm  \label{RMmnatexcsizeZ}
It follows from (M$\sp{\natural}$-EXC[$\ZZ$]) that
M$\sp{\natural}$-concave functions enjoy the following 
exchange properties under size constraints
(Lemmas 4.3 and 4.6 of Murota and Shioura 1999):
\citeH[Lemmas 4.3 and 4.6]{MS99gp}%
\\ $\bullet$ 
For any $x, y \in \ZZ\sp{N}$ with $x(N) < y(N)$, 
\begin{align}
f(x) + f(y)   &\leq 
 \max_{j \in \suppm(x-y)}  
 \{ f(x +\unitvec{j}) + f(y-\unitvec{j})  \}.
\label{mnatexcsizeZ1}
\end{align}
\\ $\bullet$ 
For any $x, y \in \ZZ\sp{N}$ with $x(N) =y(N)$
and $i \in \suppp(x-y)$, 
\begin{align}
f(x) + f(y)   &\leq 
 \max_{j \in \suppm(x-y)}  
\{ f(x-\unitvec{i}+\unitvec{j}) + f(y+\unitvec{i}-\unitvec{j})  \}.
\label{mnatexcsizeZ2}
\end{align}
The former property, in particular, implies 
the size-monotonicity of the induced choice function;
see Theorem~\ref{THchomonoMnatZuniq} and its proof.
\finbox
\end{remark}

\begin{remark} \rm  \label{RMmconcaveZ}
If $B \subseteq \ZZ\sp{N}$ lies in a hyperplane with a constant component sum
(i.e., $x(N) = y(N)$ for all  $x, y \in B$),
the exchange property (B$\sp{\natural}$-EXC[$\ZZ$]) takes a simpler form
(without the possibility of $j=0$): 
For any $x, y \in B$ and $i \in \suppp(x-y)$,
there exists some $j \in \suppm(x-y)$ such that
$x-\unitvec{i}+\unitvec{j}  \in B$, $y+\unitvec{i}-\unitvec{j} \in B$.
A set $B \subseteq \ZZ\sp{N}$ having this exchange property 
is called an {\em M-convex set}
(or {\em integral base polyhedron}).  
An M$\sp{\natural}$-concave function defined on an M-convex set
is called an {\em M-concave function}
(Murota 1996c, 2003).
\citeH{Mstein}\citeH{Mdcasiam}%
The exchange property for M-concavity reads:
A function 
$f: \mathbb{Z}\sp{N} \to \RR \cup \{ -\infty \}$
is M-concave
if and only if, 
 for any $x, y \in \mathbb{Z}\sp{N}$ and $i \in \suppp(x-y)$, it holds that
\begin{equation}
f(x) + f(y)   \leq  \max_{j \in \suppm(x-y)}  
\{ f(x-\unitvec{i}+\unitvec{j}) + f(y+\unitvec{i}-\unitvec{j})  \}.
\label{mconcavexc2Z}
\end{equation}
M-concave functions and
M$\sp{\natural}$-concave functions are equivalent concepts,
in that M$\sp{\natural}$-concave functions in $n$ variables 
can be obtained as projections of M-concave functions in $n+1$ variables.
More formally, let ``$0$'' denote a new element not in $N$ and
$\tilde{N} = \{0\} \cup N$.
A function $f : \ZZ\sp{N} \to \RR \cup \{ -\infty \}$ is M$\sp{\natural}$-concave
if and only if  the function 
$\tilde{f} : \ZZ\sp{\tilde{N}} \to \RR \cup \{ -\infty \}$ 
defined by
\begin{equation} \label{mfnmnatfnrelationcave}
  \tilde{f}(x_{0},x) = \left\{ \begin{array}{ll}
      f(x)    & \mbox{ if $x_{0} = {-}x(N)$} \\
      -\infty & \mbox{ otherwise}
    \end{array}\right.
 \qquad ( x_{0} \in \ZZ, x \in \ZZ\sp{N})
\end{equation}
is an M-concave function.
A function $f$ is called {\em M-convex} if $-f$ is M-concave.
\finbox
\end{remark}

\subsection{Maximization and single improvement property}
\label{SCmaximizationZ}

For an M$\sp{\natural}$-concave function,
the maximality of a function value is characterized by a local condition
as follows,
where $\unitvec{0}=\veczero$
(Proposition 6.23 and Theorem 6.26 of Murota 2003).
\citeH[Proposition 6.23, Theorem 6.26]{Mdcasiam}%

\begin{theorem}
\label{THmnatsetfnlocmaxZ}
Let $f: \ZZ\sp{N} \to \RR \cup \{ -\infty \}$ 
be an M$\sp{\natural}$-concave function and  $x \in \dom f$.

\noindent {\rm (1)}
If  $f(x) < f(y)$ for $y \in \dom f$,  
then $f(x) < f(x - \unitvec{i} + \unitvec{j})$
for some
$i \in \suppp(x-y) \cup \{ 0 \}$ and $j \in \suppm(x-y) \cup \{ 0 \}$.

\noindent {\rm (2)}
$x$ is a maximizer of $f$ if and only if
\begin{align} \label{mnatsetfnlocmaxZ}
 f(x)  &\geq  f(x - \unitvec{i} + \unitvec{j}) 
  \quad (\forall \,  i, j \in N \cup \{ 0 \}) .
\end{align}
\end{theorem}

For a vector
$p = ( p_{i} \mid i \in N) \in \RR\sp{N}$
we use the notation $f[-p]$ to mean the function 
$f(x) - p\sp{\top} x$,
where $p\sp{\top}$ means the transpose of $p$.
That is, 
\begin{equation} \label{f-pdefZ}
f[-p](x) = f(x) - p\sp{\top} x
\qquad (x \in \ZZ\sp{N}).
\end{equation}
By considering the properties of (1) and (2) in Theorem~\ref{THmnatsetfnlocmaxZ} 
for $f[-p]$ with varying $p$,
we are naturally led to 
(SSI[$\ZZ$]) and (SI[$\ZZ$]) below%
\footnote{
(SSI[$\ZZ$]) here is denoted as (M$\sp{\natural}$-SI[$\ZZ$]) in 
Murota (2003).
\citeH{Mdcasiam}%
}:  

\begin{description}
\item[(SSI\hbox{[$\ZZ$]})]%
For any $p \in \RR\sp{N}$ and 
$x, y \in \dom f$ with $f[-p](x) < f[-p](y)$,
there exists
$i \in \suppp(x-y) \cup \{ 0 \}$ and $j \in \suppm(x-y) \cup \{ 0 \}$ 
such that
$f[-p](x) < f[-p](x - \unitvec{i} + \unitvec{j})$.

\item[(SI\hbox{[$\ZZ$]})]
For any $p \in \RR\sp{N}$, if
$x \in \dom f$ is not a maximizer of $f[-p]$,   
there exists
$i \in N \cup \{ 0 \}$ and $j \in N \cup \{ 0 \}$ 
such that
$f[-p](x) < f[-p](x - \unitvec{i} + \unitvec{j})$.
\end{description}

The stronger version (SSI[$\ZZ$]) is 
shown to be equivalent to M$\sp{\natural}$-concavity
(Theorem 7 of Murota and Tamura 2003a).
\citeH[Theorem 7]{MTgross03}%
This property is named the {\em strong single improvement property} in 
Shioura and Tamura (2015). 
\citeH{ST15jorsj}%
The latter (SI[$\ZZ$]) is the vector version of
single improvement property 
(Section \ref{SCmaximization01}),
called the {\em multi-unit  single improvement property} by
Milgrom and Strulovici (2009).
\citeH{MS09sbst}%
We can see from 
Theorem 13 of Milgrom and Strulovici (2009)
\citeH[Theorem 13]{MS09sbst}%
that (SI[$\ZZ$]) is equivalent to M$\sp{\natural}$-concavity
under the assumption
of concave-extensibility of $f$ and boundedness of $\dom f$.

\subsection{Maximizers and gross substitutability}
\label{SCmaximizersZ}

For a vector
$p = ( p_{i} \mid i \in N) \in \RR\sp{N}$
we consider the  maximizers of the function $f[-p](x) = f(x) - p\sp{\top} x$.
We denote the set of these  maximizers by
\begin{equation} \label{DpdefZ}
 D(p ; f) = \argmax_{x} \{ f(x) - p\sp{\top} x  \}.
\end{equation}
In economic applications, 
$p$ is a price vector and
$D(p) = D(p ; f)$ represents the demand correspondence.

It is one of the most fundamental facts in 
discrete convex analysis that the
M$\sp{\natural}$-concavity of a function 
is characterized in terms of the M$\sp{\natural}$-convexity of its maximizers 
(Murota 1996c; Theorem 6.30 of Murota 2003; Murota and Shioura 1999).
\citeH{Mstein}\citeH[Theorem 6.30]{Mdcasiam}\citeH{MS99gp}%
\begin{theorem} \label{THmconcavargmaxZ}
Let 
$f: \ZZ\sp{N} \to \RR \cup \{ -\infty \}$ 
be a function with a bounded effective domain.
Then $f$ is M$\sp{\natural}$-concave
if and only if, for every vector $p \in \RR\sp{N}$, 
$D(p ; f)$ is an M$\sp{\natural}$-convex set.
That is, $f$ satisfies {\rm (M$\sp{\natural}$-EXC[$\ZZ$])} 
if and only if, for every $p \in \RR\sp{N}$, 
$D(p ; f)$ satisfies {\rm (B$\sp{\natural}$-EXC[$\ZZ$])}.
\end{theorem}

As a straightforward extension of the gross substitutes condition
from single-unit valuations 
(Section \ref{SCmaximizers01})
to multi-unit valuations it seems natural to conceive the following condition:
\begin{description}
\item[(GS\hbox{[$\ZZ$]})]
For any $p,q  \in \RR\sp{N}$ 
with $p \leq q$ and $x \in D(p ; f)$,
 there exists $y \in D(q ; f)$ such that
$x_{i} \le y_{i}$ for all $i\in N$ with $p_{i}=q_{i}$.
\end{description}
It turns out, however, that this condition
alone is too weak to be fruitful, mathematically and economically.
Subsequently, several different strengthened forms of (GS[$\ZZ$])
are proposed in the literature 
(Danilov et al.~2003,
Murota and Tamura 2003a,
Milgrom and Strulovici 2009,
Shioura and Tamura 2015).
\citeH{DKL03gr}\citeH{MTgross03}\citeH{MS09sbst}\citeH{ST15jorsj}%

Among others we start with
the {\em projected gross substitutes condition}%
\footnote{
(PRJ-GS[$\ZZ$]) is denoted as (M$\sp{\natural}$-GS[$\ZZ$]) in 
Section 6.8 of Murota (2003).
\citeH[\S 6.8]{Mdcasiam}%
}  
(PRJ-GS[$\ZZ$]) of 
Murota and Tamura (2003a):
\citeH{MTgross03}%

\begin{description}
\item[(PRJ-GS\hbox{[$\ZZ$]})]
For any $p, q  \in \RR\sp{N}$ with $p \leq q$,
any $p_{0},q_{0} \in \RR$ with $p_{0} \leq q_{0}$
and $x \in D(p - p_0 \bm{1} ; f)$,
 there exists $y \in D(q - q_0 \bm{1}; f)$ such that
(i) $x_{i} \le y_{i}$ for all $i\in N$ with $p_{i}=q_{i}$
and (ii)
$x(N) \geq y(N)$ if $p_0 = q_0$,
\end{description}
where $x(N) = \sum_{i \in N} x_{i}$ and $y(N) =\sum_{i \in N} y_{i}$.
By fixing $p_{0} = q_{0} =0 $ in (PRJ-GS[$\ZZ$]) we obtain 
the following condition:
\begin{description}
\item[(GS{\&}LAD\hbox{[$\ZZ$]})] 
For any $p,q  \in \RR\sp{N}$ 
with $p \leq q$ and $x \in D(p ; f)$,
there exists $y \in D(q ; f)$ such that
(i) $x_{i} \le y_{i}$ for all $i\in N$ with $p_{i}=q_{i}$
and 
(ii)
$x(N) \geq y(N)$.
\end{description}
As the acronym (GS{\&}LAD[$\ZZ$]) shows, 
this condition is a combination of (GS[$\ZZ$]) above 
and the law of aggregate demand: 
\begin{description}
\item[(LAD\hbox{[$\ZZ$]})] 
For any $p,q  \in \RR\sp{N}$ 
with $p \leq q$ and $x \in D(p ; f)$,
there exists $y \in D(q ; f)$ such that
$x(N) \geq y(N)$
\end{description}
considered by 
Hatfield and Milgrom (2005) 
\citeH{HM05}%
and Milgrom and Strulovici (2009).
\citeH{MS09sbst}%
Note, however, that imposing (GS{\&}LAD[$\ZZ$]) on $f$ is not the same as 
imposing (GS[$\ZZ$]) and (LAD[$\ZZ$]) on $f$, since in (GS{\&}LAD[$\ZZ$])
both (i) and (ii) must be satisfied by the same vector $y$.
Obviously, (GS{\&}LAD[$\ZZ$]) implies (GS[$\ZZ$]) and (LAD[$\ZZ$]).
The amalgamated form (GS{\&}LAD[$\ZZ$]) is given in 
Murota et al.~(2013a)\citeH{MSY13},
whereas the juxtaposition of (GS[$\ZZ$]) and (LAD[$\ZZ$]) is in 
Theorem 13 (iv) of Milgrom and Strulovici (2009).
\citeH[Theorem 13 (iv)]{MS09sbst}%
We may also consider the following variant
(Shioura and Tamura 2015, Shioura and Yang 2015)
of (GS{\&}LAD[$\ZZ$]),
where the vector $q$ takes a special form%
\footnote{
Recall that $\chi_{k}$ denotes the $k$th unit vector.
} 
 $p+\delta \chi_k$ with $k \in N$ and $\delta > 0$:
\begin{description}
\item[(GS{\&}LAD$'$\hbox{[$\ZZ$]})] 
 For any $p \in \RR\sp{N}$,  $k \in N$,
$\delta > 0$ and  
$x \in D(p ; f)$, there exists $y \in D(p + \delta \chi_{k} ; f)$
such that (i) $x_{i} \le y_{i}$ for all $i\in N \setminus \{ k \}$
and  (ii) $x(N) \geq y(N)$.
\end{description}

M$\sp{\natural}$-concavity can be characterized by these properties as follows
(Murota and Tamura 2003a\citeH{MTgross03},
Danilov et al.~2003\citeH{DKL03gr},
Theorem 13 of Milgrom and Strulovici 2009\citeH[Theorem 13]{MS09sbst},
Theorem 4.1 of Shioura and Tamura 2015\citeH[Theorem 4.1]{ST15jorsj};
Theorems 6.34 and 6.36 of Murota 2003)\citeH[Theorems 6.34, 6.36]{Mdcasiam}.
The theorem refers to two other conditions
(SWGS[$\ZZ$]) and (SS[$\ZZ$]), which are explained in Remark~\ref{RMgsSWGS-SS} below.

\begin{theorem} \label{THmnatGSLAD}
 Let $f: \ZZ\sp{N} \to \RR \cup \{-\infty \}$ be a concave-extensible 
function with a bounded effective domain.
 Then we have the following equivalence:
{\rm (M$\sp{\natural}$-EXC[$\ZZ$])}
$\iff$ {\rm (PRJ-GS[$\ZZ$])}
$\iff$ {\rm (GS{\&}LAD[$\ZZ$])}
$\iff$ {\rm (GS[$\ZZ$])} {\rm \&} {\rm (LAD[$\ZZ$])} 
$\iff$ {\rm (GS{\&}LAD$'$[$\ZZ$])}
$\iff$ {\rm (SWGS[$\ZZ$])}.
If $\dom f$ is contained in $\ZZ_{+}\sp{N}$,
each of these conditions is equivalent to {\rm (SS[$\ZZ$])}.
\end{theorem}

\begin{remark} \rm  \label{RMgsSWGS-SS}
The {\em step-wise gross substitutes condition} 
(Danilov et al.~2003)
\citeH{DKL03gr}%
means:
\begin{description}
\item[(SWGS\hbox{[$\ZZ$]})] 
For any $p \in \RR\sp{N}$, $k \in N$ and $x \in D(p;f)$, 
at least one of (i) and (ii) holds true%
\footnote{
Recall that $\chi_{k}$ denotes the $k$th unit vector.
}: 
\\
(i) $x\in D(p+ \delta \chi_k ; f)$ for all $\delta \geq 0$,
\\
(ii) there exists $\delta \geq 0$ and $y \in D(p + \delta \chi_k ; f)$
such that 
$y_{k}=x_{k}-1$ and $y_{i} \geq x_{i}$ for all $i\in N \setminus \{ k \}$.
\end{description}
The {\em strong substitute condition} 
(Milgrom and Strulovici 2009)
\citeH{MS09sbst}%
for a multi-unit valuation $f$
means the condition (GS[$\ZZ$]) for the single-unit valuation $f^{\rm B}$
corresponding to $f$:

\begin{description}
\item[(SS\hbox{[$\ZZ$]})] 
The function $f^{\rm B}$ associated with $f$ satisfies the condition (GS[$\ZZ$]).
\end{description}
More specifically, the function $f^{\rm B}$ is defined as follows.
Let $u \in \ZZ_{+}\sp{N}$ be a vector such that
$\dom f \subseteq [\bm{0}, u]_\ZZ$.
Consider a set 
$N^{\rm B} = \{(i, \beta) \mid i \in N,\ \beta \in \ZZ,\ 1 \leq \beta \leq u_{i} \}$
and define 
$f^{\rm B}: \ZZ^{N^{\rm B}} \to \RR \cup \{ -\infty \}$ 
with $\dom f^{\rm B} \subseteq \{0,1\}^{N^{\rm B}}$
by
\begin{equation} \label{fn-SGS-MS}
  {f}^{\rm B}({x}^{\rm B}) = f(x),
\quad x^{\rm B} \in \{0,1\}^{N^{\rm B}},
\quad
x_{i} = \sum_{\beta=1}^{u_{i}}x^{\rm B}_{(i, \beta)}\quad (i \in N).
\end{equation} 
\vspace{-1.5\baselineskip}
\\
\finbox
\end{remark}

\subsection{Choice function}
\label{SCchoiceMnatZ}

Let $b\in \mathbb{Z}_{+}^{N}$ be an upper bound vector and
$\mathcal{B}= \{ x\in \mathbb{Z}_{+}^{N} \mid  x\leq b \}$ 
be the set of feasible vectors.
A function $C:\mathcal{B}\to \mathcal{B}$ is called
a {\em choice function} if $C(x)\leq x$ for all $x\in \mathcal{B}$. 
Three important properties are identified in the literature 
(Alkan and Gale 2003):
\citeH{AG03stab}%

\begin{itemize}
\item
$C$ is called {\em consistent} if \ 
$C(x)\leq y\leq x$ implies $C(y)=C(x)$,

\item
$C$ is called {\em persistent} if \ 
$x\geq y$ implies $y \wedge C(x)\leq C(y)$,

\item
$C$ is called {\em size-monotone} if \  
$x\geq y$ implies $|C(x)|\geq |C(y)|$,
where $\displaystyle |C(x)|=\sum_{i \in N}C(x)_{i}$.
\end{itemize}

\begin{remark} \rm  \label{RMsignif3condZ}
Alkan and Gale (2003)\citeH{AG03stab}
considered the stable allocation model
that extends the stable matching model of 
Alkan (2002)\citeH{Alk02}.
If the choice functions are consistent and persistent,
the set of stable allocations is nonempty and forms a lattice.
Moreover, if the choice functions are also size-monotone, 
the lattice of stable allocations is distributive and has
several significant properties, called 
polarity, complementarity, and uni-size property.
\finbox
\end{remark}

For a given function 
$f: \ZZ\sp{N} \to \RR \cup \{ -\infty \}$
we define
\begin{equation}  \label{choiceByValZ}
C(z)= C(z; f)= \argmax \{ f(y) \mid y\leq z \} .
\end{equation}
In general, the maximizer may not be unique, and hence
$C(z;f) \subseteq \ZZ\sp{N}$.
We also have the possibility of $C(z;f) = \emptyset$
to express the nonexistence of a maximizer.

An important property of M$\sp{\natural}$-concave functions,
closely related to persistence,
is found in 
Lemma 1 of Eguchi et al.~(2003);
\citeH[Lemma 1]{EFT03}%
see also 
Lemma 5.2 of Fujishige and Tamura (2006).
\citeH[Lemma 5.2]{FT06market}%

\begin{theorem}
\label{THchosubstMnatZ}
Let $f : \ZZ\sp{N} \rightarrow \RR \cup \{ -\infty\}$ be 
an M$\sp{\natural}$-concave function.
Then the following hold.
\begin{description}
\item[(SC$\sp{1}$\hbox{[$\ZZ$]})]
 For any $z_{1},z_{2}\in \ZZ\sp{N}$ 
with  $z_{1}\geq z_{2}$ and
$C(z_{2}; f) \neq \emptyset$
and for any 
$x_{1} \in C(z_{1};f)$,
there exists 
$x_{2} \in C(z_{2}; f)$
such that
$ z_{2} \wedge x_{1}\leq x_{2}$.

\item[(SC$\sp{2}$\hbox{[$\ZZ$]})]
 For any $z_{1},z_{2}\in \ZZ\sp{N}$ with $z_{1}\geq z_{2}$ and
$C(z_{1}; f) \neq \emptyset$
and for any 
$x_{2} \in C(z_{2}; f)$,
there exists 
$x_{1} \in C(z_{1};f)$
such that
$z_{2} \wedge x_{1}\leq x_{2}$.
\end{description}
\end{theorem}

\begin{proof}
Assume $z_{1} \geq z_{2}$.
For 
$x_{1} \leq  z_{1}$ and $x_{2} \leq  z_{2}$
define
\[
\varPhi(x_{1}, x_{2}) = 
\sum\{(x_{1})_{i}- (x_{2})_{i} \mid i \in \suppp( (z_{2} \wedge x_{1})-x_{2}) \}.
\]

Proof of (SC$\sp{1}$[$\ZZ$]):
Let $x_{1} \in C(z_{1};f)$
and take
$x_{2} \in C(z_{2};f)$ 
with minimum $\varPhi(x_{1}, x_{2})$. 
To prove by contradiction, suppose that there exists 
$i \in \suppp( (z_{2} \wedge x_{1})-x_{2})$.
Since
$i \in \suppp( x_{1} - x_{2})$, 
(M$\sp{\natural}$-EXC[$\ZZ$]) implies
there exists 
$j \in \suppm( x_{1} - x_{2}) \cup \{ 0 \}$ 
such that
\begin{equation*}
  f(x_{1}) + f(x_{2}) \leq 
   f(x_{1}-\chi_{i}+\chi_{j}) + f(x_{2}+\chi_{i}-\chi_{j}).
\end{equation*}
Here we have $x_{1}-\chi_{i}+\chi_{j} \leq  z_{1}$ 
and $x_{2}+\chi_{i}-\chi_{j}  \leq  z_{2}$;
the former is obvious if $j=0$ and otherwise,
it follows from 
$(x_{1})_{j} < (x_{2})_{j} \leq (z_{2})_{j}\leq (z_{1})_{j}$,
and the latter follows from $(x_{2})_{i} < (z_{2})_{i}$.
This implies that
$f( x_{1}-\chi_{i}+\chi_{j}) \leq f( x_{1})$
and
$f(x_{2}+\chi_{i}-\chi_{j} ) \leq f(x_{2})$
since $x_{1} \in C(z_{1};f)$ and $x_{2} \in C(z_{2};f)$.
Therefore,
the inequalities are in fact equalities, and
$x_{1}-\chi_{i}+\chi_{j} \in C(z_{1};f)$ 
and $x_{2}+\chi_{i}-\chi_{j} \in C(z_{2};f)$.
But we have
$\varPhi(x_{1}, x_{2}+\chi_{i}-\chi_{j}) = \varPhi(x_{1}, x_{2}) - 1$,
which contradicts the choice of $x_{2}$.

Proof of (SC$\sp{2}$[$\ZZ$]):
Let $x_{2} \in C(z_{2};f)$
and take
$x_{1} \in C(z_{1};f)$ 
with minimum $\varPhi(x_{1}, x_{2})$. 
By the same argument as above we obtain
$x_{1}-\chi_{i}+\chi_{j}  \in C(z_{1};f)$ with
$\varPhi(x_{1}-\chi_{i}+\chi_{j}, x_{2}) = \varPhi(x_{1}, x_{2}) - 1$.
This is a contradiction to the choice of $x_{1}$.
\end{proof}

When the maximizer is unique in (\ref{choiceByValZ}) for every $z$,  we say
that $f$ is {\em unique-selecting}.
In the following we  assume that $f$ is unique-selecting
and 
\begin{equation}  \label{choicedomfZ}
\bm{0} \in \dom f \subseteq \mathbb{Z}_{+}^{N}.
\end{equation}
Then $C$ in (\ref{choiceByValZ}) can be regarded as a choice function
$C:\mathcal{B}\to \mathcal{B}$.

The induced choice function $C$ is obviously consistent
for any valuation function $f$.
For persistence, 
M$\sp{\natural}$-concavity plays an essential role.
The following theorem of 
Eguchi et al.~(2003)\citeH{EFT03}
can be obtained as a corollary of Theorem~\ref{THchosubstMnatZ},
since for unique-selecting valuation functions, 
 {\rm (SC$\sp{1}$[$\ZZ$])} and {\rm (SC$\sp{2}$[$\ZZ$])} are equivalent
and  both coincide with persistence.

\begin{theorem} \label{THchopersMnatZuniq}
Every unique-selecting M$\sp{\natural}$-concave function
$f: \ZZ\sp{N} \to \RR \cup \{ -\infty \}$
with {\rm (\ref{choicedomfZ})}
induces a persistent choice function.
\end{theorem}

The size-monotonicity is also implied by M$\sp{\natural}$-concavity 
(Murota and Yokoi 2015).
\citeH{MY15mor}%

\begin{theorem} \label{THchomonoMnatZuniq}
Every unique-selecting M$\sp{\natural}$-concave function
$f: \ZZ\sp{N} \to \RR \cup \{ -\infty \}$
with {\rm (\ref{choicedomfZ})}
induces a size-monotone choice function.
\end{theorem}
\begin{proof}
The proof is based on the exchange property 
(\ref{mnatexcsizeZ1}) in Remark~\ref{RMmnatexcsizeZ}.
To prove by contradiction, suppose that there exist 
$x, y\in \ZZ^{N}$ such that $x\geq y$ and
$|C(x;f)|<|C(y;f)|$.
Set $x\sp{*} = C(x;f)$ and $y\sp{*} = C(y;f)$. Then $|x\sp{*}|<|y\sp{*}|$.
By the exchange property (\ref{mnatexcsizeZ1}) 
there exists  $j \in \suppm(x\sp{*} - y\sp{*})$
such that
$f( x\sp{*})+f( y\sp{*}) \leq f(x\sp{*}+ \chi_{j})+f(y\sp{*}- \chi_{j})$. 
Here we have 
$f(x\sp{*}+ \chi_{j})<f(x\sp{*})$
since 
$x\sp{*}+ \chi_{j} \leq x$
by $x\sp{*}_{j}<y\sp{*}_{j}\leq y_{j}\leq x_{j}$
and $x\sp{*}$ 
is the unique maximizer.
We also have
$f(y\sp{*}- \chi_{j})<f(y\sp{*})$
since 
$y\sp{*}- \chi_{j}\leq y\sp{*}\leq y$
and $y\sp{*}$ is the unique maximizer.
This is a contradiction.
\end{proof}

Thus, M$\sp{\natural}$-concave valuation functions 
entail the three desired properties,
consistency,  persistence, and size-monotonicity%
\footnote{
Theorem~\ref{THchoiceMnatZ} can be extended to 
quasi M$\sp{\natural}$-concave value functions; see  
Murota and Yokoi (2015)\citeH{MY15mor}.
}. 
Recall Remark \ref{RMsignif3condZ} for the implications of this fact.


\begin{theorem} \label{THchoiceMnatZ}
For a unique-selecting M$\sp{\natural}$-concave value function
$f: \ZZ\sp{N} \to \RR \cup \{ -\infty \}$
with {\rm (\ref{choicedomfZ})},
the choice function $C$ induced from $f$ is 
consistent,  persistent, and size-monotone.
\end{theorem}

 Finally, we mention a theorem that 
characterizes  M$\sp{\natural}$-concavity 
in terms of a parametrized version of 
(SC$\sp{1}$[$\ZZ$]) and (SC$\sp{2}$[$\ZZ$]).
Recall from (\ref{f-pdefZ}) the 
notation $f[-p](x) = f(x) - p\sp{\top} x$ for 
$p \in \RR\sp{N}$ and $x \in \ZZ\sp{N}$.
If $f$ is an M$\sp{\natural}$-concave function
 (not assumed to be unique-selecting),
$f[-p]$ is also M$\sp{\natural}$-concave, and hence
is equipped with the properties (SC$\sp{1}$[$\ZZ$]) and (SC$\sp{2}$[$\ZZ$])
by Theorem~\ref{THchosubstMnatZ}.
 In other words, an M$\sp{\natural}$-concave function $f$
has the following properties.

\begin{description}
\item[(SC$\sp{1}_{\rm\bf {G}}$\hbox{[$\ZZ$]})]  
For any $p\in \RR^{N}$, $f[-p]$ satisfies {\rm (SC$\sp{1}$[$\ZZ$])}.

\item[(SC$\sp{2}_{\rm\bf {G}}$\hbox{[$\ZZ$]})]  
For any $p\in \RR^{N}$, $f[-p]$ satisfies {\rm (SC$\sp{2}$[$\ZZ$])}. 
\end{description}
The following theorem, due to 
Farooq and Shioura (2005)\citeH{FS05sbst},
states that each of these conditions characterizes M$\sp{\natural}$-concavity.

\begin{theorem} \label{THchoMnatZ}
For a function 
$f: \ZZ\sp{N} \to \RR \cup \{ -\infty \}$ 
with a bounded nonempty effective domain,
we have the equivalence:
$f$ is M$\sp{\natural}$-concave 
$\iff$ 
{\rm (SC$\sp{1}_{\rm {G}}$[$\ZZ$])}
$\iff$ 
{\rm (SC$\sp{2}_{\rm {G}}$[$\ZZ$])}.
\end{theorem}

\subsection{Twisted M$\sp{\natural}$-concavity}
\label{SCtwistMnatZ}

Let $W$ be a subset of $N$.
For any vector $x \in \ZZ\sp{N}$
we define
$\mathrm{tw}(x) \in \ZZ\sp{N}$ by specifying its $i$th component 
$\mathrm{tw}(x)_{i}$ as
\begin{equation}  \label{twistsetdefZ}
\mathrm{tw}(x)_{i}  =
   \left\{  \begin{array}{ll}
   \phantom{-} x_{i}     &   (i \in N \setminus W),  \\
   - x_{i}          &   (i  \in W) .     \\
                     \end{array}  \right.
\end{equation}
A function $f: \ZZ\sp{N} \to \RR \cup \{ -\infty \}$ 
is said to be a {\em twisted M$\sp{\natural}$-concave function}
with respect to $W$,
if the function $\tilde{f}: \ZZ\sp{N} \to \RR \cup \{ -\infty \}$ 
defined by
\begin{equation}\label{ftwistMnatdefZ}
\tilde{f}(x) = f(\mathrm{tw}(x))
\qquad (x \in \ZZ\sp{N}) 
\end{equation}
is an M$\sp{\natural}$-concave function 
(Ikebe and Tamura 2015)\citeH{IT15scnet}.
The same concept has been introduced by 
Shioura and Yang (2015)\citeH{SY15jorsj},
almost at the same time and independently, 
under the name of {\em GM-concave functions}. 
Note that $f$ is twisted M$\sp{\natural}$-concave with respect to $W$
if and only if 
it is twisted M$\sp{\natural}$-concave with respect to $U = N \setminus W$.

Mathematically, 
twisted M$\sp{\natural}$-concavity is equivalent to the original 
M$\sp{\natural}$-concavity through twisting, 
and all the properties and theorems
about M$\sp{\natural}$-concave functions 
can be translated into those about
twisted M$\sp{\natural}$-concave functions.
In such translations it is often adequate to define
the {\em twisted demand correspondence} as%
\footnote{
Note: $x \in \tilde{D}(p ; f)$  $\iff$  $\mathrm{tw}(x) \in D(p ; \tilde{f})$. 
}
\begin{equation} \label{DpdefZtwist}
 \tilde{D}(p ; f) = \argmax_{x} \{ f(x) - \mathrm{tw}(p)\sp{\top} x  \}.
\end{equation}

A twisted version of (GS\&LAD$'$[$\ZZ$]) is introduced by 
Ikebe et al.~(2015)\citeH{ISST15}
as the {\em generalized full substitutes} (GFS[$\ZZ$]) condition:

\begin{description}
\item[(GFS\hbox{[$\ZZ$]})]
(i) For any $p \in \RR\sp{N}$, $\tilde{D}(p ; f)$ is a discrete convex set%
\footnote{That is,  $\tilde{D}(p ; f)$ should coincide with the integer points
contained in the convex hull of $\tilde{D}(p ; f)$.
}.
\\
(ii) For any $p \in \RR\sp{N}$,  $k \in U$,
$\delta > 0$, and  
$x \in \tilde{D}(p ; f)$, there exists $y \in \tilde{D}(p + \delta \chi_{k} ; f)$
such that 
\begin{equation} \label{twistMGFS1}
 x_{i} \leq y_{i} \ \   (\forall i \in U \setminus \{ k \}),
\quad 
 x_{i} \geq y_{i} \ \   (\forall i \in W),
\quad 
 x(U)  -  x(W)  \geq  y(U) - y(W). 
\end{equation}
(iii) For any $p \in \RR\sp{N}$, $k \in W$,
$\delta > 0$, and  
$x \in \tilde{D}(p ; f)$, there exists $y \in \tilde{D}(p - \delta \chi_{k} ; f)$
such that 
\begin{equation} \label{twistMGFS2}
 x_{i	} \leq y_{i} \ \   (\forall i \in W \setminus \{ k \}),
\quad 
 x_{i} \geq y_{i} \ \   (\forall i \in U),
\quad 
 x(W) - x(U) \geq y(W) - y(U).
\end{equation}
\end{description}
The following theorem%
\footnote{
Theorem~\ref{THtwistMnatISST} can be understood as
 a twisted version of the equivalence 
``{\rm (GS{\&}LAD$'$[$\ZZ$])} $\Leftrightarrow$ {\rm (M$\sp{\natural}$-EXC[$\ZZ$])}''
in Theorem~\ref{THmnatGSLAD}.
} 
(Ikebe et al.~2015)\citeH{ISST15}
characterizes twisted M$\sp{\natural}$-concavity
in terms of this condition.

\begin{theorem} \label{THtwistMnatISST}
Let $f: \ZZ\sp{N} \to \RR \cup \{ -\infty \}$ be a concave-extensible%
\footnote{
The concave-extensibility of $f$ is assumed here
for the consistency with the statement of Theorem~\ref{THmnatGSLAD}.
Mathematically, this assumption can be omitted,
since the condition (i) in (GFS[$\ZZ$])
is equivalent to the concave-extensibility of $f$
and twisted M$\sp{\natural}$-concave functions are concave-extensible.
Similarly in Theorem~\ref{THtwistMnatSY}.
} 
function with a bounded effective domain.  Then 
$f$ satisfies {\rm (GFS[$\ZZ$])} if and only if 
it is a twisted M$\sp{\natural}$-concave function with respect to $W$.
\end{theorem}

In the modeling of a trading network (supply chain network),
where an agent is identified with a vertex (node) of the network,
each vertex (agent) is associated with 
a valuation function $f$ defined 
on the set of arcs incident to the vertex.
Denoting the set of in-coming arcs to the vertex by $U$
and the set of out-going arcs from the vertex by $W$,
the function $f$ is defined on $U \cup W$.
Twisted M$\sp{\natural}$-concave functions
are used effectively in this context
(Ikebe and Tamura 2015\citeH{IT15scnet},
Ikebe et al.~2015\citeH{ISST15},
Candogan et al.~2016\citeH{CEV16}).
See Section~\ref{SCtradingnet}.

With the use of the ordinary (un-twisted) demand correspondence
\begin{equation} \label{DpdefZ2}
 D(p ; f) = \argmax_{x} \{ f(x) - p\sp{\top} x  \},
\end{equation}
a similar condition was formulated by 
Shioura and Yang (2015)\citeH{SY15jorsj},
independently of 
Ikebe et al.~(2015)\citeH{ISST15},
to deal with economies with two classes of indivisible goods
such that goods in the same class are substitutable and
goods across two classes are complementary.
The condition, called
the {\em generalized gross substitutes and complements} (GGSC[$\ZZ$]) condition,
reads as follows:

\begin{description}
\item[(GGSC\hbox{[$\ZZ$]})]
(i) For any $p \in \RR\sp{N}$, $D(p ; f)$ is a discrete convex set.
\\
(ii) For any $p \in \RR\sp{N}$,  $k \in U$,
 $\delta > 0$, and  
$x \in D(p ; f)$, there exists $y \in D(p + \delta \chi_{k} ; f)$
that satisfies (\ref{twistMGFS1}).
\\
(iii) For any $p \in \RR\sp{N}$,  $k \in W$,
 $\delta > 0$, and  
$x \in D(p ; f)$, there exists $y \in D(p + \delta \chi_{k} ; f)$
that satisfies (\ref{twistMGFS2}).
\end{description}

This condition also characterizes twisted M$\sp{\natural}$-concavity
(Shioura and Yang 2015)\citeH{SY15jorsj}.

\begin{theorem} \label{THtwistMnatSY}
Let $f: \ZZ\sp{N} \to \RR \cup \{ -\infty \}$ be a 
concave-extensible 
function with a bounded effective domain.
Then 
$f$ satisfies {\rm (GGSC[$\ZZ$])} 
if and only if 
it is a twisted M$\sp{\natural}$-concave function with respect to $W$.
\end{theorem}

Although Theorems \ref{THtwistMnatISST} and \ref{THtwistMnatSY}
have significances in different contexts,
they are in fact two variants of the same mathematical statement.
Note that (GSF[$\ZZ$]) and (GGSC[$\ZZ$]) are equivalent, since
\begin{align*}
& 
D(p ; f) = \tilde{D}(\mathrm{tw}(p) ; f) , 
\qquad 
\mathrm{tw}(p + \delta \chi_{k})  =
   \left\{  \begin{array}{ll}
     \mathrm{tw}(p) + \delta \chi_{k}     &   (k \in U),  \\
     \mathrm{tw}(p) - \delta \chi_{k}       &   (k  \in W) .     \\
                     \end{array}  \right.
\end{align*}

The multi-unit (or vector) version of 
the same-side substitutability (SSS) and the cross-side complementarity (CSC) of 
Ostrovsky (2008)
\citeH{Ost08}%
can be formulated for a correspondence
$C: \ZZ\sp{N} \to 2\sp{\ZZ\sp{N}}$ 
as follows,
where, for any $z \in \mathbb{Z}^{N}$,
the subvector of $z$ on $U$ is denoted  by $z\sp{U} \in \mathbb{Z}^{U}$
and similarly 
the subvector on $W$ by $z\sp{W} \in \mathbb{Z}^{W}$.

\begin{description}
\item[(SSS-CSC$\sp{1}$\hbox{[$\ZZ$]})]
(i)
 For any $z_{1},z_{2}\in \ZZ\sp{N}$ 
with  $z_{1}\sp{U} \geq z_{2}\sp{U}$,  $z_{1}\sp{W} = z_{2}\sp{W}$
and $C(z_{2}) \neq \emptyset$
and for any  
$x_{1} \in C(z_{1})$, 
there exists 
$x_{2} \in C(z_{2})$
 such that
$ z_{2}\sp{U} \wedge x_{1}\sp{U} \leq x_{2}\sp{U}$
and
$x_{1}\sp{W} \geq x_{2}\sp{W}$,
and (ii) the same statement with $U$ and $W$ interchanged.

\item[(SSS-CSC$\sp{2}$\hbox{[$\ZZ$]})]
(i)
 For any $z_{1},z_{2}\in \ZZ\sp{N}$ 
with  $z_{1}\sp{U} \geq z_{2}\sp{U}$,  $z_{1}\sp{W} = z_{2}\sp{W}$
and $C(z_{1}) \neq \emptyset$
and for any  
$x_{2} \in C(z_{2})$
there exists 
$x_{1} \in C(z_{1})$, 
 such that
$ z_{2}\sp{U} \wedge x_{1}\sp{U} \leq x_{2}\sp{U}$
and
$x_{1}\sp{W} \geq x_{2}\sp{W}$,
and (ii) the same statement with $U$ and $W$ interchanged.
\end{description}

The following theorem  
(Ikebe and Tamura 2015)\citeH{IT15scnet} 
states that these two properties are
implied by twisted M$\sp{\natural}$-concavity.
Recall from (\ref{choiceByValZ}) 
that a valuation function $f$ induces
the  correspondence%
\footnote{
It may be that $C(z)= \emptyset$ 
if $\dom f$ is unbounded below or  
$\{ y \mid y\leq z \} \cap \dom f =\emptyset$.
The condition ``$C(z_{2}) \neq \emptyset$''
in (SSS-CSC$\sp{1}$[$\ZZ$]), for example, takes care of this possibility.
} 
$C(z)= C(z; f)= \argmax \{ f(y) \mid y\leq z \}$ $(z \in \ZZ\sp{N})$.

\begin{theorem} \label{THchotwistMnatZ}
For any twisted M$\sp{\natural}$-concave function 
$f: \ZZ\sp{N} \to \RR \cup \{ -\infty \}$,
the induced correspondence $C$ has the properties 
{\rm (SSS-CSC$\sp{1}$[$\ZZ$])} and {\rm (SSS-CSC$\sp{2}$[$\ZZ$])}.
\end{theorem}

\begin{proof}
We prove (SSS-CSC$\sp{1}$[$\ZZ$])-(i) and (SSS-CSC$\sp{2}$[$\ZZ$])-(i);
the proofs of (SSS-CSC$\sp{1}$[$\ZZ$])-(ii) and (SSS-CSC$\sp{2}$[$\ZZ$])-(ii)
are obtained by interchanging $U$ and $W$.
Assume 
 $z_{1}\sp{U} \geq z_{2}\sp{U}$,  $z_{1}\sp{W} = z_{2}\sp{W}$
and $C(z_{1}; f) \neq \emptyset$,
and let
 $\tilde{f}$ be the M$\sp{\natural}$-concave function
in (\ref{ftwistMnatdefZ}) associated with $f$.
For 
$x_{1} \leq  z_{1}$ and $x_{2} \leq  z_{2}$
define
\begin{align*} 
\varPhi(x_{1}, x_{2}) =& 
\sum\{(x_{1})_{i}- (x_{2})_{i} 
   \mid i \in U \cap \suppp( (z_{2} \wedge x_{1})-x_{2}) \}
\\ &
+ \sum\{(x_{2})_{i}- (x_{1})_{i} 
   \mid i \in W \cap \suppp( x_{2} -x_{1}) \}.
\end{align*}

Proof of (SSS-CSC$\sp{1}$[$\ZZ$])-(i): \ 
Let $x_{1} \in C(z_{1};f)$
and take
$x_{2} \in C(z_{2};f)$ with
$\varPhi(x_{1}, x_{2})$
minimum.
To prove by contradiction, suppose that there exists 
\[
i \in 
\big( 
 U \cap \suppp( (z_{2} \wedge x_{1})-x_{2})
\big)
\cup \big( 
 W \cap \suppp( x_{2} -x_{1}) 
\big).
\]
Then 
$i \in \suppp (\mathrm{tw}(x_{1}) - \mathrm{tw}(x_{2}) )$,
and 
(M$\sp{\natural}$-EXC[$\ZZ$]) for $\tilde{f}$ implies that
there exists 
$j \in \suppm( \mathrm{tw}(x_{1}) - \mathrm{tw}(x_{2}) ) \cup \{ 0 \}$
such that
\[
\tilde{f}( \mathrm{tw}(x_{1}) ) + \tilde{f}(  \mathrm{tw}(x_{2})) 
\leq  \tilde{f}( \mathrm{tw}(x_{1}) - \chi_{i} + \chi_{j}) 
      + \tilde{f}( \mathrm{tw}(x_{2}) + \chi_{i} -\chi_{j}).
\]
Letting 
$\hat{x}_{1} =   \mathrm{tw}( \mathrm{tw}(x_{1}) - \chi_{i} + \chi_{j}) )$
and
$\hat{x}_{2} =   \mathrm{tw}( \mathrm{tw}(x_{2}) + \chi_{i} -\chi_{j}) )$
we can express the above inequality as
\[
f( x_{1} ) + f( x_{2} ) \leq f( \hat{x}_{1} ) + f( \hat{x}_{2} ) .
\]
By considering all possibilities
($i \in U$ or $i \in W$, and 
$j \in U$ or $j \in W$ or $j=0$),
we can verify that $\hat{x}_{1} \leq z_{1}$ 
and $\hat{x}_{2} \leq z_{2}$, 
from which follow
$f( \hat{x}_{1} ) \leq f( x_{1})$
and
$f( \hat{x}_{2} ) \leq f(x_{2})$
since  $x_{1} \in C(z_{1};f)$
and $x_{2} \in C(z_{2};f)$.
Therefore,
the inequalities are in fact equalities, and
$\hat{x}_{1} \in C(z_{1};f)$ and $\hat{x}_{2} \in C(z_{2};f)$.
But we have
$\varPhi(x_{1}, \hat{x}_{2}) = \varPhi(x_{1}, x_{2}) - 1$,
which contradicts the choice of $x_{2}$.

Proof of (SSS-CSC$\sp{2}$[$\ZZ$])-(i): \ 
Let $x_{2} \in C(z_{2};f)$
and take
$x_{1} \in C(z_{1};f)$ 
with minimum $\varPhi(x_{1}, x_{2})$. 
By the same argument as above we obtain
$\hat{x}_{1} \in C(z_{1};f)$ 
with
$\varPhi(\hat{x}_{1}, x_{2}) = \varPhi(x_{1}, x_{2}) - 1$.
This is a contradiction to the choice of $x_{1}$.
\end{proof}

\subsection{Examples}
\label{SCmnatexampleZ}

Here are some examples of M$\sp{\natural}$-concave functions 
in integer variables.

\begin{enumerate}
\item
A linear (or affine) function
\begin{equation} \label{mlinfnZ}
 f(x)=     \alpha + \langle p,x \rangle  
\end{equation}
with  $p \in \RR\sp{N}$ and $\alpha \in \RR$
is M$\sp{\natural}$-concave
if $\dom f$ is an M$\sp{\natural}$-convex set.

\item
A quadratic function
$f: \ZZ\sp{N} \to \RR$
defined by 
\begin{equation} \label{mfnquad2Z}
 f(x) = \sum_{i= 1}\sp{n} \sum_{j = 1}\sp{n} a_{ij} x_{i} x_{j}
\end{equation}
with $a_{ij}=a_{ji} \in \RR$ $(i,j=1,\ldots,n)$
is M$\sp{\natural}$-concave if and only if
\begin{equation} \label{mfnquadhesseZ}
 a_{ij}  \leq 0  \mbox{ \  for all  $(i,j)$, \ \  and} \quad
 a_{ij}  \leq \max ( a_{ik} ,  a_{jk} ) 
 \ \mbox{ if } \  
 \{ i,j \} \cap \{ k \} = \emptyset .
\end{equation}
The Hessian matrix $H_{f}(x)=(H_{ij} (x))$ defined in (\ref{mhessianDefZ}) 
is given by $H_{ij} (x) = 2 a_{ij}$,
and (\ref{mfnquadhesseZ}) above is consistent with 
(\ref{mnatfnhesseZ1}), (\ref{mnatfnhesseZ2}) in Theorem \ref{Mmnathessian}.

\item
A function 
$f: \ZZ\sp{N} \to \RR \cup \{ -\infty \}$
is called {\em separable concave} if it can be represented as
\begin{equation}  \label{separconvZ}
 f(x)  =  \sum_{i \in N} \varphi_{i}(x_{i})
\qquad (x \in \ZZ\sp{N})
\end{equation}
for univariate concave functions%
\footnote{
Recall that $\varphi: \ZZ \to \RR \cup \{ -\infty \}$
is called concave if 
$\varphi(t-1) + \varphi(t+1) \leq 2\varphi(t)$ for all integers $t$.
} 
$\varphi_{i}: \ZZ \to \RR \cup \{ -\infty \}$
$(i \in N)$.
A separable concave function is M$\sp{\natural}$-concave.
In (\ref{mnatconcavexc2Z}) for (M$\sp{\natural}$-EXC[$\ZZ$]) 
we can always take $j=0$, i.e., (\ref{mconcav1Z}).

\item
A function 
$f: \ZZ\sp{N} \to \RR \cup \{ -\infty \}$
is called {\em laminar concave} if it can be represented as
\begin{equation}  \label{laminarconvZ}
 f(x)  =  \sum_{A \in \calT} \varphi_{A}(x(A))
\qquad (x \in \ZZ\sp{N})
\end{equation}
for a laminar family $\calT \subseteq 2\sp{N}$ 
and a family of univariate concave functions
$\varphi_{A}: \ZZ \to \RR \cup \{ -\infty \}$
indexed by $A \in \calT$,
where $x(A) = \sum_{i \in A} x_{i}$.
A laminar concave function is M$\sp{\natural}$-concave; see 
Note 6.11 of Murota (2003) 
\citeH[Note 6.11]{Mdcasiam}%
for a proof.  A special case of
(\ref{laminarconvZ}) with 
$\calT = \{ \{ 1 \},   \{ 2 \},  \ldots, \{ n \}   \}$ 
reduces to the separable convex function (\ref{separconvZ}).

\item
M$\sp{\natural}$-concave functions arise from 
the maximum weight of nonlinear network flows.
Let $G = (V, A)$ be a directed graph with 
two disjoint vertex subsets $S \subseteq V$ and $T \subseteq V$
specified as the entrance and the exit.
 Suppose that, for each arc $a \in A$, 
we are given a univariate concave function 
$\varphi_{a}: \ZZ \to \RR \cup \{ -\infty \}$
representing the weight of flow on the arc $a$.
Let $\xi \in \ZZ^A$ be a vector representing an integer flow, and
$\partial \xi \in \ZZ\sp{V}$ be the boundary of flow $\xi$ defined  by
\begin{equation}  \label{flowbounddefZ}
\begin{array}{l}
 \partial \xi(v) = \sum\{ \xi(a) \mid
 \mbox{ arc $a$ leaves $v$ } \} 
 - \sum\{\xi(a) \mid \mbox{ arc $a$ enters $v$ } \}
 \quad (v \in V).
\end{array}
\end{equation}
 Then, the maximum weight of a flow that realizes a supply/demand specification
on the exit $T$  in terms of $x \in \ZZ^T$ is expressed by
\begin{equation}  \label{weightnetworkflowZ}
  f(x) = \sup_{\xi} \big\{\sum_{a \in A} \varphi_a(\xi(a)) \mid
     (\partial \xi)(v) = - x(v) \ (v \in T), \ 
     (\partial \xi)(v) = 0\ (v \in  V \setminus (S \cup T))
       \big \} ,
\end{equation}
where no constraint is imposed on 
$(\partial \xi)(v)$ for entrance vertices $v \in S$.
This function is M$\sp{\natural}$-concave, 
provided that $f$ does not take the value $+\infty$ and 
$\dom f$ is nonempty.
If $S = \emptyset$, the function $f$ is M-concave, since
$\sum_{v \in T} x(v) = - \sum_{v \in T} (\partial \xi)(v) 
= \sum_{v \in V \setminus T} (\partial \xi)(v)  = 0$ 
in this case.
See 
Example 2.3 of Murota (1998)\citeH[Example 2.3]{Mdca}
and Section 2.2.2 of Murota (2003)\citeH[Section 2.2.2]{Mdcasiam}
for details.
The maximum weight of a matching in (\ref{weightbipartmatch}) 
can be understood as a special case of (\ref{weightnetworkflowZ}).
\end{enumerate}

\subsection{Concluding remarks of section \ref{SCmnatconcavZ}}
\label{SCconremmnatZ}

The concept of M-convex functions is formulated by
Murota (1996c)\citeH{Mstein}
as a generalization of valuated matroids of
Dress and Wenzel (1990, 1992)\citeH{DW90}\citeH{DW92}.
Then
M$\sp{\natural}$-convex functions are introduced by
Murota and Shioura (1999)\citeH{MS99gp}
as a variant of M-convex functions.
Quasi M-convex functions are introduced by
Murota and Shioura (2003)\citeH{MS03quasi}.
The concept of M-convex functions 
is extended to functions on jump systems by 
Murota (2006)\citeH{Mmjump06}; 
see also 
Kobayashi et al.~(2007)\citeH{KMT07jump}.

Unimodularity is closely related to discrete convexity.
For a fixed unimodular matrix $U$ we may consider 
a change of variables $x \mapsto Ux$ for $x \in \ZZ\sp{n}$
to define a class of functions 
$\{ f(U x) \mid f: \mbox{M$\sp{\natural}$-concave} \}$
as a variant of M$\sp{\natural}$-concave functions.
Twisted M$\sp{\natural}$-concave functions 
(Section \ref{SCtwistMnatZ}) 
are a typical example of this construction
with $U = {\rm diag}(1,\ldots,1, -1,\ldots,-1)$;  see 
Sun and Yang (2008)\citeH{SY08}
and Section \ref{SCunimodular} 
for further discussion in this direction.


\section{M$\sp{\natural}$-concave Function on $\RR\sp{n}$}
\label{SCmnatconcavR}

In Sections \ref{SCmnatconcav01} and  \ref{SCmnatconcavZ},
we have considered M$\sp{\natural}$-concave functions
on $2\sp{N}$ and $\ZZ\sp{N}$,
which correspond to valuations for indivisible goods with substitutability.
In this section we deal with M$\sp{\natural}$-concave functions 
in real vectors,
$f: \mathbb{R}\sp{N} \to \RR \cup \{ -\infty \}$,
which correspond to valuations for divisible goods with substitutability.
M$\sp{\natural}$-concave functions in real variables are investigated by
Murota and Shioura (2000, 2004a, 2004b).
\citeH{MS00poly}\citeH{MS04conjreal}\citeH{MS04fund}%

\subsection{Exchange property}
\label{SCexchangeR}

We say that a function 
$f: \mathbb{R}\sp{N} \to \RR \cup \{ -\infty \}$
is 
{\em M$\sp{\natural}$-concave}
if it is a concave function (in the ordinary sense) that satisfies
\begin{description}
\item[(M$\sp{\natural}$-EXC\hbox{[$\RR$]})] 
 For any $x, y \in \RR\sp{N}$ and $i \in \suppp(x-y)$, 
there exist $j \in \suppm(x-y) \cup \{ 0 \}$
and a positive number $\alpha_{0} \in \RR_{++}$ such that 
\begin{equation} \label{mnatconcavexc2R}
 f(x)+f(y) \leq 
   f(x-\alpha (\unitvec{i}-\unitvec{j})) + f(y+\alpha (\unitvec{i}-\unitvec{j})) 
\end{equation}
for all $\alpha \in \RR$ with $0 \leq \alpha \leq \alpha_{0}$.
\end{description}
In the following we restrict ourselves to closed proper%
\footnote{
A concave function $f: \mathbb{R}\sp{n} \to \RR \cup \{ -\infty \}$
is said to be {\em proper} if $\dom f$ is nonempty, and 
{\em closed} if the hypograph
$\{ (x,\beta) \in \RR\sp{n+1} \mid \beta \leq f(x) \}$
is a closed subset of $\RR\sp{n+1}$.
} 
 M$\sp{\natural}$-concave functions,
for which the closure of the effective domain $\dom f$ is a well-behaved polyhedron
(g-polymatroid, or M$\sp{\natural}$-convex polyhedron%
\footnote{
\label{FTmnatpolyhed}%
A polyhedron $P$ is called an {\em M$\sp{\natural}$-convex polyhedron}
if its (concave) indicator function $f$ is M$\sp{\natural}$-concave,
where $f(x)=0$ for $x \in P$ and $=-\infty$ for $x \not\in P$.
See Section 4.8 of Murota (2003)\citeH[Section 4.8]{Mdcasiam} 
for details.
});   
see Theorem 3.2 of Murota and Shioura (2008)\citeH[Theorem 3.2]{MS08cont}.
Often we are interested in polyhedral M$\sp{\natural}$-concave functions.

\begin{remark} \rm  \label{RMmnatexcsizeR}
It follows from (M$\sp{\natural}$-EXC[$\RR$]) that
M$\sp{\natural}$-concave functions enjoy the following 
exchange properties under size constraints:
\\ $\bullet$ 
For any $x, y \in \RR\sp{N}$ with $x(N) < y(N)$, 
there exists $\alpha_{0} \in \RR_{++}$ such that 
\begin{align}
f(x) + f(y)   &\leq 
 \max_{j \in \suppm(x-y)}  
 \{ f(x + \alpha  \unitvec{j}) + f(y- \alpha  \unitvec{j})  \}
\label{mnatexcsizeR1}
\end{align}
for all $\alpha \in \RR$ with $0 \leq \alpha \leq \alpha_{0}$.
\\ $\bullet$ 
For any $x, y \in \RR\sp{N}$ with $x(N) =y(N)$
and $i \in \suppp(x-y)$, 
there exists $\alpha_{0} \in \RR_{++}$ such that 
\begin{align}
f(x) + f(y)   &\leq 
 \max_{j \in \suppm(x-y)}  
 \{ f(x-\alpha (\unitvec{i}-\unitvec{j})) + f(y+\alpha (\unitvec{i}-\unitvec{j})) \}
\label{mnatexcsizeR2}
\end{align}
for all $\alpha \in \RR$ with $0 \leq \alpha \leq \alpha_{0}$.
\finbox
\end{remark}

\begin{remark} \rm  \label{RMmconcaveR}
If $\dom f \subseteq \RR\sp{N}$ lies in a hyperplane with a constant component sum
(i.e., $x(N) = y(N)$ for all  $x, y \in \dom f$),
the exchange property (M$\sp{\natural}$-EXC[$\RR$]) takes a simpler form
excluding the possibility of $j=0$.
A function 
$f: \mathbb{R}\sp{N} \to \RR \cup \{ -\infty \}$
having this exchange property 
is called an {\em M-concave function}.
That is, a concave function $f$ 
is M-concave
if and only if (\ref{mnatexcsizeR2}) holds.
\finbox
\end{remark}

\subsection{Maximizers and gross substitutability}
\label{SCmaximizersR}

For $p \in \RR\sp{N}$ we denote
the set of the  maximizers 
of $f[-p](x) = f(x) - p\sp{\top} x$
by $D(p ; f) \subseteq \RR\sp{N}$
(cf.~(\ref{DpdefZ})).
M$\sp{\natural}$-concavity of a function $f$ 
is characterized by the M$\sp{\natural}$-convexity of $D(p ; f)$
(Theorem 5.2 of Murota and Shioura 2000).
\citeH[Theorem 5.2]{MS00poly}%

\begin{theorem} \label{THmconcavargmaxR}
A polyhedral concave function 
$f: \RR\sp{N} \to \RR \cup \{ -\infty \}$ 
is M$\sp{\natural}$-concave
if and only if, for every vector $p \in \RR\sp{N}$, 
$D(p ; f)$ is an M$\sp{\natural}$-convex polyhedron%
\footnote{
See the footnote \ref{FTmnatpolyhed}.
}. 
\end{theorem}

\begin{description}
\item[(GS\hbox{[$\RR$]})]
For any $p,q  \in \RR\sp{N}$ 
with $p \leq q$ and $x \in D(p ; f)$,
 there exists $y \in D(q ; f)$ such that
$x_{i} \le y_{i}$ for all $i\in N$ with $p_{i}=q_{i}$.
\end{description}

The following theorem is given by
Danilov et al.~(2003).
\citeH{DKL03gr}%

\begin{theorem} \label{THmnatgsRpoly}
A polyhedral M$\sp{\natural}$-concave function 
$f: \RR\sp{N} \to \RR \cup \{ -\infty \}$ 
with a bounded effective domain satisfies {\rm (GS[$\RR$])}.
\end{theorem}
\begin{proof}
This follows from Theorem \ref{THmconcavebyconjfnRR} (2) and Theorem \ref{THgssubmRpoly}
in Section \ref{SCconjugacyR}.
\end{proof}

\begin{example} \rm \label{EXgsRnotMnat}
Here is an example to show 
that (GS[$\RR$]) does not imply M$\sp{\natural}$-concavity.
Let
$f: \mathbb{R}^{2} \to \mathbb{R}$
be defined by $f(x_{1}, x_{2})  = \min ( 2, x_{1} + 2 x_{2} )$ 
on $\dom f =  \mathbb{R}^{2}$.
This function is not M$\sp{\natural}$-concave because
(M$\sp{\natural}$-EXC[$\RR$]) fails for $x=(2,0)$, $y=(0,1)$ and $i=1$.
However, it satisfies (GS[$\RR$]), which can be verified easily. 
Thus the converse of Theorem \ref{THmnatgsRpoly} does not hold.
\finbox
\end{example}

\subsection{Choice function}
\label{SCchoiceMnatR}

In Theorem \ref{THchoiceMnatZ} in Section \ref{SCchoiceMnatZ}
we have seen, for the multi-unit indivisible goods,
the choice function induced from
a unique-selecting M$\sp{\natural}$-concave value function
is consistent,  persistent, and size-monotone
in the sense of 
Alkan and Gale (2003).
\citeH{AG03stab}%
In this section we point out that this is also the case with divisible goods;
recall Remark \ref{RMsignif3condZ} in Section \ref{SCchoiceMnatZ} 
for the implications of this fact.

For a choice function $C:\mathcal{B} \to \mathcal{B}$
with $\mathcal{B} = \{ x\in \mathbb{R}_{+}^{N} \mid  x \leq b \}$
for some $b\in \mathbb{R}_{+}^{N}$,
consistency means
$[\ C(x)\leq y\leq x \Rightarrow C(y)=C(x)\ ]$,
persistence means
$[\ x\geq y \Rightarrow y \wedge C(x)\leq C(y)\ ]$,
and size-monotonicity
means $[\ x\geq y \Rightarrow |C(x)|\geq |C(y)|\ ]$,
where
 $ |C(x)|=\sum_{i \in N}C(x)_{i}$ (sum of the components).

\begin{theorem} \label{THchoiceMnatR}
For a unique-selecting M$\sp{\natural}$-concave value function
$f: \mathbb{R}\sp{N} \to \RR \cup \{ -\infty \}$
with 
$\bm{0} \in \dom f \subseteq \mathbb{R}_{+}^{N}$,
the induced choice function 
$C(x;f)=\argmax \{f(y) \mid y\leq x \}$
is consistent,  persistent, and size-monotone%
\footnote{
As in Section \ref{SCchoiceMnatZ}, $f$ is said to be  unique-selecting
if $C(x;f)$ consists of a single element for every $x$.
}.  
\end{theorem}

\begin{proof}
The consistency is obvious from the definition of $C(x;f)$.

To prove persistence%
\footnote{
This proof for persistence is an adaptation of the one in 
Lemma 3.3 of Murota and Yokoi (2015)\citeH[Lemma 3.3]{MY15mor}.
} 
 by contradiction, suppose that 
$y\wedge C(x;f)\leq C(y;f)$ fails
for some $x,y\in \RR^{N}$ with $x\geq y$.
Set $x\sp{*}=C(x;f)$, $y\sp{*}=C(y;f)$.
Since $y\wedge x\sp{*}\leq y\sp{*}$ fails, 
there exists some $i \in N$ such that $y_{i} \wedge x\sp{*}_{i}>y\sp{*}_{i}$. 
Then $i\in \suppp(x\sp{*}-y\sp{*})$.
We apply (M$\sp{\natural}$-EXC[$\RR$]) to $x\sp{*}, y\sp{*}$ and $i$,
to obtain $j\in \suppm(x\sp{*}-y\sp{*})\cup\{0\}$ and $\alpha_{0} > 0$ such that
\begin{equation}
f(x\sp{*})+f(y\sp{*})\leq 
f(x\sp{*}- \alpha (\chi_{i}-\chi_{j}))+f(y\sp{*}+ \alpha (\chi_{i}-\chi_{j}))
\label{substituteR}
\end{equation}
for all $\alpha$ with $0 < \alpha \leq  \alpha_{0}$.
For sufficiently small $\alpha > 0$ we also have
$x\sp{*}- \alpha (\chi_{i}-\chi_{j}) \leq x$
and 
$y\sp{*}+ \alpha (\chi_{i}-\chi_{j}) \leq y$;
the former follows from 
$x\sp{*}_{j}<y\sp{*}_{j}\leq y_{j}\leq x_{j}$
for $j\in \suppm(x\sp{*}-y\sp{*})$,
and the latter from
$y\sp{*}_{i}<y_{i} \wedge x\sp{*}_{i}\leq y_{i}$.
On the right-hand side of (\ref{substituteR}),
we have
$f(x\sp{*}- \alpha (\chi_{i}-\chi_{j}))<f(x\sp{*})$
since
$x\sp{*}- \alpha (\chi_{i}-\chi_{j}) \leq x$
and  $x\sp{*}=C(x;f)$ is the unique maximizer of $f$ 
in 
$\{ z\in \RR^{N} \mid z\leq x \}$,
and similarly, $f(y\sp{*}+ \alpha (\chi_{i}-\chi_{j})) < f(y\sp{*})$.
This is a contradiction, proving persistence.

To prove size-monotonicity by contradiction, suppose that there exist 
$x, y\in \RR^{N}$ such that $x\geq y$ and
$|C(x;f)|<|C(y;f)|$.
Set $x\sp{*} = C(x;f)$ and $y\sp{*} = C(y;f)$. Then $|x\sp{*}|<|y\sp{*}|$.
By the exchange property (\ref{mnatexcsizeR1}) in Remark~\ref{RMmnatexcsizeR},
there exists  $j \in \suppm(x\sp{*} - y\sp{*})$
such that
$f( x\sp{*})+f( y\sp{*}) \leq f(x\sp{*}+ \alpha \chi_{j})+f(y\sp{*}- \alpha \chi_{j})$
for sufficiently small $\alpha > 0$.
Here we have 
$f(x\sp{*}+ \alpha \chi_{j})<f(x\sp{*})$
since 
$x\sp{*}+ \alpha \chi_{j} \leq x$
by $x\sp{*}_{j}<y\sp{*}_{j}\leq y_{j}\leq x_{j}$
and $x\sp{*}$ 
is the unique maximizer.
We also have
$f(y\sp{*}- \alpha \chi_{j})<f(y\sp{*})$
since 
$y\sp{*}- \alpha \chi_{j}\leq y\sp{*}\leq y$
and $y\sp{*}$ is the unique maximizer.
This is a contradiction, proving size-monotonicity.
\end{proof}

\subsection{Examples}
\label{SCmnatexampleR}

Here are some examples of M$\sp{\natural}$-concave functions 
in real variables.

\begin{enumerate}
\item

A function 
$f: \RR\sp{N} \to \RR \cup \{ -\infty \}$
 is called 
{\em laminar concave}
if it can be represented as
\begin{equation}  \label{laminarconvR}
 f(x)  =  \sum_{A \in \calT} \varphi_{A}(x(A))
\qquad (x \in \RR\sp{N})
\end{equation}
for a laminar family $\calT \subseteq 2\sp{N}$ 
and a family of univariate (closed proper) concave functions 
$\varphi_{A}: \RR \to \RR \cup \{ -\infty \}$ indexed by $A \in \calT$,
where $x(A) = \sum_{i \in A} x_{i}$.
A laminar concave function is M$\sp{\natural}$-concave.

\item
M$\sp{\natural}$-concave functions arise from 
the maximum weight of nonlinear network flows.
Let $G = (V, A)$ be a directed graph with 
two disjoint vertex subsets $S \subseteq V$ and $T \subseteq V$
specified as the entrance and the exit.
 Suppose that, for each arc $a \in A$, 
we are given a univariate (closed proper) concave function 
$\varphi_{a}: \RR \to \RR \cup \{ -\infty \}$
representing the weight of flow on the arc $a$.
Let $\xi \in \RR^A$ be a vector representing a flow, and
$\partial \xi \in \RR\sp{V}$ be the boundary of flow $\xi$ defined  by
(\ref{flowbounddefZ}).
 Then, the maximum weight of a flow that realizes a supply/demand specification
on the exit $T$  in terms of $x \in \RR^T$ is expressed by
a function $f: \RR^T \rightarrow \RR \cup \{ -\infty \}$ 
defined as (\ref{weightnetworkflowZ}).
This function is M$\sp{\natural}$-concave, 
provided that $f$ does not take the value $+\infty$ and 
$\dom f$ is nonempty.
If $S = \emptyset$, the function $f$ is M-concave.
See Section 2.2.1 of Murota (2003)\citeH[Section 2.2.1]{Mdcasiam}
and Theorem~2.10 of Murota and Shioura (2004a)\citeH[Theorem~2.10]{MS04conjreal}
for details.
\end{enumerate}

\subsection{Concluding remarks of section \ref{SCmnatconcavR}}

The concept of M-concave functions in continuous variables 
is introduced for polyhedral concave functions by
Murota and Shioura (2000)\citeH{MS00poly}
and for general concave functions by 
Murota and Shioura (2004a)\citeH{MS04conjreal}. 
This is partly motivated by a phenomenon inherent in the network flow/tension problem
described in Section \ref{SCmnatexampleR}.


\section{Operations for M$\sp{\natural}$-concave Functions}
\label{SCoperM}

\subsection{Basic operations}
\label{SCoperMbasic}

Basic operations on M$\sp{\natural}$-concave functions on $\ZZ\sp{n}$ are presented here,
whereas the most powerful operation,
transformation by networks,
is treated in Section~\ref{SCnetinduce}.

M$\sp{\natural}$-concave functions admit the following operations.

\begin{theorem}  \label{THmfnoperation}
Let $f, f_{1}, f_{2}: \ZZ\sp{N} \to \RR \cup \{ -\infty \}$ 
be M$\sp{\natural}$-concave functions.

\noindent {\rm (1)}
For nonnegative  $\alpha \in \RR_{+}$
and $\beta \in \RR$,
$\alpha f(x) + \beta$ is M$\sp{\natural}$-concave in $x$.

\noindent {\rm (2)} 
For $a \in \ZZ\sp{N}$,
$f(a-x)$ and $f(a+x)$ are M$\sp{\natural}$-concave in $x$. 

\noindent {\rm (3)} 
For $p \in \RR\sp{N}$, $f[-p]$ is M$\sp{\natural}$-concave,
where $f[-p]$ is defined by {\rm (\ref{f-pdefZ})}.

\noindent {\rm (4)} 
For univariate concave functions 
$\varphi_{i}: \ZZ \to \RR \cup \{ -\infty \}$ indexed by $i \in N$, 
\begin{equation} \label{f0sumphi}
 \tilde f(x) = f(x) + \sum_{i \in N} \varphi_{i}(x_{i})
\qquad (x \in \ZZ\sp{N})
\end{equation}
is M$\sp{\natural}$-concave, provided 
$\dom \tilde f \not= \emptyset$.

\noindent {\rm (5)} 
For $a \in (\ZZ \cup \{ -\infty \})\sp{N}$
and $b \in (\ZZ \cup \{ +\infty \})\sp{N}$,
the restriction of $f$ to the integer interval 
$[a,b]_{\ZZ} = \{ x \in \ZZ\sp{N} \mid a \leq x \leq b \}$
defined by
\begin{equation} \label{fboxrestrict}
 f_{[a,b]_{\ZZ}}(x) = 
   \left\{  \begin{array}{ll}
     f(x)      & (x \in [a,b]_{\ZZ}) ,\\
     -\infty   & (x \not\in [a,b]_{\ZZ}) 
                      \end{array}  \right.
\end{equation}
is M$\sp{\natural}$-concave, provided 
$\dom f_{[a,b]_{\ZZ}} \not= \emptyset$.

\noindent {\rm (6)} 
For $U \subseteq N$, the restriction of $f$ to $U$
defined by
\begin{eqnarray} 
 f_{U}(y) &=& f(y,\veczero_{N \setminus U})
  \qquad (y \in \ZZ\sp{U}) 
\label{fsetrestrict} 
\end{eqnarray}
is M$\sp{\natural}$-concave, provided 
$\dom f_{U} \not= \emptyset$,
where $\veczero_{N \setminus U}$ means the zero vector 
in $\ZZ\sp{N \setminus U}$.

\noindent {\rm (7)} 
For $U \subseteq N$, the projection of $f$ to $U$ defined by
\begin{eqnarray} 
  f\sp{U}(y) & = & \sup \{ f(y,z) \mid z \in \ZZ\sp{N \setminus U} \}
  \qquad (y \in \ZZ\sp{U}) 
\label{fsetproj}
\end{eqnarray}
is M$\sp{\natural}$-concave, provided
$f\sp{U} < +\infty$.

\noindent {\rm (8)} 
For $U \subseteq N$, the function $\tilde{f}$ defined by
\begin{eqnarray} 
  \tilde{f}(y,w) 
&=& 
 \sup \{ f(y, z) \mid  z(N \setminus U)=w, z \in \ZZ\sp{N \setminus U} \}
  \qquad (y \in \ZZ\sp{U}, w \in \ZZ) 
\label{fUaggreg}
\end{eqnarray}
is M$\sp{\natural}$-concave, provided 
$\tilde{f} < +\infty$.

\noindent {\rm (9)} 
{\em Integer (supremal) convolution}
$f_{1} \conv f_{2}: \ZZ\sp{N} \to \RR \cup \{ -\infty, +\infty \}$ 
defined by
\begin{equation} \label{f1f2convdef}
(f_{1} \conv f_{2})(x) =
  \sup\{ f_{1}(x_{1}) + f_{2}(x_{2}) \mid
      x= x_{1}+ x_{2}, \,  x_{1}, x_{2} \in \ZZ\sp{N}
 \}
\qquad (x \in \ZZ\sp{N})
\end{equation}
is M$\sp{\natural}$-concave, provided 
$(f_{1} \conv f_{2}) < +\infty$.
\end{theorem}
\begin{proof}
See 
Theorem 6.15 of Murota (2003)\citeH[Theorem 6.15]{Mdcasiam} 
for the proofs of (1) to (8).
In view of the importance of convolution operations we give 
a straightforward alternative proof of (9) in Remark~\ref{RMconvolproof}.
\end{proof}

\begin{remark} \rm  \label{RMecosignifconv}
Theorem \ref{THmfnoperation} (9)
for M$\sp{\natural}$-concavity of convolutions
has an implication of great economic significance.
Suppose that 
$U_1, U_2, \ldots, U_k$ represent utility functions.
Then the aggregated utility is given by their convolution
$U_1 \Box U_2 \Box \cdots \Box U_k$.
Theorem \ref{THmfnoperation} (9) means that 
substitutability is preserved in this aggregation operation.
\finbox
\end{remark}

\begin{remark} \rm  \label{RMconvolproof}
A proof for M$\sp{\natural}$-concavity of the convolution 
(\ref{f1f2convdef}) is given here%
\footnote{
This proof is an adaptation of the proof 
(Murota 2004b)\citeH{MconvolM04}
for {\rm M}-convex functions to ${\rm M}^{\natural}$-concave functions.
See Note 9.30 of Murota (2003)\citeH[Note 9.30]{Mdcasiam} 
for another proof using a network transformation.
}.  
Let $f_{1}$ and $f_{2}$ be M$\sp{\natural}$-concave functions, and 
$f = f_{1} \conv f_{2}$.
First we treat the case where 
$\dom f_{1}$ and $\dom f_{2}$ are bounded.
Then $\dom f = \dom f_{1} + \dom f_{2}$ (Minkowski sum) is bounded. 
For each $p \in \RR\sp{N}$ we have
$f[-p] = (f_{1}[-p]) \conv (f_{2}[-p])$,
from which follows
\[
 \argmax (f[-p]) = \argmax (f_{1}[-p]) + \argmax (f_{2}[-p]).
\]
In this expression, both $\argmax (f_{1}[-p])$ and $\argmax (f_{2}[-p])$
are M$\sp{\natural}$-convex sets by Theorem~\ref{THmconcavargmaxZ} (only if part),
and therefore, their Minkowski sum (the right-hand side) is M$\sp{\natural}$-convex 
(Theorem 4.23 of Murota 2003\citeH[Theorem 4.23]{Mdcasiam}).
This means that 
$\argmax (f[-p])$ is M$\sp{\natural}$-convex for each $p \in \RR\sp{N}$,
which  implies the M$\sp{\natural}$-concavity of $f$
by Theorem~\ref{THmconcavargmaxZ} (if part).

The general case without the boundedness assumption on effective domains
can be treated via limiting procedure as follows.
For $l=1,2$ and $k=1,2,\ldots$, define 
$f_{l}\sp{(k)}: \ZZ\sp{N} \to \RR \cup \{ -\infty \}$ by
\[
 f_{l}\sp{(k)}(x) = \left\{ \begin{array}{ll}
      f_{l}(x)    & \mbox{ if $\| x \|_{\infty} \leq k$} \\
      -\infty & \mbox{ otherwise},
    \end{array}\right.
\]
which is an M$\sp{\natural}$-concave function with a bounded effective domain,
provided that $k$ is large enough to ensure $\dom f_{l}\sp{(k)} \not= \emptyset$.
For each $k$, the convolution 
$f\sp{(k)}=f_{1}\sp{(k)} \conv f_{2}\sp{(k)}$ is M$\sp{\natural}$-concave
by the above argument, and moreover, 
$ \lim_{k \to \infty} f\sp{(k)}(x) = f(x)$ for each $x$.
It remains to demonstrate the property (M$\sp{\natural}$-EXC[$\ZZ$]) for $f$.
Take $x, y \in \dom f$ and $i \in \suppp(x-y)$.
There exists $k_{0}=k_{0}(x,y)$, depending on $x$ and $y$,
such that
$x, y \in \dom f\sp{(k)}$ for every $k \geq k_{0}$.
Since $f\sp{(k)}$ is M$\sp{\natural}$-concave, 
there exists $j_{k} \in \suppm(x-y) \cup \{ 0 \}$ such that  
\[
 f\sp{(k)}(x)+f\sp{(k)}(y) \leq 
 f\sp{(k)}(x-\unitvec{i}+\unitvec{j_{k}}) + f\sp{(k)}(y+\unitvec{i}-\unitvec{j_{k}}) . 
\]
Since $\suppm(x-y) \cup \{ 0 \}$ is a finite set,
at least one element of $\suppm(x-y) \cup \{ 0 \}$ appears infinitely many times in 
the sequence 
$\{ j_{k} \mid k \geq k_{0} \}$.
More precisely,  
there exists $j \in \suppm(x-y) \cup \{ 0 \}$ and an increasing subsequence
$k(1) < k(2) < \cdots$ such that
$j_{k(t)}=j$ for $t=1,2,\ldots$.
By letting $k \to \infty$ along this subsequence 
in the above inequality we obtain
\[
 f(x)+f(y) \leq  f(x-\unitvec{i}+\unitvec{j})  + f(y+\unitvec{i}-\unitvec{j})  .
\]
Thus $f = f_{1} \conv f_{2}$ satisfies (M$\sp{\natural}$-EXC[$\ZZ$]),
which proves Theorem \ref{THmfnoperation} (9).
\finbox
\end{remark}

\begin{remark} \rm  \label{RMmfnsum}
A sum of M$\sp{\natural}$-concave functions
is not necessarily M$\sp{\natural}$-concave.
This implies, in particular, that an M$\sp{\natural}$-concave function
does not necessarily remain M$\sp{\natural}$-concave 
when its effective domain is restricted to an M$\sp{\natural}$-convex set.
For example%
\footnote{
This example is a reformulation of 
Note 4.25 of Murota (2003)\citeH[Note 4.25]{Mdcasiam}
for {\rm M}-convex functions to ${\rm M}^{\natural}$-concave functions.
},  
let
$S_{1} = S_{0} \cup \{(0, 1, 1)\}$ and
$S_{2} = S_{0} \cup \{(1, 1, 0)\}$
with
$ S_{0} = \{(0, 0, 0), (1, 0, 0), (0, 1, 0),  (0, 0, 1), (1, 0, 1)\}$,
and let $f_{i}: \ZZ\sp{3} \to \RR \cup \{ -\infty \}$
be the (concave) indicator function%
\footnote{
$f_{i}(x)=0$ for $x \in S_{i}$ and $=-\infty$ for $x \not\in S_{i}$.
} 
 of $S_{i}$ for $i=1,2$.
Then $f_{1} + f_{2}$ is the
indicator function of $S_{1} \cap S_{2} = S_{0}$.
Here $S_{1}$ and $S_{2}$ are M$\sp{\natural}$-convex sets,
whereas $S_{0}$ is not%
\footnote{
(B$\sp{\natural}$-EXC[$\ZZ$]) fails for $S_{0}$ with
$x = (1, 0, 1)$, $y = (0, 1, 0)$, and $i = 1$.
}.  
Accordingly, 
 $f_{1}$ and $f_{2}$ are M$\sp{\natural}$-concave functions,
but their sum $f_{1} + f_{2}$ is not M$\sp{\natural}$-concave.
Functions represented as a sum of two M$\sp{\natural}$-concave functions 
are an intriguing mathematical object, investigated under the name of
{\em M$\sp{\natural}_{2}$-concave function} in 
Section 8.3 of Murota (2003)\citeH[Section 8.3]{Mdcasiam}.
\finbox
\end{remark}

\begin{remark} \rm  \label{RMmfnscaling4}
For a function $f: \ZZ\sp{n} \to \RR \cup \{ -\infty \}$
and a positive integer $\alpha$, the function 
$f\sp{\alpha}: \ZZ\sp{n} \to \RR \cup \{ -\infty \}$ 
defined by
$f\sp{\alpha}(x) = f(\alpha x)$ $(x \in \ZZ\sp{n})$
is called a {\em domain scaling} of $f$.
If $\alpha=2$, for instance, this amounts to 
considering the function values only on vectors of even integers.
Scaling is one of the common techniques used in designing
efficient algorithms and this is particularly true of 
network flow algorithms.
Unfortunately, M$\sp{\natural}$-concavity is not preserved under scaling.
For example%
\footnote{
This example is a reformulation of 
Note 6.18 of Murota (2003)\citeH[Note 6.18]{Mdcasiam} 
for {\rm M}-convex functions to ${\rm M}^{\natural}$-concave functions.
},  
let $f$ be the indicator function of a set 
$S =  \{ c_{1} (1,0,-1) +c_{2} (1,0,0) + c_{3} (0,1,-1) + c_{4} (0,1,0)
\mid c_{i} \in \{ 0,1 \} \} \subseteq \ZZ^{3}$.
This $f$ is an ${\rm M}^{\natural}$-concave function,
but $f^{2}$ (=$f^{\alpha}$ with $\alpha=2$),
being the indicator function of $ \{ (0,0,0), (1,1,-1) \}$,
is not ${\rm M}^{\natural}$-concave.
Nevertheless, scaling of an M$\sp{\natural}$-concave function is useful 
in designing efficient algorithms 
(Section 10.1 of Murota 2003)\citeH[Section 10.1]{Mdcasiam}.
It is worth mentioning that 
some subclasses of M$\sp{\natural}$-concave functions are closed under scaling operation;
linear, quadratic, separable, and laminar M$\sp{\natural}$-concave functions,
respectively, form such subclasses.
\finbox
\end{remark}

\begin{remark} \rm  \label{RMmatbaseval}
A class of set functions,
named matroid-based valuations,
is defined by 
Ostrovsky and Paes Leme (2015)
\citeH{OP15gs}%
with the use of the convolution operation
as well as the contraction operation.
For set functions 
$f_{1}, f_{2}: 2\sp{N} \to \RR \cup \{ -\infty \}$,
the convolution of $f_{1}$ and $f_{2}$
is defined by
$(f_{1} \conv f_{2})(X)  = \max_{Y \subseteq X}( f_{1}(Y) + f_{2}(X \setminus Y))$
for $X \subseteq N$.
For a set function 
$f: 2\sp{N} \to \RR \cup \{ -\infty \}$
and a subset $T$ of $N$,
the {\em contraction} of $T$
is defined as
$f_{T}(X)  = f(X \cup T ) - f(T)$ for $X \subseteq N \setminus T$.
A set function $f$ is said to be a {\em matroid-based valuation},
if it can be constructed by repeated application of convolution and contraction
to weighted matroid valuations (\ref{weightedrank}).
By Theorem \ref{THmfnoperation}, 
matroid-based valuations are 
M$\sp{\natural}$-concave functions.
It is conjectured in 
Ostrovsky and Paes Leme (2015)
\citeH{OP15gs}%
that every  M$\sp{\natural}$-concave function is a matroid-based valuation.
\finbox
\end{remark}

\subsection{Transformation by networks}
\label{SCnetinduce}

\begin{figure}\begin{center}
 \includegraphics[width=0.8\textwidth,clip]{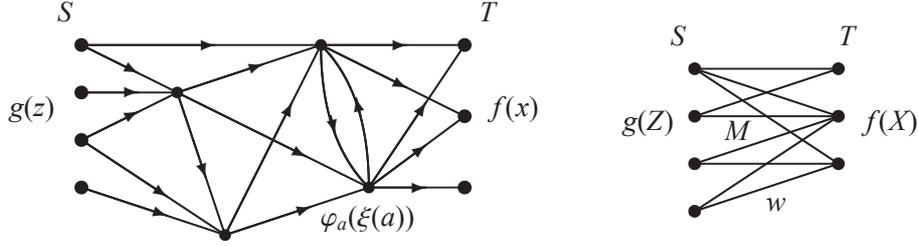}
 \caption{Transformation by a network and a bipartite graph}
 \label{FGnettrans}
\end{center}\end{figure}

M$\sp{\natural}$-concave functions can be transformed through networks.
Let $G = (V, A)$ be a directed graph with 
two disjoint vertex subsets $S \subseteq V$ and $T \subseteq V$
specified as the entrance and the exit (Fig.~\ref{FGnettrans}, left).
 Suppose that, for each arc $a \in A$, 
we are given a univariate concave function 
$\varphi_{a}: \ZZ \to \RR \cup \{ -\infty \}$
representing the weight of flow on the arc $a$.
Let $\xi \in \ZZ^A$ be a vector representing a flow, and
$\partial \xi \in \ZZ\sp{V}$ be the boundary of flow $\xi$ defined  by
(\ref{flowbounddefZ}).

Given a function
$g: \ZZ\sp{S} \to \RR \cup \{ -\infty \}$
on the entrance set $S$, 
we define a function
$f: \ZZ\sp{T} \to \RR \cup \{ -\infty, +\infty \}$
on the exit set $T$ by 
\begin{equation} \label{netindftilde} 
 f(x) =  \sup_{\xi,z}\{  g(z) +  \sum_{a \in A} \varphi_{a}(\xi(a)) \mid
  \xi \in \ZZ\sp{A},   \partial\xi =  (z, -x, \veczero) 
  \in \ZZ\sp{S} \times \ZZ\sp{T} \times \ZZ\sp{V \setminus (S \cup T)} \} 
 \quad  (x \in \ZZ\sp{T}) .
\end{equation}
This function $f(x)$ represents the maximum weight to meet
the demand specification $x$ at the exit,
subject to the flow conservation at the vertices not in $S \cup T$.
The weight consists of two parts,
the weight $g(z)$ of supply $z$ at the entrance $S$
and the weight $\sum_{a \in A} \varphi_{a}(\xi(a))$ in the arcs.

We can regard (\ref{netindftilde}) as a transformation of $g$ to $f$ by the network.
If the given function $g$ is M$\sp{\natural}$-concave, 
the resultant function $f$
is also M$\sp{\natural}$-concave,
provided that $f$ does not take the value $+\infty$ and 
$\dom f$ is nonempty.
In other words, the transformation (\ref{netindftilde}) by a network
preserves M$\sp{\natural}$-concavity.
See 
Section 9.6 of Murota (2003)\citeH[Section 9.6]{Mdcasiam} 
for a proof.  An alternative proof is given by
Kobayashi et al.~(2007)\citeH{KMT07jump}.

In particular, an M$\sp{\natural}$-concave set function 
is transformed to another M$\sp{\natural}$-concave set function 
through a bipartite graph
(Fig.~\ref{FGnettrans}, right).
Let $G=(S,T; E)$ be a bipartite graph with vertex bipartition
$(S, T)$ and edge set $E$,
with weight $w_{e} \in \RR$ associated with each edge $e \in E$.
Given an M$\sp{\natural}$-concave set function 
$g: 2\sp{S} \to \RR \cup \{ -\infty \}$ on $S$,
define a set function $f$ on $T$ by
\begin{equation}  \label{weightmatchinduce}
 f(X) = \max \{ g(Z) + w(M) 
         \mid \mbox{$M$ is a matching, $S \cap \partial M = Z$, 
                $T \cap \partial M=X$}  \}
\quad (X \subseteq T) ,
\end{equation}
where $f(X)= -\infty$ if no such $M$ exists for $X$.
If $g$ is M$\sp{\natural}$-concave, then $f$
is also M$\sp{\natural}$-concave,
as long as $\dom f$ is nonempty.
A proof tailored to set functions is given 
in the proof of Theorem~5.2.18 of Murota (2000a).
\citeH[Theorem 5.2.18]{Mspr2000}%

\subsection{Concluding remarks of section \ref{SCoperM}}

Efficient algorithms are available for 
the operations listed in Theorem \ref{THmfnoperation}.
In particular, the convolution (\ref{f1f2convdef}),
corresponding to the aggregation of utility functions,
can be computed efficiently
(Murota and Tamura 2003b)\citeH{MTcompeq03}.
The transformation by networks is also accompanied by efficient algorithms.
For M$\sp{\natural}$-concave function maximization algorithms,
see Chapter 10 of Murota (2003)\citeH[Chapter 10]{Mdcasiam},
and more recent papers, e.g., 
Shioura (2004)\citeH{Shimin04},
Tamura (2005)\citeH{Tam05scale},
Murota (2010)\citeH{Mrims10},
Moriguchi et al.~(2011)\citeH{MST08Mrelax},
Fujishige et al.~(2015)\citeH{FGHPZ15conges},
and Shioura (2015)\citeH{Shi15budget}.



\section{Conjugacy and L$\sp{\natural}$-convexity}
\label{SCconjuLconv}

Conjugacy under the Legendre  transformation 
is one of the most appealing facts in convex analysis.
This is also the case in discrete convex analysis.
The conjugacy theorem in discrete convex analysis says
that the Legendre transformation 
gives a one-to-one correspondence between M$\sp{\natural}$-concave
functions and L$\sp{\natural}$-convex functions.
Since M$\sp{\natural}$-concavity expresses
substitutability of valuation or utility functions,
L$\sp{\natural}$-convexity characterizes
substitutability in terms of indirect utility functions.
This fact has a significant application to auction theory, 
to be expounded in Section \ref{SCauction}.

\subsection{L$\sp{\natural}$-convex function}

The concept of L$\sp{\natural}$-convexity is defined for
functions in discrete (integer) variables and for those in 
continuous (real) variables.
We start with discrete variables.

\paragraph{L$\sp{\natural}$-convex function on $\ZZ\sp{n}$:}

First recall that
a function $g: \ZZ^{n} \to \RR \cup \{ +\infty \}$
is called {\em submodular} if 
\begin{equation} \label{gsubmineq}
 g(p) + g(q) \geq g(p \vee q) + g(p \wedge q) 
\qquad (p,q \in \ZZ^{n}),
\end{equation}
where $p \vee q$ and $p \wedge q$ mean
the vectors of componentwise maximum and minimum of $p$ and $q$, respectively.
To define L$\sp{\natural}$-convexity of $g$,
we consider a function $\tilde g$ in $n+1$ variables
$(p_{0},p) =(p_{0}, p_{1}, \ldots, p_{n})$
defined as 
\begin{equation}\label{lfnlnatfnrelation}
 \tilde g(p_{0},p) = g(p - p_{0} \vecone)
 \qquad ( p_{0} \in \ZZ, p \in \ZZ\sp{n}),
\end{equation}
where $\bm{1}=(1,1,\ldots, 1)$.
Then we say that $g: \ZZ^{n} \to \RR \cup \{ +\infty \}$ is 
{\em L$\sp{\natural}$-convex}
if the associated function 
$\tilde g: \ZZ\sp{n+1} \to \RR \cup \{ +\infty \}$
is a submodular function in $(p_{0},p)$,
i.e., if
\begin{equation}\label{tildegsubmineqZ}
  g(p - p_{0} \vecone) +  g(q - q_{0} \vecone)
 \geq  g( (p \vee q) - (p_{0} \vee q_{0}) \vecone) 
    +  g( (p \wedge q) - (p_{0} \wedge q_{0}) \vecone) 
 \quad ( p_{0}, q_{0} \in \ZZ, p, q \in \ZZ\sp{n}).
\end{equation}

\begin{remark} \rm  \label{RMonedimsubm}
The significance of the extra variable $p_{0}$ in the definition of
 L$\sp{\natural}$-convexity is most transparent when $n = 1$.
When $n=1$ we have
$(p \vee q, p \wedge q) = ( p, q )$ or $( q, p )$,
according to whether $p \geq q$ or $p \leq q$.
Hence the submodular inequality (\ref{gsubmineq}) is always satisfied,
and every function $g: \ZZ \to \RR \cup \{ +\infty \}$ is submodular.
On the other hand, 
the inequality (\ref{tildegsubmineqZ}) for $(p_{0},p) = (1,t)$ and
$(q_{0},q) = (0,t+1)$ yields
$ g(t-1) +  g(t+1) \geq  2g(t)$ for $t \in \ZZ$,
which shows the convexity of $g$ on $\ZZ$. 
The converse is also true.
Therefore, 
a function $g: \ZZ \to \RR \cup \{ +\infty \}$ is 
 L$\sp{\natural}$-convex if and only if 
$ g(t-1) +  g(t+1) \geq  2g(t)$ for all $t \in \ZZ$.
\finbox
\end{remark}

\begin{remark} \rm  \label{RMlnatsubm01}
For a set function
$\mu: 2\sp{N} \to \RR \cup \{ +\infty \}$,
L$\sp{\natural}$-convexity is equivalent to submodularity (\ref{setfnsubm}).
Recall the notation $\unitvec{X}$ for the characteristic vector 
of a subset $X$;
see (\ref{charvecdefnotat}).
A set function $\mu$ can be identified with a function 
$g: \mathbb{Z}\sp{N} \to \RR \cup \{ +\infty \}$
with  $\dom g \subseteq \{ 0, 1 \}\sp{N}$
by $\mu(X) = g(\unitvec{X})$ for $X \subseteq N$,
and $\mu$ is submodular if and only if the corresponding $g$ 
is L$\sp{\natural}$-convex.
\finbox
\end{remark}

\begin{remark} \rm  \label{RMmatrankdualchar}
Matroid rank functions have a dual character of 
being both L$\sp{\natural}$-convex and M$\sp{\natural}$-concave.
It is L$\sp{\natural}$-convex as it is submodular,
and M$\sp{\natural}$-concave as already mentioned in Section~\ref{SCmnatexample01}.
\finbox
\end{remark}

L$\sp{\natural}$-convexity can be characterized by a number of equivalent conditions
(Favati and Tardella 1990\citeH{FT90}, 
Fujishige and Murota 2000\citeH{FM00}, 
Murota 2003\citeH{Mdcasiam}).

\begin{theorem}
 \label{THlnatcondZ}
For $g: \ZZ^{n} \to \RR \cup \{ +\infty \}$
the following conditions, {\rm (a)} to {\rm (d)}, are equivalent:

\noindent \ \ {\rm (a)}
L$\sp{\natural}$-convexity, i.e.,
{\rm (\ref{tildegsubmineqZ})}.

\noindent \ \ {\rm (b)}
{\em Translation-submodularity}%
\footnote{
This condition is labeled as (SBF$\sp{\natural}$[$\ZZ$]) in 
Section 7.1 of Murota (2003)\citeH[Section 7.1]{Mdcasiam}.
Note that $\alpha$ is restricted to be nonnegative, and 
the inequality (\ref{lnatftrsubmZ}) for $\alpha=0$ 
coincides with submodularity (\ref{gsubmineq}).
}{\rm :}  
\begin{equation} \label{lnatftrsubmZ}
  g(p) + g(q) \geq g((p - \alpha \vecone) \vee q) 
                 + g(p \wedge (q + \alpha \vecone))
\qquad  (\forall p, q \in \ZZ\sp{n},  \forall \alpha \in \ZZ_{+}).
\end{equation}

\noindent \ \ {\rm (c)}
{\em Discrete midpoint convexity:} 
\begin{equation} \label{disfnmidconv}
 g(p) + g(q) \geq
   g \left(\left\lceil \frac{p+q}{2} \right\rceil\right) 
  + g \left(\left\lfloor \frac{p+q}{2} \right\rfloor\right) 
\qquad (p, q \in \ZZ\sp{n})   ,
\end{equation}
where $\lceil \cdot \rceil$ and $\lfloor \cdot \rfloor$ denote
the integer vectors obtained by
componentwise rounding-up and rounding-down to the nearest integers,
respectively.

\noindent \ \ {\rm (d)}
For any $p, q \in \ZZ\sp{n}$ with $\suppp(p-q)\not= \emptyset$, 
it holds that%
\footnote{
This condition is labeled as  (L$\sp{\natural}$-APR[$\ZZ$]) in 
Section 7.2 of Murota (2003)\citeH[Section 7.2]{Mdcasiam}.
Recall the notation $\unitvec{A}$ for the characteristic vector of $A$,
as defined in (\ref{charvecdefnotat}).
} 
\begin{equation} \label{lnatAPR}
 g(p) + g(q) \geq g(p - \unitvec{A}) + g(q + \unitvec{A}) ,
\end{equation}
where $\displaystyle A = \argmax_{i} \{ p_{i} - q_{i} \}$.
\end{theorem}

It is known 
(Theorem 7.20 of Murota 2003\citeH[Theorem 7.20]{Mdcasiam}) 
that an L$\sp{\natural}$-convex function 
$g: \mathbb{Z}\sp{n} \to \RR \cup \{ +\infty \}$
is {\em convex-extensible}, 
i.e., there exists a convex function 
$\overline{g}: \RR\sp{n} \to \RR \cup \{ +\infty \}$
such that
$\overline{g}(p) = g(p)$ for all $p \in \ZZ\sp{n}$.
Moreover, the convex extension $\overline{g}$ can be constructed 
by a simple procedure; see 
Theorem 7.19 of Murota (2003)\citeH[Theorem 7.19]{Mdcasiam}.

\begin{remark} \rm  \label{RMlnatconvexsetZ}
A nonempty set $P \subseteq \mathbb{Z}\sp{n}$
is called an {\em L$\sp{\natural}$-convex set}
if its indicator function%
\footnote{
$g(p)=0$ for $p \in P$ and $=+\infty$ for $p \not\in P$.
} 
 is an L$\sp{\natural}$-convex function.
In other words, $P \not= \emptyset$ is an L$\sp{\natural}$-convex set
if it satisfies one of the following equivalent conditions,
where $ p, q \in \ZZ\sp{n}$ and $ p_{0}, q_{0} \in \ZZ$:

\begin{itemize}
\item[(a)]
$p - p_{0} \vecone, \  q - q_{0} \vecone \  \in P$
\ $\Longrightarrow$ \ 
$ (p \vee q) - (p_{0} \vee q_{0}) \vecone, \   
    (p \wedge q) - (p_{0} \wedge q_{0}) \vecone \ \in P$.

\item[(b)]
$p,q \in P$, $\alpha \in \ZZ_{+}$
\ $\Longrightarrow$ \ 
$(p - \alpha \vecone) \vee q, \  p \wedge (q + \alpha \vecone) \ \in P$.

\item[(c)]
$p,q \in P$ 
\ $\Longrightarrow$ \ 
$\left\lceil \frac{p+q}{2} \right\rceil, \  
\left\lfloor \frac{p+q}{2} \right\rfloor \ \in P$.

\item[(d)]
$p,q \in P$, \  $\suppp(p-q)\not= \emptyset$ 
\ $\Longrightarrow$ \ 
$ p - \unitvec{A}, \  q + \unitvec{A} \  \in P$
with $\displaystyle A = \argmax_{i} \{ p_{i} - q_{i} \}$.
\end{itemize}
For an L$\sp{\natural}$-convex function $g$, 
the effective domain $\dom g$ and the set of minimizers $\argmin g$
are L$\sp{\natural}$-convex sets. 
See 
Section 5.5 of Murota (2003)\citeH[Section 5.5]{Mdcasiam}
for more about L$\sp{\natural}$-convex sets. 
\finbox
\end{remark}

\begin{remark} \rm  \label{RMlconvexZ}
A function $g: \mathbb{Z}\sp{n} \to \RR \cup \{ +\infty \}$
is called an {\em L-convex function}
if it is an L$\sp{\natural}$-convex function such that 
there exists $r \in \RR$ for which	
$g(p + \vecone) = g(p) +  r$ for all $p \in \ZZ\sp{n}$.
L-convex functions and
L$\sp{\natural}$-convex functions are equivalent concepts,
in that L$\sp{\natural}$-convex functions in $n$ variables 
can be identified, up to the constant $r$, with
L-convex functions in $n+1$ variables.
Indeed, a function $g : \ZZ\sp{n} \to \RR \cup \{ +\infty \}$ is L$\sp{\natural}$-convex
if and only if  the function 
$\tilde{g} : \ZZ\sp{n+1} \to \RR \cup \{ +\infty \}$ 
in (\ref{lfnlnatfnrelation}) is an L-convex function (with $r=0$).
\finbox
\end{remark}

\paragraph{L$\sp{\natural}$-convex function on $\RR\sp{n}$:}
\label{SClnatconvR}

We turn to continuous variables.
A function 
$g: \mathbb{R}\sp{n} \to \RR \cup \{ +\infty \}$
is said to be  
{\em L$\sp{\natural}$-convex}
if it is a convex function (in the ordinary sense) such that
$\tilde g(p_{0},p) = g(p - p_{0} \vecone)$
$( p_{0} \in \RR, p \in \RR\sp{n})$
is a submodular function in $n+1$ variables,
i.e.,
\begin{equation}\label{tildegsubmineqR}
  g(p - p_{0} \vecone) +  g(q - q_{0} \vecone)
 \geq
  g( (p \vee q) - (p_{0} \vee q_{0}) \vecone) 
  +  g( (p \wedge q) - (p_{0} \wedge q_{0}) \vecone) 
 \quad ( p_{0}, q_{0} \in \RR, p, q \in \RR\sp{n}).
\end{equation}
In the following 
we restrict ourselves to closed proper L$\sp{\natural}$-convex functions%
\footnote{
A convex function $g: \mathbb{R}\sp{n} \to \RR \cup \{ +\infty \}$
is said to be {\em proper} if $\dom g$ is nonempty, and 
{\em closed} if the epigraph $\{ (p,\alpha) \in \RR\sp{n+1} \mid  \alpha \geq g(p) \}$
is a closed subset of $\RR\sp{n+1}$.
},   
for which the closure of the effective domain $\dom g$ is a well-behaved polyhedron
(L$\sp{\natural}$-convex polyhedron%
\footnote{
A polyhedron is called an {\em L$\sp{\natural}$-convex polyhedron}
if its (convex) indicator function is L$\sp{\natural}$-convex.
See Section~5.6 of Murota (2003)\citeH[Section 5.6]{Mdcasiam}
for details.
});   
see Theorem 3.3 of Murota and Shioura (2008)\citeH[Theorem 3.3]{MS08cont}.
For a closed proper convex function $g: \RR^{n} \to \RR \cup \{ +\infty \}$,
the condition {\rm (\ref{tildegsubmineqR})} for L$\sp{\natural}$-convexity  
is equivalent to translation-submodularity:
\begin{equation} \label{lnatftrsubmR}
  g(p) + g(q) \geq g((p - \alpha \vecone) \vee q) 
                 + g(p \wedge (q + \alpha \vecone))
\qquad  (\forall p, q \in \RR\sp{n},  \forall \alpha \in \RR_{+}).
\end{equation}
Often we are interested in polyhedral L$\sp{\natural}$-convex functions.

L$\sp{\natural}$-convex functions in real variables are investigated by 
Murota and Shioura (2000, 2004a, 2004b, 2008)%
\citeH{MS00poly}\citeH{MS04conjreal}\citeH{MS04fund}\citeH{MS08cont}.

\subsection{Conjugacy}
\label{SCconjugacy}

\paragraph{Functions in continuous variables:}
\label{SCconjugacyR}

For a function 
$f: \RR\sp{n} \to \mathbb{R} \cup \{ +\infty  \}$
(not necessarily convex) with
$\dom f \not= \emptyset$,
the {\em convex conjugate}
$f\sp{\bullet}: \RR\sp{n} \to \mathbb{R} \cup \{ +\infty  \}$
is defined by
\begin{equation} \label{conjvexOdefR}
 f\sp{\bullet}(p) 
 = \sup\{  \langle p, x \rangle - f(x)     \mid x \in \RR\sp{n} \}
\qquad ( p \in \RR\sp{n}),
\end{equation}
where  
$ \langle p, x \rangle = \sum_{i=1}\sp{n} p_{i} x_{i}$ 
is the inner product of 
$p=(p_{i}) \in \RR\sp{n}$ and 
$x=(x_{i}) \in \RR\sp{n}$.
The function $f\sp{\bullet}$ is also referred to as
the (convex) {\em Legendre(--Fenchel) transform}
of $f$,
and the mapping $f \mapsto f\sp{\bullet}$ as the (convex) 
{\em Legendre(--Fenchel) transformation}.
A fundamental theorem in convex analysis states that 
the Legendre transformation 
gives a symmetric one-to-one correspondence in the class of
all closed proper convex functions.
That is, for a closed proper convex function $f$,
the conjugate function
$f\sp{\bullet}$ is a closed proper convex function
and the {\em biconjugacy} $(f\sp{\bullet})\sp{\bullet} = f$ holds.

To formulate the correspondence between 
concave functions 
$f: \RR\sp{n} \to \mathbb{R} \cup \{ -\infty  \}$
and convex functions
$g: \RR\sp{n} \to \mathbb{R} \cup \{ +\infty  \}$
with $\dom f \not= \emptyset$ and $\dom g \not= \emptyset$,
we introduce the following variants of the transformation (\ref{conjvexOdefR}):
\begin{align} 
 f\sp{\triangledown}(p) 
 &= \sup\{  f(x) - \langle p, x \rangle   \mid x \in \RR\sp{n} \}
\qquad ( p \in \RR\sp{n}),
\label{conjcave2vexR}
\\
 g\sp{\triangle}(x) 
 &= \inf\{  g(p) + \langle p, x \rangle   \mid p \in \RR\sp{n} \}
\qquad ( x \in \RR\sp{n}) ,
\label{conjvex2caveR}
\end{align}
where
$f\sp{\triangledown}: \RR\sp{n} \to \mathbb{R} \cup \{ +\infty  \}$
and 
$g\sp{\triangle}: \RR\sp{n} \to \mathbb{R} \cup \{ -\infty  \}$.
The biconjugacy is expressed as 
$(f\sp{\triangledown})\sp{\triangle} = f$, 
$(g\sp{\triangle})\sp{\triangledown} = g$
for closed proper concave functions $f$ and closed proper convex functions $g$.

\begin{theorem} \label{THOconjcavevexR}
\quad  

\noindent {\rm (1)}
The transformations {\rm (\ref{conjcave2vexR})} and {\rm (\ref{conjvex2caveR})}
give a one-to-one correspondence between the classes of all
closed proper concave functions $f$ and closed proper convex functions $g$.

\noindent {\rm (2)}
For a closed proper concave function 
$f: \RR\sp{n} \to \mathbb{R} \cup \{ -\infty  \}$,
the conjugate function 
$f\sp{\triangledown}: \RR\sp{n} \to \mathbb{R} \cup \{ +\infty  \}$ 
is a closed proper convex function
and $(f\sp{\triangledown})\sp{\triangle} = f$.

\noindent {\rm (3)}
For a closed proper convex function 
$g: \RR\sp{n} \to \mathbb{R} \cup \{ +\infty  \}$,
the conjugate function 
$g\sp{\triangle}: \RR\sp{n} \to \mathbb{R} \cup \{ -\infty  \}$ 
is a closed proper concave function
and $(g\sp{\triangle})\sp{\triangledown} = g$.
\end{theorem}

Addition of combinatorial ingredients
to the above theorem yields the conjugacy theorem between
M$\sp{\natural}$-concave and L$\sp{\natural}$-convex functions
(Murota and Shioura 2004a)\citeH{MS04conjreal}.

\begin{theorem}  \label{THlmconjcavevexR}
\quad  

\noindent {\rm (1)}
The transformations {\rm (\ref{conjcave2vexR})} and {\rm (\ref{conjvex2caveR})}
give a one-to-one correspondence between the classes of all
closed proper M$\sp{\natural}$-concave functions $f$ 
and closed proper L$\sp{\natural}$-convex functions $g$.

\noindent {\rm (2)}
For a closed proper M$\sp{\natural}$-concave function
$f: \RR\sp{n} \to \mathbb{R} \cup \{ -\infty  \}$,
the conjugate function 
$f\sp{\triangledown}: \RR\sp{n} \to \mathbb{R} \cup \{ +\infty  \}$ 
is a closed proper L$\sp{\natural}$-convex function
and $(f\sp{\triangledown})\sp{\triangle} = f$.

\noindent {\rm (3)}
For a closed proper L$\sp{\natural}$-convex function 
$g: \RR\sp{n} \to \mathbb{R} \cup \{ +\infty  \}$,
the conjugate function 
$g\sp{\triangle}: \RR\sp{n} \to \mathbb{R} \cup \{ -\infty  \}$ 
is a closed proper M$\sp{\natural}$-concave function
and $(g\sp{\triangle})\sp{\triangledown} = g$.
\end{theorem}

The M$\sp{\natural}$/L$\sp{\natural}$-conjugacy 
is also valid for polyhedral concave/convex functions
(Murota and Shioura 2000, Theorem 8.4 of Murota 2003).
\citeH{MS00poly}\citeH[Theorem 8.4]{Mdcasiam}%

\begin{theorem} \label{THlmconjcavevexP}
\quad  

\noindent {\rm (1)}
The transformations {\rm (\ref{conjcave2vexR})} and {\rm (\ref{conjvex2caveR})}
give a one-to-one correspondence between the classes of all
polyhedral M$\sp{\natural}$-concave functions $f$ 
and polyhedral L$\sp{\natural}$-convex functions $g$.

\noindent {\rm (2)}
For a polyhedral M$\sp{\natural}$-concave function 
$f: \RR\sp{n} \to \mathbb{R} \cup \{ -\infty  \}$,
the conjugate function 
$f\sp{\triangledown}: \RR\sp{n} \to \mathbb{R} \cup \{ +\infty  \}$ 
is a polyhedral L$\sp{\natural}$-convex function
and $(f\sp{\triangledown})\sp{\triangle} = f$.

\noindent {\rm (3)}
For a polyhedral L$\sp{\natural}$-convex function 
$g: \RR\sp{n} \to \mathbb{R} \cup \{ +\infty  \}$,
the conjugate function 
$g\sp{\triangle}: \RR\sp{n} \to \mathbb{R} \cup \{ -\infty  \}$ 
is a polyhedral M$\sp{\natural}$-concave function
and $(g\sp{\triangle})\sp{\triangledown} = g$.
\end{theorem}

As corollaries of the conjugacy theorems,
the following characterizations of 
M$\sp{\natural}$-concavity
and L$\sp{\natural}$-convexity
in terms of the conjugate functions are obtained.

\begin{theorem}  \label{THmconcavebyconjfnRR}
\quad  

\noindent {\rm (1)}
A function
$f: \RR\sp{n} \to \mathbb{R} \cup \{ -\infty  \}$
is closed proper M$\sp{\natural}$-concave 
if and only if
the conjugate function 
$f\sp{\triangledown}: \RR\sp{n} \to \mathbb{R} \cup \{ +\infty  \}$
by {\rm (\ref{conjcave2vexR})} is closed proper L$\sp{\natural}$-convex.

\noindent {\rm (2)}
A function
$f: \RR\sp{n} \to \mathbb{R} \cup \{ -\infty  \}$
is polyhedral M$\sp{\natural}$-concave 
if and only if
the conjugate function 
$f\sp{\triangledown}: \RR\sp{n} \to \mathbb{R} \cup \{ +\infty  \}$
by {\rm (\ref{conjcave2vexR})} is polyhedral L$\sp{\natural}$-convex.
\end{theorem}

\begin{theorem}  \label{THlconvbyconjfnRR}
\quad  

\noindent {\rm (1)}
A function
$g: \RR\sp{n} \to \mathbb{R} \cup \{ +\infty  \}$
is closed proper L$\sp{\natural}$-convex 
if and only if
the conjugate function 
$g\sp{\triangle}: \RR\sp{n} \to \mathbb{R} \cup \{ -\infty  \}$
by {\rm (\ref{conjvex2caveR})} is closed proper M$\sp{\natural}$-concave.

\noindent {\rm (2)}
A function
$g: \RR\sp{n} \to \mathbb{R} \cup \{ +\infty  \}$
is polyhedral L$\sp{\natural}$-convex 
if and only if
the conjugate function 
$g\sp{\triangle}: \RR\sp{n} \to \mathbb{R} \cup \{ -\infty  \}$
by {\rm (\ref{conjvex2caveR})} is polyhedral M$\sp{\natural}$-concave.
\end{theorem}

L$\sp{\natural}$-convexity,
being equivalent to translation-submodularity,
is a stronger property than mere submodularity.
When we replace 
L$\sp{\natural}$-convexity of $f\sp{\triangledown}$ 
in Theorem \ref{THmconcavebyconjfnRR} (2) with submodularity,
we obtain a larger class of polyhedral concave functions $f$ 
than M$\sp{\natural}$-concave functions.
The following theorem is ascribed to 
Danilov and Lang (2001)\citeH{DL01gs}
in Danilov et al.~(2003)\citeH{DKL03gr};
see also Appendix of Shioura and Tamura (2015)\citeH[Appendix]{ST15jorsj}
for technical supplements.

\begin{theorem} \label{THgssubmRpoly}
Let $f: \RR\sp{N} \to \RR \cup \{ -\infty \}$ be a polyhedral concave 
function with a bounded effective domain.
Then the following conditions are equivalent%
\footnote{
Recall the definition of (GS[$\RR$]) from Section \ref{SCmaximizersR}.
Also recall from Theorem \ref{THmnatgsRpoly}
that polyhedral M$\sp{\natural}$-concave functions satisfy (GS[$\RR$]).
}:  

\noindent \ \ {\rm (a)} 
$f$ satisfies {\rm (GS[$\RR$])}.

\noindent \ \ {\rm (b)} 
For every $p \in \RR\sp{N}$, 
each edge (one-dimensional face) of $D(p ; f)$ 
is parallel to a vector $d$ with $|\suppp(d)| \leq 1$ and $|\suppm(d)| \leq 1$.

\noindent \ \  {\rm (c)} 
$f\sp{\triangledown}: \RR\sp{N} \to \RR \cup \{ +\infty \}$
by {\rm (\ref{conjcave2vexR})} is a submodular function.
\end{theorem}

\begin{remark} \rm  \label{RMpolybasic}
In Danilov et al.~(2003)\citeH{DKL03gr}
 a bounded polyhedron $P$ is called a {\em quasi-polymatroid}
if each edge (one-dimensional face) 
is parallel to a vector $d$ with $|\suppp(d)| \leq 1$ and $|\suppm(d)| \leq 1$.
It follows from 
Theorem 3.1 of Fujishige et al.~(2004)\citeH[Theorem 3.1]{FMTK04}
that every face of a quasi-polymatroid
whose normal vector has the full support $N$ 
is obtained from an M-convex polyhedron (base polyhedron)
by a scaling along axes.
We mention in passing that a pointed convex polyhedron is called {\em polybasic} 
if each edge is parallel to a vector $d$ with 
$|\suppp(d)| + |\suppm(d)| \leq 2$ 
(Fujishige et al.~2004\citeH{FMTK04}).
\finbox
\end{remark}

\begin{remark} \rm  \label{RMsmoothGSsubm}
In the canonical situation, where
$f: \RR\sp{n} \to \mathbb{R}$ is a strictly concave smooth function,
the equivalence between (GS[$\RR$]) of $f$ and 
the submodularity of $g = f\sp{\triangledown}$ is easily derived by simple calculus.
Let $x(p)$ be the unique maximizer of 
$f(x) - \langle p, x \rangle$.
We have
$p_{i} = \partial f / \partial x_{i}$ for $i=1,\ldots, n$,
and 
$g(p) = f(x(p)) - \langle p, x(p) \rangle$.
This implies
$\partial g / \partial p_{i} = - x_{i}$ $(i=1,\ldots, n)$,
and hence
$\partial\sp{2} g / \partial p_{i}\partial p_{j} 
 = - \partial x_{i}/ \partial p_{j}$  $(i,j=1,\ldots, n)$.
On the other hand, the submodularity of $g$ is equivalent to 
$\partial\sp{2} g / \partial p_{i}\partial p_{j} \leq 0$  $(i \not= j)$,
and (GS[$\RR$]) of $f$ is represented as
$\partial x_{i}/ \partial p_{j} \geq 0$ $(i \not= j)$.
\finbox
\end{remark}

\paragraph{Functions in discrete variables:}
\label{SCconjugacyZ}

We turn to functions defined on integer vectors.
For functions 
$f: \ZZ\sp{n} \to \mathbb{R} \cup \{ -\infty  \}$ and 
$g: \ZZ\sp{n} \to \mathbb{R} \cup \{ +\infty  \}$
with $\dom f \not= \emptyset$ and $\dom g \not= \emptyset$,
the transformations 
(\ref{conjcave2vexR}) and (\ref{conjvex2caveR})
are modified to
\begin{align} 
 f\sp{\triangledown}(p) 
 &= \sup\{  f(x) - \langle p, x \rangle   \mid x \in \ZZ\sp{n} \}
\qquad ( p \in \RR\sp{n}),
\label{conjcave2vexZ}
\\
 g\sp{\triangle}(x) 
 &= \inf\{  g(p) + \langle p, x \rangle   \mid p \in \ZZ\sp{n} \}
\qquad ( x \in \RR\sp{n}) ,
\label{conjvex2caveZ}
\end{align}
where
$f\sp{\triangledown}: \RR\sp{n} \to \mathbb{R} \cup \{ +\infty  \}$
and 
$g\sp{\triangle}: \RR\sp{n} \to \mathbb{R} \cup \{ -\infty  \}$.

The conjugacy  between
M$\sp{\natural}$-concavity and L$\sp{\natural}$-convexity in this case
reads as follows%
\footnote{
In Theorem \ref{THlmconjZR} (1),
${}\sp{\triangledown}$ is defined by {\rm (\ref{conjcave2vexZ})} and 
${}\sp{\triangle}$ by {\rm (\ref{conjvex2caveR})}.
In (2), ${}\sp{\triangle}$ is defined by {\rm (\ref{conjvex2caveZ})} and 
${}\sp{\triangledown}$ by {\rm (\ref{conjcave2vexR})}.
}.  

\begin{theorem}  \label{THlmconjZR}
\quad  

\noindent {\rm (1)}
For an M$\sp{\natural}$-concave function 
$f: \ZZ\sp{n} \to \mathbb{R} \cup \{ -\infty  \}$,
the conjugate function 
$f\sp{\triangledown}: \RR\sp{n} \to \mathbb{R} \cup \{ +\infty  \}$
is a {\rm (}locally polyhedral\/{\rm )}
L$\sp{\natural}$-convex function,
and $(f\sp{\triangledown})\sp{\triangle}(x) = f(x)$ for $x \in \ZZ\sp{n}$.

\noindent {\rm (2)}
For an L$\sp{\natural}$-convex function
$g: \ZZ\sp{n} \to \mathbb{R} \cup \{ +\infty  \}$,
the conjugate function 
$g\sp{\triangle}: \RR\sp{n} \to \mathbb{R} \cup \{ -\infty  \}$
is a {\rm (}locally polyhedral\/{\rm )}
M$\sp{\natural}$-concave function,
and $(g\sp{\triangle})\sp{\triangledown}(p) = g(p)$ for $p \in \ZZ\sp{n}$.
\end{theorem}

For integer-valued functions $f$ and $g$, 
$f\sp{\triangledown}(p)$ and $g\sp{\triangle}(x)$ 
are  integers for integer vectors $p$ and $x$.
Hence 
(\ref{conjcave2vexZ}) with $p \in \ZZ\sp{n}$ 
and (\ref{conjvex2caveZ}) with $x \in \ZZ\sp{n}$, i.e.,
\begin{align} 
 f\sp{\triangledown}(p) 
 &= \sup\{  f(x) - \langle p, x \rangle   \mid x \in \ZZ\sp{n} \}
\qquad ( p \in \ZZ\sp{n}),
\label{conjcave2vexZZ}
\\
 g\sp{\triangle}(x) 
 &= \inf\{  g(p) + \langle p, x \rangle   \mid p \in \ZZ\sp{n} \}
\qquad ( x \in \ZZ\sp{n}) ,
\label{conjvex2caveZZ}
\end{align}
define  transformations of 
$f: \ZZ\sp{n} \to \mathbb{Z} \cup \{ -\infty  \}$ to
$f\sp{\triangledown}: \ZZ\sp{n} \to \mathbb{Z} \cup \{ +\infty  \}$
and 
$g: \ZZ\sp{n} \to \mathbb{Z} \cup \{ +\infty  \}$ to
$g\sp{\triangle}: \ZZ\sp{n} \to \mathbb{Z} \cup \{ -\infty  \}$,
respectively.

The conjugacy theorem for integer-valued discrete-variable  M$\sp{\natural}$-concave 
and L$\sp{\natural}$-convex functions reads as follows
(Murota 1998, Theorem 8.12 of Murota 2003)\citeH{Mdca}\citeH[Theorem 8.12]{Mdcasiam}.

\begin{theorem}  \label{THlmconjcavevexZ}
\quad  

\noindent {\rm (1)}
The transformations 
{\rm (\ref{conjcave2vexZZ})} and {\rm (\ref{conjvex2caveZZ})}
give a one-to-one correspondence between the classes of all
integer-valued M$\sp{\natural}$-concave functions $f$ 
and integer-valued L$\sp{\natural}$-convex functions $g$.

\noindent {\rm (2)}
For an integer-valued M$\sp{\natural}$-concave function
$f: \ZZ\sp{n} \to \mathbb{Z} \cup \{ -\infty  \}$,
the conjugate function 
$f\sp{\triangledown}: \ZZ\sp{n} \to \mathbb{Z} \cup \{ +\infty  \}$
is an integer-valued L$\sp{\natural}$-convex function
and $(f\sp{\triangledown})\sp{\triangle} = f$.

\noindent {\rm (3)}
For an integer-valued L$\sp{\natural}$-convex function
$g: \ZZ\sp{n} \to \mathbb{Z} \cup \{ +\infty  \}$,
the conjugate function 
$g\sp{\triangle}: \ZZ\sp{n} \to \mathbb{Z} \cup \{ -\infty  \}$ 
is an integer-valued M$\sp{\natural}$-concave function
and $(g\sp{\triangle})\sp{\triangledown} = g$.
\end{theorem}

As corollaries of the conjugacy theorems,
the following characterizations of 
M$\sp{\natural}$-concavity
and L$\sp{\natural}$-convexity
in terms of the conjugate functions are obtained.

\begin{theorem}  \label{THmconcavebyconjfnZR}
\quad  

\noindent {\rm (1)}
A function
$f: \ZZ\sp{n} \to \mathbb{R} \cup \{ -\infty  \}$
is M$\sp{\natural}$-concave 
if and only if
the conjugate function 
$f\sp{\triangledown}: \RR\sp{n} \to \mathbb{R} \cup \{ +\infty  \}$
by {\rm (\ref{conjcave2vexZ})} is (locally polyhedral) L$\sp{\natural}$-convex.

\noindent {\rm (2)}
A function
$f: \ZZ\sp{n} \to \mathbb{Z} \cup \{ -\infty  \}$
is M$\sp{\natural}$-concave 
if and only if
the conjugate function 
$f\sp{\triangledown}: \ZZ\sp{n} \to \mathbb{Z} \cup \{ +\infty  \}$
by {\rm (\ref{conjcave2vexZZ})} is L$\sp{\natural}$-convex.
\end{theorem}

\begin{theorem}  \label{THlconvbyconjfnZR}
\quad  

\noindent {\rm (1)}
A function
$g: \ZZ\sp{n} \to \mathbb{R} \cup \{ +\infty  \}$
is L$\sp{\natural}$-convex 
if and only if
the conjugate function 
$g\sp{\triangle}: \RR\sp{n} \to \mathbb{R} \cup \{ -\infty  \}$
by {\rm (\ref{conjvex2caveZ})} is (locally polyhedral) M$\sp{\natural}$-concave.

\noindent {\rm (2)}
A function
$g: \ZZ\sp{n} \to \mathbb{Z} \cup \{ +\infty  \}$
is L$\sp{\natural}$-convex 
if and only if
the conjugate function 
$g\sp{\triangle}: \ZZ\sp{n} \to \mathbb{Z} \cup \{ -\infty  \}$
by {\rm (\ref{conjvex2caveZZ})}
is M$\sp{\natural}$-concave.
\end{theorem}

L$\sp{\natural}$-convexity,
being equivalent to translation-submodularity,
is a stronger property than mere submodularity.
Naturally, we may wonder if L$\sp{\natural}$-convexity of 
$f\sp{\triangledown}$ in Theorem \ref{THmconcavebyconjfnZR}
can be replaced by submodularity.
However, the following example denies this possibility%
\footnote{
Example 7.4 of Shioura and Tamura (2015)\citeH[Example 7.4]{ST15jorsj}
also shows this.
See Theorem \ref{THgssubmRpoly} for the continuous case.
}. 

\begin{example} \rm \label{EXmconjIfsubmNot}
Here is an example of a function $f$ 
such that the conjugate function $f\sp{\triangledown}$ is submodular, 
but $f$ is not M$\sp{\natural}$-concave.
Let  
$f: \mathbb{Z}^{2} \to \mathbb{R} \cup \{ -\infty  \}$
be defined by $f(x_{1}, x_{2})  = \min ( 2, x_{1} + 2 x_{2} )$ on 
$\dom f = \{ (x_{1}, x_{2}) \in \mathbb{Z}^{2} 
\mid 0 \leq x_{1} \leq 2,  0 \leq x_{2} \leq 1 \}$, whose numerical values are
\[
 f(0,0) = 0, \ f(1,0) = 1, \  f(2,0) = 2; \quad
 f(0,1) = f(1,1) =  f(2,1) = 2.
\]
This function is not M$\sp{\natural}$-concave because
(M$\sp{\natural}$-EXC[$\ZZ$]) fails for $x=(2,0)$, $y=(0,1)$ and $i=1$.
The conjugate function
$f\sp{\triangledown}: \mathbb{R}^{2} \to \mathbb{R} \cup \{ +\infty  \}$
of (\ref{conjcave2vexZ})
is given  by
\[
 f\sp{\triangledown}(p_{1}, p_{2}) 
 = \max ( 0, 2- 2 p_{1}, 2 - p_{2},2 - 2 p_{1} - p_{2} ) =
   \left\{  \begin{array}{ll}
   0  &  ( p_{1} \geq 1 , p_{2} \geq 2) ,  \\
   2  - 2 p_{1}  &  (2 p_{1} \leq \min(2 , p_{2}), p_{2} \geq 0 ) ,  \\
   2 - p_{2}  &  (p_{2} \leq \min (2, 2 p_{1}),  p_{1} \geq 0 ) ,  \\
   2 - 2 p_{1} - p_{2}  &  (p_{1} \leq 0, p_{2} \leq 0 ) .  \\
             \end{array}  \right.
\]
The function $f\sp{\triangledown}$ is submodular, as is easily verified, 
but it is not L$\sp{\natural}$-convex
since the translation-submodularity (\ref{lnatftrsubmR}) fails 
for $g = f\sp{\triangledown}$, $p=(1,2)$, $q=(0,0)$ and $\alpha = 1$
with $g(p) + g(q) = 0 + 2 =2$ and 
$ g((p - \alpha \vecone) \vee q) + g(p \wedge (q + \alpha \vecone))
= g(0,1) + g(1,1) = 2 + 1 = 3$.
It is also noted that 
$f\sp{\triangledown}(p_{1}/2, p_{2})$ is L$\sp{\natural}$-convex in $(p_{1}, p_{2})$.
\finbox
\end{example}

In spite of the above example,
M$\sp{\natural}$-concavity of a set function
$f: 2\sp{N} \to \RR \cup \{ -\infty \}$
can be characterized by submodularity of 
the conjugate function $f\sp{\triangledown}$, which is defined by
\begin{align} 
 f\sp{\triangledown}(p) 
 &= \max\{  f(X) - p(X)   \mid X \subseteq N \}
\qquad ( p \in \RR\sp{n})
\label{conjcave2vex01}
\end{align}
as an adaptation of (\ref{conjcave2vexZ}).

\begin{theorem}  \label{THmconcavebyconjfn01}
A set function
$f: 2\sp{N} \to \mathbb{R} \cup \{ -\infty  \}$
is M$\sp{\natural}$-concave 
if and only if
the conjugate function 
$f\sp{\triangledown}: \RR\sp{n} \to \mathbb{R} \cup \{ +\infty  \}$
in {\rm (\ref{conjcave2vex01})} is submodular.
\end{theorem}

This theorem can be derived from a combination of 
Theorem 10 of Ausubel and Milgrom (2002)\citeH[Theorem 10]{AM02}
with Theorem \ref{THmconcavgross} in Section~\ref{SCmaximizers01}; 
see also 
Section 7.2.2 of Shioura and Tamura (2015)\citeH[Section 7.2.2]{ST15jorsj}
for an alternative proof.

\subsection{Minimization of L$\sp{\natural}$-convex functions}
\label{SCalgLmin}

The set of the minimizers of an L$\sp{\natural}$-convex function on $\ZZ\sp{n}$
forms a well-behaved ``discrete convex'' subset of $\ZZ\sp{n}$.
Recall from Remark \ref{RMlnatconvexsetZ} that
a nonempty set $P \subseteq \ZZ\sp{n}$ is
called an {L$\sp{\natural}$-convex set}
if
\begin{equation} \label{lnatsettrsubmZ}
 p, q \in P    \ \Longrightarrow  \  
(p - \alpha \vecone) \vee q, \ 
 p \wedge (q + \alpha \vecone) \in P
\qquad (\forall \alpha \in \ZZ_{+}).
\end{equation}
This condition with $\alpha = 0$ gives
\begin{equation} \label{lnatsetsubmZ}
 p, q \in P    \ \Longrightarrow  \  
 p \vee q, \   p \wedge q \in P ,
\end{equation}
which shows that 
an L$\sp{\natural}$-convex set forms a sublattice of $\ZZ\sp{n}$.
A bounded L$\sp{\natural}$-convex set 
has the (uniquely determined) maximal element
and the (uniquely determined) minimal element.

\begin{theorem}  \label{THlfnargminZ}
Let $g: \ZZ\sp{N} \to \RR \cup \{ +\infty \}$ 
be an L$\sp{\natural}$-convex function and assume 
$\argmin g \not= \emptyset$.
Then the set of the minimizers $\argmin g$ is an L$\sp{\natural}$-convex set.
If $\argmin g$ is bounded, there exist the maximal and the minimal minimizer of $g$.
\end{theorem}
\begin{proof}
This follows easily from the translation-submodularity 
in Theorem \ref{THlnatcondZ} (b).
\end{proof}

For an L$\sp{\natural}$-convex function,
the minimality of a function value is characterized by a local condition
as follows 
(Theorem 7.14 of Murota 2003)\citeH[Theorem 7.14]{Mdcasiam}.
Recall the notation $\unitvec{Y}$ for the characteristic vector of a subset $Y$;
see (\ref{charvecdefnotat}).

\begin{theorem}  \label{THlfnlocminZ}
Let $g: \ZZ\sp{N} \to \RR \cup \{ +\infty \}$ 
be an L$\sp{\natural}$-convex function and $p \in \dom g$. 

\noindent {\rm (1)}
If $g(p) > g(q)$ for $q \in \dom g$, then 
$g(p) > g(p + \chi_{Y})$ 
for some $Y \subseteq \suppp(q-p)$
or
$g(p) > g(p - \chi_{Z})$ for some $Z \subseteq \suppm(q-p)$.

\noindent {\rm (2)}
$p$ is a minimizer of $g$ if and only if
\begin{equation} \label{lnatfnlocminZ}
g(p) \leq  g(p + \chi_{Y})
\quad  (\forall \, Y \subseteq N),
\qquad
g(p) \leq  g(p - \chi_{Z})  
\quad  (\forall \, Z \subseteq N).
\end{equation}
\end{theorem}
\begin{proof}
(1)
This follows from Theorem \ref{THlfnlocmindecZ} below.
If $g(q) < g(p)$ in (\ref{lfnlocdecineq}),
 $g(p+\unitvec{Y_{k}}) - g(p) < 0 $ 
for some $k$
or $g(p-\unitvec{Z_{j}}) - g(p) < 0$ for some $j$.
(2) This is immediate from (1).
\end{proof}

\begin{theorem}  \label{THlfnlocmindecZ}
Let $g: \ZZ\sp{N} \to \RR \cup \{ +\infty \}$ 
be an L$\sp{\natural}$-convex function.
For $p, q \in \dom g$ we have
\begin{equation} \label{lfnlocdecineq}
 g(q) \geq g(p) 
+ \sum_{k=1}\sp{m}[ g(p+\unitvec{Y_{k}}) - g(p) ] 
+ \sum_{j=1}\sp{l}[ g(p-\unitvec{Z_{j}}) - g(p) ] ,
\end{equation}
where%
\footnote{
The decomposition (\ref{lnatq-pdec}) is uniquely determined:
$m  = \max(0, q_{1} - p_{1}, \ldots,q_{n} - p_{n} )$, \ 
$ Y_{k} = \{ i \mid q_{i}- p_{i} \geq m+1-k \}$
$(k=1,\ldots,m)$;
$l  = \max(0, p_{1} - q_{1}, \ldots,p_{n} - q_{n} )$, \ 
$ Z_{j} = \{ i \mid p_{i}- q_{i} \geq l+1-j  \}$ 
$(j=1,\ldots,l)$.
} 
$\emptyset \not= Y_{1} \subseteq Y_{2}  
  \subseteq \cdots \subseteq Y_{m} = \suppp(q-p)$, \ 
$\emptyset \not= Z_{1} \subseteq Z_{2}  \subseteq \cdots \subseteq Z_{l}
= \suppm(q-p)$, and
\begin{equation} \label{lnatq-pdec}
q - p = \sum_{k=1}^{m} \unitvec{Y_{k}} -  \sum_{j=1}^{l} \unitvec{Z_{j}} .
\end{equation}
\end{theorem}

\begin{proof}
(1)
If $\suppp(q-p)$ is nonempty,
(\ref{lnatAPR}) for $(q,p)$ implies 
\[
 g(q) \geq  g(p+\unitvec{Y_{1}})  + g(q-\unitvec{Y_{1}}) - g(p)
= [ g(p+\unitvec{Y_{1}}) - g(p) ] + g(q_{2}),
\]
where  $q_{2}=q-\unitvec{Y_{1}}$.
If $\suppp(q_{2}-p)$ is nonempty,
(\ref{lnatAPR}) for $(q_{2},p)$ implies
\[
 g(q_{2}) \geq  g(p+\unitvec{Y_{2}}) + g(q_{2}-\unitvec{Y_{2}})  - g(p)  
=  [ g(p+\unitvec{Y_{2}}) - g(p) ] + g(q_{3}),
\]
where $q_{3}=q_{2}-\unitvec{Y_{2}}
= q-\unitvec{Y_{1}}-\unitvec{Y_{2}}$.
Repeating this, we obtain
$q'=q-\sum_{k=1}\sp{m} \unitvec{Y_{k}} = p \wedge q$
and
\begin{equation} \label{ineq1}
 g(q) \geq g(q') 
+ \sum_{k=1}\sp{m}[ g(p+\unitvec{Y_{k}}) - g(p) ] .
\end{equation}
By the similar procedure starting with $(p, q')$
we obtain 
$p=q'+\sum_{j=1}\sp{l} \unitvec{Z_{j}}$ and
\begin{equation} \label{ineq2}
 g(q') \geq g(p) + \sum_{j=1}\sp{l}[ g(p-\unitvec{Z_{j}}) - g(p) ] .
\end{equation}
Adding (\ref{ineq1}) and (\ref{ineq2})  we obtain (\ref{lfnlocdecineq}).
\end{proof}

\paragraph{Algorithms for L$\sp{\natural}$-convex minimization:}
Algorithms for L$\sp{\natural}$-convex function minimization
are considered by 
Murota (2000b)\citeH{Malg00},
Kolmogorov and Shioura (2009)\citeH{KS09lnatmin},
\RED{
Murota and Shioura (2014, 2017)\citeH{MS14exbndLmin}\citeH{MS16Lmin2ph},
}%
Murota et al.~(2016)\citeH{MSY16auction},
and
Shioura (2017)\citeH{Shi17L};
see also Section 10.3 of Murota (2003)\citeH[Section 10.3]{Mdcasiam}.
Among others we present here the following two algorithms%
\footnote{
Algorithm {\sc Greedy} is called ``steepest descent algorithm'' in 
Section 10.3.1 of Murota (2003)\citeH[Section 10.3.1]{Mdcasiam}.
}. 

\begin{tabbing}       
\= {\bf Algorithm} {\sc Greedy}  \\
\> \quad  Step 0: 
    Find a vector $p\sp{\circ} \in \dom g$ 
  \   and set $p := p\sp{\circ}$. \\
\> \quad  Step 1:
   Find $\varepsilon \in \{ +1, -1 \}$ and $X \subseteq N$
   that minimize $g(p+ \varepsilon \chi_{X})$.  \\
\> \quad  Step 2: 
    If $g(p) \leq g(p + \varepsilon \chi_{X})$,  then output $p$ and stop.  \\ 
\> \quad  Step 3: 
    Set  $p:= p + \varepsilon \chi_{X}$ \ and \ go to Step~1. 
\end{tabbing}

\begin{tabbing}      
\= {\bf Algorithm} {\sc GreedyUpMinimal} \\
\> \quad  Step 0: 
    Find a vector $p\sp{\circ} \in \dom g$ such that 
    $\{ q \mid  q \geq p\sp{\circ} \} \cap \argmin g  \not= \emptyset$ 
    \  \ and set $p := p\sp{\circ}$.  \\
\> \quad  Step 1:
   Find the minimal minimizer $X \subseteq N$ of $g(p+  \chi_{X})$.  \\
\> \quad  Step 2: 
    If $X = \emptyset$,  then output $p$ and stop.   \\ 
\> \quad  Step 3: 
    Set  $p:= p + \chi_{X}$ \ and \ go to Step~1. 
\end{tabbing}

The algorithm {\sc Greedy} can start with an arbitrary initial vector $p\sp{\circ}$
in the effective domain, and the vector $p$ may increase or
decrease depending on $\varepsilon =+1$ or $-1$.
The output of the algorithm {\sc Greedy} is not uniquely determined,
varying with the choice of $\varepsilon$ and $X$ in case of ties
in minimizing $g(p+ \varepsilon \chi_{X})$ in Step~1. 
Step~1 amounts to minimizing two set functions
$\rho_{+}(X)=g(p+\chi_{X})-g(p)$ and
$\rho_{-}(X)=g(p-\chi_{X})-g(p)$
over all subsets $X$ of $N$.
As a consequence of submodularity of $g$, 
both $\rho_{+}$ and $\rho_{-}$ are submodular set functions
and they can be minimized efficiently (i.e., in strongly polynomial time).
The second algorithm, {\sc GreedyUpMinimal}, keeps increasing the vector $p$,
until it reaches the smallest minimizer of $g$ 
that is greater than or equal to $p\sp{\circ}$.
Accordingly, the initial vector $p\sp{\circ}$ must be small enough to 
ensure $\{ q \mid  q \geq p\sp{\circ} \} \cap \argmin g  \not= \emptyset$.
If $g$ has the minimal minimizer $p\sp{*}_{\min}$ and
$p\sp{\circ} \leq p\sp{*}_{\min}$, 
then the algorithm {\sc GreedyUpMinimal} outputs $p\sp{*}_{\min}$.

The correctness of the algorithms, at their termination,
is guaranteed by Theorem \ref{THlfnlocminZ},
whereas the following exact bounds for the number of updates of $p$
are established recently 
by Murota and Shioura (2014)\citeH{MS14exbndLmin}.

\begin{theorem}\label{THlminalgiter}
\quad  

\noindent {\rm (1)}
The number of updates of $p$ in the algorithm {\sc Greedy} is exactly equal to 
\begin{equation}  \label{distmudef}
 \mu(p\sp{\circ})
 =   \min\{\| p\sp{\circ} - p\sp{*} \|_\infty\sp{+}
+ \| p\sp{\circ} - p\sp{*} \|_\infty\sp{-} \mid p\sp{*} \in \arg\min g\} 
\end{equation}
under the assumption of $\argmin g \not= \emptyset$,
where
$\| q \|_\infty\sp{+} = \max(0, q_{1}, q_{2}, \ldots, q_{n} )$
and 
$\| q \|_\infty\sp{-} = \max(0, -q_{1}, -q_{2}, \ldots, -q_{n} )$.

\noindent {\rm (2)}
The number of updates of $p$ in the algorithm {\sc GreedyUpMinimal} 
is exactly equal to%
\footnote{
We have $\hat\mu(p\sp{\circ})=+\infty$ if there is no
$p\sp{*} \in \arg\min g$ with $p\sp{*} \geq p\sp{\circ}$.
It can be shown that
$\hat\mu(p\sp{\circ}) \in \{ \mu(p\sp{\circ}), +\infty \}$
holds for all $p\sp{\circ} \in \ZZ\sp{n}$;  see 
Shioura (2017)\citeH{Shi17L}
for the proof.
} 
\begin{equation}  \label{disthatmudef}
 \hat\mu(p\sp{\circ})
= \min\{\| p\sp{\circ} - p\sp{*}  \|_\infty 
   \mid p\sp{*} \in \arg\min g,\ p\sp{\circ} \leq p\sp{*}  \} 
\end{equation}
under the assumption of
$\{ q \mid  q \geq p\sp{\circ} \} \cap \argmin g  \not= \emptyset$.
If the minimal minimizer $p\sp{*}_{\min}$ exists and
$p\sp{\circ} \leq p\sp{*}_{\min}$, then 
$ \hat\mu(p\sp{\circ}) =
\| p\sp{\circ} - p\sp{*}_{\min} \|_{\infty}$.
\end{theorem}

We can conceive variants of
 {\sc GreedyUpMinimal} 
by changing  ``{\sc Up}'' to  ``{\sc Down}''
and/or  
``{\sc Minimal}'' to  ``{\sc Maximal}'' according to
Table \ref{TBgreedyLalg} (a).
For example, the algorithm {\sc GreedyDownMinimal} 
is obtained from {\sc GreedyUpMinimal} 
by changing Steps 0 and 1 to:
\begin{tabbing}      
\= \quad  Step 0: 
    Find a vector $p\sp{\circ} \in \dom g$ such that 
    \underline{$\{ q \mid  q \leq p\sp{\circ} \}$} $\cap \argmin g  \not= \emptyset$ 
    \  \ and set $p := p\sp{\circ}$.  \\
\> \quad  Step 1:
   Find the \underline{maximal} minimizer $X \subseteq N$ of 
\underline{$g(p-  \chi_{X})$}.  
\end{tabbing}
Starting with an initial vector $p\sp{\circ}$  large enough to 
ensure $\{ q \mid  q \leq p\sp{\circ} \} \cap \argmin g  \not= \emptyset$,
this algorithm keeps decreasing the vector $p$.
If $g$ has the minimal minimizer $p\sp{*}_{\min}$,
the algorithm stops when it reaches $p\sp{*}_{\min}$.
The number of updates of $p$ in {\sc GreedyDownMinimal} 
is exactly equal to $\| p\sp{\circ} - p\sp{*}_{\min} \|_{\infty}$
(Proposition 3.7 of Murota et al.~2016).
\citeH[Proposition 3.7]{MSY16auction}%
Table \ref{TBgreedyLalg} (b) shows the output 
and the number of updates of $p$ for the four algorithms.

\begin{table}
\caption{Algorithms {\sc Greedy-\{Up, Down\}-\{Minimal, Maximal\}} }
\label{TBgreedyLalg}
\begin{center}
(a) Description of the algorithms
\\
\begin{tabular}{ll|l|l}
\hline 
\multicolumn{2}{l|}{{\sc Greedy}}
    & \multicolumn{1}{c|}{ {\sc Minimal} } & \multicolumn{1}{c}{ {\sc Maximal} }
\\ \hline 
{\sc Up} &  Step 0
     & \multicolumn{2}{c}{$p\sp{\circ}$ such that \ 
             $\{ q \mid  q \geq p\sp{\circ} \} \cap \argmin g  \not= \emptyset$
   \quad (i.e., $p\sp{*}_{\max} \geq p\sp{\circ}$)}
\\  &   Step 1
    &  minimal minimizer $X$ of $g(p+ \chi_X)$    
     &  maximal minimizer $X$ of $g(p+ \chi_X)$
\\ \hline 
{\sc Down} &  Step 0 
     & \multicolumn{2}{c}{$p\sp{\circ}$ such that \  
     $\{ q \mid  q \leq p\sp{\circ} \} \cap \argmin g  \not= \emptyset$
   \quad (i.e., $p\sp{*}_{\min} \leq p\sp{\circ}$)}
\\   & Step 1
     &  maximal minimizer $X$ of $g(p- \chi_X)$
     &  minimal minimizer $X$ of $g(p- \chi_X)$    
\\ \hline 
\\
     \multicolumn{4}{c}{(b) Output and the exact number of updates of $p$}
\\ \hline 
\multicolumn{2}{l|}{{\sc Greedy}}
     & \multicolumn{1}{c|}{ {\sc Minimal} } & \multicolumn{1}{c}{ {\sc Maximal} }
\\ \hline 
{\sc Up}  & Output
     & $p\sp{*}_{\min}$ if $p\sp{*}_{\min} \geq p\sp{\circ}$; otherwise 
     & $p\sp{*}_{\max}$
\\   &
     & \quad $\min( \{ q \mid  q \geq p\sp{\circ} \} \cap \argmin g)$
     &
\\  \multicolumn{2}{r|}{\# Updates}  
   &  $\| p\sp{\circ} - p\sp{*}_{\min} \|_{\infty}$
       if $p\sp{*}_{\min} \geq p\sp{\circ}$; 
                        & $\| p\sp{\circ} - p\sp{*}_{\max} \|_{\infty}$
\\
   &  &  \quad  otherwise   $\hat\mu(p\sp{\circ})$   & 
\\ \hline 
{\sc Down}   & Output
     & $p\sp{*}_{\min}$
     & $p\sp{*}_{\max}$ if $p\sp{*}_{\max} \leq p\sp{\circ}$; otherwise 
\\   &  &
     & \quad $\max( \{ q \mid  q \leq p\sp{\circ} \} \cap \argmin g)$
\\  \multicolumn{2}{r|}{\# Updates}  
   &  $\| p\sp{\circ} - p\sp{*}_{\min} \|_{\infty}$ 
   &  $\| p\sp{\circ} - p\sp{*}_{\max} \|_{\infty}$
       if $p\sp{*}_{\max} \leq p\sp{\circ}$;
\\ & &  &  \quad otherwise $\check\mu(p\sp{\circ})$
\\ \hline 
   \multicolumn{4}{c}{$p\sp{\circ}$: initial vector, \ \  
      $p\sp{*}_{\min}$: minimal minimizer of $g$, \ \ 
      $p\sp{*}_{\max}$: maximal minimizer of $g$}
\\
\multicolumn{4}{c}{$
 \hat\mu(p\sp{\circ}) = \min\{\| p\sp{\circ} - p\sp{*}  \|_\infty 
   \mid p\sp{*} \in \arg\min g,\ p\sp{\circ} \leq p\sp{*}  \}$}
\\
\multicolumn{4}{c}{$
 \check\mu(p\sp{\circ}) = \min\{\| p\sp{\circ} - p\sp{*}  \|_\infty 
   \mid p\sp{*} \in \arg\min g,\ p\sp{\circ} \geq p\sp{*}  \}$}
\\ \hline 
\end{tabular}
\end{center}
\end{table}

In Section \ref{SCauction} we shall discuss connection of
L$\sp{\natural}$-convex function minimization to iterative auctions.
The algorithm {\sc GreedyUpMinimal}
corresponds to ascending (English) auctions,
and {\sc GreedyDownMaximal} to descending (Dutch) auctions.
In connection to two-phase (English--Dutch) auctions 
it is natural to consider two-phase algorithms
for L$\sp{\natural}$-convex function minimization.

The combination of 
{\sc GreedyUpMinimal} and {\sc GreedyDownMaximal} 
results in the following algorithm:

\begin{tabbing}      
\= {\bf Algorithm} {\sc TwoPhaseMinMax} \\
\> \quad  Step 0: 
    Find a vector $p\sp{\circ} \in \dom g$  \ and set $p := p\sp{\circ}$.
    Go to Up Phase.  \\
\> \quad Up Phase: 
\\
\> \qquad  Step U1:
   Find the minimal minimizer $X \subseteq N$ of $g(p+  \chi_{X})$.  \\
\> \qquad  Step U2: 
 If $X = \emptyset$, then go to Down Phase.  \\
\> \qquad  Step U3: 
    Set  $p:= p + \chi_{X}$ \ and \ go to Step~U1. \\
\> \quad Down Phase: 
\\
\> \qquad  Step D1:
   Find the minimal minimizer $X \subseteq N$ of $g(p -  \chi_{X})$.  \\
\> \qquad  Step D2: 
 If $X = \emptyset$, then output $p$ and stop.  \\
\> \qquad  Step D3: 
    Set  $p:= p - \chi_{X}$ \ and \ go to Step~D1. 
\end{tabbing}

It can be shown from Theorem \ref{THlnatcondZ} (d) that, 
at the end of the up phase, the vector $p$ satisfies the condition 
\ $\{ q \mid  q \leq p \} \cap \argmin g  \not= \emptyset$ \ 
required for an initial vector of {\sc GreedyDownMaximal}.
Therefore, the output of {\sc TwoPhaseMinMax} 
is guaranteed to be a minimizer of $g$.
An upper bound on the number of updates of $p$ is
given in 
Theorem 4.13 of Murota et al.~(2016),
\citeH[Theorem 4.13]{MSY16auction}%
which is improved to the following statement by 
\RED{
Murota and Shioura (2017)\citeH{MS16Lmin2ph};
}%
see also Remark~\ref{RMtwophaseMinMaxbnd}.
Recall the definition of $\mu(p\sp{\circ})$ from (\ref{distmudef}).

\begin{theorem}\label{THlmintwophaseMinMax}
For any initial vector $p\sp{\circ}$,
the algorithm {\sc TwoPhaseMinMax} terminates
by outputting some minimizer of $g$.
The number of updates of the vector $p$ 
is bounded by $\mu(p\sp{\circ})$ in the up phase 
and by $\mu(p\sp{\circ})$ in the down phase; 
in total, bounded by $2\mu(p\sp{\circ})$.
\end{theorem}

For the analysis of the Vickrey--English--Dutch auction algorithm
(Section \ref{SCunitdemandauction}),
it is convenient to consider the combination of 
{\sc GreedyUpMinimal} and {\sc GreedyDownMinimal}.
The resulting two-phase algorithm is called {\sc TwoPhaseMinMin},
which is the same as {\sc TwoPhaseMinMax} except that  
Step D1 is replaced by 
\begin{tabbing}      
\= \qquad  Step D1:
   Find the maximal minimizer $X \subseteq N$ of $g(p -  \chi_{X})$. 
\end{tabbing}
An upper bound on the number of updates of $p$ is
given in 
Theorem 4.12 of Murota et al.~(2016)\citeH[Theorem 4.12]{MSY16auction},
which is improved by 
\RED{
Murota and Shioura (2017)\citeH{MS16Lmin2ph}
}
to the following statement;
see also Remark~\ref{RMtwophaseMinMinbnd}.
Recall the notation
$\| q \|_\infty\sp{+} = \max(0, q_{1}, q_{2}, \ldots, q_{n} )$
for $q \in \ZZ\sp{n}$.

\begin{theorem}\label{THlmintwophaseMinMin}
For any initial vector $p\sp{\circ}$,
the algorithm {\sc TwoPhaseMinMin} terminates
by outputting the minimal minimizer $p\sp{*}_{\min}$ of $g$,
if $p\sp{*}_{\min}$ exists.
The number of updates of the vector $p$ 
is bounded by $\mu(p\sp{\circ})$ in the up phase
and is exactly equal to
$\| p\sp{\circ} - p\sp{*}_{\min} \|_\infty\sp{+}$
in the down phase;
in total, bounded by 
$\mu(p\sp{\circ}) + \| p^{\circ} - p\sp{*}_{\min} \|_\infty\sp{+}$.
\end{theorem}

\begin{remark} \rm  \label{RMtwophaseMinMaxbnd}
For the algorithm {\sc TwoPhaseMinMax},
Theorem 4.13 of Murota et al.~(2016)\citeH[Theorem 4.13]{MSY16auction}
shows that the number of updates of $p$ is bounded by
$\eta(p\sp{\circ}, p\sp{*}) =  \| p\sp{\circ} - p\sp{*} \|_\infty\sp{+} 
  + \| p\sp{\circ} - p\sp{*} \|_\infty\sp{-} $
in the up phase,
by $2 \eta(p\sp{\circ}, p\sp{*})$ 
in the down phase,
and  in total
by $3 \eta(p\sp{\circ}, p\sp{*})$,
where $p\sp{*}$ denotes the output of the algorithm.
Theorem \ref{THlmintwophaseMinMax}
gives an improved bound
since $\eta(p\sp{\circ}, p\sp{*}) \geq \mu(p\sp{\circ})$.
Theorem 3.2 of Murota et al.~(2013a)\citeH[Theorem 3.2]{MSY13},
though stated as a bound for a two-phase auction algorithm,
implies that
the number of updates of $p$ in {\sc TwoPhaseMinMax} 
is bounded by
$\mu(p\sp{\circ})$ in the up phase,
by $2 \mu(p\sp{\circ})$
in the down phase,
and in total by $3 \mu(p\sp{\circ})$;
see Murota et al.~(2013b)\citeH{MSY13metr}
for the proof.
\finbox
\end{remark}

\begin{remark} \rm  \label{RMtwophaseMinMinbnd}
For the algorithm {\sc TwoPhaseMinMin},
Theorem 4.12 of Murota et al.~(2016)\citeH[Theorem 4.12]{MSY16auction}
shows that the number of updates of $p$ is bounded by
$\eta(p\sp{\circ}, p\sp{*}_{\min}) =  \| p\sp{\circ} - p\sp{*}_{\min} \|_\infty\sp{+} 
  + \| p\sp{\circ} - p\sp{*}_{\min} \|_\infty\sp{-} $
in the up phase,
by $2 \eta(p\sp{\circ}, p\sp{*}_{\min})$ in the down phase,
and  in total
by $3 \eta(p\sp{\circ}, p\sp{*}_{\min})$.
Theorem \ref{THlmintwophaseMinMin}
gives an improved bound since
$\eta(p\sp{\circ}, p\sp{*}_{\min}) \geq \mu(p\sp{\circ})$
and
$\eta(p\sp{\circ}, p\sp{*}_{\min}) \geq 
\| p\sp{\circ} - p\sp{*}_{\min} \|_\infty\sp{+}$.
\finbox
\end{remark}

\begin{remark} \rm  \label{RMlminTwoPhaseOther}
Besides {\sc TwoPhaseMinMax} and {\sc TwoPhaseMinMin},
we can obtain other variants of two-phase algorithms
by choosing appropriate combinations 
from among the algorithms {\sc Greedy-\{Up, Down\}-\{Minimal, Maximal\}}
listed in Table \ref{TBgreedyLalg}.
\finbox
\end{remark}

\subsection{Concluding remarks of section \ref{SCconjuLconv}}

In this paper we put more emphasis on 
M$\sp{\natural}$-concave functions
and give L$\sp{\natural}$-convex functions
only a secondary role as the conjugate of M$\sp{\natural}$-concave functions,
though, in fact,  they are equally important and play 
symmetric roles in discrete convex analysis.

The concept of L-convex functions is formulated by
Murota (1998)\citeH{Mdca},
compatibly with the accepted understanding of the relationship 
between submodularity and convexity expounded by
Lov\'{a}sz (1983)\citeH{Lov83}.
Then L$\sp{\natural}$-convex functions are introduced by
Fujishige and Murota (2000)\citeH{FM00}
as a variant of L-convex functions,
together with the observation that
they coincide with submodular integrally convex functions
considered earlier by 
Favati and Tardella  (1990)\citeH{FT90}.
The concept of quasi L-convex functions is also introduced by
Murota and Shioura (2003)\citeH{MS03quasi},
in accordance with  quasisupermodularity of
Milgrom and Shannon (1994)\citeH{MS94eco}.
L-convex functions in continuous variables 
are defined by
Murota and Shioura (2000, 2004a)\citeH{MS00poly}\citeH{MS04conjreal},
partly motivated by 
a phenomenon inherent in the network flow/tension problem
described in 
Section 2.2.1 of Murota (2003)\citeH[Section 2.2.1]{Mdcasiam}.

Recently, the concept of L-convex functions 
is extended to functions on graph structures,
which are more general than  $\ZZ\sp{n}$.
See
Kolmogorov (2011)\citeH{Kol11},
Huber and Kolmogorov (2012)\citeH{HK12},
Fujishige (2014)\citeH{Fuj14bisubm},
and 
\RED{
Hirai (2015, 2016a, 2018)\citeH{Hir15Lextprox}\citeH{Hir16min0ext}\citeH{Hir16Lgraph}
}%
for the recent development.



\section{Iterative Auctions}
\label{SCauction}

This section presents a unified method of analysis 
for iterative auctions (dynamic auctions)
by combining the Lyapunov function approach of 
Ausubel (2006)\citeH{Aus06}
with discrete convex analysis.
We are mainly concerned with the multi-item multi-unit model, 
where there are multiple indivisible goods for sale and 
each good may have several units.
The bidders' valuation functions are assumed to have gross substitutes property.
This section is mostly based on 
Murota et al.~(2013a, 2016)\citeH{MSY13}\citeH{MSY16auction}
with some new results from 
\RED{
Murota and Shioura (2017)\citeH{MS16Lmin2ph}.
}

\subsection{Auction models and Walrasian equilibrium}
\label{SCauctionmodel}

Fundamental concepts about auctions
are introduced here only briefly; see, e.g., 
Milgrom (2004)\citeH{Mil04book}, 
Cramton et al.~(2006)\citeH{CSS06},
and Blumrosen and Nisan (2007)\citeH{BN07}
for comprehensive accounts.

 In the auction market, there are
 $n$ types of items or goods, denoted by $N= \{1, 2, \ldots, n\}$, and
 $m$ bidders,
 denoted by $M=\{1, 2,\ldots, m\}$, where $m \geq 2$.
 We have $u_{i}$ units available for each item $i \in N$,
where $u_{i}$ is a positive integer.
 We denote the integer interval as
$[\bm{0}, u]_{\ZZ} = \{x \in \ZZ\sp{n} \mid \ \bm{0} \leq x \leq u\}$,  
where
$u=(u_{1}, u_{2}, \ldots, u_{n})$.
 Each vector $x \in [\bm{0}, u]_{\ZZ}$ is called a {\it bundle};
a bundle 
$x=(x_{1}, x_{2}, \ldots, x_{n})$ 
corresponds to a (multi-)set of items, where $x_{i}$
represents the multiplicity of item $i \in N$.
 Each bidder $j \in M$ has his valuation function
$f_{j}: [\bm{0}, u]_{\ZZ} \to \RR$;
the number 
$f_{j}(x)$ represents the
value of the bundle $x$ worth to bidder $j$.
 The case with $u_{i}=1$ for all $i \in N$
is referred to as {\em single-unit auction},
while the general case with $u \in \ZZ_{++}\sp{n}$ as {\em multi-unit auction}.
 Note that $[\bm{0}, \bm{1}]_{\ZZ} =\{0,1\}\sp{n}$,
where $\bm{1} = (1,1,\ldots, 1)$.
 A further special case where 
each bidder is interested in getting at most one item is called
{\em unit-demand auction}.

 In an auction, we want to find an
efficient allocation
and market clearing prices.
 An {\it allocation} of items is defined as a set of bundles
$x_1, x_2, \ldots, x_m \in [\bm{0}, u]_{\ZZ}$
satisfying $\sum_{j=1}^m x_{j} = u$.
 Given a price vector $p \in \RR\sp{n}_{+}$, each bidder $j \in M$ wants to
have a bundle $x$ which maximizes the value $f_{j}(x) - p^\top x$.
 For $j \in M$ and $p \in \RR\sp{n}_{+}$, define
\begin{eqnarray}
D_{j}(p) & = & D(p; f_{j}) 
= \arg\max\{f_{j}(x) - p^\top x \mid x \in [\bm{0}, u]_{\ZZ} \, \}.
\label{eqn:def-demand-set}
\end{eqnarray}
We call 
the set $D_{j}(p) \subseteq [\bm{0}, u]_{\ZZ}$
the {\em demand set}.
 The auctioneer wants to find 
a pair of a price vector $p^*$ and an allocation
$x_1^*, x_2^*, \ldots, x_m^*$
such that $x_{j}^* \in D_{j}(p^*)$
for all $j \in M$.
 Such a pair is called a {\em (Walrasian) equilibrium}
and $p^*$ is called a {\em (Walrasian) equilibrium price vector}.

 Although the Walrasian equilibrium
 possesses several desirable properties, it does not always exist.
 Some condition has to be imposed on bidders' valuation functions
 before the existence of a Walrasian equilibrium can be  guaranteed.
 Throughout this section we assume the following conditions for
 bidders' valuation functions $f_{j}$ $(j=1,2,\ldots, m)$:
\begin{quote}
{\bf (A0)} $f_{j}$ is monotone nondecreasing,
\\
{\bf (A1)}  $f_{j}$  is an M$\sp{\natural}$-concave function, 
\\
{\bf (A2)}   $f_{j}$ takes integer values.
\end{quote}
Recall from Sections \ref{SCmaximizers01} and \ref{SCmaximizersZ}
that a valuation function is M$\sp{\natural}$-concave
if and only if it 
has the gross substitutes (GS) property (in its stronger form);
see Theorems \ref{THmconcavgross} and \ref{THmnatGSLAD}, in particular.

\begin{remark} \rm  \label{RMauctIntroSingUnit}
Whereas we are mainly concerned with the multi-unit model here,
the single-unit model is treated more extensively in the literature,
e.g., 
Kelso and Crawford (1982)\citeH{KC82},
Gul and Stacchetti (1999, 2000)\citeH{GS99}\citeH{GS00},
Milgrom (2004)\citeH{Mil04book},
Blumrosen and Nisan (2007)\citeH{BN07},
Cramton et al.~(2006)\citeH{CSS06},
and Milgrom and Strulovici (2009)\citeH{MS09sbst}.
The method of analysis presented in this section 
remains meaningful and interesting also for the single-unit model.
\finbox
\end{remark}

\begin{remark} \rm  \label{RMauctIntroUnitDemand}
 Iterative auctions for unit-demand auction are
discussed extensively in the literature,
e.g.,
Vickrey (1961)\citeH{Vic61},
Demange et al.~(1986)\citeH{DGS86},
Mo et al.~(1988)\citeH{MTL88},
Sankaran (1994)\citeH{San94},
Mishra and Parkes (2009)\citeH{MP09vd},
Andersson et al.~(2013)\citeH{AAT13},
and Andersson and Erlanson (2013)\citeH{AE13}.
Specifically,
the Vickrey--English auction by 
Demange et al.~(1986)\citeH{DGS86},
the Vickrey--Dutch auction by 
Mishra and Parkes (2009)\citeH{MP09vd},
and the Vickrey--English--Dutch auction by 
Andersson and Erlanson (2013)\citeH{AE13} 
are such iterative auctions.
Although these three algorithms are proposed independently of the iterative
auction algorithms for the multi-unit model,
it is possible to give a unified treatment 
of these iterative auction algorithms by
revealing their relationship to
the Lyapunov function approach (Section \ref{SCunitdemandauction}).
\finbox
\end{remark}

\subsection{Lyapunov function approach to iterative auctions}
\label{SClyapapproach}

In this section we describe the Lyapunov function-based iterative auctions,
which is proposed by 
Ausubel (2006)\citeH{Aus06}.
Our objective is to clarify the underlying mathematical structure
with the aid of discrete convex analysis, 
and to derive sharp upper or exact bounds on the
number of iterations in the iterative auctions.

 For $j \in M$ and $p \in \RR\sp{n}_{+}$, we define
the {\em indirect utility function}
 $V_{j}: \RR\sp{n}_{+} \to \RR$ by 
\begin{eqnarray}
V_{j}(p) &=& V(p; f_{j}) = \max\{f_{j}(x) - p^\top x \mid x \in [\bm{0}, u]_{\ZZ} \, \},
\label{indirutildef}
\end{eqnarray}
and the {\em Lyapunov function} by
\begin{equation}  \label{Lyapfndef}
  L(p) = \sum_{j=1}^m V_{j}(p) + u^\top p \qquad (p \in \RR\sp{n}),
\end{equation}
where the vector $u \in \ZZ\sp{n}_{+}$ represents
the numbers of available units for items in $N$.

Under the assumptions (A0)--(A2) it can be shown%
\footnote{
The integrality follows from the fact that
 an integer-valued M$\sp{\natural}$-concave function $f$ on $\ZZ\sp{n}$
has an integral subgradient (or supergradient) at every point $x$ in $\dom f$.
} 
that there exists an equilibrium price vector $p^*$ 
whose components are nonnegative integers.
Henceforth we assume that the price vector $p$
in iterative auctions is always chosen to be a nonnegative integer vector,
i.e., $p \in \ZZ_{+}\sp{n}$.
Accordingly, we regard $V_{j}$ and $L$ as integer-valued functions
defined on nonnegative integers, i.e.,
$V_{j}: \ZZ_{+}\sp{n} \to \ZZ$ and
$L: \ZZ_{+}\sp{n} \to \ZZ$.

The ascending auction algorithm based on the 
Lyapunov function 
(Ausubel 2006\citeH{Aus06})
is as follows:

\begin{tabbing}     
\= {\bf Algorithm} {\sc AscendMinimal} \\
\> \quad  Step 0: 
 Set $p:=p\sp{\circ}$, where
$p\sp{\circ} \in \ZZ_{+}\sp{n}$ is an arbitrary vector satisfying 
$p\sp{\circ} \leq p\sp{*}_{\min}$ \ (e.g., $p\sp{\circ} = \bm{0}$).  \\
\> \quad  Step 1:
    Find the minimal minimizer $X \subseteq N$ of $L(p+ \chi_X)$.  \\
\> \quad  Step 2: 
    If $X = \emptyset$,  then output $p$ and stop.   \\ 
\> \quad  Step 3: 
    Set  $p:= p + \chi_{X}$ \ and \ go to Step~1. 
\end{tabbing}

The above algorithm can be interpreted in auction terms as follows%
\footnote{
See Appendix B of Ausubel (2006)\citeH[Appendix B]{Aus06}
for details about the implementation of Steps 2 and 3.
}:  

\begin{tabbing}     
\= {\bf Algorithm} {\sc AscendMinimal} (in auction terms) \\
\> \quad  Step 0: 
The auctioneer sets
 $p:=p\sp{\circ}$, where
$p\sp{\circ} \in \ZZ\sp{n}$ should satisfy $p\sp{\circ} \leq p\sp{*}_{\min}$.
\\
\> \quad  Step 1:
The auctioneer asks the bidders to report their demand sets $D_{j}(p)$ $(j \in M)$,
\\
\> \quad  \phantom{Step 1:} 
 and finds the minimal minimizer $X \subseteq N$ of $L(p+ \chi_X)$.
\\
\> \quad  Step 2: 
The auctioneer checks if $X = \emptyset$; 
if $X = \emptyset$  holds, then the auctioneer
\\
\> \quad  \phantom{Step 1:}
 reports $p$ as the final price vector and stop.
\\
\> \quad  Step 3: 
   The auctioneer sets $p:=p + \chi_X$ and returns to Step~1. 
\end{tabbing}

The analysis of the algorithm {\sc AscendMinimal} 
can be made transparent 
by using concepts and results from discrete convex analysis.
Before presenting formal theorems, we enumerate 
the major mathematical ingredients.

\begin{itemize}
\item
As pointed out by  
Ausubel (2006)\citeH{Aus06},
the Walrasian equilibrium price vector can be characterized as a 
minimizer of the Lyapunov function $L$
and an iterative auction algorithm can be understood as 
a minimization process of the Lyapunov function $L(p)$.
See Theorem~\ref{THlyapminzerequil}.

\item
The conjugate function of an M$\sp{\natural}$-concave function
is an L$\sp{\natural}$-convex function, and vice versa
(the conjugacy theorem in Section \ref{SCconjugacy}).
Hence the indirect utility function $V_{j}$
is an L$\sp{\natural}$-convex function
and therefore, the Lyapunov function $L$
is an L$\sp{\natural}$-convex function. 
See Theorem~\ref{THlyaplnatfn}.

\item
The L$\sp{\natural}$-convexity of the Lyapunov function $L$
implies a nice combinatorial structure of the equilibrium prices.
The set of the equilibrium prices is an L$\sp{\natural}$-convex set
(Remark \ref{RMlnatconvexsetZ}),
which is more special than just being a sublattice.
See Theorem~\ref{THlyapequilLnatset}.

\item
The L$\sp{\natural}$-convexity of the Lyapunov function $L$
enables us to utilize general results on 
L$\sp{\natural}$-convex function minimization (Section \ref{SCalgLmin}) 
to analyze the behavior of iterative auction algorithms, 
such as convergence to an equilibrium price and the number of iterations
needed to reach the equilibrium price.
See 
Theorem~\ref{THlyapnumiterAscMin} as well as Theorem~\ref{THunitdemanditerbound}.
\end{itemize}

We now present the theorems substantiating the above-mentioned points.
The conditions (A0)--(A2) are assumed implicitly 
in the following four theorems.
The first theorem is due to
Ausubel (2006)\citeH{Aus06}.

\begin{theorem} \label{THlyapminzerequil}
A vector $p \in \ZZ_{+}\sp{n}$ is an equilibrium price vector
 if and only if it is a minimizer of the Lyapunov function $L$.
\end{theorem}
\begin{proof}
The key of the proof is the fact  
that the set of excess supply vectors
at a price vector~$p$, i.e., 
$\{u - \sum_{j=1}^m x_{j} \mid x_{j} \in D_{j}(p) \ (j=1,2,\ldots, m)\}$,
coincides with the set of subgradients of  the Lyapunov function $L$ at $p$;
see Ausubel (2006)\citeH{Aus06}.
\end{proof}

\begin{theorem} \label{THlyaplnatfn}
\quad   

\noindent
{\rm (1)} 
For each $j \in M$, the indirect utility function $V_{j}$
is an L$\sp{\natural}$-convex function.

\noindent
{\rm (2)} 
The Lyapunov function $L$ is an L$\sp{\natural}$-convex function. 
\end{theorem}
\begin{proof}
(1) 
When regarded as  $V_{j}: \ZZ_{+}\sp{n} \to \ZZ$,
the definition (\ref{indirutildef}) of $V_{j}$
shows that $V_{j}$ is the conjugate function of $f_{j}$ in the sense of
 (\ref{conjcave2vexZZ}).
That is,
$V_{j} = f_{j}\sp{\triangledown}$
in the notation of Section \ref{SCconjugacy}.
Then Theorem~\ref{THlmconjcavevexZ} (2) shows the L$\sp{\natural}$-convexity
of $V_{j}$.

(2) In the definition (\ref{Lyapfndef}) of $L$,
each $V_{j}$ is L$\sp{\natural}$-convex by (1), 
and the linear term $u^\top p$ is obviously L$\sp{\natural}$-convex.
The sum of L$\sp{\natural}$-convex functions is again
L$\sp{\natural}$-convex by Theorem~\ref{THlnatcondZ}.
Hence the Lyapunov function $L$ is L$\sp{\natural}$-convex.
\end{proof}

\begin{theorem} \label{THlyapequilLnatset}
The equilibrium price vectors form 
a bounded L$\sp{\natural}$-convex set%
\footnote{
See Remark \ref{RMlnatconvexsetZ} for L$\sp{\natural}$-convex sets.
If we consider real price vectors, 
the equilibrium price vectors form an L$\sp{\natural}$-convex polyhedron.
}.  
That is, for two equilibrium price vectors $p\sp{*}$, $q\sp{*}$ 
and any nonnegative integer $\alpha$,
both $(p\sp{*} - \alpha \vecone) \vee q\sp{*}$
and $p\sp{*} \wedge (q\sp{*} + \alpha \vecone)$
are equilibrium price vectors.
In particular, the minimal equilibrium price vector $p\sp{*}_{\min}$
and the maximal equilibrium price vector $p\sp{*}_{\max}$
are uniquely determined.
\end{theorem}
\begin{proof}
This follows from the L$\sp{\natural}$-convexity of the Lyapunov function
(Theorem~\ref{THlyaplnatfn})
and the L$\sp{\natural}$-convexity of the set of the minimizers
(Remark \ref{RMlnatconvexsetZ});
the boundedness is easily shown.
\end{proof}

\begin{theorem} \label{THlyapnumiterAscMin}
For an initial vector $p^\circ$ with $p\sp{\circ} \leq p\sp{*}_{\min}$,
the algorithm {\sc AscendMinimal} outputs 
the minimal equilibrium price vector $p\sp{*}_{\min}$
and the number of updates of the price vector 
is  exactly equal to $\|p\sp{*}_{\min} - p\sp{\circ} \|_\infty$.
\end{theorem}
\begin{proof}
 The Lyapunov function $L$ is an L$\sp{\natural}$-convex function 
by Theorem~\ref{THlyaplnatfn}, 
and the algorithm {\sc AscendMinimal} is nothing but
the algorithm {\sc GreedyUpMinimal} in Section \ref{SCalgLmin}
applied to $L$.
Since the minimal minimizer of the Lyapunov function $L$ is 
the minimal equilibrium price vector $p\sp{*}_{\min}$ 
by Theorem~\ref{THlyapminzerequil},
the auction algorithm {\sc AscendMinimal} 
yields the minimal equilibrium price vector $p\sp{*}_{\min}$. 
The number of updates of the price vector 
is equal to $\|p\sp{*}_{\min} - p\sp{\circ} \|_\infty$
by Theorem~\ref{THlminalgiter} (2).
\end{proof}

Theorem~\ref{THlyapnumiterAscMin} is due to 
Murota et al.~(2016)\citeH{MSY16auction},
while the finite termination is noted in 
Ausubel (2006)\citeH{Aus06}.
The bound for the number of iterations
in {\sc AscendMinimal} is given as the
$\ell_\infty$-distance from the initial price vector $p\sp{\circ}$ to the
minimal equilibrium price vector $p\sp{*}_{\min}$.
This implies, in particular, that the trajectory of the price vector
generated by the ascending auction
is the ``shortest'' path between the initial vector
and the minimal equilibrium price vector.

\paragraph{Variants of auction algorithms:}
A variant of the ascending auction algorithm,
called {\sc AscendMaximal},
is obtained through the application of 
the algorithm {\sc GreedyUpMaximal} 
in Section \ref{SCalgLmin} to the Lyapunov function $L$.
Two other variants of the descending auction algorithm,
called {\sc DescendMaximal} and {\sc DescendMinimal},
are obtained through
the application of 
the algorithms {\sc GreedyDownMaximal} and {\sc GreedyDownMinimal}
in Section \ref{SCalgLmin} to the Lyapunov function $L$,
where {\sc DescendMaximal} 
coincides with the descending auction algorithm in
Ausubel (2006)\citeH{Aus06}.
The general results for L$\sp{\natural}$-convex function minimization 
summarized in Table \ref{TBgreedyLalg} (b) in Section \ref{SCalgLmin} imply
the following exact bounds 
(Murota et al.~2016)\citeH{MSY16auction}.

\begin{theorem} \label{THauctionvariants}
\quad   

\noindent
{\rm (1)} 
For an initial vector $p^\circ$ with $p^\circ \leq p\sp{*}_{\max}$,
the algorithm {\sc AscendMaximal}
outputs $p\sp{*}_{\max}$ and
the number of updates of the price vector 
is exactly equal to $\|  p\sp{*}_{\max} - p^\circ \|_\infty$.

\noindent
{\rm (2)} 
For an initial vector $p^\circ$ with $p^\circ \geq p\sp{*}_{\max}$,
the algorithm {\sc DescendMaximal} 
outputs $p\sp{*}_{\max}$ and
the number of updates of the price vector 
is exactly equal to $\|  p\sp{*}_{\max} - p^\circ \|_\infty$.

\noindent
{\rm (3)} 
For any initial vector $p^\circ$ with $p^\circ \geq p\sp{*}_{\min}$, 
the algorithm {\sc DescendMinimal} 
outputs $p\sp{*}_{\min}$ and
the number of updates of the price vector 
is exactly equal to $\| p\sp{*}_{\min} - p^\circ \|_\infty$.
\end{theorem}

A two-phase auction algorithm, consisting of
an ascending auction phase followed by a descending phase,
can be obtained by applying the algorithm
{\sc TwoPhaseMinMax} in Section \ref{SCalgLmin} to the Lyapunov function $L$.
Another two-phase auction algorithm can be obtained from {\sc TwoPhaseMinMin}.
Then Theorems \ref{THlmintwophaseMinMax} and \ref{THlmintwophaseMinMin}
imply the following
\RED{
(Murota and Shioura 2017)\citeH{MS16Lmin2ph}.
}%

\begin{theorem} \label{THauctiontwophase}
\quad   

\noindent
{\rm (1)} 
For any initial vector $p^\circ$,
the two-phase algorithm {\sc TwoPhaseMinMax} 
outputs some equilibrium price $p\sp{*}$.
The number of updates of the vector $p$ 
is bounded by $\mu(p\sp{\circ})$ in the ascending phase 
and by $\mu(p\sp{\circ})$ in the descending phase; 
in total, bounded by $2\mu(p\sp{\circ})$.

\noindent
{\rm (2)} 
For any initial vector $p^\circ$,
the two-phase algorithm {\sc TwoPhaseMinMin} 
outputs the minimal equilibrium price $p\sp{*}_{\min}$.
The number of updates of the vector $p$ 
is bounded by $\mu(p\sp{\circ})$ in the ascending phase
and is exactly equal to
$\| p\sp{\circ} - p\sp{*}_{\min} \|_\infty\sp{+}$
in the descending phase;
in total, bounded by 
$\mu(p\sp{\circ}) + \| p^{\circ} - p\sp{*}_{\min} \|_\infty\sp{+}$.
\end{theorem}

Two-phase algorithms with more flexibility
are given in 
Murota et al.~(2013a)\citeH{MSY13},
and
\RED{
Murota and Shioura (2017)\citeH{MS16Lmin2ph}.
}%

\begin{remark} \rm  \label{RMauctionTwoPhase}
The algorithm {\sc TwoPhaseMinMax}, when applied to
valuation functions on $\{0,1\}\sp{N}$
(single-unit valuations), 
coincides with a special case of  
``Global Dynamic Double-Track (GDDT) procedure'' proposed in 
Sun and Yang (2009)\citeH{SY09}.
The ``global Walrasian t{\^a}tonnement algorithm'' proposed by
Ausubel (2006)\citeH{Aus06}
repeats ascending and descending phases
until some equilibrium is found.
Theorem~\ref{THlmintwophaseMinMax} 
shows that the global Walrasian t\^atonnement algorithm terminates
after only one ascending phase and only one descending phase.
Put differently, the behavior of the global Walrasian t\^atonnement 
algorithm coincides with that of {\sc TwoPhaseMinMax}.
\finbox
\end{remark}

\begin{remark} \rm  \label{RMauctionTwoPhaseOther}
Besides {\sc TwoPhaseMinMax},
we can obtain many variants of two-phase algorithms
by choosing appropriate combinations 
from among the algorithms {\sc Greedy-\{Up, Down\}-\{Minimal, Maximal\}}
listed in Table \ref{TBgreedyLalg}.
In Section \ref{SCunitdemandauction}, for example,
we consider the combination of 
{\sc GreedyUpMinimal} and {\sc GreedyDownMinimal}.
\finbox
\end{remark}

\subsection{Unit-demand auctions}
\label{SCunitdemandauction}

Fundamental multi-item unit-demand auction algorithms
such as the
Vickrey--English,
Vickrey--Dutch,
Vickrey--English--Dutch auctions
can be reformulated 
in the  framework of the Lyapunov function approach.
In so doing we can derive bounds for the number of iterations
in these auction algorithms from the corresponding results about 
L$\sp{\natural}$-convex function minimization 
presented in Section \ref{SCalgLmin}.

The unit-demand auction model
is a special case of the single-unit auction model, where each bidder is
a {\em unit-demand} bidder,  being 
interested in getting at most one item.
 We continue to use notations
 $N=\{1,2,\ldots, n\}$ for the set of items 
and $M=\{1,2,\ldots, m\}$ for the set of bidders. 
 For each item $i$ and each bidder $j$, 
we denote by $v_{ji}$ the valuation of item $i$ by bidder $j$,
which is assumed to be a nonnegative integer, i.e., $v_{ji} \in \ZZ_{+}$.
The valuation function $f_{j}: 2\sp{N} \to \ZZ_{+}$ 
of bidder $j$ is given by
\begin{equation}
 \label{unitdemandfndef}
 f_{j}(X) = \left\{
\begin{array}{ll}
 \max\{v_{ji} \mid i \in X \} & (\mbox{if }X \neq \emptyset),\\
0 & (\mbox{if }X = \emptyset).
\end{array}
\right.
\end{equation}
 A valuation function of this form, often called a 
{\em unit-demand valuation}\/%
\footnote{
See, e.g.,  
Section~9.2.2 of Cramton et al.~(2006)\citeH[Section~9.2.2]{CSS06}
and  Definition 11.17 of Blumrosen and Nisan (2007)\citeH[Definition 11.17]{BN07}.
}, 
is a gross substitutes valuation, as pointed out by 
Gul and Stacchetti (1999)\citeH{GS99}.
In other words, a unit-demand valuation is M$\sp{\natural}$-concave; see (\ref{unitdemutil01}).
We are interested in finding the minimal Walrasian equilibrium price vector
$p\sp{*}_{\min} \in \ZZ_{+}\sp{N}$ by iterative auctions.

Fundamental iterative auction algorithms such as the
Vickrey--English auction of 
Demange et al.~(1986)\citeH{DGS86} 
(the variant by 
Mo et al.~(1988)\citeH{MTL88}
and Sankaran (1994)\citeH{San94}, 
to be more specific),
the Vickrey--Dutch auction of 
Mishra and Parkes (2009)\citeH{MP09vd},
and the Vickrey--English--Dutch auction of 
Andersson and Erlanson (2013)\citeH{AE13}
can be recast into the Lyapunov function-based framework.
The following theorem is due to
Murota et al.~(2016)\citeH[Theorem 5.5]{MSY16auction};
the specific forms of the auction algorithms are described in Remark~\ref{RMunitdemaucnalg}.

\begin{theorem}   \label{THunitdemconnectLyap}
Let $L: \ZZ_{+}\sp{N} \to \ZZ$ be the Lyapunov function
associated with the unit-demand valuations {\rm (\ref{unitdemandfndef})}.  
\\
{\rm (1)}
For any initial price vector $p\sp{\circ}$ with $p\sp{\circ} \leq p\sp{*}_{\min}$, 
the sequence of price vectors $p$ generated by 
the algorithm {\sc Vickrey\_English} 
is the same as that of 
{\sc GreedyUpMinimal} applied to $L$.
\\
{\rm (2)}
For any initial price vector $p\sp{\circ}$ with $p\sp{\circ} \geq p\sp{*}_{\min}$,
the sequence of price vectors $p$ generated by
the algorithm {\sc Vickrey\_Dutch}
is the same as that of
{\sc GreedyDownMinimal} applied to $L$.
\\
{\rm (3)}
For any initial price vector $p\sp{\circ}$,
the sequence of price vectors $p$ generated by
the algorithm {\sc Vickrey\_English\_Dutch}  
is the same as that of 
{\sc TwoPhaseMinMin} applied to $L$.
\end{theorem}

Theorem~\ref{THunitdemconnectLyap} above is established on the basis of
the following technical observations  
(Lemma 5.7 of Murota et al.~2016),
  \citeH[Lemma 5.7]{MSY16auction}%
which relate the descending directions of the Lyapunov function 
with ``sets in excess demand'' 
(see Remark~\ref{RMunitdemaucnalg})
used in the Vickrey--English, Vickrey--Dutch, Vickrey--English--Dutch 
auction algorithms.

\begin{proposition} \label{PRexcessLyapunov}
 Let $p \in \ZZ_{+}\sp{N}$ be a price vector.

\noindent
{\rm (1)}
 A set $X \subseteq N$  is the maximal set in excess demand at price $p$
if and only if $X$ is the minimal minimizer of $L(p + \chi_X) - L(p)$.

\noindent
{\rm (2)}
 A set $Z \subseteq \suppp(p)$  
is the maximal set in positive excess demand at price $p$
if and only if $X=\suppp(p) \setminus Z$ 
is the maximal minimizer of $L(p - \chi_X) - L(p)$.
\end{proposition}

Theorem~\ref{THunitdemconnectLyap} enables us to 
resort to the general results for 
L$\sp{\natural}$-convex function minimization 
in Section \ref{SCalgLmin}
to establish the following (exact or upper) bounds on the number of iterations
in the unit-demand auction algorithms,
where (1) and (2) are given in 
Corollary 2 of Andersson and Erlanson (2013)\citeH[Corollary 2]{AE13},
and (3) is in
\RED{
Murota and Shioura (2017)\citeH{MS16Lmin2ph}.
}%

\begin{theorem}\label{THunitdemanditerbound}
\quad   

\noindent
{\rm (1)}
For any initial price vector $p\sp{\circ}$ with $p\sp{\circ} \leq p\sp{*}_{\min}$, 
the number of updates of the price vector in 
the algorithm {\sc Vickrey\_English} is exactly equal to 
$\| p\sp{\circ} - p\sp{*}_{\min} \|_\infty$.
\\
{\rm (2)}
For any initial price vector $p\sp{\circ}$ with $p\sp{\circ} \geq p\sp{*}_{\min}$,
the number of updates of the price vector in 
the algorithm {\sc Vickrey\_Dutch} is exactly equal to 
$\| p\sp{\circ} - p\sp{*}_{\min} \|_\infty$.
\\
{\rm (3)}
For any initial price vector $p\sp{\circ}$,
the number of updates of the price vector in 
the algorithm {\sc Vickrey\_English\_Dutch} 
is bounded by $\mu(p\sp{\circ})$ in the ascending phase
and is exactly equal to
$\| p\sp{\circ} - p\sp{*}_{\min} \|_\infty\sp{+}$
in the descending phase;
in total, bounded by 
$\mu(p\sp{\circ}) + \| p^{\circ} - p\sp{*}_{\min} \|_\infty\sp{+}$.
\end{theorem}
\begin{proof}
We prove the claims to illustrate the use of 
the general results in Section \ref{SCalgLmin}.
(1) follows from Theorem~\ref{THunitdemconnectLyap} (1)
and Theorem~\ref{THlminalgiter} (2).
(2) follows from Theorem~\ref{THunitdemconnectLyap} (2)
and Table \ref{TBgreedyLalg} (b). 
(3) follows from Theorem~\ref{THunitdemconnectLyap} (3)
and Theorem~\ref{THlmintwophaseMinMin}.
\end{proof}

\begin{remark} \rm  \label{RMunitdemaucnalg}
The Vickrey--English, Vickrey--Dutch,
Vickrey--English--Dutch auction algorithms are described here,
following
Andersson and Erlanson (2013)\citeH{AE13}
and Andersson et al.~(2013)\citeH{AAT13}.
Denote by $0$ an artificial item (null-item)
which has no value (i.e., $v_{j0}=0$ for all $j \in M$) 
and is available in an infinite number of units.
 For each bidder $j \in M$ and a price vector $p \in \ZZ_{+}\sp{N}$, 
define $D_{j}(p) \subseteq N \cup \{ 0 \}$ by
\begin{align*}
  D_{j}(p)
 = \arg\max\{v_{ji} - p_{i}  \mid i \in N \cup \{ 0 \}  \}
 = \{i \in N \cup \{ 0 \} \mid  v_{ji} - p_{i} \geq v_{ji'} - p_{i'} 
   \ (\forall i' \in N \cup \{ 0 \} ) \,  \},
\end{align*}
where $p_{0} = 0$. 
 For an item  set $Y \subseteq N$ and a price vector $p \in \ZZ_{+}\sp{N}$,  define
\begin{align*}
  O(Y, p) & = \{j \in M \mid D_{j}(p) \subseteq Y\}, 
\\
 U(Y, p) & = \{j \in M \mid D_{j}(p) \cap Y \neq \emptyset\}.
\end{align*}
 The set $O(Y,p)$ 
consists of bidders who only demand items in $Y$
at price $p$, while $U(Y, p)$
is the set of bidders who demand some item in $Y$
at price $p$.
 Obviously, $O(Y, p) \subseteq U(Y,p)$.
 A set $X \subseteq N$ is said to be {\em in excess demand}
at price $p$ if it satisfies
\[
 |U(Y, p) \cap O(X, p)| > |Y| \qquad (\emptyset \neq \forall Y \subseteq X).
\]
For each price vector $p$ there uniquely exists a maximal set in excess demand%
\footnote{
See 
Proposition 1 of Mo et al.~(1988)\citeH[Proposition 1]{MTL88}
and also 
Proposition 1 of Andersson and Erlanson (2013)\citeH[Proposition 1]{AE13}, 
and Theorem 1 of Andersson et al.~(2013)\citeH[Theorem 1]{AAT13}.
}. 
 The Vickrey-English auction algorithm  due to 
Mo et al.~(1988)\citeH{MTL88}
 and Sankaran (1994)\citeH{San94},
a variant of the one in 
Demange et al.~(1986)\citeH{DGS86},
is as follows:

\begin{tabbing}     
\= {\bf Algorithm} {\sc Vickrey\_English} \\
\> \quad  Step 0: 
 Set $p:=p\sp{\circ}$, where
$p\sp{\circ} \in \ZZ_{+}\sp{N}$ is an arbitrary vector satisfying 
$p\sp{\circ} \leq p\sp{*}_{\min}$ \ (e.g., $p\sp{\circ} = \bm{0}$).  \\
\> \quad  Step 1:
    Find the maximal set $X \subseteq N$ in excess demand at price $p$. \\
\> \quad  Step 2: 
    If $X = \emptyset$,  then output $p$ and stop.   \\ 
\> \quad  Step 3: 
    Set  $p:= p + \chi_{X}$ \ and \ go to Step~1. 
\end{tabbing}

The Vickrey-Dutch auction algorithm
refers to the variants of the sets $D_{j}(p)$ and $O(Y, p)$
defined as 
\begin{align*}
 D\sp{+}_{j}(p) & = D_{j}(p) \cap \suppp(p),
\\
 O\sp{+}(Y, p) & = \{j \in M \mid D\sp{+}_{j}(p) \subseteq Y\}.
\end{align*}
\noindent
 A set $X \subseteq N$ is said to be \textit{in positive excess demand}
at price $p$ if 
$X \subseteq \suppp(p)$ and 
\[
 |U(Y, p) \cap O\sp{+}(X, p)| > |Y| \qquad (\emptyset \neq \forall Y \subseteq X).
\]
For each price vector $p$ there uniquely exists 
a maximal set in positive excess demand%
\footnote{
See Theorem 2 of Andersson and Erlanson (2013)\citeH[Theorem 2]{AE13}.
}. 
The Vickrey--Dutch auction by 
Mishra and Parkes (2009)\citeH{MP09vd}
is as follows:

\begin{tabbing}     
\= {\bf Algorithm} {\sc Vickrey\_Dutch} \\
\> \quad  Step 0: 
     Set $p:=p\sp{\circ}$, where
    $p\sp{\circ} \in \ZZ_{+}\sp{N}$ is an arbitrary vector satisfying 
     $p\sp{\circ} \geq p\sp{*}_{\min}$. \\
\> \quad  Step 1:
    Find the maximal set $Z \subseteq N$ in positive excess demand at price $p$,
    and  \\
\> \quad  \phantom{Step 1:}
     set $X := \suppp(p) \setminus Z$.  \\
\> \quad  Step 2: 
    If $X = \emptyset$,  then output $p$ and stop.   \\ 
\> \quad  Step 3: 
    Set  $p:= p - \chi_{X}$ \ and \ go to Step~1. 
\end{tabbing}

The Vickrey--English--Dutch auction by 
Andersson and Erlanson (2013)\citeH{AE13}
is a combination of the Vickrey--English and Vickrey--Dutch auctions,
as follows:

\begin{tabbing}      
\= {\bf Algorithm} {\sc Vickrey\_English\_Dutch} \\
\> \quad  Step 0: 
   Set $p:=p\sp{\circ}$, where  $p\sp{\circ} \in \ZZ^N_{+}$ is an arbitrary vector. 
    Go to Ascending Phase.\\
\> \quad  Ascending Phase: \\
\> \qquad  Step A1:
  Find the maximal set $X \subseteq N$ in excess demand at price $p$. \\
\> \qquad  Step A2: 
 If $X = \emptyset$, then go to Descending Phase.  \\
\> \qquad  Step A3: 
    Set  $p:= p + \chi_{X}$ \ and \ go to Step~A1. \\
\> \quad Descending Phase: \\
\> \qquad  Step D1:
    Find the maximal set $Z \subseteq N$ in positive excess demand at price $p$,
    and  \\
\> \qquad  \phantom{Step D1:}
    set $X := \suppp(p) \setminus Z$.  \\
\> \qquad  Step D2: 
 If $X = \emptyset$, then output $p$ and stop.  \\
\> \qquad  Step D3: 
    Set  $p:= p - \chi_{X}$ \ and \ go to Step~D1. 
\end{tabbing}
\vspace{-\baselineskip}
\finbox
\end{remark}

\subsection{Concluding remarks of section \ref{SCauction}}

Use of discrete convex analysis
in the Lyapunov function approach
is also conceived by 
Drexl and Kleiner (2015)\citeH{DK15drexl}.
Besides the basic form of ascending auction, 
the paper proposes and analyzes the ``singleton-based t{\^ a}tonnement''
which reflects a certain practice in auction design. 
It also discusses the double-track adjustment process of 
Sun and Yang (2009)\citeH{SY09}
as an application of the framework of Section \ref{SClyapapproach};
the underlying key fact here is that gross substitutes and complements
are represented 
by twisted M$\sp{\natural}$-concave functions
(Section \ref{SCtwistMnat01}).
Lehmann et al.~(2006)\citeH{LLM06}
shows a connection 
between discrete convex analysis and combinatorial auctions.
Sun and Yang (2014)\citeH{SY14}
considers super-additive utility functions.



\section{Intersection and Separation Theorems}
\label{SCintersepar}

\subsection{Separation theorem}
\label{SCseparthm}

The {\em duality}
principle in convex analysis can be
expressed in a number of different forms.
One of the most appealing statements
is in the form of the separation theorem,
which asserts the existence of a separating affine function
$y = \alpha\sp{*} + \langle p\sp{*}, x \rangle$  
for a pair of convex and concave functions.
In application to economic problems,
 the separating vector $p\sp{*}$ gives the equilibrium price.

In the continuous case we have the following.

\begin{theorem}  \label{THOsep}
\index{separation theorem}
Let $f: \RR\sp{n} \to \RR \cup \{ +\infty \}$ 
and 
$h: \RR\sp{n} \to \RR \cup \{ -\infty \}$ 
be convex and concave functions, respectively
{\rm (}satisfying certain regularity conditions{\rm )}.
If
\[
 f(x) \geq h(x)  
 \qquad  (\forall x \in \RR\sp{n}) ,
\]
there exist 
$\alpha\sp{*} \in \RR$ and $p\sp{*} \in \RR\sp{n}$
such that
\[
f(x) \geq \alpha\sp{*} + \langle p\sp{*}, x \rangle  \geq h(x)  
 \qquad  (\forall x \in \RR\sp{n}).
\]
\end{theorem}

In the discrete case we are concerned with functions defined on integer points:
$f: \ZZ\sp{n} \to \RR \cup \{ +\infty \}$ and 
$h: \ZZ\sp{n} \to \RR \cup \{ -\infty \}$.
A {\em discrete separation theorem}
 means a statement like:
\begin{quote}
For any 
$f: \ZZ\sp{n} \to \RR \cup \{ +\infty \}$ and 
$h: \ZZ\sp{n} \to \RR \cup \{ -\infty \}$ belonging to
certain classes of functions,
if $f(x)  \geq h(x)$  for all $x \in \ZZ\sp{n}$,
then there exist
$\alpha\sp{*} \in \RR$ and $p\sp{*} \in \RR\sp{n}$
such that
\[ 
f(x) \geq \alpha\sp{*} + \langle p\sp{*}, x \rangle  \geq h(x)  
 \qquad  (\forall x \in \ZZ\sp{n}).
\] 
Moreover, if $f$ and $h$ are integer-valued, there exist integer-valued
$\alpha\sp{*} \in \ZZ$ and $p\sp{*} \in \ZZ\sp{n}$.
\end{quote}
In application to economic problems,
 the separating vector $p\sp{*}$ 
in a discrete separation theorem
often gives the equilibrium price in markets with indivisible goods.

Discrete separation theorems capture
deep combinatorial properties in spite of the apparent similarity to 
the separation theorem in the continuous case.
In this connection we note the following facts
that indicate the difficulty 
inherent in discrete separation theorems%
\footnote{
See Examples 1.5 and 1.6 of Murota (2003)\citeH[Examples 1.5 and 1.6]{Mdcasiam}
for concrete examples.
}.  
Let 
$f: \ZZ\sp{n} \to \RR \cup \{ +\infty \}$ 
be a convex-extensible function,
with the convex closure $\overline{f}$.
Also let $h: \ZZ\sp{n} \to \RR \cup \{ -\infty \}$ 
be a concave-extensible function,
with the concave closure $\overline{h}$.
In the following statements,  
\   $\Longrightarrow$ \hspace{-6mm}$\not$\hspace{5mm} \  
stands for ``does not imply.''

\begin{enumerate}
\item
$f(x) \geq h(x)$ $(\forall x \in \ZZ\sp{n})$
\   $\Longrightarrow$ \hspace{-6mm}$\not$\hspace{5mm} \  
$\overline{f}(x) \geq \overline{h}(x)$ $(\forall x \in \RR\sp{n})$.

\item
$f(x) \geq h(x)$ $(\forall x \in \ZZ\sp{n})$
\   $\Longrightarrow$ \hspace{-6mm}$\not$\hspace{5mm} \  
existence of $\alpha\sp{*} \in \RR$ and $p\sp{*} \in \RR\sp{n}$.

\item
existence of $\alpha\sp{*} \in \RR$ and $p\sp{*} \in \RR\sp{n}$
\   $\Longrightarrow$ \hspace{-6mm}$\not$\hspace{5mm} \  
existence of $\alpha\sp{*} \in \ZZ$ and $p\sp{*} \in \ZZ\sp{n}$. 
\end{enumerate}

It is known that discrete separation theorems hold 
for M$\sp{\natural}$-convex/M$\sp{\natural}$-concave functions
and for L$\sp{\natural}$-convex/L$\sp{\natural}$-concave functions.
The M$\sp{\natural}$-separation theorem (Theorem \ref{THmfnsep}) is shown by
Murota (1996c, 1998, 1999)\citeH{Mstein}\citeH{Mdca}\citeH{Msbmfl}
in terms of M-convex/concave functions,
and the L$\sp{\natural}$-separation theorem (Theorem \ref{THlfnsep}) by
Murota (1998)\citeH{Mdca}
in terms of L-convex/concave functions.
The assumptions of the theorems refer to
the convex and concave conjugate functions of $f$ and $h$
 defined, respectively, by%
\footnote{
We have $f\sp{\bullet}(p) = - f\sp{\triangle}(-p)$ and
$ h\sp{\circ}(p) = - h\sp{\triangledown}(p)$
in the notation of (\ref{conjcave2vexZ}) and (\ref{conjvex2caveZ}).
}
\begin{eqnarray}
 f\sp{\bullet}(p) 
 &=& \sup\{  \langle p, x \rangle - f(x) 
        \mid x \in \ZZ\sp{n} \}
\qquad ( p \in \RR\sp{n}), 
\label{disconjvex6def} \\
 h\sp{\circ}(p) 
 &=& \inf\{  \langle p, x \rangle - h(x) 
        \mid x \in \ZZ\sp{n} \}
\qquad ( p \in \RR\sp{n}).
 \label{disconjcav6def}
\end{eqnarray}

\begin{theorem}[M$\sp{\natural}$-separation theorem]     \label{THmfnsep}
Let $f: \ZZ\sp{n} \to \RR \cup \{ +\infty \}$
be an M$\sp{\natural}$-convex function
and 
$h: \ZZ\sp{n} \to \RR \cup \{ -\infty \}$
be an M$\sp{\natural}$-concave function such that 
$\domZ f \cap \domZ h \not= \emptyset$
or
$\domR f\sp{\bullet} \cap \domR h\sp{\circ} \not= \emptyset$.
If
$f(x) \geq h(x)$ $(\forall x \in \ZZ\sp{n})$, 
there exist 
$\alpha\sp{*} \in \RR$ and $p\sp{*} \in \RR\sp{n}$ such that
\[ 
 f(x) \geq \alpha\sp{*} + \langle p\sp{*}, x \rangle  \geq h(x)  
 \qquad  (\forall x \in \ZZ\sp{n}).
\] 
Moreover, if $f$ and $h$ are integer-valued,
there exist integer-valued 
$\alpha\sp{*} \in \ZZ$ and $p\sp{*} \in \ZZ\sp{n}$.
\end{theorem}

\begin{theorem}[L$\sp{\natural}$-separation theorem] \label{THlfnsep}
Let $g: \ZZ\sp{n} \to \RR \cup \{ +\infty \}$
be an L$\sp{\natural}$-convex function
and 
$k: \ZZ\sp{n} \to \RR \cup \{ -\infty \}$
be an L$\sp{\natural}$-concave function
such that 
$\domZ g \cap \domZ k \not= \emptyset$
or
$\domR g\sp{\bullet} \cap \domR k\sp{\circ} \not= \emptyset$.
If
$g(p) \geq k(p)$ $(\forall p \in \ZZ\sp{n})$, 
there exist 
$\beta\sp{*} \in \RR$ and $x\sp{*} \in \RR\sp{n}$ such that
\[ 
 g(p) \geq \beta\sp{*} + \langle p, x\sp{*} \rangle  
           \geq k(p)  
 \qquad  (\forall p \in \ZZ\sp{n}).
\] 
Moreover, if $g$ and $k$ are integer-valued,
there exist integer-valued 
$\beta\sp{*} \in \ZZ$ and $x\sp{*} \in \ZZ\sp{n}$.
\end{theorem}

As an immediate corollary of the M$\sp{\natural}$-separation theorem
we can obtain an optimality criterion
for the problem of maximizing the sum of two M$\sp{\natural}$-concave functions,
which we call the 
{\em M$\sp{\natural}$-concave intersection problem}.
Note that the sum of M$\sp{\natural}$-concave functions 
is no longer M$\sp{\natural}$-concave
and Theorem \ref{THmnatsetfnlocmaxZ} does not apply.
Recall the notation $f[-p](x) = f(x) - \langle p, x \rangle$.

\begin{theorem}[M$\sp{\natural}$-concave intersection theorem] \label{THmfcaveninteropt}
For M$\sp{\natural}$-concave functions
$f_{1}, f_{2}: \ZZ\sp{n} \to \RR \cup \{ -\infty \}$
and a point $x\sp{*} \in \domZ f_{1} \cap \domZ f_{2}$
we have
\[ 
f_{1}(x\sp{*})+f_{2}(x\sp{*})  \geq f_{1}(x)+f_{2}(x)
 \qquad (\forall x \in \ZZ\sp{n})
\] 
if and only if there exists $p\sp{*} \in \RR\sp{n}$ such that
\begin{align*}
f_{1}[-p\sp{*}](x\sp{*}) & \geq f_{1}[-p\sp{*}](x)
 \qquad (\forall x \in \ZZ\sp{n}),
\\
 f_{2}[+p\sp{*}](x\sp{*}) & \geq f_{2}[+p\sp{*}](x)
 \qquad (\forall x \in \ZZ\sp{n}) .
\end{align*}
These conditions 
are equivalent, respectively, to
\begin{align*}
 f_{1}[-p\sp{*}](x\sp{*}) & \geq f_{1}[-p\sp{*}](x\sp{*}+\chi_{i}-\chi_{j})
 \qquad (\forall \, i, j  \in  \{ 0,1, \ldots,n \}) ,
 \\
 f_{2}[+p\sp{*}](x\sp{*}) & \geq f_{2}[+p\sp{*}](x\sp{*}+\chi_{i}-\chi_{j})
 \qquad (\forall \, i, j  \in  \{ 0,1, \ldots,n \}) ,
\end{align*}
and for such $p\sp{*}$ we have
\[ 
  \argmaxZ (f_{1}+f_{2}) = \argmaxZ f_{1}[-p\sp{*}] \cap \argmaxZ f_{2}[+p\sp{*}].
\] 
Moreover, if
$f_{1}$ and $f_{2}$ are integer-valued,
we can choose integer-valued $p\sp{*} \in \ZZ\sp{n}$.
\end{theorem}

An extension of the M$\sp{\natural}$-concave intersection theorem
is given in Theorem \ref{THgameaggregate},
which constitutes the technical pivot 
in the Fujishige--Tamura model that unifies
the stable marriage and the assignment game
(see Remark \ref{RMgameFTdualcases}).

\begin{remark} \rm  \label{RMminterproof}
Three different proofs are available 
for the M$\sp{\natural}$-concave intersection theorem.
The original proof (Murota 1996c)\citeH{Mstein}
is based on the reduction of 
the M$\sp{\natural}$-concave intersection problem 
to the M-convex submodular flow problem;
see Remark \ref{RMintersubmfl} in Section \ref{SCsubmflproblem}.
Then Theorem \ref{THmfcaveninteropt} is derived from 
the negative-cycle optimality criterion (Theorem~\ref{THsbmfcyccritZ})
for the M-convex submodular flow problem.
The second proof is based on the reduction to 
the discrete separation theorem,
which is proved by the polyhedral-combinatorial method
using the (standard) separation theorem in convex analysis;
see the proof of Theorem 8.15 of Murota (2003)\citeH{Mdcasiam}.
The third proof (Murota 2004a)\citeH{MproofMinter04}
is a direct constructive proof based on the successive shortest path algorithm.
\finbox
\end{remark}

\subsection{Fenchel duality}
\label{SCfencheldual}

Another expression of the duality principle is in the form of
the Fenchel duality.
This is a min-max relation 
between a pair of convex and concave functions
and their conjugate functions.
Such a min-max theorem is computationally useful 
in that it affords a certificate of optimality.

We start with the continuous case.
For a function 
$f: \RR\sp{n} \to \RR \cup \{ +\infty \}$
with $\dom f \not= \emptyset$,
the convex conjugate
$f\sp{\bullet}: \RR\sp{n} \to \RR \cup \{ +\infty \}$
is defined by%
\footnote{
We have $f\sp{\bullet}(p) = - f\sp{\triangle}(-p)$ and
$ h\sp{\circ}(p) = - h\sp{\triangledown}(p)$
in the notation of (\ref{conjcave2vexR}) and (\ref{conjvex2caveR}).
}
\begin{equation} \label{conjvexOdef}
 f\sp{\bullet}(p) 
 = \sup\{  \langle p, x \rangle - f(x) 
        \mid x \in \RR\sp{n} \}
\qquad ( p \in \RR\sp{n}).
\end{equation}
For $h: \RR\sp{n} \to \RR \cup \{ -\infty \}$,
the {concave conjugate}
$h\sp{\circ}: \RR\sp{n} \to \RR \cup \{ -\infty \}$
is defined by
\begin{equation} \label{conjcavOdef}
 h\sp{\circ}(p) 
 = \inf\{  \langle p, x \rangle - h(x) 
        \mid x \in \RR\sp{n} \}
\qquad ( p \in \RR\sp{n}).
\end{equation}

\begin{theorem}  \label{THOfenc}
Let $f: \RR\sp{n} \to \RR \cup \{ +\infty \}$ 
and 
$h: \RR\sp{n} \to \RR \cup \{ -\infty \}$ 
be convex and concave functions, respectively
{\rm (}satisfying certain regularity conditions{\rm )}. Then
\[
  \inf\{ f(x) - h(x) \mid x \in \RR\sp{n}  \}
 = 
  \sup\{ h\sp{\circ}(p) - f\sp{\bullet}(p)   
  \mid   p \in \RR\sp{n} \}.  
\]
\end{theorem}

We now turn to the discrete case.
For any functions
$f: \ZZ\sp{n} \to \ZZ \cup \{ +\infty \}$
and  
$h: \ZZ\sp{n} \to \ZZ \cup \{ -\infty \}$,
we define the discrete versions of 
(\ref{conjvexOdef}) and (\ref{conjcavOdef}) as
\begin{eqnarray}
 f\sp{\bullet}(p) 
 &=& \sup\{  \langle p, x \rangle - f(x) 
        \mid x \in \ZZ\sp{n} \}
\qquad ( p \in \ZZ\sp{n}), 
\label{conjvexZZ} \\
 h\sp{\circ}(p) 
 &=& \inf\{  \langle p, x \rangle - h(x) 
        \mid x \in \ZZ\sp{n} \}
\qquad ( p \in \ZZ\sp{n}).
 \label{conjcavZZ}
\end{eqnarray}
Then we have a chain of inequalities:
\begin{equation}  \label{fencweak}
\begin{array}{ccc}
 \inf\{ f(x) - h(x) \mid x \in \ZZ\sp{n}  \}
& &  \sup\{ h\sp{\circ}(p) - f\sp{\bullet}(p) \mid  p \in \ZZ\sp{n} \}  
\\
 \raisebox{0.25cm}{\rotatebox{270}{$\geq$}}  &   & 
 \raisebox{0cm}{\rotatebox{90}{$\geq$}}  
\\
 \inf\{ \overline{f}(x) - \overline{h}(x) \mid x \in \RR\sp{n}  \}
 &\geq &   \sup\{ \overline{h}\sp{\circ}(p)
         - \overline{f}\sp{\bullet}(p)
        \mid  p \in \RR\sp{n} \} ,
\end{array}
\end{equation}
where $\overline{f}$ and $\overline{h}$ are the convex and concave
closures of $f$ and $h$, respectively,
and $\overline{f}\sp{\bullet}$ and $\overline{h}\sp{\circ}$
are defined by (\ref{conjvexOdef}) for $\overline{f}$
 and (\ref{conjcavOdef}) for $\overline{h}$.
We observe that
\begin{enumerate}
\item
The second inequality ($\geq$) in the middle of (\ref{fencweak})
 is in fact an equality ($=$)
(under mild regularity conditions)
by the Fenchel duality theorem  in convex analysis  (Theorem~\ref{THOfenc});

\item
The first inequality 
(\,\raisebox{0.25cm}{\rotatebox{270}{$\geq$}}\,)
in the left of (\ref{fencweak}) 
can be strict (i.e., $\not=$) even when
$f$ is  convex-extensible and $h$ is concave-extensible,
and similarly for the third inequality 
(\,\raisebox{0cm}{\rotatebox{90}{$\geq$}}\,)
in the right.
See Examples \ref{EXnorealfenc} and \ref{EXnointfenc} below%
\footnote{
These examples are taken from 
Murota (2009)\citeH{Mbonn09}.
}. 
\end{enumerate}

\begin{example}  \rm \label{EXnorealfenc}
For $f, h: \ZZ\sp{2} \to \ZZ$ defined as
\[
f(x_{1},x_{2}) = |x_{1}+x_{2}-1|,
\qquad
h(x_{1},x_{2}) = 1- |x_{1}-x_{2}|
\]
we have
$ \inf\{ f - h \} = 0$, 
$\inf\{ \overline{f} - \overline{h} \}=-1$.
The conjugate functions
(\ref{conjvexZZ}) and (\ref{conjcavZZ}) 
are given by
\[
f\sp{\bullet}(p_{1}, p_{2})  =
   \left\{  \begin{array}{ll}
    p_{1}            &   ((p_{1}, p_{2}) \in S)      \\
   + \infty      &   (\mbox{otherwise}),  \\
                      \end{array}  \right.
\quad
 h\sp{\circ}(p_{1}, p_{2})  =
   \left\{  \begin{array}{ll}
    -1            &   ((p_{1}, p_{2}) \in T)      \\
   - \infty      &   (\mbox{otherwise})  \\
                      \end{array}  \right.
\]
with
$S=\{ (-1,-1), (0,0), (1,1) \}$ and
$T=\{ (-1,1), (0,0), (1,-1) \}$.
Hence
$\sup\{ h\sp{\circ} - f\sp{\bullet} \}
 = h\sp{\circ}(0,0) - f\sp{\bullet}(0,0) = -1 - 0 = -1$.
Then (\ref{fencweak}) reads as
\[
  \begin{array}{ccccccc}
 \inf\{ f - h \}
 & > & \inf\{ \overline{f} - \overline{h} \}
 & = &   \sup\{ \overline{h}\sp{\circ}
         - \overline{f}\sp{\bullet}  \}
 & = &  \sup\{ h\sp{\circ} - f\sp{\bullet} \}  .
\\
 (0) &  & (-1)  & & (-1)  && (-1) 
  \end{array} 
\]
\vspace{-1.5\baselineskip}
\\
\finbox
\end{example}

\begin{example}  \rm \label{EXnointfenc}
For $f, h: \ZZ\sp{2} \to \ZZ$ defined as
\[
f(x_{1},x_{2}) = \max(0,x_{1}+x_{2}),
\qquad
h(x_{1},x_{2}) = \min(x_{1},x_{2})
\]
we have
$ \inf\{ f - h \} 
= \inf\{ \overline{f} - \overline{h} \}=0$.
The conjugate functions (\ref{conjvexZZ}) and (\ref{conjcavZZ}) 
are given as
$f\sp{\bullet} = \delta_{S}$ and
$h\sp{\circ} = - \delta_{T}$
in terms of the (convex) indicator functions%
\footnote{
$\delta_{S}(p)=0$ for $p \in S$ and $=+\infty$ for $p \not\in S$.
} 
of $S=\{ (0,0), (1,1) \}$ and $T=\{ (1,0), (0,1) \}$.
Since $S \cap T = \emptyset$, the function
$h\sp{\circ} - f\sp{\bullet}$ is identically equal to $-\infty$,
whereas 
$\sup\{ \overline{h}\sp{\circ} - \overline{f}\sp{\bullet} \} = 0$
since 
$\overline{f}\sp{\bullet} = \delta_{\overline{S}}$,
$\overline{h}\sp{\circ} = -\delta_{\overline{T}}$
and
$\overline{S} \cap \overline{T} = \{ (1/2, 1/2) \}$.
Then (\ref{fencweak}) reads as
\[
  \begin{array}{ccccccc}
 \inf\{ f - h \}
 & = & \inf\{ \overline{f} - \overline{h} \}
 & = &   \sup\{ \overline{h}\sp{\circ}
         - \overline{f}\sp{\bullet}  \}
 & > &  \sup\{ h\sp{\circ} - f\sp{\bullet} \} . 
\\
 (0) &  &  (0) & & (0) && (-\infty)
  \end{array} 
\]
\vspace{-1.5\baselineskip}
\\
\finbox
\end{example}

The Fenchel-type duality  holds for
M$\sp{\natural}$-convex/M$\sp{\natural}$-concave functions and
L$\sp{\natural}$-convex/L$\sp{\natural}$-concave functions.
The Fenchel-type duality theorem originates in 
Murota (1996c)\citeH{Mstein}
(see also Murota 1998)\citeH{Mdca}
and formulated into the following form in
Murota (2003)\citeH{Mdcasiam}.
The essence of the theorem is the assertion that
the first and third inequalities in (\ref{fencweak})
are in fact equalities for
M$\sp{\natural}$-convex/M$\sp{\natural}$-concave functions and
L$\sp{\natural}$-convex/L$\sp{\natural}$-concave functions.

\begin{theorem}[Fenchel-type duality theorem]     \label{THmlfencdual}
\quad  

\noindent {\rm (1)}
Let $f: \ZZ\sp{n} \to \ZZ \cup \{ +\infty \}$
be an integer-valued M$\sp{\natural}$-convex function and 
$h: \ZZ\sp{n} \to \ZZ \cup \{ -\infty \}$
be an integer-valued M$\sp{\natural}$-concave function
such that 
$\domZ f \cap \domZ h \not= \emptyset$
or
$\domZ f\sp{\bullet} \cap \domZ h\sp{\circ} \not= \emptyset$,
where
$f\sp{\bullet}$ and $h\sp{\circ}$ 
are defined by {\rm (\ref{conjvexZZ})} and {\rm (\ref{conjcavZZ})}.
Then we have
\begin{equation} \label{mlfencminmaxM}
  \inf\{ f(x) - h(x) \mid x \in \ZZ\sp{n}  \}
 =   \sup\{ h\sp{\circ}(p) - f\sp{\bullet}(p) \mid p \in \ZZ\sp{n} \} .
\end{equation}
If this common value is finite, 
the infimum and the supremum are attained.

\noindent {\rm (2)}
Let $g: \ZZ\sp{n} \to \ZZ \cup \{ +\infty \}$
be an integer-valued L$\sp{\natural}$-convex function and 
$k: \ZZ\sp{n} \to \ZZ \cup \{ -\infty \}$
be an integer-valued L$\sp{\natural}$-concave function
 such that 
$\domZ g \cap \domZ k \not= \emptyset$
or
$\domZ g\sp{\bullet} \cap \domZ k\sp{\circ} \not= \emptyset$,
where
$g\sp{\bullet}$ and $k\sp{\circ}$ 
are defined by {\rm (\ref{conjvexZZ})} and {\rm (\ref{conjcavZZ})}.
Then we have
\begin{equation} \label{mlfencminmaxL}
  \inf\{ g(p) - k(p) \mid p \in \ZZ\sp{n}  \}
 =   \sup\{ k\sp{\circ}(x) - g\sp{\bullet}(x) \mid x \in \ZZ\sp{n} \} .
\end{equation}
If this common value is finite, 
the infimum and the supremum are attained.
\end{theorem}

The Fenchel-type duality theorem can be formulated 
for real-valued functions 
$f, g: \ZZ\sp{n} \to \RR \cup \{ +\infty \}$
and
$h, k: \ZZ\sp{n} \to \RR \cup \{ -\infty \}$
as well; see 
Theorem 8.21 of Murota (2003)\citeH[Theorem 8.21]{Mdcasiam}.

\begin{remark} \rm  \label{RMfencmixNG}
For the Fenchel-type duality,
the two	 functions must be consistent with respect to 
the types (M$\sp{\natural}$ or L$\sp{\natural}$).
In Example \ref{EXnorealfenc},
$f$ is M$\sp{\natural}$-convex and $h$ is L$\sp{\natural}$-concave.
This is also the case in Example \ref{EXnointfenc}.
\finbox
\end{remark}

\begin{remark} \rm  \label{RMseparfenc}
Whereas the L$\sp{\natural}$-separation and M$\sp{\natural}$-separation theorems 
are parallel or conjugate to each other
in their statements, the Fenchel-type duality theorem is self-conjugate,
in that the substitution of $f=g\sp{\bullet}$ and $h=k\sp{\circ}$
into (\ref{mlfencminmaxM}) results in (\ref{mlfencminmaxL})
by virtue of the biconjugacy $g=(g\sp{\bullet})\sp{\bullet}$ 
and $k=(k\sp{\circ})\sp{\circ}$
(Theorem \ref{THlmconjcavevexZ}).
With the knowledge of M-/L-conjugacy (Section~\ref{SCconjugacy}), 
these three duality theorems are almost equivalent to one another;
once one of them is established, the other two theorems can be derived
by relatively easy formal calculations.
\finbox
\end{remark}

\subsection{Concluding remarks of section \ref{SCintersepar}}
\label{SCinterseparcondrem}

The significance of the duality theorems of this section 
in combinatorial optimization is mentioned here.
Frank's discrete separation theorem 
(Frank 1982)\citeH{Fra82}
for submodular/supermodular set functions
is a special case of the L$\sp{\natural}$-separation theorem.
Frank's weight splitting theorem 
(Frank 1981)\citeH{Fra81}  
for the weighted
matroid intersection problem is a special case of 
the M$\sp{\natural}$-concave intersection problem.
Edmonds's intersection theorem 
(Edmonds 1970)\citeH{Edm70}
for (poly) matroids in the integral case is a special case of 
the Fenchel-type duality (Theorem \ref{THmlfencdual} (1)).
Fujishige's Fenchel-type duality theorem 
(Fujishige 1984\citeH{Fuj84};  
also Section 6.1 of Fujishige 2005\citeH[Section 6.1]{Fuj05})
for submodular set functions
is a special case of Theorem \ref{THmlfencdual} (2).
Section 8.2.3 of Murota (2003)
  \citeH[Section 8.2.3]{Mdcasiam}%
gives more details.



\section{Stable Marriage and Assignment Game}
\label{SCmarrigeassigngame}

Two-sided matching 
(Roth and Sotomayor 1990, Abdulkadiro{\u g}lu and S{\"o}nmez 2013)
  \citeH{RS90}\citeH{AS13matmark}%
affords a fairly general framework in game theory,
including the stable matching of 
Gale and Shapley (1962) 
  \citeH{GS62}%
and the assignment model of 
Shapley and Shubik (1972) 
  \citeH{SS72}%
as special cases.
An even more general framework has been proposed by
Fujishige and Tamura (2007),
  \citeH{FT07market}%
in which the existence of an equilibrium is established 
on the basis of a novel duality-related property
of M$\sp{\natural}$-concave functions.
The results of
Fujishige and Tamura (2007)
  \citeH{FT07market}%
are described in this section%
\footnote{
This section is based on
Section 11.10 of Murota (2009). 
  \citeH[Section 11.10]{Mbonn09}%
}.  

\subsection{Fujishige--Tamura model}

Let $P$ and $Q$ be finite sets and put
\[
 E =  P \times  Q  = \{ (i,j) \mid i \in P, j \in Q \} ,
\]
where we think of $P$ as a set of workers
and $Q$ as a set of firms, respectively.
We suppose that worker $i$ works at firm $j$ for 
$x_{ij}$ units of time, gaining a salary $s_{ij}$ per unit time.
Then the {\em labor allocation} is represented by an integer vector
\[
x = (x_{ij} \mid (i,j) \in E) \in \ZZ\sp{E}
\]
and the salary by a real vector
$s = ( s_{ij} \mid (i,j) \in E) \in \RR\sp{E}$.
We are interested in the stability of a pair $(x,s)$
in the sense to be made precise later.

For $i\in P$ and $j\in Q$ we put
\[
 E_{(i)} = \{i\} \times  Q  = \{ (i,j) \mid  j \in Q \},
\qquad
 E_{(j)} =  P \times \{j\}  = \{ (i,j) \mid  i \in P \} ,
\]
and for a vector $y$ on $E$ we denote by 
$y_{(i)}$ and $y_{(j)}$ 
the restrictions of $y$ to $E_{(i)}$
and $E_{(j)}$, respectively.
For example, for the labor allocation $x$ we obtain
\[
x_{(i)} = (x_{ij} \mid j \in Q) \in \ZZ\sp{E_{(i)}},
\qquad
x_{(j)} = (x_{ij} \mid i \in P) \in \ZZ\sp{E_{(j)}} 
\]
and this convention also applies to the salary vector $s$ to yield  
$s_{(i)}$ and $s_{(j)}$.

It is supposed that for each $(i,j) \in E$ 
lower and upper bounds on the salary $s_{ij}$ are given,
denoted by 
$\underline{\pi}_{ij} \in \RR \cup \{ -\infty \}$ 
and $\overline{\pi}_{ij} \in \RR \cup \{ +\infty \}$,
where $\underline{\pi}_{ij} \leq \overline{\pi}_{ij}$.
A salary $s$ is called {\em feasible} if 
$\underline{\pi}_{ij} \leq s_{ij} \leq \overline{\pi}_{ij}$
for all $(i,j) \in E$.
We put
\[
\underline{\pi} = (\underline{\pi}_{ij} \mid (i,j) \in E)
    \in (\RR \cup \{ -\infty \})\sp{E},
\qquad
\overline{\pi} = (\overline{\pi}_{ij} \mid (i,j) \in E) 
    \in (\RR \cup \{ +\infty \})\sp{E} .
\]

Each agent (worker or firm)
$k \in  P \cup  Q$ evaluates his/her state $x_{(k)}$ 
of labor allocation in monetary terms
through a function
$f_{k} : \ZZ\sp{E_{(k)}} \to \RR \cup \{ -\infty \}$.
Here the effective domain
$ 
 \dom f_{k} = \{ z  \in \ZZ\sp{E_{(k)}} \mid f_{k}(z) > -\infty \} 
$
is assumed to satisfy the following natural condition:
\begin{equation} \label{gameassumpA}
 \mbox{$\dom f_{k}$ is bounded and hereditary, with unique minimal element $\veczero$,}
\end{equation}
where $\dom f_{k}$ being hereditary means that
$\veczero \leq z \leq y \in \dom f_{k}$ implies $z \in \dom f_{k}$.
In what follows we always assume that $x$ is feasible in the sense that
\[ 
x_{(i)} \in \dom f_{i}
\quad
(i \in  P),
\qquad
x_{(j)} \in \dom f_{j}
\quad
(j \in  Q).
\] 
A pair $(x,s)$ of 
feasible allocation $x$ and feasible salary $s$
is called an {\em outcome}.

\begin{example} \rm  \label{EXgamestablemarriage}
The {\em stable marriage problem}\index{stable marriage problem}
can be formulated as a special case
of the present setting. Put
$\underline{\pi}=\overline{\pi}=\veczero$ and
define
$f_{i} : \ZZ\sp{E_{(i)}} \to \RR \cup \{ -\infty \}$
for  $i \in  P$
and
$f_{j} : \ZZ\sp{E_{(j)}} \to \RR \cup \{ -\infty \}$
for  $j \in Q$ as
\begin{equation} \label{gameufnPiQj} 
  f_{i}(y) = \left\{
   \begin{array}{ll}
    a_{ij}         & (y = \chi_{j}, j \in Q), \\
    0  &  (y = \veczero),   \\
    -\infty & (\mbox{otherwise}),
   \end{array} \right.                           
\quad
  f_{j}(z) = \left\{
   \begin{array}{ll}
    b_{ij}         & (z = \chi_{i}, i \in P), \\
    0  &  (z = \veczero),   \\
    -\infty & (\mbox{otherwise}) ,
   \end{array} \right.                           
\end{equation}
where the vector 
$(a_{ij} \mid j \in Q) \in {\RR}\sp{Q}$
represents (or, is an encoding of) 
the preference of ``man'' $i \in P$ over ``women'' $Q$,
and
$(b_{ij} \mid i \in P) \in {\RR}\sp{P}$
the preference of ``woman'' $j \in Q$ over ``men'' $P$.
Then a matching $X$ is
stable if and only if $(x,s)=(\chi_{X},\veczero)$ is stable 
in the present model.
\finbox
\end{example}

\begin{example} \rm  \label{EXgameassignment}
The {\em assignment model}\index{assignment model}
 is a special case where
$\underline{\pi}=(-\infty, \ldots, -\infty)$,
$\overline{\pi}=(+\infty, \ldots, +\infty)$
and the functions 
$f_{i}$ and $f_{j}$ are of the form of  (\ref{gameufnPiQj})
with some
$a_{ij},b_{ij} \in \RR$ for all $i \in P, j \in Q$.
\finbox
\end{example}

\subsection{Market equilibrium}

Given an outcome $(x,s)$ 
the payoff of worker $i \in  P$ is defined to be
the sum of his/her evaluation of $x_{(i)}$ and the total income from firms:
\begin{equation} \label{gamePigain}
 f_{i}(x_{(i)}) + \sum_{j \in  Q} s_{ij}x_{ij}
\qquad (=: (f_{i}+s_{(i)})(x_{(i)}) ) .
\end{equation}
Similarly, the payoff of firm $j \in  Q$ is defined as
\begin{equation} \label{gameQjgain}
f_{j}(x_{(j)}) - \sum_{i \in  P} s_{ij}x_{ij}
\qquad (=: (f_{j}-s_{(j)})(x_{(j)}) ) .
\end{equation}
Each agent ($i \in P$ or $j \in Q$) naturally wishes to maximize his/her
payoff function%
\footnote{
We have
$ (f_{i}+s_{(i)})(x_{(i)})=f_{i}[+s_{(i)}](x_{(i)}) $ and
$(f_{j}-s_{(j)})(x_{(j)}) = f_{j}[-s_{(j)}](x_{(j)})$
in the notation of (\ref{f-pdefZ}).
}. 

A {\em market equilibrium} is defined as an outcome $(x,s)$ 
that is stable under reasonable actions 
(i) by each worker $i$, (ii) by each firm $j$, and 
(iii) by each worker-firm pair $(i,j)$.
To be specific, we say that
$(x,s)$ is stable with respect to $i \in P$ if
\begin{equation} \label{incentivePi}
    (f_{i}+s_{(i)})(x_{(i)}) = 
            \max\{ (f_{i}+s_{(i)})(y) \mid y \leq x_{(i)}\} .
\end{equation}
Similarly, 
$(x,s)$ is said to be stable with respect to $j \in Q$ if
\begin{equation} \label{incentiveQj}
    (f_{j}-s_{(j)})(x_{(j)}) =
            \max\{ (f_{j}-s_{(j)})(z) \mid z \leq x_{(j)}\} .
\end{equation}
In technical terms $(x,s)$ is said to satisfy the 
{\em incentive constraint}\index{incentive constraint}
if it satisfies (\ref{incentivePi}) and (\ref{incentiveQj}).

The stability of $(x,s)$  with respect to 
$(i,j)$ is defined as follows.
Suppose that worker $i$ and firm $j$ think of a change of their contract
to a new salary $\alpha \in [\underline{\pi}_{ij},\overline{\pi}_{ij}]_{\RR}$
and a new working time of $\beta \in \ZZ_{+}$ units.
Worker $i$ will be happy with this contract if
there exists $y \in \ZZ\sp{E_{(i)}}$ such that
\begin{eqnarray}
 && y_{j} = \beta, \qquad
 y_{k} \leq x_{ik} \quad (k \in  Q\setminus\{j\}) ,
\label{newallocPi} \\
 &&(f_{i}+s_{(i)})(x_{(i)}) < (f_{i}+(s_{(i)}\sp{-j},\alpha))(y) ,
\label{newutilityPi}
\end{eqnarray}
where $(s_{(i)}\sp{-j},\alpha)$ denotes
the vector $s_{(i)}$ with its $j$-th component
replaced by $\alpha$.
Note that $y$ means the new labor allocation of worker $i$
with an increased payoff given on the right-hand side of (\ref{newutilityPi}).
Similarly, firm $j$ is motivated to make the new contract if
there exists $z \in \ZZ\sp{E_{(j)}}$ such that
\begin{eqnarray}
 && z_{i} = \beta, \qquad
 z_{k} \leq x_{kj} \quad (k \in  P\setminus\{i\}) ,
\label{newallocQj} \\
    && (f_{j}-s_{(j)})(x_{(j)}) < (f_{j}-(s_{(j)}\sp{-i},\alpha))(z) ,
\label{newutilityQj}
\end{eqnarray}
where $(s_{(j)}\sp{-i},\alpha)$ is
the vector $s_{(j)}$ with its $i$-th component
replaced by $\alpha$.
Then we say that $(x,s)$ is stable with respect to $(i,j)$ if
there exists no $(\alpha, \beta, y, z)$
that simultaneously satisfies
(\ref{newallocPi}), (\ref{newutilityPi}), (\ref{newallocQj})
and (\ref{newutilityQj}).

We now define an outcome $(x,s)$ to be {\em stable} 
if, for every $i \in  P$, $j \in  Q$,
$(x,s)$ is (i) stable with respect to $i$,
(ii) stable with respect to $j$, and
(iii) stable with respect to $(i,j)$.
This is our concept of market equilibrium.

A remarkable fact, found by 
Fujishige and Tamura (2007)\citeH{FT07market},
is that a market equilibrium
exists if the functions $f_{k}$ are M$\sp{\natural}$-concave.

\begin{theorem} \label{THgamestable}
Assume that $\underline{\pi} \leq \overline{\pi}$
and, for each $k \in  P \cup  Q$, $f_{k}$ is an
M$\sp{\natural}$-concave function satisfying {\rm (\ref{gameassumpA})}.
Then a stable outcome $(x,s) \in \ZZ\sp{E} \times \RR\sp{E}$ exists.
Furthermore,  we can take an integral $s \in \ZZ\sp{E}$
if $\underline{\pi} \in (\ZZ \cup \{ -\infty \})\sp{E}$,
$\overline{\pi} \in (\ZZ \cup \{ +\infty \})\sp{E}$,
and $f_{k}$ is integer-valued for every $k \in  P \cup  Q$.
\end{theorem}

\subsection{Technical ingredients}

The technical ingredients of the above theorem can be 
divided into the following two theorems, due to 
Fujishige and Tamura (2007).
  \citeH{FT07market}%
Note also that sufficiency part of Theorem \ref{THgamechar}
(which we need here) 
is independent of M$\sp{\natural}$-concavity.

\begin{theorem} \label{THgamechar}
Under the same assumption as in Theorem {\rm \ref{THgamestable}}
let $x$ be a feasible allocation.
Then $(x,s)$ is a stable outcome for some $s$
if and only if there exist 
$p \in \RR\sp{E}$, 
$u =  (u_{(i)} \mid i \in P) \in (\ZZ \cup \{ +\infty \})\sp{E}$ and 
$v =  (v_{(j)} \mid j \in Q) \in (\ZZ \cup \{ +\infty \})\sp{E}$
such that
\begin{eqnarray}
&&    x_{(i)} \in \argmax\{ (f_{i}+p_{(i)})(y) \mid y \leq u_{(i)}\},
\label{gamechar1} \\
&&    x_{(j)} \in \argmax\{ (f_{j}-p_{(j)})(z) \mid z \leq v_{(j)}\},
\label{gamechar2} \\
&&  \underline{\pi} \leq p \leq \overline{\pi},
\label{gamechar3}  \\
&& (i,j) \in E, u_{ij} < +\infty 
  \ \Longrightarrow \ p_{ij} = \underline{\pi}_{ij}, v_{ij}= +\infty,
\label{gamechar4} \\
&& (i,j) \in E, v_{ij} < +\infty 
  \ \Longrightarrow \ p_{ij} = \overline{\pi}_{ij}, u_{ij}= +\infty.
\label{gamechar5} 
\end{eqnarray}
Moreover, $(x,p)$ is a stable outcome for any $(x, p, u, v)$
satisfying the above conditions.
\end{theorem}

\begin{theorem} \label{THgameexist}
Under the same assumption as in Theorem {\rm \ref{THgamestable}}
there exists  $(x, p, u, v)$ that satisfies 
{\rm (\ref{gamechar1})--(\ref{gamechar5})}.
Furthermore,  we can take an integral $p \in \ZZ\sp{E}$
if $\underline{\pi} \in (\ZZ \cup \{ -\infty \})\sp{E}$,
$\overline{\pi} \in (\ZZ \cup \{ +\infty \})\sp{E}$,
and $f_{k}$ is integer-valued for every $k \in  P \cup  Q$.
\end{theorem}

It is worth while noting that the essence of 
Theorem \ref{THgameexist} is an intersection-type theorem
for a pair of M$\sp{\natural}$-concave functions,
Theorem \ref{THgameaggregate} below, due to 
Fujishige and Tamura (2007).
  \citeH{FT07market}%
Indeed Theorem \ref{THgameexist} can be derived easily from 
Theorem \ref{THgameaggregate}  
applied to
\begin{equation} \label{gamefPfQ}
 f_{P}(x) = \sum_{i \in P} f_{i}(x_{(i)}),
\qquad
 f_{Q}(x) = \sum_{j \in Q} f_{j}(x_{(j)}) .
\end{equation}

\begin{theorem} \label{THgameaggregate}
Assume $\underline{\pi} \leq \overline{\pi}$
for $\underline{\pi} \in (\RR \cup \{ -\infty \})\sp{E}$
and
$\overline{\pi} \in (\RR \cup \{ +\infty \})\sp{E}$,
and let $f, g: \ZZ\sp{E} \to \RR \cup \{ -\infty \}$
be M$\sp{\natural}$-concave functions
such that the effective domains are
bounded and hereditary, with unique minimal element $\veczero$.
Then there exist
$x \in \dom f \cap \dom g$, $p \in \RR\sp{E}$, 
$u  \in (\ZZ \cup \{ +\infty \})\sp{E}$ and $v  \in (\ZZ \cup \{ +\infty \})\sp{E}$ 
such that
\begin{eqnarray}
&&    x \in \argmax\{ (f+p)(y) \mid y \leq u \},
\label{gamegeneral1}
 \\
&&    x \in \argmax\{ (g-p)(z) \mid z \leq v \},
\label{gamegeneral2} 
\\
&&  \underline{\pi} \leq p \leq \overline{\pi},
\label{gamegeneral3}
  \\
&& e \in E, u_{e} < +\infty 
  \ \Longrightarrow \ p_{e} = \underline{\pi}_{e}, v_{e}= +\infty,
\label{gamegeneral4}
 \\
&& e \in E, v_{e} < +\infty 
  \ \Longrightarrow \ p_{e} = \overline{\pi}_{e}, u_{e}= +\infty.
\label{gamegeneral5} 
\end{eqnarray}
Furthermore,  we can take an integral $p \in \ZZ\sp{E}$
if $\underline{\pi} \in (\ZZ \cup \{ -\infty \})\sp{E}$,
$\overline{\pi} \in (\ZZ \cup \{ +\infty \})\sp{E}$,
and $f$ and $g$ are integer-valued.
\end{theorem}

\begin{remark} \rm  \label{RMgameFTdualcases}
Two special cases of Theorem \ref{THgameaggregate} are worth mentioning.

\begin{itemize}
\item
The first case is where 
$\underline{\pi}=(-\infty, \ldots, -\infty)$ and
$\overline{\pi}=(+\infty, \ldots, +\infty)$.
In this case, (\ref{gamegeneral3}) is void, and 
we must have 
$u_{e} = v_{e} = +\infty$ for all $e \in E$
by (\ref{gamegeneral4}) and (\ref{gamegeneral5}).
Therefore, the assertion of Theorem \ref{THgameaggregate} reduces
to:
There exist
$x \in \dom f \cap \dom g$ and $p \in \RR\sp{E}$
such that
$x \in \argmax (f+p)$ and 
$x \in \argmax (g-p)$,
which coincides with the M$\sp{\natural}$-concave intersection theorem
(Theorem~\ref{THmfcaveninteropt}).

\item
The second case is where $\underline{\pi}=\overline{\pi}=\veczero$,
which corresponds to the discrete concave stable marriage model of 
Eguchi et al.~(2003).
  \citeH{EFT03}%
Let $w$ be a vector such that $y \leq w$ for all $y \in \dom f \cap \dom g$.
By (\ref{gamegeneral3}) we must have 
$p_{e} = 0$ for all $e \in E$.
For each $e \in E$, we must have
$u_{e} = +\infty$ or $v_{e} = +\infty$ (or both)
by (\ref{gamegeneral4}) and (\ref{gamegeneral5}).
Therefore, the assertion of Theorem \ref{THgameaggregate} reduces
to:
There exist
$x \in \dom f \cap \dom g$,
$u  \in \ZZ\sp{E}$, and $v  \in \ZZ\sp{E}$ 
such that
$w =  u \vee v$,
$x \in \argmax\{ f(y) \mid y \leq u \}$, and 
$x \in \argmax\{ g(z) \mid z \leq v \}$.
This is the main technical result of 
Eguchi et al.~(2003)
  \citeH{EFT03}%
 that implies
the existence of a stable allocation in their model.
\finbox
\end{itemize}
\end{remark}

\subsection{Concluding remarks of section \ref{SCmarrigeassigngame}}
\label{SCmarriagecondrem}

The Fujishige--Tamura model contains 
several recently proposed matching models
such as
Eriksson and Karlander (2000), 
  \citeH{EK00}%
Fleiner (2001),
  \citeH{Flei01}%
Sotomayor (2002)
  \citeH{Sot02}%
as well as 
Eguchi and Fujishige (2002),
  \citeH{EF02}%
Eguchi et al.~(2003),
  \citeH{EFT03}%
Fujishige and Tamura (2006)
  \citeH{FT06market}%
as special cases.
In particular, the hybrid model of 
Eriksson and Karlander (2000), 
  \citeH{EK00}%
with flexible and rigid agents,
is a special case
where $P$ and $Q$ are partitioned as
$P=P_{\infty} \cup P_{0}$ and $Q=Q_{\infty} \cup Q_{0}$,
and 
$\underline{\pi}_{ij}=-\infty$, $\overline{\pi}_{ij}=+\infty$
for $(i,j) \in P_{\infty} \times Q_{\infty}$ and
$\underline{\pi}_{ij}=\overline{\pi}_{ij}=0$
for other $(i,j)$.
Realistic constraints on matchings such as lower quotas can 
be expressed in terms of matroids
(Fleiner 2001,
    \citeH{Flei01}%
Fleiner and Kamiyama 2016,
    \citeH{FK16}%
\RED{ 
Kojima et al.~2018,   \citeH{KTY14}%
}%
\RED{
Goto et al.~2017,  \citeH{GKKTY16}%
}%
\RED{
Yokoi 2017).
  \citeH{Yok16mor}%
}



\section{Valuated Assignment Problem}

\label{SCviap}

As we have seen in Sections \ref{SCmnatexample01} and \ref{SCnetinduce},
M$\sp{\natural}$-concave set functions 
are amenable to (bipartite) graph structures.
As a further step in this direction 
we describe the valuated (independent) assignment problem,
introduced by 
Murota (1996a, 1996b)\citeH{MvmiI96}\citeH{MvmiII96}.
In contrast to the original formulation of the problem
in terms of valuated matroids (or M-convex set functions),
 we present here a reformulation
in terms of  M$\sp{\natural}$-concave set functions
for the convenience of applications to economics and game theory.

\subsection{Problem description}
\label{SCviapprob}

The problem we consider is the following%
\footnote{
This problem is a variant of the valuated independent assignment problem.
}: 

\smallskip
{\bf [M$\sp{\natural}$-concave matching problem]}
Given a bipartite graph
$G=(V\sp{+}, V\sp{-}; A)$,
a pair of M$\sp{\natural}$-concave set functions
$f\sp{+}: 2\sp{V\sp{+}} \to \RR \cup \{ -\infty \}$
and
\RED{
$f\sp{-}: 2\sp{V\sp{-}} \to \RR \cup \{ -\infty \}$,
}%
and a  weight function  $w: A \to \RR$
(see Fig.~\ref{FGviapgraph}),
find a matching $M \ (\subseteq A)$ that maximizes
\begin{equation} \label{Omegadef}
 w(M)+ f\sp{+}(\partial\sp{+} M)  + f\sp{-}(\partial\sp{-} M),
\end{equation}
where
$w(M) = \sum \{w(a) \mid a \in M \}$, 
and  $\partial\sp{+} M$ (resp., $\partial\sp{-} M$)
denotes the set of the vertices in $V\sp{+}$ (resp., $V\sp{-}$)
incident to $M$.
For (\ref{Omegadef}) to be finite, we have implicit constraints that
\begin{equation} \label{indepmatching}
\partial\sp{+} M \in \dom f\sp{+},
\qquad
\partial\sp{-} M \in \dom f\sp{-}.
\end{equation}

\begin{figure}\begin{center}
 \includegraphics[width=0.25\textwidth,clip]{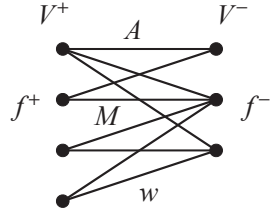}
 \caption{Valuated assignment problem}
 \label{FGviapgraph}
\end{center}\end{figure}

In applications the empty set often belongs to 
$\dom f\sp{+}$ (resp.,  $\dom f\sp{-}$),
in which case 
$\dom f\sp{+}$ (resp.,  $\dom f\sp{-}$) forms
the family of independent sets of a matroid.
If
$f\sp{+} \equiv 0$
and $f\sp{-} \equiv 0$
(with $\dom f\sp{+} = 2\sp{V\sp{+}}$ and $\dom f\sp{-} = 2\sp{V\sp{-}}$),
this problem coincides with  the conventional weighted matching problem.

An important special case of 
the M$\sp{\natural}$-concave matching problem
arises from a very special underlying graph 
$G_{\equiv}=(V\sp{+}, V\sp{-}; A_{\equiv})$
that represents a one-to-one correspondence between
$V\sp{+}$ and $V\sp{-}$.
In other words,
given a pair of M$\sp{\natural}$-concave set functions
$f_{1}, f_{2}: 2\sp{V} \to \RR \cup \{ -\infty \}$
and a  weight function  $w: V \to \RR$,
let
$V\sp{+}$ and $V\sp{-}$ be disjoint copies of $V$ and
$A_{\equiv} = \{ (v\sp{+},v\sp{-}) \mid v \in V \}$,
where $v\sp{+} \in V\sp{+}$ and $v\sp{-} \in V\sp{-}$ denote the
copies of $v \in V$.
The given functions  $f_{1}$ and $f_{2}$ are regarded as set functions 
on $V\sp{+}$ and $V\sp{-}$, respectively.
Then we obtain the following problem:

\smallskip
{\bf [M$\sp{\natural}$-concave intersection problem]} 
Given a pair of M$\sp{\natural}$-concave set functions
$f_{1}, f_{2}: 2\sp{V} \to \RR \cup \{ -\infty \}$
and a  weight function  $w: V \to \RR$, 
find a subset $X$ that maximizes 
\begin{equation} \label{Mconvinterobj01}
w(X) + f_{1}(X)  + f_{2}(X),
\end{equation}
where $w(X) = \sum_{v \in X} w(v)$.

\subsection{Optimality criterion by potentials}
\label{SCvmpotential}

We show the optimality criterion for 
the M$\sp{\natural}$-concave matching problem
in terms of potentials, where
a {\em potential}
means a function $p: V\sp{+} \cup V\sp{-} \to \RR$ 
(or a vector $p \in \RR\sp{V\sp{+} \cup V\sp{-}}$)
on the vertex set $V\sp{+} \cup V\sp{-}$. 
In the following theorem due to 
Murota (1996a)\citeH{MvmiI96}
(see also Theorem 5.2.39 of Murota 2000a\citeH[Theorem 5.2.39]{Mspr2000}),
the statement (1) refers to the existence of 
an appropriate potential,
whereas its reformulation in (2) reveals the duality nature%
\footnote{
Compare the identity in (2) with the Fenchel-type duality in Theorem~\ref{THmlfencdual}.
}.  
For each arc $a =(u,v) \in A$, $\partial\sp{+} a$ denotes the 
initial (tail) vertex of $a$, and
$\partial\sp{-}a$ the terminal (head) vertex of $a$,
i.e., $\partial\sp{+} a = u \in V\sp{+}$ and 
$\partial\sp{-} a = v \in V\sp{-}$,
where all the arcs are assumed to be directed from $V\sp{+}$ to $V\sp{-}$.

\begin{theorem}[Potential criterion]   \label{THvmpotential}
Let $M$ be a matching in $G=(V\sp{+}, V\sp{-}; A)$
satisfying {\rm (\ref{indepmatching})}
for the M$\sp{\natural}$-concave matching problem
to maximize {\rm (\ref{Omegadef})}.

\noindent
{\rm (1)}
$M$ is an optimal matching
if and only if
there exists a potential
$p: V\sp{+} \cup V\sp{-} \to \RR$ such that

{\rm (i)} \ 
${\displaystyle
   w(a) - p(\partial\sp{+} a)+p(\partial\sp{-} a)   \ 
   \left\{  \begin{array}{ll}
  \leq 0  &   (a \in A)  ,    \\
     = 0  &   (a \in M)  ,     
                      \end{array}  \right.
}$

{\rm (ii)} \ 
$\partial\sp{+} M$ is a maximizer of $f\sp{+}[+ p\sp{+}]$,

{\rm (iii)}  \ 
$\partial\sp{-} M$ is a maximizer of $f\sp{-}[-p \sp{-}]$,

\noindent
where $p\sp{+}$ and $p\sp{-}$ are the restrictions of $p$ to $V\sp{+}$ and $V\sp{-}$, 
respectively, and 
$f\sp{+}[+ p\sp{+}]$ and $f\sp{-}[-p\sp{-}]$
are defined by
\begin{align*}
 f\sp{+}[+p\sp{+}](X) &=
  f\sp{+}(X)  +  \sum \{ p(u) \mid u \in X \} 
\qquad (X \subseteq V\sp{+}),
\nonumber \\
 f\sp{-}[-p\sp{-}](Y) &=
  f\sp{-}(Y)  -  \sum \{ p(v) \mid v \in Y \} 
\qquad (Y \subseteq V\sp{-}).
\end{align*}

\noindent
{\rm (2)} \ 
${\displaystyle
 \max_{M} \{  w(M)+ f\sp{+}(\partial\sp{+} M)  + f\sp{-}(\partial\sp{-} M) \}}$
\par \ 
${\displaystyle 
=\min_{p} \{  \max (f\sp{+}[+ p\sp{+}]) 
           + \max (f\sp{-}[-p \sp{-}])  \mid 
w(a) - p(\partial\sp{+} a)+p(\partial\sp{-} a) \leq 0
          \ (a \in A) \}  .
}$

\noindent
{\rm (3)}
If $f\sp{+}$, $f\sp{-}$ and $w$ are all integer-valued,
the potential $p$ in {\rm (1)} and {\rm (2)}
can be chosen to be integer-valued.

\noindent
{\rm (4)}
Let $p$ be a potential that satisfies {\rm (i)}--{\rm (iii)} 
in {\rm (1)} for some (optimal) matching $M=M_{0}$.  
A matching $M'$ is 
optimal if and only if it satisfies {\rm (i)}--{\rm (iii)}
(with $M$ replaced by $M'$).
\end{theorem}

In connection to (ii) and (iii) in (1) in Theorem \ref{THvmpotential},
Theorem \ref{THmnatsetfnlocmax} shows:
\begin{eqnarray}
 && X \in \argmax (f\sp{+}[+ p\sp{+}]) 
 \nonumber \\
 & & \iff
   \left\{  \begin{array}{ll}
   f\sp{+}(X) - f\sp{+}(X-u+v) +p(u) - p(v) \geq 0 
    & (\forall \,  u \in X, \ \forall \, v \in V\sp{+} \setminus X),  \\
   f\sp{+}(X) - f\sp{+}(X-u) +p(u)  \geq 0 
    & (\forall \,  u \in X),  \\
   f\sp{+}(X) - f\sp{+}(X+v) - p(v) \geq 0 
    & (\forall \, v \in V\sp{+} \setminus X),  \\
             \end{array}  \right.
\label{POTigenderiv01} 
\\
 && Y \in \argmax (f\sp{-}[- p\sp{-}]) 
 \nonumber \\
 & & \iff  
   \left\{  \begin{array}{ll}
   f\sp{-}(Y) - f\sp{-}(Y-u+v) -p(u) + p(v) \geq 0 
    & (\forall \,  u \in Y, \ \forall \, v \in V\sp{-} \setminus Y),  \\
   f\sp{-}(Y) - f\sp{-}(Y-u) -p(u)  \geq 0 
    & (\forall \,  u \in Y),  \\
   f\sp{-}(Y) - f\sp{-}(Y+v) + p(v) \geq 0 
    & (\forall \, v \in V\sp{-} \setminus Y).  \\
             \end{array}  \right.
\label{POTiicaseMderiv01} 
\end{eqnarray}
These expressions are crucial in deriving the second
optimality criterion (Theorem \ref{THvmcycle}) 
in Section \ref{SCvmnegcycle} 
and in designing efficient algorithms 
for the M$\sp{\natural}$-concave matching problem.

The optimality condition for 
the M$\sp{\natural}$-concave intersection problem 
(\ref{Mconvinterobj01})
deserves a separate statement
in the form of weight splitting,
though it is an immediate corollary of the above theorem.

\begin{theorem}[Weight splitting for M$\sp{\natural}$-concave intersection]  
            \label{THwsplit01}
\quad

\noindent
{\rm (1)}
A subset $X \subseteq V$
maximizes  $w(X)+ f_{1}(X)  + f_{2}(X)$
if and only if
there exist $w_{1}, w_{2}: V \to \RR$ such that

{\rm (i)} 
{\rm [}``weight splitting''{\rm ]}\index{weight splitting}  \ 
$w(v) = w_{1}(v) + w_{2}(v)$ \  $(v \in V)$,

{\rm (ii)} \ 
$X$ is a maximizer of $f_{1}[+w _{1}]$,

{\rm (iii)}  \ 
$X$ is a maximizer of $f_{2}[+w _{2}]$.

\noindent
{\rm (2)}
${\displaystyle 
 \max_{X} \{ w(X)+ f_{1}(X)  + f_{2}(X) \} 
}$
\par \ 
${\displaystyle  =  
\min_{w_{1},w_{2}} \{ 
      \max (f_{1}[+w_{1}]) + \max (f_{2}[+w_{2}]) 
    \mid 
        w(v) = w_{1}(v) + w_{2}(v) \  (v \in V) \}  .
}$

\noindent
{\rm (3)}
If $f_{1}$, $f_{2}$ and $w$ are all integer-valued,
we may assume that $w_{1}, w_{2}: V \to \ZZ$.
\end{theorem}

\subsection{Optimality criterion by negative-cycles}
\label{SCvmnegcycle}

As the second criterion for optimality 
we describe the negative-cycle criterion.
First we need to introduce the auxiliary graph
$G_{M} = (\tilde V,  A_{M})$
associated with a matching $M$
satisfying 
$\partial\sp{+} M \in \dom f\sp{+}$ and $\partial\sp{-} M \in \dom f\sp{-}$
in (\ref{indepmatching}).
Define $X=\partial\sp{+} M$ and $Y=\partial\sp{-} M$.

The vertex set $\tilde V$ of the auxiliary graph $G_{M}$ is given by
$\tilde V = V\sp{+} \cup V\sp{-} \cup \{ s\sp{+}, s\sp{-} \}$,
where $s\sp{+}$ and $s\sp{-}$ are new vertices often referred to as
``source vertex'' and ``sink vertex'' respectively.
The arc set $A_{M}$ 
consists of nine disjoint parts:
\[
 A_{M} =  (A\sp{\circ} \cup M\sp{\circ})
      \cup (A\sp{+} \cup F\sp{+}  \cup S\sp{+})
      \cup (A\sp{-}  \cup F\sp{-} \cup S\sp{-}) \cup R,
\]
where%
\footnote{
The {\em reorientation} of an arc $a=(u,v)$ means the arc $(v,u)$, 
to be denoted as $\overline{a}$.
}  
\begin{eqnarray}
 A\sp{\circ} &=& \{ a \mid a \in A \}
  \qquad \mbox{(copy of $A$)},
\nonumber \\
 M\sp{\circ} &=& \{ \overline{a} \mid a \in M \}
  \qquad \mbox{($\overline{a}$: reorientation of $a$)},
\nonumber \\
 A\sp{+} &=& \{ (u,v) \mid 
        u \in X, \  v \in V\sp{+} \setminus X, \  
        X - u + v \in \dom f\sp{+}  \}, 
\nonumber \\
 F\sp{+} &=& \{ (u,s\sp{+}) \mid  
\RED{
u \in X, \
X - u \in \dom f\sp{+} \},
}%
\label{viapkarcdef} \\
 S\sp{+} &=& \{ (s\sp{+},v) \mid
\RED{
  v \in V\sp{+} \setminus X, \  
  X + v \in \dom f\sp{+} \},
}%
\nonumber \\
 A\sp{-} &=& \{ (v,u) \mid 
       u \in Y, \  v \in V\sp{-} \setminus Y,  \   
        Y - u + v \in \dom f\sp{-}  \} ,
\nonumber \\
 F\sp{-} &=& \{ (s\sp{-},u) \mid  
\RED{
 u \in Y, \ 
 Y - u  \in \dom f\sp{-}  \},
}%
\nonumber \\
 S\sp{-} &=& \{ (v,s\sp{-}) \mid 
\RED{
v \in V\sp{-} \setminus Y,  \  
Y + v \in \dom f\sp{-}  \},
}
\nonumber \\
 R \  &=& 
\RED{
\{ (s\sp{-},s\sp{+}),  (s\sp{+},s\sp{-})   \}.
}%
\nonumber 
\end{eqnarray}
The arc length $\ell_{M}(a)$ for $a \in A_{M}$ is defined by
\begin{equation} \label{gamdef2k}
 \ell_{M}(a) =
   \left\{  \begin{array}{ll}
  -w(a)             &   (a \in A\sp{\circ}),      \\
   w(\overline{a})  &   (a=(u,v) \in M\sp{\circ}, \  \overline{a}=(v,u) \in M), \\
 f\sp{+}(X)-f\sp{+}(X-u+v)  &   (a=(u,v)\in A\sp{+}) ,  \\
 f\sp{+}(X)-f\sp{+}(X-u)  &   (a=(u,s\sp{+})\in F\sp{+}) ,  \\
 f\sp{+}(X)-f\sp{+}(X+v)  &   (a=(s\sp{+},v)\in S\sp{+}) ,  \\
 f\sp{-}(Y)-f\sp{-}(Y-u+v)  &   (a=(v,u)\in A\sp{-}) , \\
 f\sp{-}(Y)-f\sp{-}(Y-u)  &   (a=(s\sp{-},u)\in F\sp{-}) , \\
 f\sp{-}(Y)-f\sp{-}(Y+v)  &   (a=(v,s\sp{-})\in S\sp{-}) , \\
 0   & 
\RED{
  (a \in R) . 
}%
                      \end{array}  \right.
\end{equation}

A directed cycle in $G_{M}$ of a negative length 
with respect to the arc length $\ell_{M}$
is called a {\em negative cycle}.
As is well known in network flow theory,
there exists no negative cycle  in $(G_{M},\ell_{M})$ 
if and only if there exists a potential 
$p: \tilde{V} \to \RR$ such that
\begin{equation} \label{potfeasible01}
 \ell_{M}(a) +  p(\partial\sp{+} a)  -  p(\partial\sp{-} a)
 \geq 0
\qquad (a \in A_{M}) ,
\end{equation}
where $\partial\sp{+} a$ denotes the 
initial (tail) vertex of $a$, and
$\partial\sp{-}a$ the terminal (head) vertex of $a$.
With the use of 
(\ref{POTigenderiv01}), (\ref{POTiicaseMderiv01}) and (\ref{potfeasible01}),
Theorem \ref{THvmpotential} is translated into the following theorem;
see Remark \ref{RMpotonsources}.
This theorem 
(Murota 1996a,
  \citeH{MvmiI96}%
Theorem 5.2.42 of Murota 2003)
  \citeH[Theorem 5.2.42]{Mspr2000}%
gives an optimality criterion in terms of negative cycles.

\begin{theorem}[Negative-cycle criterion] \label{THvmcycle}
In the M$\sp{\natural}$-concave matching problem
to maximize {\rm (\ref{Omegadef})},
a matching $M$ 
satisfying {\rm (\ref{indepmatching})}
is optimal 
if and only if
there exists in the auxiliary graph $G_{M}$ no negative cycle 
with respect to the arc length $\ell_{M}$.
\end{theorem}

\begin{remark} \rm  \label{RMpotonsources}
The condition (\ref{potfeasible01}) 
for $a \in (F\sp{+}  \cup S\sp{+}) \cup (F\sp{-} \cup S\sp{-})$
refers to 
$p(s\sp{+})$ and $p(s\sp{-})$,
while the potential $p$ in Theorem~\ref{THvmpotential}
is defined only on $V\sp{+} \cup V\sp{-}$.
To derive (\ref{potfeasible01}) from Theorem~\ref{THvmpotential} 
we may define $p(s\sp{+})=p(s\sp{-})=0$.
Indeed, the conditions imposed on $p(s\sp{+})$ by (\ref{potfeasible01}) are
\begin{align*}
& f\sp{+}(X)-f\sp{+}(X-u) +  p(u)  -  p(s\sp{+}) \geq 0
\qquad( u \in X) ,
\\
& f\sp{+}(X)-f\sp{+}(X+v) +  p(s\sp{+})  -  p(v) \geq 0
\qquad( v \in V\sp{+} \setminus X),
\end{align*}
which are satisfied 
by (\ref{POTigenderiv01}) if $p(s\sp{+})=0$.
Similarly for $p(s\sp{-})$.
\finbox
\end{remark}

\subsection{Concluding remarks of section \ref{SCviap}}
\label{SCviapcondrem}

Theorems \ref{THvmpotential} and \ref{THvmcycle} 
contain several standard results in matroid optimization,
such as 
Frank's weight splitting theorem 
(Frank 1981)\citeH{Fra81}
for the weighted matroid intersection problem.
The proofs of Theorems \ref{THvmpotential}
and \ref{THvmcycle} can be found in 
Murota (1996a)\citeH{MvmiI96}
and 
Section 5.2 of Murota (2000a)\citeH[Section 5.2]{Mspr2000}.
There are two key lemmas, called ``upper-bound lemma''
and ``unique-max lemma,'' which capture the 
essential properties inherent in M-concavity.
On the basis of these optimality criteria
efficient algorithms can be designed
for the M$\sp{\natural}$-concave matching problem.
For algorithmic issues, see 
Murota (1996b)\citeH{MvmiII96} 
and Section 6.2 of Murota (2000a)\citeH[Section 6.2]{Mspr2000}.

The valuated matching problem treated in this section
is generalized to the submodular flow problem in Section \ref{SCsubmflow}.



\section{Submodular Flow Problem}
\label{SCsubmflow}

\subsection{Submodular flow problem}
\label{SCsubmflproblem}

Let $G=(V,A)$ be a directed graph with vertex set $V$ and arc set $A$.
Suppose that each arc $a \in A$ is associated with
upper-capacity $\overline{c}(a)$, lower-capacity $\underline{c}(a)$,
and  cost $\gamma(a)$ per unit flow.
Furthermore, for each vertex $v \in V$, the amount of flow supply at $v$
is specified by $x(v)$.

The minimum cost flow problem
is to find a flow $\xi=(\xi(a) \mid a \in A)$ 
that minimizes the total cost
$\langle \gamma, \xi \rangle_{A}
 = \sum_{a \in A} \gamma(a) \xi(a)$
subject to the capacity constraint 
$\underline{c}(a) \leq \xi(a) \leq \overline{c}(a)$
$(a \in A)$
and the supply specification. 
Here the supply specification means a constraint 
that the boundary $\partial\xi$ of $\xi$ defined by
\begin{equation} \label{bounddef5}
  \partial\xi(v) = 
    \sum \{ \xi(a) \mid a \in \delta\sp{+}v \}  
    - \sum \{ \xi(a) \mid a \in \delta\sp{-}v \}
\qquad (v \in V)
\end{equation}
should be equal to a given value $x(v)$,
where
$\delta\sp{+}v$ and $\delta\sp{-}v$
denote the sets of arcs
leaving (going out of) $v$ and entering (coming into) $v$, respectively.
We can interpret 
$x(v)=\partial\xi(v)$ as the net amount of flow
entering the network at $v$ from outside.

We consider the integer flow problem,
which is described by 
an integer-valued upper-capacity
$\overline{c}: A \to \ZZ \cup \{ +\infty \}$,
an integer-valued lower-capacity
$\underline{c}: A \to \ZZ \cup \{ -\infty \}$,
a real-valued cost function $\gamma: A \to \RR$,
and an integer supply vector $x: V \to \ZZ$,
where it is assumed that
$\overline{c}(a) \geq \underline{c}(a)$ for each $a \in A$.
The variable to be optimized is an integral flow $\xi: A \to \ZZ$.

\smallskip

{\bf [Minimum cost flow problem  MCFP (linear arc cost)]}%
\footnote{
MCFP stands for Minimum Cost Flow Problem.
}
\begin{eqnarray}
\mbox{Minimize \ \ } & & 
 \Gamma_{0}(\xi) = 
 \sum_{a \in A} \gamma(a) \xi(a)
 \label{mincostflowlin} \\
 \mbox{subject to \ }  &  &  
  \underline{c}(a) \leq \xi(a) \leq \overline{c}(a)
  \qquad (a \in A),
 \label{capconstlin} \\
  &  &  \partial\xi =x,
 \label{bndxi=xlin} \\
  &  &  \xi(a) \in \ZZ       \qquad (a \in A) .
 \label{flowreallin} 
\end{eqnarray}

A generalization of the minimum cost flow problem MCFP 
is obtained by relaxing the supply specification 
$\partial\xi = x$ to the constraint that the flow boundary 
$\partial\xi$ should belong to a given subset $B$ of $\ZZ\sp{V}$ 
representing  ``feasible'' or ``admissible'' supplies%
\footnote{
By the flow conservation law,
the sum of the components of $\partial \xi$ is equal to zero,
i.e., $\partial \xi(V)=0$, for any flow $\xi$.
Accordingly we assume that $B$ is contained in the hyperplane
$\{ x \in \RR\sp{V} \mid x(V) = 0 \}$.
}: 
\begin{equation} \label{flowbndinB} 
 \partial\xi \in B.
\end{equation}
Such problem is
called the {\em submodular flow problem},
if $B$ is an M-convex set (integral base polyhedron; see Remark~\ref{RMmconcaveZ})%
\footnote{
In the conventional formulation 
(Chapter III of Fujishige 2005),
  \citeH[Chapter III]{Fuj05}%
the M-convex set $B$ is given by an integer-valued submodular set function
that describes $B$; see also 
Section 4.4 of Murota (2003).
  \citeH[Section 4.4]{Mdcasiam}%
}. 
This problem is introduced by Edmonds and Giles (1977)\citeH{EG77}.

\smallskip

{\bf [Submodular flow problem  MSFP$_{1}$ (linear arc cost)]}%
\footnote{
MSFP stands for M-convex Submodular Flow Problem.
We use denotation MSFP$_{i}$ with $i=1,2,3$ to indicate
the hierarchy of generality in the problems.
}
\begin{eqnarray}
\mbox{Minimize \ \ } & & 
 \Gamma_{1}(\xi) = 
 \sum_{a \in A} \gamma(a) \xi(a)
 \label{sbmflowlin} \\
 \mbox{subject to \ } & &  \underline{c}(a) \leq \xi(a) \leq \overline{c}(a)
  \qquad (a \in A),
 \label{sbmfcapconstlin} \\
  &  &  \partial \xi  \in B, 
 \label{sbmfbndxilin} \\
  &  &  \xi(a) \in \ZZ   
     \qquad (a \in A) .
 \label{sbmflowreallin} 
\end{eqnarray}

A further generalization of the problem is obtained by 
introducing a cost function
for the flow boundary $\partial\xi$ rather than merely imposing 
the constraint $\partial\xi \in B$.
Namely, with a function 
$f: \ZZ\sp{V} \to \RR\cup\{+\infty\}$
we add a new term
$f(\partial\xi)$ to the objective function,
thereby imposing constraint
$\partial\xi \in B=\dom f$ implicitly.
If the function $f$ is M-convex,
the generalized problem is called the 
{\em M-convex submodular flow problem}, introduced by 
Murota (1999).
  \citeH{Msbmfl}%

\smallskip

{\bf [M-convex submodular flow problem  MSFP$_{2}$ (linear arc cost)]}
\begin{eqnarray}
\mbox{Minimize \ \ } & & 
 \Gamma_{2}(\xi) = 
 \sum_{a \in A} \gamma(a) \xi(a)
  + f(\partial\xi)
 \label{mincostflowmlin} \\
 \mbox{subject to \ } & &  \underline{c}(a) \leq \xi(a) \leq \overline{c}(a)
  \qquad (a \in A),
 \label{capconstmlin} \\
  &  &  \partial \xi  \in \dom f,
 \label{bndximlin} \\
  &  &  \xi(a) \in \ZZ   
     \qquad (a \in A) .
 \label{flowrealmlin} 
\end{eqnarray}
The special case of the M-convex submodular flow problem MSFP$_{2}$
with a $\{ 0, +\infty\}$-valued $f$ reduces to
the submodular flow problem MSFP$_{1}$.

A still further generalization is possible by
replacing the linear arc cost in $\Gamma_{2}$
with a separable convex function\index{cost function!flow}.
Namely, using univariate convex functions%
\footnote{
$f_{a}(t-1) + f_{a}(t+1) \geq 2 f_{a}(t)$ for all integers $t$.
}  
$f_{a}: \ZZ \to \RR\cup\{+\infty\}$ $(a \in A)$,
we consider
${\displaystyle \sum_{a \in A} f_{a}(\xi(a))}$
instead of 
${\displaystyle \sum_{a \in A} \gamma(a) \xi(a)}$
to obtain MSFP$_{3}$ below.

\smallskip

{\bf [M-convex submodular flow problem  MSFP$_{3}$ (nonlinear arc cost)]}
\begin{eqnarray}
\mbox{Minimize \ \ } & & 
 \Gamma_{3}(\xi) = 
  \sum_{a \in A} f_{a}(\xi(a))  + f(\partial\xi)
\label{mincostflowmconv} \\
 \mbox{subject to \ } & &  \xi(a) \in \dom f_{a}   \qquad (a \in A),
 \label{capconstmconv} \\
  &  &  \partial \xi  \in \dom f,
 \label{bndxi=xmconv} \\
  &  &  \xi(a) \in \ZZ   
     \qquad (a \in A) .
 \label{flowrealmconv} 
\end{eqnarray}
Obviously,  MSFP$_{2}$ is a special case of MSFP$_{3}$ with
\begin{equation} \label{falin5m}
 f_{a}(t) =
   \left\{  \begin{array}{ll}
   \gamma(a) t & (t \in [\underline{c}(a), \overline{c}(a)]_{\ZZ}) , \\
   +\infty      & (\mbox{otherwise}) . 
            \end{array}  \right.
\end{equation}
Conversely,  MSFP$_{3}$ can be put into the form MSFP$_{2}$;
see Remark~\ref{RMarccosttobound}.

\begin{remark} \rm  \label{RMarccosttobound}
Problem MSFP$_{3}$ on $G=(V, A)$
can be written in the form of MSFP$_{2}$
on a larger graph $\tilde G=(\tilde V, \tilde A)$.
We replace each arc $a =(u,v) \in A$ with a pair of arcs,
$a\sp{+} = (u,v_{a}\sp{-})$ and 
$a\sp{-} = (v_{a}\sp{+}, v)$,
where $v_{a}\sp{+}$ and $v_{a}\sp{-}$ are newly introduced vertices.
Accordingly, we have
$\tilde A =  \{ a\sp{+}, a\sp{-} \mid a \in A \}$
and 
$\tilde V = V \cup \{ v_{a}\sp{+},v_{a}\sp{-} \mid a \in A \}$.
For each $a \in A$ we consider a function
$\tilde f_{a}: \ZZ\sp{2} \to \RR \cup \{ +\infty \}$
given by
\[
\tilde f_{a}(t,s)=
   \left\{  \begin{array}{ll}
   f_{a}(t)   &   (t+s =0) ,     \\
   +\infty        &     (\mbox{otherwise}) ,
                      \end{array}  \right.
\]
and define $\tilde f: \ZZ\sp{\tilde V} \to \RR\cup\{+\infty\}$
by
\[
 \tilde f(\tilde x)
=  \sum_{a \in A} \tilde f_{a}  (\tilde x(v_{a}\sp{+}), \tilde x(v_{a}\sp{-}))  
 + f(\tilde x |_{V}) 
\qquad (\tilde x \in \ZZ\sp{\tilde V}),
\]
where
$\tilde x |_{V}$ denotes the restriction of $\tilde x$ to $V$.
For a flow $\tilde \xi: \tilde A \to \ZZ$, we have
$\tilde \xi(a\sp{+}) = \tilde \xi(a\sp{-})$
if 
$(\partial \tilde \xi(v_{a}\sp{+}),\partial \tilde \xi(v_{a}\sp{-}))
\in \dom \tilde f_{a}$.
Problem MSFP$_{3}$ is thus reduced to MSFP$_{2}$ with objective function
$\tilde \Gamma_{2}(\tilde \xi) = \tilde f(\partial \tilde \xi)$,
where the function $\tilde f$ is M-convex. 
\finbox
\end{remark}

\begin{remark} \rm  \label{RMintersubmfl}
The M$\sp{\natural}$-concave intersection problem (Section \ref{SCseparthm}) 
can be formulated as an M-convex submodular flow problem.
Suppose we want to maximize the sum 
$f_{1}(x) + f_{2}(x)$
of two M$\sp{\natural}$-concave functions
$f_{1}, f_{2}: \ZZ\sp{n} \to \RR \cup \{ -\infty \}$.
Let
$\tilde{f}_{1}, \tilde{f}_{2}: \ZZ\sp{n+1} \to \RR \cup \{ -\infty \}$
be the associated M-concave functions;
see (\ref{mfnmnatfnrelationcave}).
We consider an M-convex submodular flow problem
on the bipartite graph $G=(V_{1} \cup V_{2},A)$
in Fig.~\ref{FGmcaveinter},
where 
$V_{i}=\{ v_{i0},  v_{i1}, \ldots, v_{in}  \}$
for $i=1,2$
and
$A = \{ (v_{1j}, v_{2j}) \mid j=0,1,\ldots,n \}$.
The boundary cost function 
$f: \ZZ\sp{V_{1}}\times \ZZ\sp{V_{2}} \to \RR \cup \{ +\infty \}$
is defined by
$f(x_{1},x_{2}) = -\tilde{f}_{1}(x_{1}) - \tilde{f}_{2}(-x_{2})$
for  
$x_{1} \in \ZZ\sp{V_{1}}$ and $x_{2} \in \ZZ\sp{V_{2}}$,
which is an M-convex function.
The arc costs are identically zero 
and no capacity constraints are imposed
($\gamma(a)=0$, 
$\overline{c}(a)=+\infty$, $\underline{c}(a)=-\infty$ for all $a \in A$).
Since $x_{1}=-x_{2}$ if $(x_{1},x_{2})=\partial \xi$ for a flow $\xi$
in this network, this M-convex submodular flow problem
is equivalent to the problem of maximizing $f_{1}(x) + f_{2}(x)$.
Theorem \ref{THmfcaveninteropt} for the M-convex intersection problem
can be regarded as a special case of Theorem \ref{THsbmfpotcritZ}
for the M-convex submodular flow problem.
\finbox
\end{remark}

 \begin{figure}\begin{center}
 \includegraphics[width=0.45\textwidth,clip]{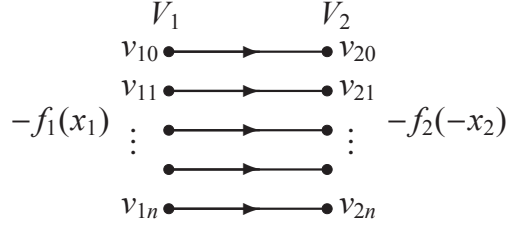}
\caption{M-convex submodular flow problem for M$\sp{\natural}$-concave intersection problem}
\label{FGmcaveinter}
 \end{center}\end{figure}

In subsequent sections we show optimality criteria
for the M-convex submodular flow problem
in terms of potentials and negative cycles.

\subsection{Optimality criterion by potentials}
\label{SCsbmfpotentialcrit}

We show the optimality criterion for 
the M-convex submodular flow problem MSFP$_{3}$
in terms of potentials.
A {\em potential}
means a function $p: V \to \RR$ (or a vector $p \in \RR\sp{V}$)
on the vertex set $V$. 
The {\em coboundary}
of a potential $p$ is a function $\delta p: A \to \RR$ 
on the arc set $A$ defined by
\begin{equation} \label{cobounddef5}
  \delta  p(a) =  
  p(\partial\sp{+} a)  -  p(\partial\sp{-} a)
\qquad (a \in A) ,
\end{equation}
where, for each arc $a \in A$, $\partial\sp{+} a$ denotes the 
initial (tail) vertex of $a$ and, $\partial\sp{-}a$ the terminal (head) vertex of $a$.
The following theorem is due to
Murota (1999);
  \citeH{Msbmfl}%
see also 
Section 9.4 of Murota (2003).
  \citeH[Section 9.4]{Mdcasiam}%

\begin{theorem}[Potential criterion]   \label{THsbmfpotcritZ}
Consider the M-convex submodular flow problem {\rm MSFP$_{3}$}.

\noindent {\rm (1)} 
For a feasible flow $\xi: A \to \ZZ$, 
two conditions {\rm (OPT)} and {\rm (POT)} below are equivalent.

{\rm (OPT)} $\xi$ is an optimal flow.

{\rm (POT)} There exists a potential $p: V \to \RR$ such that%
\footnote{
By notation (\ref{f-pdefZ}),
$f_{a}[+\delta p(a)]$ means the function defined as
$f_{a}[+\delta p(a)](t) = f_{a}(t) + 
( p(\partial\sp{+} a)  -  p(\partial\sp{-} a) )t $ for all $t \in \ZZ$.
} 
\par
\quad {\rm (i)} \ 
 $\xi(a) \in \argmin f_{a}[+\delta p(a)]$ \  for every $a \in A$, \  and
\par
\quad {\rm (ii)} \ $\partial\xi \in \argmin f[-p]$.

\noindent {\rm (2)}
Suppose that a potential $p: V \to \RR$ satisfies
{\rm (i)} and {\rm (ii)} above 
for an optimal flow $\xi$.
A feasible flow $\xi'$ is optimal if and only if
\par
\quad {\rm (i)} \ 
 $\xi'(a) \in \argmin f_{a}[+\delta p(a)]$ \  for every $a \in A$, \  and
\par
\quad {\rm (ii)} \ 
$\partial\xi' \in \argmin f[-p]$.

\noindent {\rm (3)}
If the cost functions
$f_{a}$ $(a \in A)$ and $f$
are integer-valued, 
there exists an integer-valued potential $p: V \to \ZZ$ 
in {\rm (POT)}.
Moreover, the set of integer-valued optimal potentials,
\[
\Pi\sp{*} = \{ p \mid p: \mbox{integer-valued optimal potential\,} \},
\]
is an L-convex set\,%
\footnote{
A nonempty set $P \subseteq \mathbb{Z}\sp{n}$
is called an {\em L-convex set}
if it is an L$\sp{\natural}$-convex set 
(Remark \ref{RMlnatconvexsetZ}) such that
$p \in P$ implies $p + \vecone, p -\vecone \in P$.
See 
Chapter 5 of Murota (2003)
  \citeH[Chapter 5]{Mdcasiam}%
for details.
}. 
\end{theorem}

In connection to (i) and (ii) in {\rm (POT)} in Theorem \ref{THsbmfpotcritZ},
note the equivalences:
\begin{eqnarray}
 && \xi(a) \in \argmin f_{a}[+\delta p(a)]
 \nonumber \\
 & & \iff  
 f_{a}(\xi(a)+d) - f_{a}(\xi(a)) + d[p(\partial\sp{+} a) - p(\partial\sp{-} a)]
  \geq 0
\qquad (d=\pm1),
\label{POTigenderivZ} 
\\
 && \partial\xi \in \argmin f[-p]
 \nonumber \\
 & & \iff  
 \Delta f(\partial \xi;v,u) +p(u)-p(v) \geq 0
 \qquad (\forall u,v \in V),
\label{POTiicaseMderivZ} 
\end{eqnarray}
where  
\begin{equation} \label{fvuderivdef}
 \Delta f(z;v,u)  = 
f(z+\unitvec{v}-\unitvec{u}) - f(z) 
   \qquad (z \in \dom f; u,v \in V) .
\end{equation}
These expressions are crucial in deriving the second
optimality criterion (Theorem \ref{THsbmfcyccritZ}) in Section \ref{SCsbmfnegcyclecrit} 
and in designing efficient algorithms 
for the M-convex submodular flow problem.

\subsection{Optimality criterion by negative cycles}
\label{SCsbmfnegcyclecrit}

The optimality of an M-convex submodular flow can also be characterized
by the nonexistence of negative cycles in an auxiliary network.
This fact leads to the cycle-cancelling algorithm.
We consider the M-convex submodular flow problem
MSFP$_{2}$ that has a linear arc cost.
This is not restrictive, since MSFP$_{3}$
can be put in the form of MSFP$_{2}$ (Remark~\ref{RMarccosttobound}).

For a feasible flow $\xi: A \to \ZZ$
we define an auxiliary network as follows.
Let $G_{\xi}=(V,A_{\xi})$
be a directed graph with vertex set $V$ and 
arc set $ A_{\xi} = A_{\xi}\sp{\circ} \cup B_{\xi}\sp{\circ} \cup C_{\xi}$
consisting of three disjoint parts:
\begin{eqnarray}
 A_{\xi}\sp{\circ} & = &
    \{ a \mid a \in A, \  \xi(a) < \overline{c}(a) \},
\nonumber \\
 B_{\xi}\sp{\circ} &  = &
    \{ \overline{a} \mid a \in A,  \  \underline{c}(a) < \xi(a)  \}
  \qquad \mbox{($\overline{a}$: reorientation of $a$)},
\nonumber \\
  C_{\xi} & = & \{ (u,v) \mid 
    u, v \in V,  \  u\not=v, \  
    \partial\xi-  (\unitvec{u}-\unitvec{v}) \in \dom f \}.
\label{auggrCxidefZ} 
\end{eqnarray}
We define an arc length function $\ell_{\xi}: A_{\xi} \to \RR$ by
\begin{equation}\label{sbmfauxgraphlengthdef5Z}
 \ell_{\xi}(a) =
   \left\{  \begin{array}{ll}
   \gamma(a )             &   (a \in A_{\xi}\sp{\circ}) ,     \\
   -\gamma(\overline{a})  &   (a \in B_{\xi}\sp{\circ}, \  \overline{a} \in A) ,\\
  \Delta f(\partial \xi;v,u)  &   (a=(u,v)\in C_{\xi}).   \\
                      \end{array}  \right.
\end{equation}
We refer to $(G_{\xi},\ell_{\xi})$ as the {auxiliary network}.

A directed cycle in $G_{\xi}$ of a negative length with respect to 
the arc length $\ell_{\xi}$
is called a {\em negative cycle}.
As is well known in network flow theory,
there exists no negative cycle 
    in $(G_{\xi},\ell_{\xi})$ 
if and only if there exists a potential 
$p: V \to \RR$ such that
\begin{equation} \label{potfeasible}
 \ell_{\xi}(a) +  p(\partial\sp{+} a)  -  p(\partial\sp{-} a)
 \geq 0
\qquad (a \in A_{\xi}) .
\end{equation}
With the use of 
(\ref{POTigenderivZ}), (\ref{POTiicaseMderivZ}) and (\ref{potfeasible}),
 Theorem \ref{THsbmfpotcritZ}
is translated into the following theorem
(Murota 1999;
  \citeH{Msbmfl}%
 see also 
Section 9.5 of Murota 2003),
  \citeH[Section 9.5]{Mdcasiam}%
which gives an optimality criterion in terms of negative cycles.

\begin{theorem}[Negative-cycle criterion] \label{THsbmfcyccritZ}
For a feasible flow $\xi: A \to \ZZ$
to the M-convex submodular flow problem {\rm MSFP$_{2}$},
the conditions {\rm (OPT)} and {\rm (NNC)} below are equivalent.

{\rm (OPT)} $\xi$ is an optimal flow.

{\rm (NNC)} 
There exists no negative cycle 
    in the auxiliary network $(G_{\xi},\ell_{\xi})$ 
    with $\ell_{\xi}$ of {\rm (\ref{sbmfauxgraphlengthdef5Z})}. 
\end{theorem}

\paragraph{Cycle cancellation:}
\label{SCsbmfcyclecancel}

The negative-cycle optimality criterion 
states that the existence of a negative cycle implies the non-optimality
of a feasible flow.
This suggests the possibility of 
improving a non-optimal feasible flow by the cancellation of 
a suitably chosen negative cycle.

Suppose that negative cycles exist in the 
auxiliary network $(G_{\xi},\ell_{\xi})$
for a feasible flow $\xi$, where the arc length $\ell_{\xi}$
is defined by (\ref{sbmfauxgraphlengthdef5Z}).
Choose a negative cycle having the smallest number of arcs,
and let $Q$ $(\subseteq A_{\xi})$ be the set of its arcs.
Modifying the flow $\xi$ along $Q$ by a unit amount we obtain a new flow
$\overline{\xi}$ defined by
\begin{equation}   \label{flowud5}
 \overline{\xi}(a) =
   \left\{  \begin{array}{ll}
   \xi(a)+1  &   (a \in  Q \cap A_{\xi}\sp{\circ})    ,  \\
   \xi(a)-1   &   (\overline{a} \in Q \cap B_{\xi}\sp{\circ}) , \\
   \xi(a)   &   \mbox{(otherwise)}.
                      \end{array}  \right.
\end{equation}
The following theorem%
\footnote{
The inequality (\ref{GammdecreseNegCyc}) is by no means obvious. See 
Murota (1999)
  \citeH{Msbmfl}%
and Section 10.4 of Murota (2003)
  \citeH[Section 10.4]{Mdcasiam}%
for the proof.
}  
 shows that the updated flow $\overline{\xi}$ is a feasible 
flow with an improvement in the objective function in (\ref{mincostflowmlin}): 
\[
 \Gamma_{2}(\xi) = 
 \sum_{a \in A} \gamma(a) \xi(a) + f(\partial\xi) .
\]

\begin{theorem}   \label{THncaug}
For a feasible flow $\xi: A \to \ZZ$
to the M-convex submodular flow problem {\rm MSFP$_{2}$},
let $Q$ be a negative cycle having the smallest number of arcs
in $(G_{\xi},\ell_{\xi})$.
Then $\overline{\xi}$ in {\rm (\ref{flowud5})} is a feasible flow and 
\begin{equation}  \label{GammdecreseNegCyc}
\Gamma_{2}(\overline{\xi}) \leq 
 \Gamma_{2}(\xi) + \ell_{\xi}(Q)  <  \Gamma_{2}(\xi) .
\end{equation}
\end{theorem}

\subsection{Concluding remarks of section \ref{SCsubmflow}}
\label{SCsugmflowcondrem}

On the basis of the optimality criteria in Theorems \ref{THsbmfpotcritZ}
and \ref{THsbmfcyccritZ}
we can design efficient algorithms 
for the M-convex submodular flow problem,
where the expressions (\ref{POTigenderivZ}) and   
(\ref{POTiicaseMderivZ})
are crucial.
For algorithmic issues, see 
Murota (1999),
  \citeH{Msbmfl}%
Section 10.4 of Murota (2003),
  \citeH[Section 10.4]{Mdcasiam}%
Iwata and Shigeno (2003),
  \citeH{IS03}%
Murota and Tamura (2003b),
  \citeH{MTcompeq03}%
and 
Iwata et al.~(2005).
  \citeH{IMM03submflow}%



\section{Discrete Fixed Point Theorem}
\label{SCfixedpoint}

Discrete fixed point theorems in discrete convex analysis
originate in the theorem of 
Iimura et al.~(2005)\citeH{IMT05}
based on Iimura (2003)\citeH{Iim03}, 
which is described in this section.
Subsequent development and other types of discrete fixed point theorems 
are mentioned in Section \ref{SCfixptcondrem}.

\subsection{Discrete fixed point theorem}
\label{SCfixptthm}

To motivate the discrete fixed point theorem of 
Iimura et al.~(2005)\citeH{IMT05},
we first take a glimpse at Kakutani's fixed point theorem.

Let $S$ be a subset of $\RR\sp{n}$
and $F$ be a set-valued mapping (correspondence)
 from $S$ to itself, which is denoted as
$F: S \to\to S$ (or $F: S \to 2\sp{S}$).
A point $x \in S$ satisfying $x \in F(x)$ is said to be a 
{\em fixed point}\index{fixed point} of $F$.
Kakutani's fixed point theorem reads as follows.

\begin{theorem}    \label{THkakutanifixpt}
A set-valued mapping $F: S \to\to S$,
where  $S \subseteq \RR\sp{n}$, has a fixed point if

{\rm (a)}
 $S$ is a bounded closed convex subset of $\RR\sp{n}$,

{\rm (b)}
 For each $x \in S$, $F(x)$ is a nonempty closed convex set, and

{\rm (c)}
 $F$ is upper-hemicontinuous.
\end{theorem}

In the discrete fixed point theorem (Theorem \ref{THimtfixpt} below) we are
concerned with 
$F: S \to\to S$, where $S$ is a subset of $\ZZ\sp{n}$.
The three conditions (a) to (c) in Theorem \ref{THkakutanifixpt} above
are ``discretized'' as follows.

\begin{itemize}
\item
Condition (a) assumes that the domain of definition $S$
is nicely-shaped or well-behaved.
In the discrete case we assume $S$ to be ``integrally convex.''
\item
Condition (b) assumes that each value $F(x)$
is nicely-shaped or well-behaved.
In the discrete case we assume that
$F(x) = \overline{F(x)} \cap \ZZ\sp{n}$,
where $\overline{F(x)}$ denotes the convex hull of $F(x)$.
\item
Condition (c) assumes that mapping $F$ is continuous in some appropriate sense.
In the discrete case we assume $F$ to be ``direction-preserving.'' 
\end{itemize}

The key concepts, 
``integrally convex set'' and ``direction-preserving mapping,'' 
are explained in 
Section \ref{SCfixpticsetdirpres}.
The discrete fixed point theorem of 
Iimura et al.~(2005)\citeH{IMT05}
is the following.

\begin{theorem}  \label{THimtfixpt}
A set-valued mapping $F: S \to\to S$,
where  $S \subseteq \ZZ\sp{n}$,
 has a fixed point if

{\rm (a)}
 $S$ is a nonempty finite integrally convex subset of $\ZZ\sp{n}$,

{\rm (b)}
 For each $x \in S$, $F(x)$ is nonempty and 
$F(x) = \overline{F(x)} \cap \ZZ\sp{n}$, and

{\rm (c)}
 $F$ is direction-preserving.
\end{theorem}

\subsection{Integrally convex set and direction-preserving mapping}
\label{SCfixpticsetdirpres}

\paragraph{Integrally convex set:}

 \begin{figure}\begin{center}
 \includegraphics[width=0.5\textwidth,clip]{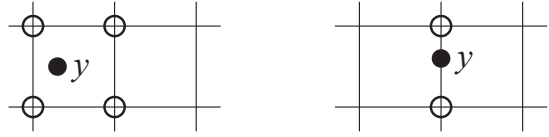}
\caption{Integral neighbor $N(y)$ of $y$ \ \ 
{\rm (}$\circ$: point of $N(y)${\rm )}}
\label{FGneighbor}
 \end{center}\end{figure}

 \begin{figure}\begin{center}
 \includegraphics[width=0.78\textwidth,clip]{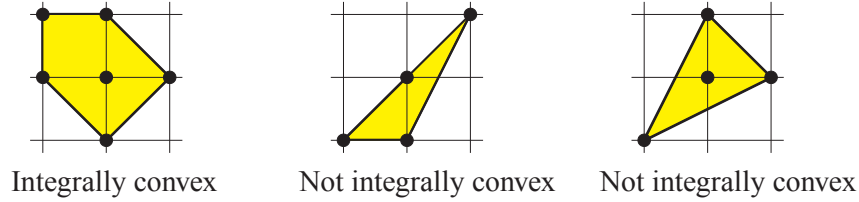}
\caption{Concept of integrally convex sets}
\label{FGintconvset}
 \end{center}\end{figure}

The {\em integral neighborhood} of a point $y \in \RR\sp{n}$
is defined as
\begin{equation}
 N(y) = \{ z \in \ZZ\sp{n} \mid \|z-y\|_{\infty} < 1 \} .
\end{equation}
See Fig.~\ref{FGneighbor}.
A set $S  \subseteq \ZZ\sp{n}$ is said to be 
{\em integrally convex}\index{integrally convex set}
if
\begin{equation}\label{setintconvnscond}
 y \in \overline{S} \ \Longrightarrow  y \in  \overline{S \cap N(y)}
\end{equation}
for any $y \in \RR\sp{n}$ 
(Favati and Tardella 1990)\citeH{FT90}.
Figure \ref{FGintconvset} illustrates this concept.
We have $S = \overline{S} \cap \ZZ\sp{n}$ for an integrally convex set $S$. 
It is known that L$\sp{\natural}$-convex sets and
M$\sp{\natural}$-convex sets are integrally convex.
See Section 3.4 of Murota (2003)\citeH[Section 3.4]{Mdcasiam}
 and Moriguchi et al.~(2016)\citeH{MMTT16}
for more about integral convexity.

\paragraph{Direction-preserving mapping:}

\begin{figure}\begin{center}
 \includegraphics[width=0.45\textwidth,clip]{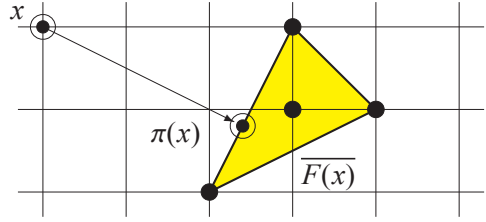}
\caption{Projection $\pi(x)$ with $\sigma(x) = \sign(\pi(x)-x) = (+1, -1)$}
\label{FGdirectsign}
\end{center}\end{figure}

Let $S$ be a subset of $\ZZ\sp{n}$
and $F: S \to\to S$ be a set-valued mapping (correspondence) from $S$ to $S$.
For $x =(x_{1},\ldots, x_{n})  \in \ZZ\sp{n}$ we denote by 
$\pi(x) = (\pi_{1}(x), \ldots, \pi_{n}(x) ) \in \RR\sp{n}$
the projection of $x$ to $\overline{F(x)}$;
see Fig.~\ref{FGdirectsign}.
This means that $\pi(x)$ is the point of $\overline{F(x)}$
that is nearest to $x$ with respect to the Euclidean norm.
We define the direction sign vector
$\sigma(x) \in \{+1,0,-1 \}\sp{n}$ as
\[ 
 \sigma(x) =
 ( \sigma_{1}(x), \ldots, \sigma_{n}(x) )
 = ( \sign(\pi_{1}(x) - x_{1}), \ldots, \sign(\pi_{n}(x) -  x_{n}) ) ,
\] 
where
\[
 \sign(y) =  
   \left\{  \begin{array}{rl}
   +1  & (y > 0),  \\
   0  & (y = 0),  \\
   -1  & (y < 0).  \\
             \end{array}  \right.
\]
Then we say that $F$ is 
{\em direction-preserving}
if for all $x, z \in S$ with $\|x-z\|_{\infty} \leq 1$ 
it holds that
\begin{equation} \label{dirpresdef1}
   \sigma_{i}(x) > 0 
 \ \Longrightarrow \ 
 \sigma_{i}(z) \geq 0 
   \qquad (i=1,\ldots,n).
\end{equation}
Note that this is equivalent to saying that
$\sigma_{i}(x) \sigma_{i}(z) \not= -1$ for each $i=1,\ldots,n$
if $x, z \in S$ and $\|x-z\|_{\infty} \leq 1$.
Being direction-preserving is interpreted as being ``continuous'' 
in the discrete setting.

\subsection{Illustrative examples}
\label{SCfixptexample}

\begin{example} \rm \label{EXfixonedim}
The significance of being direction-preserving is most transparent
in the case of $n=1$.
Let $S = [ a, b ]_{\ZZ}$ be an integer interval
with $a, b \in \ZZ$ and $a \leq b$.
Consider $F: S \to\to S$
represented as
$F(x) = [ \alpha(x), \beta(x) ]_{\ZZ}$,
where $\alpha(x), \beta(x) \in \ZZ$ and $a \leq  \alpha(x) \leq \beta(x) \leq  b$.
The projection $\pi(x)$ and the direction sign vector $\sigma(x)$ are given by
\[
\pi(x)  =
   \left\{  \begin{array}{ll}
    x           &   (\alpha(x) \leq x \leq \beta(x)) ,     \\
   \alpha(x)    &   (x \leq \alpha(x)-1),  \\
   \beta(x)     &   (x \geq \beta(x)+1),  \\
                      \end{array}  \right.
\quad
\sigma(x)  =
   \left\{  \begin{array}{ll}
    0           &   (\alpha(x) \leq x \leq \beta(x))  ,    \\
    +1     &   (x \leq \alpha(x)-1) , \\
    -1     &   (x \geq \beta(x)+1) .  \\
                      \end{array}  \right.
\]
Suppose that $F$ is direction-preserving,
which means 
$\sigma(x) \sigma(x+1) \not= -1$ 
for all $x$ with $a \leq x < b$.
There are three possibilities:

(i) $\sigma(x)=+1$ for all $x \in S$, 

(ii) $\sigma(x)=-1$ for all $x \in S$,

(iii) $\sigma(x)=0$ for some $x \in S$.

\noindent
In the first case (i) 
we must have
$x+1 \leq \alpha(x) \leq b$
for all $x \in S$,
but this is impossible for $x = b$.
Similarly, the second case (ii) is not possible, either.
Therefore, we must have the third case (iii),
and then the $x$ satisfying $\sigma(x)=0$ is a fixed point of $F$.
\finbox
\end{example}

\begin{figure}
\begin{center}
 \includegraphics[width=0.5\textwidth,clip]{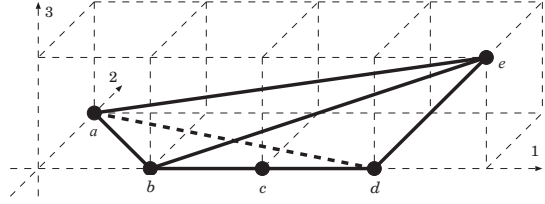}
 \caption{Necessity of the assumption  of integral convexity}
 \label{FGfixptcounter}
\end{center}
\end{figure}

\begin{example} \rm \label{EXfixptcounter}
The assumption (a) of integral convexity in Theorem \ref{THimtfixpt}
cannot be weakened to the ``hole-free'' property:
$S = \overline{S} \cap \ZZ\sp{n}$.
Let $n=3$ and consider a subset $S$ of $\ZZ\sp{3}$ (Fig.~\ref{FGfixptcounter}) given by
\[
  S = \{ a=(0,1,0), \ b=(1,0,0), 
      \ c=(2,0,0), \ d=(3,0,0), \ e=(4,0,1) \} ,
\]
which is not integrally convex, but satisfies $S = \overline{S} \cap \ZZ\sp{n}$.
Define $F : S \to\to S$ by
\[
 F(a)=F(b)= \{ e \}, \quad
 F(c)=\{ a,e \}, \quad
 F(d)=F(e)= \{ a \}.
\]
For each $x \in S$, $F(x)$ is a nonempty subset of $S$ 
satisfying 
$F(x) = \overline{F(x)} \cap \ZZ\sp{n}$.
Furthermore, $F$ is direction-preserving.
Indeed we have
\[
   \begin{array}{lcl@{}rrr@{}l}
    \pi(a)-a &=& (&  4,&  -1,&   1 &), \\
    \pi(b)-b &=& (&  3,&   0,&   1 &), \\
    \pi(c)-c &=& (&  0,& 1/2,& 1/2 &), \\
    \pi(d)-d &=& (& -3,&   1,&   0 &), \\
    \pi(e)-e &=& (& -4,&   1,&  -1 &),
   \end{array}
\quad
   \begin{array}{lcl@{}rrr@{}l}
    \sigma(a) &=& (& +1,&  -1,&  +1 &), \\
    \sigma(b) &=& (& +1,&   0,&  +1 &), \\
    \sigma(c) &=& (&  0,&  +1,&  +1 &), \\
    \sigma(d) &=& (& -1,&  +1,&   0 &), \\
    \sigma(e) &=& (& -1,&  +1,&  -1 &)
   \end{array}
\]
and the condition (\ref{dirpresdef1}) holds 
for every pair $(x,z)$ with $\|x-z\|_{\infty} \leq 1$,
i.e., for
$(x,z) =(a,b), (b,c), (c,d), (d,e)$.
Thus, $F$ meets the conditions
(b) and (c) in Theorem \ref{THimtfixpt}, but $F$ has no fixed point.
\finbox
\end{example}

\subsection{Proof outline}
\label{SCfixptproof}

The proof of Theorem~\ref{THimtfixpt} consists of the following three major steps;
the reader is referred to 
Iimura et al.~(2005)\citeH{IMT05} 
for the detail.

\begin{enumerate}
\item
An integrally convex set $S$  
has a simplicial decomposition $\calT$ with a nice property.
For each $y \in \RR\sp{n}$ contained in the convex hull of $S$,
let $T(y)$ denote the smallest simplex in $\calT$ that contains $y$.
Then the simplicial decomposition $\calT$ has the property
that all the vertices of $T(y)$ belong to the integral neighborhood $N(y)$ of $y$.
That is, the set of the vertices of $T(y)$, to be denoted by $V(y)$,
is given as $V(y) = T(y) \cap N(y)$.

\item

With reference to the simplicial decomposition $\calT$,
we define a piecewise linear extension, say, $f$
of the projection $\pi$ by
\[ 
 f(y) = \sum_{x \in V(y)} \lambda_{x} \pi(x)
\qquad
 ( y = \sum_{x \in V(y)} \lambda_{x} x,
  \ \ 
   \sum_{x \in V(y)} \lambda_{x} = 1, \quad
   \lambda_{x} \geq 0 ) .
\] 
By Brouwer's fixed point theorem applied to 
$f: \overline{S} \to \overline{S}$,
we obtain a fixed point $y\sp{*} \in \overline{S}$ of $f$,
i.e., $y\sp{*} = f(y\sp{*})$.

\item
   From the equations
\[ 
 \sum_{x \in V(y\sp{*})}\lambda_{x} (\pi(x)-x)
 =  \sum_{x \in V(y\sp{*})}\lambda_{x} \pi(x)
   -  \sum_{x \in V(y\sp{*})}\lambda_{x} x
 = f(y\sp{*}) - y\sp{*}  = \veczero
\] 
and the assumption of $F$ being direction-preserving,
we see that $\pi(x)-x = \veczero$ for some $x \in V(y\sp{*})$.
Let $x\sp{*}$ be such a point in $V(y\sp{*})$.
Then $x\sp{*}$ is a fixed point of $F$, since 
$x\sp{*} = \pi(x\sp{*}) \in \overline{F(x\sp{*})}$,
from which follows
$x\sp{*} \in \overline{F(x\sp{*})} \cap \ZZ\sp{n} = F(x\sp{*})$
by condition (b).
\end{enumerate}

\subsection{Concluding remarks of section \ref{SCfixedpoint}}
\label{SCfixptcondrem}

The discrete fixed point theorem initiated by 
Iimura (2003)\citeH{Iim03}
and Iimura et al.~(2005)\citeH{IMT05}
aims at a discrete version of Brouwer's fixed point theorem.
Related work in this direction includes
van der Laan et al.~(2006)\citeH{LTY06zpt},
Danilov and Koshevoi (2007)\citeH{DK07fixpt},
Chen and Deng (2006, 2008, 2009)\citeH{CD06lat}\citeH{CD08brw}\citeH{CD09simp},
Yang (2008, 2009)\citeH{Yan08comp}\citeH{Yan09fixpt},
Talman and  Yang (2009)\citeH{TY09},
Iimura and Yang (2009)\citeH{IY09},
Iimura (2010)\citeH{Iim10},
Deng et al.~(2011)\citeH{DQSZ11},
van der Laan et al.~(2011)\citeH{LTY11},
and Iimura et al.~(2012)\citeH{IMT12}.
Discrete fixed point theorems are used successfully 
in showing the existence of
a competitive equilibrium under indivisibility,
a pure Nash equilibrium with discrete strategy sets,
etc.

Efforts are made to weaken the condition (c) 
of ``direction preserving'' in Theorem \ref{THimtfixpt}.
Weaker conditions called
``locally gross direction preserving''
and
``simplicially locally gross direction preserving''
are considered by
Yang (2008, 2009)\citeH{Yan08comp}\citeH{Yan09fixpt},
Iimura and Yang (2009)\citeH{IY09},
Iimura (2010)\citeH{Iim10}.
Further variants are found in 
Talman and  Yang (2009)\citeH{TY09},
van der Laan et al.~(2011)\citeH{LTY11},
and Iimura et al.~(2012)\citeH{IMT12}.
These studies, however, share the framework
of  mappings and correspondences defined 
on integrally convex sets or their simplicial divisions.

The proof of Theorem \ref{THimtfixpt} by
Iimura et al.~(2005)\citeH{IMT05}
is not constructive,
relying on Brouwer's fixed point theorem.
Constructive proofs are given by
van der Laan et al.~(2006)\citeH{LTY06zpt}
and van der Laan et al.~(2011)\citeH{LTY11}.
Computational complexity of finding a fixed point
for direction-preserving mappings 
is discussed by
Chen and Deng (2006, 2008, 2009)\citeH{CD06lat}\citeH{CD08brw}\citeH{CD09simp}
and Deng et al.~(2011)\citeH{DQSZ11}.

Another type of (discrete) fixed point theorem,
the lattice-theoretical fixed point theorem of 
Tarski (1955)\citeH{Tar55fixpt},
is a powerful tool used extensively in economics and game theory; see 
Milgrom and Roberts (1990)\citeH{MR90eco},
Vives (1990)\citeH{Viv90},
and Topkis (1998)\citeH{Top98}.
For stable matchings, use and power of Tarski's fixed point theorem
are demonstrated by
Adachi (2000)\citeH{Ada00},
Fleiner (2003)\citeH{Flei03},
and Farooq et al.~(2012)\citeH{FFT12}.
It may be said, however, that 
Tarski's fixed point theorem is rather independent of discrete convex analysis.

Yet another type of discrete fixed point theorems are considered in the
literature, including
Robert (1986)\citeH{Rob86},
Shih and Dong (2005)\citeH{SD05},
Richard (2008)\citeH{Ric08},
Sato and Kawasaki (2009)\citeH{SK09fixpt}
and
Kawasaki et al.~(2013)\citeH{KKK13}.



\section{Other Topics}
\label{SCothertopic}

\subsection{Matching market and economy with indivisible goods}

Since the seminal paper by
Kelso and Crawford (1982)\citeH{KC82},
the concept of gross substitutes with its variants
has turned out to be pivotal
in discussing matching market and economy with indivisible goods.
The literature includes, e.g.,
Roth and Sotomayor (1990)\citeH{RS90},
Bikhchandani and Mamer (1997)\citeH{BM97eco},
Gul and Stacchetti (1999)\citeH{GS99},
Ausubel and Milgrom (2002)\citeH{AM02},
Milgrom (2004)\citeH{Mil04book},
Hatfield and Milgrom (2005)\citeH{HM05},
Ausubel (2006)\citeH{Aus06},
Sun and Yang (2006)\citeH{SY06},
Milgrom and Strulovici (2009)\citeH{MS09sbst},
and Hatfield et al. (2016)\citeH{HKNOW16fullsubst}.

Application of discrete convex analysis to 
economics was started by
Danilov et al.~(1998, 2001)\citeH{DKM98}\citeH{DKM01} 
for the  Walrasian equilibrium of indivisible markets
(see also Chapter 11 of Murota 2003)\citeH[Chapter 11]{Mdcasiam}.
The interaction between economics and discrete convex analysis
was reinforced decisively by the observation of 
Fujishige and Yang (2003)\citeH{FY03gr}
that M$\sp{\natural}$-concavity (of set functions) is equivalent
to the gross substitutes property 
(Theorem \ref{THmconcavgross} in Section \ref{SCmaximizers01}).
This equivalence is extended to functions in integer variables
(Section \ref{SCmaximizersZ}).
While the reader is referred to 
Tamura (2004)\citeH{Tam04}
and Chapter 11 of Murota (2003)\citeH[Chapter 11]{Mdcasiam}
for this earlier development, 
we mention more recent papers below.

As described in Section \ref{SCmarrigeassigngame},
the Fujishige-Tamura model of two-sided matching markets,
proposed by Fujishige and Tamura (2006, 2007),
is a common generalization of the stable marriage model
 (Gale and Shapley 1962)\citeH{GS62}
and the assignment game
 (Shapley and Shubik 1972)\citeH{SS72}.

Inoue (2008)\citeH{Ino08} 
uses the property of M$\sp{\natural}$-convex sets 
that they are closed under (Minkowski) summation, 
to show that the weak core in a finite exchange economy
is nonempty if every agent's upper contour set is M$\sp{\natural}$-convex.
\RED{
Kojima et al.~(2018)\citeH{KTY14} 
}%
present a unified treatment of 
two-sided matching markets with a variety of distributional
constraints that can be represented 
by M$\sp{\natural}$-concave functions.
It is shown that the generalized
deferred acceptance algorithm is strategy-proof and yields a stable matching.
Yokote (2016)\citeH{Yokote16} 
considers a market in which 
each buyer demands at most one unit of commodity
and each seller produces multiple units of several types of commodities. 
The core and the competitive equilibria are shown to exist and coincide
under the assumption that 
the cost function of each seller is M$\sp{\natural}$-convex.

Algorithmic aspects of Walrasian equilibria
 are fully investigated by
\RED{
Paes Leme and Wong (2020)\citeH{PLW16}
}%
in a general setting,
in which the algorithms from discrete convex analysis
are singled out as efficient methods 
for the gross substitutes case.
See also 
\RED{
Paes Leme (2017)\citeH{PLem14gs}
}%
as well as 
Murota and Tamura (2003b)\citeH{MTcompeq03} and
Section 11.5 of Murota (2003)\citeH[Chapter 11]{Mdcasiam}.

\subsection{Trading networks}
\label{SCtradingnet}

M$\sp{\natural}$-concavity plays a substantial role 
in the modeling and analysis of 
vertical 
trading networks (supply chain networks)
introduced  by
Ostrovsky (2008)\citeH{Ost08}
and investigated 
in a more general setting
by Hatfield et al.~(2013)\citeH{HKNOW13tradnet},
Fleiner (2014)\citeH{Flei14},
Fleiner et al.~(2015)\citeH{FJTT15hj},
Ikebe et al.~(2015)\citeH{ISST15},
Ikebe and Tamura (2015)\citeH{IT15scnet},
and Candogan et al.~(2016)\citeH{CEV16}.

In a trading network,
an agent is identified with a vertex (node) of the network.
In-coming arcs to a vertex represent the trades
in which the agent acts as a buyer and 
out-going arcs represent the trades
in which the agent acts as a seller.
Each vertex $v$ of the network is associated
with a choice function $C_{v}$ and/or a valuation function $f_{v}$ of the agent,
defined on the set $U_{v} \cup W_{v}$ of the arcs incident to $v$,
where $U_{v}$ is the set of in-coming arcs to the vertex $v$
and $W_{v}$ is the set of out-going arcs from $v$.
In particular, the function $f_{v}$ 
is a set function on $U_{v} \cup W_{v}$ in the single-unit case,
whereas it is a function on $\ZZ\sp{U_{v} \cup W_{v}}$
in the multi-unit case.

In the single-unit case,
Ostrovsky (2008)\citeH{Ost08}
identifies the key property of a choice function, called 
the same-side substitutability (SSS) and the cross-side complementarity (CSC),
which are discussed in Section~\ref{SCtwistMnat01}.
These properties are 
satisfied by the choice function induced from a unique-selecting 
twisted M$\sp{\natural}$-concave valuation function $f_{v}$,
with twisting by $W_{v}$;
see Theorem \ref{THchotwistMnatuniq01}.
The multi-unit case is treated by Ikebe and Tamura (2015)\citeH{IT15scnet}.
The conditions (SSS) and (CSC) are generalized to
(SSS-CSC$\sp{1}$\hbox{[$\ZZ$]}) and (SSS-CSC$\sp{2}$\hbox{[$\ZZ$]}),
and these conditions are shown to be satisfied 
by the choice function induced from 
a unique-selecting twisted M$\sp{\natural}$-concave valuation $f_{v}$;
see Theorem \ref{THchotwistMnatZ} in Section \ref{SCtwistMnatZ}.

Discrete convex analysis is especially relevant and useful 
when valuation functions and 
the price vector $p$ are explicitly involved in the model
as in Hatfield et al.~(2013), Ikebe et al.~(2015), 
and Candogan et al.~(2016).
\citeH{HKNOW13tradnet}\citeH{ISST15}\citeH{CEV16}%
Specifically, we can use the results from discrete convex analysis
as follows:

\begin{itemize}

\item
The existence of a competitive equilibrium 
(Hatfield et al.~2013, Definition 3)\citeH{HKNOW13tradnet} 
can be proved with the aid of 
the M$\sp{\natural}$-concave intersection theorem
(Theorem \ref{THmfcaveninteropt}).

\item
The lattice structure of the equilibrium price vectors can be shown 
through the conjugacy relationship between 
M$\sp{\natural}$-concavity and L$\sp{\natural}$-convexity
(Section \ref{SCconjugacy}).

\item
The equivalence of chain stability and stability
can be established with the aid of the negative-cycle criterion 
for the M-convex submodular flow problem 
(Theorem~\ref{THsbmfcyccritZ}).
Recall from Remark \ref{RMintersubmfl} that
the M$\sp{\natural}$-concave intersection problem 
can be formulated as an M-convex submodular flow problem.

\item
Fundamental computational problems for a trading network,
such as checking stability,  computing a competitive equilibrium,
and maximizing the welfare, 
can often 
be solved with the aid of algorithms known in discrete convex analysis, 
such as those for maximizing M$\sp{\natural}$-concave functions 
and  for solving the M-convex submodular flow problem.
See Candogan et al.~(2016)\citeH{CEV16}
as well as
Murota and Tamura (2003b)\citeH{MTcompeq03},
Section 11.5 of Murota (2003)\citeH[Chapter 11]{Mdcasiam},
and
Ikebe et al.~(2015)\citeH{ISST15}.
\end{itemize}

\subsection{Congestion games}

Congestion games (Rosenthal 1973)\citeH{Ros73},
which are equivalent to
(exact) ``finite'' potential games 
(Monderer and Shapley 1996)\citeH{MS96potgame},
are a class of games possessing a Nash equilibrium in pure strategies.
There are various generalizations of potential games, such as:
ordinal and generalized ordinal 
(Monderer and Shapley 1996)\citeH{MS96potgame}
and best-response 
(Voorneveld 2000)\citeH{Voo00}
potential games.
For algorithmic aspects of congestion games, we refer to
Roughgarden (2007)\citeH{Rou07}
and 
Tardos and Wexler (2007)\citeH{TW07netgame}.

Recently, a connection is made by
Fujishige et al.~(2015)\citeH{FGHPZ15conges}
between congestion games on networks
and discrete convex analysis.
It has been known (Fotakis 2010)\citeH{Fot10} 
that for every congestion game 
on an extension-parallel network, 
considered by 
Holzman and Law-yone (2003)\citeH{HL03},
any best-response sequence reaches a pure Nash equilibrium of the game
in $n$ steps, where $n$ is the number of players.
It is pointed out by
Fujishige et al.~(2015)\citeH{FGHPZ15conges}
that the fast convergence of best-response sequences 
is a consequence of M$\sp{\natural}$-convexity of the associated potential function,
which is a laminar convex function and hence is
M$\sp{\natural}$-convex; 
see (\ref{laminarconvZ}) in Section \ref{SCmnatexampleZ}.

In economics, potential games on some subset of
a Euclidean space are more widely studied.
A maximizer of (some sort of)  potential function
is a Nash equilibrium.
We also have the converse if the potential function is ``concave,''
since local optimality implies the global optimality there.
Ui (2006, 2008) studies the condition
for a local maximizer of a function on the integer lattice
to become a global maximizer of the function as well,
with application to best-response potential games
on the integer lattice.
In Ui (2008), it is shown that
a condition analogous to midpoint concavity,
called ``larger midpoint property,''  is sufficient
for the equivalence of local optimality and global optimality,
and shows the equivalence of
a Nash equilibrium and a maximizer of the best-response potential function.
A more general condition for
the equivalence of local and global optimality
is studied in Ui (2006),
along with its relation to M-, L-, L$^\natural$-, and M$^\natural$-convex functions.

\subsection{Integrally concave games}

Another study on the games on the integer lattice $\ZZ\sp{n}$
is found in Iimura and Watanabe (2014)\citeH{IW14},
which deals with
$n$-person symmetric games with integrally concave payoff functions
defined on the $n$-product of a finite integer interval.
Here, the integral concavity is in the sense of 
Favati and Tardella (1990)\citeH{FT90};
see also Section 3.4 of Murota (2003)\citeH[Section 3.4]{Mdcasiam}.
It is shown that every game in this class of games
has a (not necessarily symmetric) Nash equilibrium,
which is located within a unit distance from
the diagonal of strategy space.
Although assuming concavity on the entire strategy space is somewhat stringent,
this result generalizes the result of 
Cheng et al.~(2004)\citeH{CRVW04}
that every $n$-person symmetric ``two-strategy'' game has
a (not necessarily symmetric) Nash equilibrium,
because any real-valued function on the $n$-product of a doubleton
is integrally concave.
A further generalization has been made by Iimura and Watanabe (2016)\citeH{IW}, 
which implies the existence of an equilibrium in discrete Cournot game 
with concave industry revenue, convex cost, and nonincreasing inverse demand.

\subsection{Unimodularity and tropical geometry}
\label{SCunimodular}

Unimodular coordinate transformations are a natural operation
for discrete convexity.
In Section~\ref{SCconremmnatZ} we have mentioned 
that a function $f$ is twisted M$\sp{\natural}$-concave 
if and only if it is represented as $f(x) = g(U x)$
with $U = {\rm diag}(1,\ldots,1, -1,\ldots,-1)$
for some M$\sp{\natural}$-concave function $g$.
Another such example is multimodular functions 
of Hajek (1985)\citeH{Haj85}
used in discrete-event control 
(Altman et al.~2000)\citeH{AGH00}.
A function $f: \ZZ\sp{n} \to \RR \cup \{ +\infty \}$
is said to be {\em multimodular}
if the function  $\tilde f: \ZZ\sp{n+1} \to \RR \cup \{ +\infty \}$
defined by 
$ \tilde f(x_{0}, x) = f(x_{1}-x_{0},  x_{2}-x_{1}, \ldots, x_{n}-x_{n-1})$  
for $x_{0} \in \ZZ$ and $x \in \ZZ\sp{n}$
is submodular in $n+1$ variables. 
This means that $f$ is multimodular if and only if
the function $g(x)=f(Dx)$ is L$\sp{\natural}$-convex,
where
$D=(d_{ij} \mid 1 \leq i,j \leq n)$ 
is a bidiagonal matrix defined by
$d_{ii}=1 \ (i=1,\ldots,n)$ and
$d_{i+1,i}=-1 \ (i=1,\ldots,n-1)$.
This matrix $D$ is unimodular, and its inverse
$D\sp{-1}$ is an integral matrix
with $(D\sp{-1})_{ij}=1$ for $i \geq j$ and 
$(D\sp{-1})_{ij}=0$ for $i < j$.
Therefore,
a function $f$ is multimodular
if and only if it is represented as $f(x) = g(U x)$
with $U = D\sp{-1}$
for some L$\sp{\natural}$-convex function $g$.

The fundamental role of unimodularity for discrete convexity,
beyond unimodular coordinate transformations,
is investigated in
Danilov and Koshevoy (2004)\citeH{DK04uni}
under the name of ``unimodular systems.''
An application of unimodular systems 
to competitive equilibrium
is found in  Danilov et al.~(2001)\citeH{DKM01}.

Another recent topic, of a similar flavor, is tropical geometry.
\RED{
Baldwin and Klemperer (2019)\citeH{BK16}
}%
investigate indivisibility issues 
in terms of tropical geometry.
The Ricardian theory of international trade
is treated by Shiozawa (2015)\citeH{Shiz15},
mechanism design by 
Crowell and Tran (2016)\citeH{CT16},
and dominant strategy implementation by 
Weymark (2016)\citeH{Wey16}.
The interaction of tropical geometry with economics has just begun%
\footnote{
A summer school entitled
``Economics and Tropical Geometry''
was organized by Ngoc Tran and Josephine Yu
at Hausdorff Center for Mathematics, Bonn, May 2016.
}. 


\subsection*{Acknowledgements}

The author would like to thank Zaifu Yang for 
offering the opportunity of this survey paper. 
Special thanks go to Akiyoshi Shioura and Akihisa Tamura
for carefully reading all the manuscript and making constructive comments.
The author is also indebted to 
Satoru Fujishige,
Takuya Iimura,
Satoko Moriguchi,
and Yu Yokoi
for helpful suggestions.
This work was supported by The Mitsubishi Foundation, CREST, JST, 
and JSPS KAKENHI Grant Number 26280004.


\addtocontents{toc}{\protect\contentsline {section}{\numberline {}References}{\thepage}}


\newpage
\tableofcontents

\newpage
\listoffigures

\end{document}